\documentclass{article}
\usepackage{amssymb}
\usepackage{amsmath}

\newtheorem{theorem}{Theorem}

\newtheorem{definition}[theorem]{Definition}

\newtheorem{lemma}[theorem]{Lemma}

\newtheorem{proposition}[theorem]{Proposition}
\newtheorem{remark}[theorem]{Remark}

\newenvironment{proof}[1][Proof]{\noindent\textbf{#1.} }{\ \rule{0.5em}{0.5em}}
\input{tcilatex}

\begin{document}

\title{Decay and Continuity of Boltzmann Equation in Bounded Domains}
\author{Yan Guo \\
Division of Applied Mathematics, Brown University}
\maketitle

\begin{abstract}
Boundaries occur naturally in kinetic equations and boundary effects are
crucial for dynamics of dilute gases governed by the Boltzmann equation. We
develop a mathematical theory to study the time decay and continuity of
Boltzmann solutions for four basic types of boundary conditions: inflow,
bounce-back reflection, specular reflection, and diffuse reflection. We
establish exponential decay in $L^{\infty }$ norm for hard potentials for
general classes of smooth domains near an absolute Maxwellian. Moreover, in
convex domains, we also establish continuity for these Boltzmann solutions
away from the grazing set of the velocity at the boundary. Our contribution
is based on a new $L^{2}$ decay theory and its interplay with delicate $%
L^{\infty }$ decay analysis for the linearized Boltzmann equation, in the
presence of many repeated interactions with the boundary.
\end{abstract}

\section{Introduction}

Boundary effect plays a crucial role in the dynamics of gases governed by
the Boltzmann equation: 
\begin{equation}
\partial _{t}F+v\cdot \nabla _{x}F=Q(F,F)_{\text{ }}  \label{boltzmann}
\end{equation}%
where $F(t,x,v)$ is the distribution function for the gas particles at time $%
t\geq 0,$ position $x\in \Omega ,~~$and $v\in \mathbf{R}^{3}.$ Throughout
this paper, the collision operator takes the form%
\begin{eqnarray}
Q(F_{1},F_{2}) &=&\int_{\mathbf{R}^{3}}\int_{\mathbf{S}^{2}}|v-u|^{\gamma
}F_{1}(u^{\prime })F_{2}(v^{\prime })q_{0}(\theta )d\omega du  \notag \\
&&-\int_{\mathbf{R}^{3}}\int_{\mathbf{S}^{2}}|v-u|^{\gamma
}F_{1}(u)F_{2}(v)q_{0}(\theta )d\omega du  \notag \\
&\equiv &Q_{\text{gain}}(F_{1},F_{2})-Q_{\text{loss}}(F_{1},F_{2}),
\label{qgl}
\end{eqnarray}%
where $u^{\prime }=u+(v-u)\cdot \omega ,$ $v^{\prime }=v-(v-u)\cdot \omega ,$
$\cos \theta =(u-v)\cdot \omega /|u-v|,$ $0\leq \gamma \leq 1$ (hard
potential) and $0\leq q_{0}(\theta )\leq C|\cos \theta |$ (angular cutoff)$.$
The mathematical study of the particle-boundary interaction in a bounded
domain and its effect on the global dynamics is one of the fundamental
problems in the Boltzmann theory. There are four basic types of boundary
conditions for $F(t,x,v)$ at the boundary $\partial \Omega :$ (1) In-flow
injection: in which the incoming particles are prescribed; (2) bounce-back
reflection: in which the particles bounce back at the reverse the velocity;
(3) specular reflection: in which the particle bounce back specularly; (4)
diffuse reflection (stochastic): in which the incoming particles are a
probability average of the outgoing particles. Due to its importance, there
have been many contributions in the mathematical study of different aspects
of the Boltzmann boundary value problems [A], [AC], [AEMN], [AEP], [AH],
[C2], [C3], [CC], [De], [Gui], [H], [LY], [MS], [M], [US], [YZ], among
others, see also references in the books [C1], [CIP] and [U1].

According to Grad (p243, [Gr1]), one of the basic problems in the Boltzmann
study is to prove existence and uniqueness of its solutions, as well as
their time-decay toward an absolute Maxwellian, in the presence\ of
compatible physical boundary conditions in a general domain. In spite of
those contributions to the study of Boltzmann boundary problems, there are
few mathematical results of uniqueness, regularity, and time decay-rate for
Boltzmann solutions toward an Maxwellian. In [SA], it was announced that
Boltzmann solutions near a Maxwellian would decay exponentially to it in a
smooth bounded convex domain with specular reflection boundary conditions.
Unfortunately, we are not aware of any complete proof for such a result over
the last thirty years [U2]. Recently, important progress has been made in
[DeV] and [V] to establish almost exponential decay rate for Boltzmann
solutions with large amplitude for general collision kernels and general
boundary conditions, provided certain a-priori strong Sobolev estimates can
be verified. Even though these estimates had been established for spatially
periodic domains [G3-4] near Maxwellians, their validity is completely open
for the Boltzmann solutions, even local in time, in a bounded domain. As a
matter of fact, such kind of strong Sobolev estimates may not be expected
for a general non-convex domain [G4]. This is because even for simplest
kinetic equations with the differential operator $v\cdot \nabla _{x},$ the
phase boundary $\partial \Omega \times \mathbf{R}^{3}$ is always
characteristic but not uniformly characteristic at the grazing set $\gamma
_{0}=\{(x,v):x\in \partial \Omega ,~$and $v\cdot n(x)=0\}$ where $n(x)$
being the outward normal at $x$. Hence it is very challenging and delicate
to obtain regularity from general \ theory of hyperbolic PDE. Moreover, in
comparison with the half-space problems studied, for instance in [LY], [YZ],
the complication of the geometry makes it difficult to employ spatial
Fourier transforms in $x.$ There are many cycles (bouncing characteristics)
interacting with the boundary repeatedly, and analysis of such cycles is one
of the key mathematical difficulties.

The purpose of this article is to develop an unified $L^{2}-L^{\infty }$
theory in the near Maxwellian regime, to establish exponential decay toward
a normalized Maxwellian $\mu =e^{-\frac{|v|^{2}}{2}},$ for all four basic
types of the boundary conditions and rather general domains. Consequently,
uniqueness among these solutions can be obtained. For convex domains, these
solutions are shown to be continuous away from the singular grazing set $%
\gamma _{0}$.

\subsection{\protect\bigskip Domain and Characteristics}

Throughout this paper, we define $\Omega =\{x:\xi (x)<0\}\ $is connected and
bounded with $\xi (x)$ being a smooth function. We assume $\nabla \xi
(x)\neq 0$ at the boundary $\xi (x)=0$. The outward normal vector at $%
\partial \Omega $ is given by 
\begin{equation}
n(x)=\frac{\nabla \xi (x)}{|\nabla \xi (x)|},  \label{outwardnormal}
\end{equation}%
and it can be extended smoothly near $\partial \Omega =\{x:\xi (x)=0\}.$ We
say $\Omega $ is real analytic if $\xi $ is real analytic in $x.$ We define $%
\Omega $ is strictly convex if there exists $c_{\xi }>0$ such that 
\begin{equation}
\partial _{ij}\xi (x)\zeta ^{i}\zeta ^{j}\geq c_{\xi }|\zeta |^{2}
\label{convexity}
\end{equation}%
for all $x$ such that $\xi (x)\leq 0,$ and all $\zeta \in \mathbf{R}^{3}.$
We say that $\Omega $ has a rotational symmetry, if there are vectors $x_{0}$
and $\varpi ,$ such that for all $x\in \partial \Omega $ 
\begin{equation}
\{(x-x_{0})\times \varpi \}\cdot n(x)\equiv 0.  \label{axis}
\end{equation}

We denote the phase boundary in the space $\Omega \times \mathbf{R}^{3}$ as $%
\gamma =\partial \Omega \times \mathbf{R}^{3},$ and split it into outgoing
boundary $\gamma _{+},$ the incoming boundary $\gamma _{-},$ and the
singular boundary $\gamma _{0}$ for grazing velocities$:$ 
\begin{eqnarray*}
\gamma _{+} &=&\{(x,v)\in \partial \Omega \times \mathbf{R}^{3}:\text{ \ }%
n(x)\cdot v>0\}, \\
\gamma _{-} &=&\{(x,v)\in \partial \Omega \times \mathbf{R}^{3}:\text{ \ }%
n(x)\cdot v<0\}, \\
\gamma _{0} &=&\{(x,v)\in \partial \Omega \times \mathbf{R}^{3}:\text{ \ }%
n(x)\cdot v=0\}.
\end{eqnarray*}

Given $(t,x,v),$ let $[X(s),V(s)]=[X(s;t,x,v),V(s;t,x,v)]=[x+(s-t)v,v]$ be
the trajectory (or the characteristics) for the Boltzmann equation (\ref%
{boltzmann}): 
\begin{equation}
\frac{dX(s)}{ds}=V(s),\text{ \ \ \ \ \ \ \ \ \ \ \ \ \ \ \ \ \ \ }\frac{dV(s)%
}{ds}=0.  \label{ode}
\end{equation}%
with the initial condition: $[X(t;t,x,v),V(t;t,x,v)]=[x,v].$ For any $(x,v)$
such that $x\in \bar{\Omega},v\neq 0,$ we define its \textbf{backward exit
time} $t_{\mathbf{b}}(x,v)\geq 0$ to be the the last moment at which the
back-time straight line $[X(s;0,x,v),V(s;0,x,v)]$ remains in $\bar{\Omega}:$ 
\begin{equation}
t_{\mathbf{b}}(x,v)=\sup \{\tau \geq 0:x-\tau v\in \bar{\Omega}\}.
\label{exit}
\end{equation}%
We therefore have $x-t_{\mathbf{b}}v\in \partial \Omega $ and $\xi (x-t_{%
\mathbf{b}}v)=0.$ We also define 
\begin{equation}
x_{\mathbf{b}}(x,v)=x(t_{\mathbf{b}})=x-t_{\mathbf{b}}v\in \partial \Omega .
\label{xb}
\end{equation}%
We always have $v\cdot n(x_{\mathbf{b}})\leq 0.$

\subsection{Boundary Condition and Conservation Laws}

In terms of the standard perturbation $f$ such that $F=\mu +\sqrt{\mu }f,$
the Boltzmann equation can be rewritten as 
\begin{equation*}
\left\{ \partial _{t}+v\cdot \nabla +L\right\} f=\Gamma (f,f),\text{ \ \ \ \ 
}f(0,x,v)=f_{0}(x,v),
\end{equation*}%
where the standard linear Boltzmann operator see [G] is given by 
\begin{equation}
Lf\equiv \nu f-Kf=-\frac{1}{\sqrt{\mu }}\{Q(\mu ,\sqrt{\mu }f)+Q(\sqrt{\mu }%
f,\mu )\}  \label{mu-k}
\end{equation}%
with the collision frequency $\nu (v)\equiv \int |v-u|^{\gamma }\mu
(u)q_{0}(\theta )dud\theta \backsim \{1+|v|\}^{\gamma }$ for $0\leq \gamma
\leq 1;$ and 
\begin{equation}
\Gamma (f_{1},f_{2})=\frac{1}{\sqrt{\mu }}Q(\sqrt{\mu }f_{1},\sqrt{\mu }%
f_{2})\equiv \Gamma _{\text{gain}}(f_{1},f_{2})-\Gamma _{\text{loss}%
}(f_{1},f_{2}).  \label{gamma}
\end{equation}%
In terms of $f,$ we formulate the boundary conditions as

(1) The in-flow boundary condition: for $(x,v)\in \gamma _{-}$ 
\begin{equation}
f|_{\gamma _{-}}=g(t,x,v)  \label{inflow}
\end{equation}

(2) The bounce-back boundary condition: for $x\in \partial \Omega ,$

\begin{equation}
f(t,x,v)|_{\gamma _{-}}=f(t,x,-v)  \label{bounceback}
\end{equation}

(3) Specular reflection: for $x\in \partial \Omega ,$ let 
\begin{equation}
R(x)v=v-2(n(x)\cdot v)n(x),  \label{r}
\end{equation}%
and 
\begin{equation}
f(t,x,v)|_{\gamma _{-}}=f(x,v,v-2(n(x)\cdot v)n(x))=f(x,v,R(x)v)
\label{specular}
\end{equation}

(4) Diffusive reflection: assume the natural normalization, 
\begin{equation}
c_{\mu }\int_{v\cdot n(x)>0}\mu (v)|n(x)\cdot v|dv=1,  \label{diffusenormal}
\end{equation}%
then for $(x,v)\in \gamma _{-},$ 
\begin{equation}
f(t,x,v)|_{\gamma _{-}}=c_{\mu }\sqrt{\mu (v)}\int_{v^{\prime }\cdot
n(x)>0}f(t,x,v^{\prime })\sqrt{\mu (v^{\prime })}\{n_{x}\cdot v^{\prime
}\}dv^{\prime }.  \label{diffuse}
\end{equation}

For both the bounce-back and specular reflection conditions (\ref{bounceback}%
) and (\ref{specular}), it is well-known that both mass and energy are
conserved for (\ref{boltzmann}). Without loss of generality, we may always
assume that the mass-energy conservation laws hold for $t\geq 0,$ in terms
of the perturbation $f$: 
\begin{eqnarray}
\int_{\Omega \times \mathbf{R}^{3}}f(t,x,v)\sqrt{\mu }dxdv &=&0,
\label{mass} \\
\int_{\Omega \times \mathbf{R}^{3}}|v|^{2}f(t,x,v)\sqrt{\mu }dxdv &=&0.
\label{energy}
\end{eqnarray}%
Moreover, if the domain $\Omega $ has \textit{any} axis of rotation symmetry
(\ref{axis}), then we further assume the corresponding conservation of
angular momentum is valid for all $t\geq 0:$%
\begin{equation}
\int_{\Omega \times \mathbf{R}^{3}}\{(x-x_{0})\times \varpi \}\cdot vf(t,x,v)%
\sqrt{\mu }dxdv=0.  \label{axiscon}
\end{equation}

For the diffuse reflection (\ref{diffuse}), the mass conservation (\ref{mass}%
) is assumed to be valid.

\subsection{Main Results}

We introduce the weight function 
\begin{equation}
w(v)=(1+\rho |v|^{2})^{\beta }e^{\theta |v|^{2}}.  \label{weight}
\end{equation}%
where $0\leq \theta <\frac{1}{4},\rho >0$ and $\beta \in \mathbf{R}^{1}.$

\begin{theorem}
\label{inflownl}Assume $w^{-2}\{1+|v|\}^{3}\in L^{1}$ in (\ref{weight})$.$
There exists $\delta >0$ such that if $F_{0}=\mu +\sqrt{\mu }f_{0}\geq 0,$
and 
\begin{equation*}
||wf_{0}||_{\infty }+\sup_{0\leq t\leq \infty }e^{\lambda
_{0}t}||wg(t)||_{\infty }\leq \delta ,
\end{equation*}%
with $\lambda _{0}>0,$ then there there exists a unique solution $%
F(t,x,v)=\mu +\sqrt{\mu }f\geq 0$ to the inflow boundary value problem (\ref%
{inflow}) for the Boltzmann equation (\ref{boltzmann}). There exists $%
0<\lambda <\lambda _{0}$ such that 
\begin{equation*}
\sup_{0\leq t\leq \infty }e^{\lambda t}||wf(t)||_{\infty }\leq
C\{||wf_{0}||_{\infty }+\sup_{0\leq t\leq \infty }e^{\lambda
_{0}t}||wg(s)||_{\infty }\}.
\end{equation*}%
Moreover, if $\Omega $ is strictly convex (\ref{convexity}), and if $%
f_{0}(x,v)$ is continuous except on $\gamma _{0},$ and $g(t,x,v)$ is
continuous in $[0,\infty )\times \{\partial \Omega \times \mathbf{R}%
^{3}\setminus \gamma _{0}\}$ satisfying 
\begin{equation*}
f_{0}(x,v)=g(x,v)\text{ }\ \text{on }\gamma _{-},
\end{equation*}%
then $f(t,x,v)$ is continuous in $[0,\infty )\times \{\bar{\Omega}\times 
\mathbf{R}^{3}\setminus \gamma _{0}\}.$
\end{theorem}

\begin{theorem}
\label{bouncebacknl}Assume $w^{-2}\{1+|v|\}^{3}\in L^{1}$ in (\ref{weight}).
Assume the conservation of mass (\ref{mass}) and energy (\ref{energy}) are
valid for $f_{0}$. Then there exists $\delta >0$ such that if $%
F_{0}(x,v)=\mu +\sqrt{\mu }f_{0}(x,v)\geq 0$ and $||wf_{0}||_{\infty }\leq
\delta ,$ there exists a unique solution $F(t,x,v)=\mu +\sqrt{\mu }%
f(t,x,v)\geq 0$ to the bounce-back boundary value problem (\ref{bounceback})
for the Boltzmann equation (\ref{boltzmann}) such that 
\begin{equation*}
\sup_{0\leq t\leq \infty }e^{\lambda t}||wf(t)||_{\infty }\leq
C||wf_{0}||_{\infty }.
\end{equation*}%
Moreover, if $\Omega $ is strictly convex (\ref{convexity}), and if
initially $f_{0}(x,v)$ is continuous except on $\gamma _{0}$ and 
\begin{equation*}
f_{0}(x,v)=f_{0}(x,-v)\text{ on }\partial \Omega \times \mathbf{R}%
^{3}\setminus \gamma _{0},
\end{equation*}%
then $f(t,x,v)$ is continuous in $[0,\infty )\times \{\bar{\Omega}\times 
\mathbf{R}^{3}\setminus \gamma _{0}\}.$
\end{theorem}

\begin{theorem}
\label{specularnl}Assume $w^{-2}\{1+|v|\}^{3}\in L^{1}$ in (\ref{weight}).
Assume that $\xi $ is both strictly convex (\ref{convexity}) and analytic,
and the mass (\ref{mass}) and energy (\ref{energy}) are conserved for $f_{0}$%
. In the case of $\Omega $ has any rotational symmetry (\ref{axis}), we
require that the corresponding angular momentum (\ref{axiscon}) is conserved
for $f_{0}$. Then there exists $\delta >0$ such that if $F_{0}(x,v)=\mu +%
\sqrt{\mu }f_{0}(x,v)\geq 0$ and $||wf_{0}||_{\infty }\leq \delta ,$ there
exists a unique solution $F(t,x,v)=\mu +\sqrt{\mu }f(t,x,v)\geq 0$ to the
specular boundary value problem (\ref{specular}) for the Boltzmann equation (%
\ref{boltzmann}) such that 
\begin{equation*}
\sup_{0\leq t\leq \infty }e^{\lambda t}||wf(t)||_{\infty }\leq
C||wf_{0}||_{\infty }.
\end{equation*}%
Moreover, if $f_{0}(x,v)$ is continuous except on $\gamma _{0}$ and 
\begin{equation*}
f_{0}(x,v)=f_{0}(x,R(x)v)\text{ on }\partial \Omega
\end{equation*}%
then $f(t,x,v)$ is continuous in $[0,\infty )\times \{\bar{\Omega}\times 
\mathbf{R}^{3}\setminus \gamma _{0}\}.$
\end{theorem}

\begin{theorem}
\label{diffusenl}Assume (\ref{diffusenormal}). There is $\theta _{0}(\nu
_{0})>0$ such that 
\begin{equation}
\text{ \ \ \ \ }\theta _{0}(\nu _{0})<\theta <\frac{1}{4}\text{ },\ \text{\
\ \ and \ \ }\rho \text{ }\ \ \text{is}\ \text{sufficiently small}
\label{wdiffuse}
\end{equation}%
for weight function $w$ in (\ref{weight}). Assume the mass conservation (\ref%
{mass}) is valid for $f_{0}$. If $F_{0}(x,v)=\mu +\sqrt{\mu }f_{0}(x,v)\geq
0 $ and $||wf_{0}||_{\infty }\leq \delta $ sufficiently small$,$ then there
exists a unique solution $F(t,x,v)=\mu +\sqrt{\mu }f(t,x,v)\geq 0$ to the
diffuse boundary value problem (\ref{diffuse}) for the Boltzmann equation (%
\ref{boltzmann}) such that 
\begin{equation*}
\sup_{0\leq t\leq \infty }e^{\lambda t}||wf(t)||_{\infty }\leq
C||wf_{0}||_{\infty }.
\end{equation*}%
Moreover, if $\xi $ is strictly convex, and if $f_{0}(x,v)$ is continuous
except on $\gamma _{0}$ with 
\begin{equation*}
f_{0}(x,v)|_{\gamma _{-}}=c_{\mu }\sqrt{\mu }\int_{\{n_{x}\cdot v^{\prime
}>0\}}f_{0}(x,v^{\prime })\sqrt{\mu (v^{\prime })}\{n(x)\cdot v^{\prime
}\}dv^{\prime }\text{ }
\end{equation*}%
then $f(t,x,v)$ is continuous in $[0,\infty )\times \{\bar{\Omega}\times 
\mathbf{R}^{3}\setminus \gamma _{0}\}.$
\end{theorem}

\subsection{Velocity Lemma and Analyticity}

In section 2, we first establish some important analytical tools. The first
Velocity Lemma \ref{velocity} plays the most important role in the study of
continuity and the cycles (bouncing generalized trajectories) in the
specular case. It implies that in a strictly convex domain (\ref{convexity}%
), the singular set $\gamma _{0}$ can not be reached via the trajectory $%
\frac{dx}{dt}=v,$ $\frac{dv}{dt}=0$ from interior points inside $\Omega ,$
and hence $\gamma _{0}$ does not really participate or interfere with the
interior dynamics. By Lemma \ref{huang}, no singularity would be created
from $\gamma _{0}$ and it is possible to perform calculus for the back-time
exit time $t_{\mathbf{b}}(x,v).$ This is the foundation for future
regularity study. Moreover, the Velocity Lemma \ref{velocity} also provides
the lower bound away from the singular set $\gamma _{0},$ which leads to the
estimates for repeating bounces in the specular reflection cases. Such a
Velocity Lemma \ref{velocity} was first discovered in [G3-4], in the study
of regularity of Vlasov-Poisson (Maxwell) system with flat geometry. It then
was generalized in [HH] for Vlasov-Poisson system in a ball, and it is the
starting point for the recent final resolution to the Vlasov-Poisson in a
general convex domain [HV] with specular boundary condition.

Lemma \ref{kernel} gives refined estimates for the operator $K_{w}$ with $w$ 
$\backsim $ $\mu ^{-1/2}.\,$Similar estimates were established in [SG].
Lemma \ref{analytic} states that the zero set of a analytic function is of
measure zero unless such a analytic function is identically zero. This
provides a very convenient tool to verify certain geometric conditions of
general domains for particularly specular reflections.

\subsection{$L^{2}$ Decay Theory}

Since no spatial Fourier transform is available, we first establish linear $%
L^{2}$ exponential decay estimates in Section 3 via a functional analytical
approach. It turns out that it suffices to establish the following
finite-time estimate (Proposition \ref{lower})%
\begin{equation}
\int_{0}^{1}||\mathbf{P}f(s)||_{\nu }^{2}ds\leq M\left\{ \int_{0}^{1}||\{%
\mathbf{I}-\mathbf{P\}}f(s)||_{\nu }^{2}+\text{boundary contributions}%
\right\}  \label{pi-p}
\end{equation}%
for any solution $f$ $\ $to the linear Boltzmann equation 
\begin{equation}
\partial _{t}f+v\cdot \nabla _{x}f+Lf=0,\text{ \ \ \ \ \ }f(0,x,v)=f_{0}(x,v)
\label{lboltzmann}
\end{equation}%
with all four boundary conditions (\ref{inflow}), (\ref{bounceback}), (\ref%
{specular}) and (\ref{diffuse}). Here for any fixed $(t,x),$ the standard
projection $\mathbf{P}$ onto the hydrodynamic part is given by 
\begin{eqnarray}
\mathbf{P}f &=&\{a(t,x)+b(t,x)\cdot v+c(t,x)|v|^{2}\}\sqrt{\mu (v)},
\label{hydro} \\
\mathbf{P}_{a}f &=&a(t,x)\sqrt{\mu (v)},\text{ \ \ }\mathbf{P}_{\mathbf{b}%
}f=b(t,x)v\sqrt{\mu (v)},\text{ \ \ }\mathbf{P}_{c}f=c(t,x)|v|^{2}\sqrt{\mu
(v)},  \notag
\end{eqnarray}%
and $||\cdot ||_{\nu }$ is the weighted $L^{2}$ norm with the collision
frequency $\nu (v).$

Similar types of estimates like (\ref{pi-p}), but with strong Sobolev norms,
have been established in recent years [G1] via so-called the macroscopic
equations for the coefficients $a,b$ and $c.~$\ The key of the analysis was
based on the ellipticity for $b$ which satisfies the Poisson's equation $%
\Delta b=\partial ^{2}\{\mathbf{I}-\mathbf{P\}}f,$ where $\partial ^{2}$ is
some second order differential operators. In the presence of the boundary
condition $b\cdot n(x)=0$ (bounce-back and specular) or $b\equiv 0$ (inflow
and diffuse) at $\partial \Omega $, such an ellipticity is very difficult to
employ for the weak $L^{2}$, instead of $H^{1},$ estimate for $b$ in (\ref%
{pi-p}). This is due to lack of regularity of $b$ in (\ref{pi-p}), even the
trace of $b$ is hard to define. Instead, we employ the hyperbolic
(transport) feature rather than elliptic feature of the problem to prove (%
\ref{pi-p}). $\ $By a method of contradiction, we can find $f_{k}$ such that
if (\ref{pi-p}) is not valid, then the normalized $Z_{k}(t,x,v)\equiv \frac{%
f_{k}(t,x,v)}{\sqrt{\int_{0}^{1}||\mathbf{P}f_{k}(s)||_{\nu }^{2}ds}}$
satisfies $\int_{0}^{1}||\mathbf{P}Z_{k}(s)||_{\nu }^{2}ds\equiv 1,$ and 
\begin{equation}
\int_{0}^{1}||(\mathbf{I}-\mathbf{P)}Z_{k}(s)||_{\nu }^{2}ds\leq \frac{1}{k}.
\label{i-p1/k}
\end{equation}%
Denote a weak limit of $Z_{k}$ to be $Z,$ we expect that $Z=\mathbf{P}Z=0,$
by each of the four boundary conditions. The key is to prove that $%
Z_{k}\rightarrow Z$ strongly to reach a contradiction. By the averaging
Lemma [DL2], we know that $Z_{k}(s)\rightarrow Z$ strongly in the interior
of $\Omega .$ As expected, the most delicate part is to exclude possible
concentration near the boundary $\partial \Omega .$ Since $Z_{k}$ is a
solution to the transport equation, it then follows (Proposition \ref{strong}%
) that near $\partial \Omega ,$ on set of the non-grazing velocity $v\cdot
n(x)\neq 0$ can be reached via a trajectory from the interior of $\Omega ,$
which implies that $Z_{k}$ can be controlled on such non-grazing set with no
concentration$.$ On the other hand, over the remaining almost grazing set $%
v\cdot n(x)\thicksim 0,$ thanks to the fact (\ref{i-p1/k})$,$ we know that 
\begin{equation*}
Z_{k}\thicksim \mathbf{P}Z_{k}=\{a_{k}(t,x)+b_{k}(t,x)\cdot
v+c_{k}(t,x)|v|^{2}\}\sqrt{\mu (v)}.
\end{equation*}%
We observe that such special form of velocity distribution $\mathbf{P}Z_{k}$
can \textit{not} have concentration on the almost grazing set $v\cdot
n(x)\thicksim 0$ (Lemma \ref{nomove}), and we therefore conclude (\ref{pi-p}%
). Clearly, the hyperbolic or the transport property is crucial to control
boundary behaviors via the interior compactness of $Z_{k}$.

\subsection{$L^{\infty }$ Decay Theory}

Section 4 is devoted to the study of linear $L^{\infty }$ decay for all four
different types of boundary conditions: in-flow, bounceback, specular and
diffuse (stochastic) reflection. \ In order to control the nonlinear term $%
\Gamma (f,f),$ we need to estimate the weighted $L^{\infty }$ of $wf.$ We
recall that $L=\nu -K,$ and study the $L^{\infty }$ (pointwise) decay of the
linear Boltzmann equation (\ref{lboltzmann}). We denote a weight function 
\begin{equation}
h(t,x,v)=w(v)f(t,x,v),  \label{h}
\end{equation}%
and study the equivalent linear Boltzmann equation: 
\begin{equation}
\{\partial _{t}+v\cdot \nabla _{x}+\nu -K_{w}\}h=0,\text{ \ \ \ \ \ }%
h(0,x,v)=h_{0}(x,v)\equiv wf_{0},  \label{lboltzmannh}
\end{equation}%
where 
\begin{equation}
K_{w}h=wK(\frac{h}{w}),  \label{kw}
\end{equation}%
together with various boundary conditions (\ref{inflow}), (\ref{bounceback}%
), (\ref{specular}), or (\ref{diffuse}). In bounce-back, specular, diffuse
reflection as well as the inflow case with $g\equiv 0$, we denote the
semigroup $U(t)h_{0}$ to be the solution to (\ref{lboltzmannh}), and the
semigroup $G(t)h_{0}$ to be the solution to the simpler transport equation
without collision $K_{w}:$%
\begin{equation}
\{\partial _{t}+v\cdot \nabla _{x}+\nu \}h=0,\text{ \ \ \ }%
h(0,x,v)=h_{0}(x,v)=wf_{0}\text{.\ \ }  \label{transport}
\end{equation}%
Notice that neither $G(t)$ nor $U(t)$ is a strongly continuous semigroup in $%
L^{\infty }$ [U1]$.$

We first obtain explicit representation of $G(t)$ in the presence of various
boundary conditions. We then can obtain the explicit exponential decay
estimate for $G(t).$ Moreover, we also establish the continuity for $G(t)$
with a forcing term $q$ if $\Omega $ is strictly convex (\ref{convexity})
based on the Velocity Lemma \ref{velocity}. To study the $L^{\infty }$ decay
for $U(t),$ we make use of the Duhamal Principle : 
\begin{equation}
U(t)=G(t)+\int_{0}^{t}G(t-s_{1})K_{w}U(s_{1})ds_{1}.  \label{duhamel}
\end{equation}%
Following the pioneering work of Vidav [Vi], we iterate (\ref{duhamel}) back
to get: 
\begin{equation}
U(t)=G(t)+\int_{0}^{t}G(t-s_{1})K_{w}G(s_{1})ds_{1}+\int_{0}^{t}%
\int_{0}^{s_{1}}G(t-s_{1})K_{w}G(s_{1}-s)K_{w}U(s)dsds_{1}.  \label{duhamel2}
\end{equation}

The idea is to estimate the last double integral in terms of the $L^{2}$
norm of $f=\frac{h}{w}$, which decays exponential by $L^{2}$ decay theory in
Section 2. The key difficulty lies in the presence of different boundary
conditions which could lead to complicated bouncing trajectories. Each of
the boundary condition presents different difficulties.

Section 4.1 is devoted to the study of inflow boundary condition (\ref%
{inflow}), in which the back-time trajectory is either from the initial
plane or from the boundary. Even though when $g\neq 0,$ the solution
operators for (\ref{transport}) and (\ref{lboltzmannh}) are not semigroups,
for any $(t,x,v),$ a similar representation as $%
G(t-s_{1})K_{w}G(s_{1}-s)K_{w}U(s)$ $\ $is still possible. With the compact
property of $K_{w}$ (Lemma \ref{kernel})$,$ we are led to the main
contribution in (\ref{duhamel2}) roughly of the form 
\begin{equation}
\int_{0}^{t}\int_{0}^{s_{1}}\int_{v^{\prime },v^{\prime \prime }\text{bounded%
}}|h(s,X(s;s_{1},X(s_{1};t,x,v),v^{\prime }),v^{\prime \prime })|dv^{\prime
}dv^{\prime \prime }dsds_{1}.  \label{k-k}
\end{equation}%
The $v^{\prime }$ integral is estimated by a change of variable introduced
in [Vi], (see also in [LY]) 
\begin{equation}
y\equiv X(s;s_{1},X(s_{1};t,x,v),v^{\prime
})=x-(t-s_{1})v-(s_{1}-s)v^{\prime }.  \label{ychange}
\end{equation}%
Since $\det (\frac{dy}{dv^{\prime }})\neq 0$ almost always true$,$ the $%
v^{\prime }$ and $v^{\prime \prime }-$integration in (\ref{k-k}) can be
bounded by ($h=wf$) 
\begin{equation*}
\int_{\Omega ,v^{\prime \prime }\text{ bounded}}|h(s,y,v^{\prime \prime
})|dydv^{\prime \prime }\leq C\left( \int_{\Omega ,v^{\prime \prime }\text{
bounded}}|f(s,y,v^{\prime \prime })|^{2}dydv^{\prime \prime }\right) ^{1/2}.
\end{equation*}

For bounce-back, specular or diffuse reflections, the characteristic
trajectories repeatedly interact with the boundary. Instead of $X(s;t,x,v),$
we should use the generalized characteristics, defined as cycles, $X_{%
\mathbf{cl}}(s;t,x,v)$ in (\ref{k-k}) as in Definitions \ref%
{bouncebackcycles}, \ref{specularcycles} and \ref{diffusecycles}. The key
question is wether or not the change of variable 
\begin{equation}
y\equiv X_{\mathbf{cl}}(s;s_{1},X_{\mathbf{cl}}(s_{1};t,x,v),v^{\prime })
\label{ycycle}
\end{equation}%
is valid, i.e., to determine if it is almost always true 
\begin{equation}
\det \left\{ \frac{dX_{\mathbf{cl}}(s;s_{1},X_{\mathbf{cl}%
}(s_{1};t,x,v),v^{\prime })}{dv^{\prime }}\right\} \neq 0.
\label{detnotzero}
\end{equation}

Section 4.2 is devoted to the study of the bounce-back reflection. The
bounce-back cycles $X_{\mathbf{cl}}(s;t,x,v)$ from a given point $(t,x,v)$
is relatively simple, just going back and forth between two fixed points $x_{%
\mathbf{b}}(x,v)$ and $x_{\mathbf{b}}(x_{\mathbf{b}}(x,v),-v).$ Now the
change of variable (\ref{ycycle}) and (\ref{detnotzero}) can be established
by the study of set $S_{x}(v)$ in Section 4.2.2.

Section 4.3 is devoted to the study of specular reflection. The specular
cycles $X_{\mathbf{cl}}(s;t,x,v)$ reflect repeatedly with the boundary in
general, and $\frac{dX_{\mathbf{cl}}(s;s_{1},X_{\mathbf{cl}%
}(s_{1};t,x,v),v^{\prime })}{dv^{\prime }}$ is very complicated to compute
and (\ref{detnotzero}) is extremely difficult to verify, even in a convex
domain. This is in part due to the fact that there is no apparent way to
analyze $\frac{dX_{\mathbf{cl}}(s;s_{1}X_{\mathbf{cl}}(s_{1};t,x,v),v^{%
\prime })}{dv^{\prime }}$ inductively with finite bounces. To overcome such
a difficulty, in Section 4.3.2, $\det \frac{dv_{k}}{dv_{1}}$ can be computed
asymptotically in a delicate iterative fashion for special cycles almost
tangential to the boundary, which undergo many small bounces near the
boundary. It then follows that $\det \{\frac{dX_{\mathbf{cl}}(s;s_{1},X_{%
\mathbf{cl}}(s_{1};t,x,v),v^{\prime })}{dv^{\prime }}\}\neq 0$ for these
special cycles. This crucial observation is then combined with analyticity
of $\xi $ and Lemma \ref{analytic} to conclude that the set of $\det \{\frac{%
dX_{\mathbf{cl}}(s;X_{\mathbf{cl}}(s_{1},x,v),v^{\prime })}{dv^{\prime }}%
\}=0 $ is arbitrarily small (Lemma \ref{specularlower}), and the change of
variable (\ref{ycycle}) is almost always valid. The analyticity plays an
important role in our proof.

Section 4.4 is devoted to the study of diffuse reflection. The diffuse
cycles $X_{\mathbf{cl}}(s;t,x,v)$ contain more and more independent
variables and (\ref{k-k}) involves their integrations. Similar change of
variable (\ref{ychange}) is expected with respect to one of such independent
variables. However, the main difficulty in this case is the $L^{\infty }$
control of $G(t)$ which satisfies (\ref{transport}). The most natural $%
L^{\infty }$ estimate for $G(t)$ is for the weight $w=\mu ^{-\frac{1}{2}},$
in which the diffuse boundary condition takes the form 
\begin{equation*}
h(t,x,v)=c_{\mu }\int_{v^{\prime }\cdot n(x)>0}h(t,x,v^{\prime })\mu
(v^{\prime })\{v^{\prime }\cdot n(x)\}dv^{\prime }
\end{equation*}%
with $c_{\mu }\int_{v^{\prime }\cdot n(x)>0}\mu (v^{\prime })\{v^{\prime
}\cdot n(x)\}dv^{\prime }=1.$ But such a natural (strong) weight $\mu ^{-%
\frac{1}{2}}$ exactly makes the whole linear Boltzmann theory break down
(Lemma \ref{kernel}), so we have to use a weaker weight, which is closer to $%
\mu ^{-\frac{1}{2}}.$ This new weight will introduce serious new difficulty
since no natural maximum principle is available now from the new boundary
condition (\ref{hdiffuse}). Moreover, for any $(t,x,v),$ since there are
always particles moving almost tangential to the boundary in the bounce-back
reflection, it is impossible to reach down the initial plane no matter how
many cycles the particles take. In other words, there is no explicit
expression for $G(t)$ in terms of initial data completely. To establish the $%
L^{\infty }$ estimate, we make crucial observation in Lemma \ref{small} that
the measure of those particles can not reach initial plane after $k-$bounces
is small when $k$ is large. We therefore can obtain an approximate
representation formula for $G(t)$ by the initial datum, with only finite
number of bounces.

Section 5 is devoted to the proofs of the main nonlinear decay and
continuity results of this paper.

Our contribution opens new lines of research about such interesting
questions as specular reflections in non-convex domains, decay for the soft
potentials, higher regularity in a convex domain, characterization of
propagation of singularity in a non-convex domain, as well as charged
particles\ interacting with fields. Moreover, our new $L^{2}-L^{\infty }$
theory will shed new lights in other investigations of Boltzmann equation,
in which regularity of the solutions is difficult to employ [EGM] [GS].

\section{Preliminary}

\begin{lemma}
\label{velocity}Let $\Omega $ be strictly convex defined in (\ref{convexity}%
). Define the functional along the trajectories $\frac{dX(s)}{ds}=V(s),\frac{%
dV(s)}{ds}=0$ in (\ref{ode}) as: 
\begin{equation}
\alpha (s)\equiv \xi ^{2}(X(s))+[V(s)\cdot \nabla \xi
(X(s))]^{2}-2\{V(s)\cdot \nabla ^{2}\xi (X(s))\cdot V(s)\}\xi (X(s)).
\label{alpha}
\end{equation}%
Let $X(s)\in \bar{\Omega}$ for $t_{1}\leq s\leq t_{2}$. Then there exists
constant $C_{\xi }>0$ such that 
\begin{eqnarray}
e^{C_{\xi }(|V(t_{1})|+1)t_{1}}\alpha (t_{1}) &\leq &e^{C_{\xi
}(|V(t_{1})|+1)t_{2}}\alpha (t_{2});  \label{velocitybound} \\
e^{-C_{\xi }(|V(t_{1})|+1)t_{1}}\alpha (t_{1}) &\geq &e^{-C_{\xi
}(|V(t_{1})|+1)t_{2}}\alpha (t_{2}).  \notag
\end{eqnarray}
\end{lemma}

\begin{proof}
Under the convexity assumption (\ref{convexity}), we notice that $\alpha
(s)\geq 0$ for $t_{1}\leq s\leq t_{2}.$ Since $\frac{dV(s)}{ds}\equiv 0$ by (%
\ref{ode}), we compute the derivative of $\alpha (s)$ in (\ref{alpha}) as 
\begin{eqnarray*}
\frac{d\alpha (s)}{ds} &=&2\xi (X(s))[\nabla \xi (X(s))\cdot
V(s)]+2[V(s)\nabla ^{2}\xi (X(s))V(s))][V(s)\cdot \nabla \xi (X(s))] \\
&&-2\{V(s)\cdot \nabla ^{2}\xi (X(s))\cdot V(s)\}[\nabla \xi (X(s))\cdot
V(s)] \\
&&-2[V(s)\cdot \nabla ^{3}\xi (X(s))V(s)\cdot V(s)]\xi (X(s)).
\end{eqnarray*}%
Note that the second and the third terms on the right hand side cancel
exactly$.$ By the convexity (\ref{convexity}), there exists $C_{\xi }>0$ so
that we can further bound the last term by $\alpha (s)$ as: 
\begin{eqnarray*}
&&|2[V(s)\cdot \nabla ^{3}\xi (X(s))V(s)\cdot V(s)]\xi (X(s))| \\
&\leq &C_{\xi }|V(t_{1})|\times |\{-2V(s)\cdot \nabla ^{2}\xi (X(s))\cdot
V(s)\}\xi (X(s))| \\
&\leq &C_{\xi }|V(t_{1})|\alpha (s).
\end{eqnarray*}%
We therefore have from the definition (\ref{alpha}): 
\begin{eqnarray*}
|\frac{d\alpha (s)}{ds}| &\leq &\xi ^{2}(X(s))+[\nabla \xi (X(s)\cdot
V(s)]^{2}+C_{\xi }|V(t_{1})|\alpha (s) \\
&\leq &\{C_{\xi }|V(t_{1})|+1\}\alpha (s).
\end{eqnarray*}%
Our lemma thus follows from the standard Gronwall inequality.
\end{proof}

\begin{lemma}
\label{huang} Let $(t,x,v)$ be connected with $(t-t_{\mathbf{b}},x_{\mathbf{b%
}},v)$ backward in time through a trajectory of (\ref{ode}).

(1) The backward exit time $t_{\mathbf{b}}(x,v)$ is lower semicontinuous.

(2) If 
\begin{equation}
v\cdot n(x_{\mathbf{b}})=v\cdot \frac{\nabla \xi (x_{\mathbf{b}})}{|\nabla
\xi (x_{\mathbf{b}})|}<0,  \label{negative}
\end{equation}%
then $(t_{\mathbf{b}}(x,v),x_{\mathbf{b}}(x,v))$ are smooth functions of $%
(x,v)$ so that%
\begin{eqnarray}
\text{ \ \ \ }\nabla _{x}t_{\mathbf{b}} &=&\frac{n(x_{\mathbf{b}})}{v\cdot
n(x_{\mathbf{b}})},\text{ \ \ \ }\nabla _{v}t_{\mathbf{b}}=\frac{t_{\mathbf{b%
}}n(x_{\mathbf{b}})}{v\cdot n(x_{\mathbf{b}})},  \notag \\
\text{ \ \ \ }\nabla _{x}x_{\mathbf{b}} &=&I+\nabla _{x}t_{\mathbf{b}%
}\otimes v,\text{ \ \ }\nabla _{v}x_{\mathbf{b}}=t_{\mathbf{b}}I+\nabla
_{v}t_{\mathbf{b}}\otimes v.  \label{tbderivative}
\end{eqnarray}%
Furthermore, if $\xi $ is real analytic, then $(t_{\mathbf{b}}(x,v),x_{%
\mathbf{b}}(x,v))$ are also real analytic$.$

(3) Let $x_{i}\in \partial \Omega ,$ for $i=1,2,$ and let $(t_{1},x_{1},v)$
and $(t_{2},x_{2},v)$ be connected with the trajectory $\frac{dX(s)}{ds}%
=V(s),\frac{dV(s)}{ds}=0$ which lies inside $\bar{\Omega}$. Then there
exists a constant $C_{\xi }>0$ such that 
\begin{equation}
|t_{1}-t_{2}|\geq \frac{|n(x_{1})\cdot v|}{C_{\xi }|v|^{2}}.  \label{tlower}
\end{equation}

(4) Define the boundary mapping 
\begin{equation}
\Phi _{\mathbf{b}}:(t,x,v)\rightarrow (t-t_{\mathbf{b}},x_{\mathbf{b}%
}(x,v),v)\in \mathbf{R\times }\{\gamma _{0}\cup \gamma _{-}\}.
\label{boundarymap}
\end{equation}%
Then $\Phi _{\mathbf{b}}$ and $\Phi _{\mathbf{b}}^{-1}$ maps zero measure
sets to zero-measure sets between either $\{t\}\times \Omega \mathbf{\times R%
}^{3}$ and $\mathbf{R\times }\{\gamma _{0}\cup \gamma _{+}\}$ or $\mathbf{%
R\times }\{\gamma _{0}\cup \gamma _{+}\}\rightarrow \mathbf{R\times }%
\{\gamma _{0}\cup \gamma _{-}\}$.
\end{lemma}

\begin{proof}
\textbf{\ (1):} We need to show that the set $\{(x,v):x\in \Omega \,$\ and $%
t_{\mathbf{b}}(x,v)>c\}$ is open for all $c\in \mathbf{R}.$ Let $t_{\mathbf{b%
}}(x_{0},v_{0})>c+\varepsilon ,$ for some $\varepsilon >0$ small. From the
definition of $t_{\mathbf{b}}(x,v)$ in (\ref{exit}), $x_{0}-sv_{0}\in \Omega 
$ for all $0\leq s\leq c+\varepsilon <t_{\mathbf{b}}(x_{0},v_{0}).$ Since $%
\Omega $ is open, we deduce that for $(x,v)$ close to $(x_{0},v_{0}),$ by
continuity, $x-vs\in \Omega $ for all $c\leq s\leq c+\varepsilon .$ This
implies that $t_{\mathbf{b}}(x,v)>c$. Hence $t_{\mathbf{b}}(x,v)$ is lower
semicontinuous.

\textbf{Proof of (2):} Since $x_{\mathbf{b}}\in \partial \Omega ,$ $\xi (x_{%
\mathbf{b}}(x,v))=\xi (x-t_{\mathbf{b}}v)=0.$ But from (\ref{negative}) and
the fact $|\nabla \xi (x_{\mathbf{b}})|\neq 0,$ we have 
\begin{equation*}
\partial _{t_{\mathbf{b}}}\xi (x-t_{\mathbf{b}}v)|_{t_{\mathbf{b}}}=-v\cdot
\nabla \xi (x_{\mathbf{b}})>0.
\end{equation*}%
Therefore, by the Implicit Function Theorem, we can solve for smooth $t_{%
\mathbf{b}}(x,v)$ and deduce (\ref{tbderivative}). Furthermore, when $\xi $
is analytic, so are $t_{\mathbf{b}}$ and $x_{\mathbf{b}}$ by Theorem 15.3 in
[D].

\textbf{Proof of (3):} Notice that for $x_{1}\in \partial \Omega ,$%
\begin{equation*}
\lim_{\substack{ y\rightarrow x_{1}  \\ y\in \partial \Omega }}\frac{%
|\{x_{1}-y\}\cdot n(x_{1})|}{|x_{1}-y|}=0.
\end{equation*}%
Hence, we have $|\{x_{1}-y\}\cdot n(x_{1})|\leq C_{\xi }|x_{1}-y|^{2}$ for
all $y\in \partial \Omega .$ Taking inner product of $%
x_{1}-x_{2}=v(t_{1}-t_{2})$ with $n(x_{1}),$ we get 
\begin{equation*}
C_{\xi }|v|^{2}|t_{1}-t_{2}|^{2}\geq C_{\xi }|x_{1}-x_{2}|^{2}\geq
|\{x_{1}-x_{2}\}\cdot n(x_{1})|=|v(t_{1}-t_{2})\cdot n(x_{1})|.
\end{equation*}%
We thus deduce (\ref{tlower}) by dividing $|t_{1}-t_{2}|$.

\textbf{Proof of (4):} We define a map from $\mathbf{R\times }\partial
\Omega \mathbf{\times R}^{3}$ to $\{t\}\times \mathbf{R}^{3}\times \mathbf{R}%
^{3}$ as%
\begin{equation}
\Psi _{t}:(t^{\prime },x^{\prime },v^{\prime })=(t,x^{\prime }+v^{\prime
}(t-t^{\prime }),\text{ }v^{\prime }).  \label{psi}
\end{equation}

Recall the boundary map $\Phi _{\mathbf{b}}(t,x,v)$ in (\ref{boundarymap}).
From the definition of $t_{\mathbf{b}}$ in (\ref{exit}), $\Phi _{\mathbf{b}}$
is one to one from either from $\{t\}\times \Omega \times \mathbf{R}^{3}$ or
from $\mathbf{R\times }\{\gamma _{0}\cup \gamma _{+}\}$ to $\mathbf{R\times }%
\{\gamma _{0}\cup \gamma _{-}\}$. Denote its inverse by $\Phi _{\mathbf{b}%
}^{-1}.$

In the case that $v\cdot n(x_{\mathbf{b}})=0,$ i.e., $\Phi _{\mathbf{b}%
}(t,x,v)\in \gamma _{0},$ the grazing set$,$ then $(t,x,v)\in $ $\Psi
_{t}(\gamma _{0})$. In (\ref{psi}), $\gamma _{0}$ is characterized by the
five-dimensional space: $x^{\prime }\in \partial \Omega ,$ $v^{\prime }\cdot
n(x^{\prime })=0,$ $t^{\prime }\in \mathbf{R.}$ Since $\Psi _{t}$ is a
smooth map, $\Psi _{t}(\gamma _{0})$ is also a five-dimensional space
locally at $(x^{\prime },v^{\prime },t^{\prime })\in \gamma _{0}.$ This
implies that $\Phi _{\mathbf{b}}^{-1}(\gamma ^{0})\subset \Psi _{t}(\gamma
_{0})$ is a zero-measure set in $\{t\}\times \Omega \times \mathbf{R}^{3}$.

In the case that $v\cdot n(x_{\mathbf{b}})\neq 0,$ we consider the map $\Psi
_{t}$ where $\xi (x^{\prime })=0.$ We may assume that $\partial _{x_{1}}\xi
(x^{\prime })\neq 0,$ and $x_{1}^{\prime }=\eta (x_{2}^{\prime
},x_{3}^{\prime })$ with some smooth function $\eta .$ Now $\Psi
_{t}=(t,\eta (x_{2}^{\prime },x_{3}^{\prime })+v_{1}^{\prime }(t-t^{\prime
}),x_{2}^{\prime }+v_{2}^{\prime }(t-t^{\prime }),x_{3}^{\prime
}+v_{3}^{\prime }(t-t^{\prime }),v^{\prime })$ and we compute the Jacobian
matrix of $\Psi _{t\text{ }}$for $(t^{\prime },x_{2}^{\prime },x_{3}^{\prime
},v^{\prime })$ to get 
\begin{equation*}
\left( 
\begin{array}{cccccc}
-v_{1}^{\prime } & \partial _{x_{2}^{\prime }}\eta & \partial
_{x_{3}^{\prime }}\eta & t-t^{\prime } & 0 & 0 \\ 
-v_{2}^{\prime } & 1 & 0 & 0 & t-t^{\prime } & 0 \\ 
-v_{3}^{\prime } & 0 & 1 & 0 & 0 & t-t^{\prime } \\ 
0 & 0 & 0 & 1 & 0 & 0 \\ 
0 & 0 & 0 & 0 & 1 & 0 \\ 
0 & 0 & 0 & 0 & 0 & 1%
\end{array}%
\right)
\end{equation*}%
and its determinant is exactly $\pm \{v^{\prime }\cdot n(x^{\prime })\}\sqrt{%
1+|\nabla \eta |^{2}}\neq 0$ if $(x^{\prime },v^{\prime })\notin \gamma
_{0}. $ Hence, locally, $\Psi _{t}$ is a smooth homeomorphism preserving
zero-measure sets away from $\gamma _{0}$. Notice that from the uniqueness
in part (2) of Lemma \ref{huang}, $\Phi _{\mathbf{b}}^{-1}=\Psi _{t}$ and $%
\Phi _{\mathbf{b}}^{{}}=\Psi _{t}^{-1}$ locally if $v\cdot n(x_{\mathbf{b}%
})\neq 0.$ Hence, for any $k\geq 1,$ by a finite covering for a compact set, 
$\Phi _{\mathbf{b}}^{-1}$preserves the zero-measure sets on 
\begin{equation*}
A_{k}\equiv \{(t^{\prime },x^{\prime },v^{\prime })|\text{ }x^{\prime }\in
\partial \Omega ,|v^{\prime }\cdot n(x^{\prime })|\geq \frac{1}{k}%
,|v^{\prime }|\leq k,|t^{\prime }|\leq k\}\cap \Phi _{\mathbf{b}%
}(\{t\}\times \Omega \times \mathbf{R}^{3}).
\end{equation*}

To prove $\Phi _{\mathbf{b}}^{{}}$ preserves the zero-measure sets between $%
\{t\}\times \Omega \times \mathbf{R}^{3}$ and $\mathbf{R\times }\{\gamma
_{0}\cup \gamma _{+}\},$ we take any set $S\in \mathbf{R\times }\{\gamma
_{0}\cup \gamma _{+}\}$ with $|S|=0.$ Clearly, $S=\{\mathbf{R}\times \gamma
_{0}\cap S\}\cup _{k=1}^{\infty }\{A_{k}\cap S\}.$ Therefore, $\Phi _{%
\mathbf{b}}^{-1}(S)\subset \Psi _{t}(\gamma _{0})\cup _{k=1}^{\infty }\Phi _{%
\mathbf{b}}^{-1}(A_{k}\cap S),$ and $|\Phi _{\mathbf{b}}^{-1}(A_{k}\cap
S)|=0 $ and $|\Psi _{t}(\gamma _{0})|=0.$ On the other hand, if $S\in $ $%
\{t\}\times \Omega \times \mathbf{R}^{3},$ we have $\Phi _{\mathbf{b}}(S)=\{%
\mathbf{R}\times \gamma _{0}\cap \Phi _{\mathbf{b}}(S)\}\cup _{k=1}^{\infty
}\{A_{k}\cap \Phi _{\mathbf{b}}(S)\}.$ If $|S|=0,$ then $|A_{k}\cap \Phi _{%
\mathbf{b}}(S)|=0,$ because $\Phi _{\mathbf{b}}^{-1}\{A_{k}\cap \Phi _{%
\mathbf{b}}(S)\}\subset S$ has measure zero and $\Phi _{\mathbf{b}}=\Psi
_{t}^{-1}$ maps zero-measure sets to zero-measure sets.

To prove $\Phi _{\mathbf{b}}^{{}}$ preserves the zero-measure sets from $%
\mathbf{R\times }\{\gamma _{0}\cup \gamma _{+}\}\rightarrow \mathbf{R\times }%
\{\gamma _{0}\cup \gamma _{-}\},$ we take $S\in \mathbf{R\times \{}\gamma
_{0}\cup \gamma _{+}\}$ with $|S|=0.$ Consider the set $\Phi _{\mathbf{b}%
}^{-1}\{\Phi _{\mathbf{b}}(S)\setminus \gamma _{0}\}$. For any point $%
(t,x,v)\in \Phi _{\mathbf{b}}^{-1}\{\Phi _{\mathbf{b}}(S)\setminus \gamma
_{0}\},$ we know that $v\cdot n(x_{\mathbf{b}})\neq 0.$ This implies from (%
\ref{tlower}) that $t_{\mathbf{b}}>0$ and $t-t_{\mathbf{b}}<t.$ We can
choose a fixed $s$ between $t$ and $t-t_{\mathbf{b}}$. Locally around $%
(t,x,v),$ 
\begin{equation}
\Phi _{\mathbf{b}}=\Phi _{\mathbf{b}}(s)\circ \Psi _{s}.  \label{composition}
\end{equation}%
where $\Psi _{s}$ defined in (\ref{psi}) which maps $\mathbf{R\times }%
\{\gamma _{0}\cup \gamma _{+}\}$ to the plane $\{s\}\times \Omega \times 
\mathbf{R}^{3},$ and $\Phi _{\mathbf{b}}(s)$ maps $\{s\}\times \Omega \times 
\mathbf{R}^{3}$ to $\mathbf{R\times }\{\gamma _{0}\cup \gamma _{-}\}.$ Since 
$\Phi _{\mathbf{b}}$ is one to one, we have $\Phi _{\mathbf{b}}^{-1}\{\Phi _{%
\mathbf{b}}(S)\setminus \gamma _{0}\}\subset S,$ so that $|\Phi _{\mathbf{b}%
}^{-1}\{\Phi _{\mathbf{b}}(S)\setminus \gamma _{0}\}|=0.$ Therefore, from
previous arguments, $\Psi _{s}$ preserves zero-measure sets from $\mathbf{%
R\times }\{\gamma _{0}\cup \gamma _{+}\}~$to $\{s\}\times \Omega \times 
\mathbf{R}^{3},$ while $\Phi _{\mathbf{b}}(s)$ preserves zero-measure sets
from $\{s\}\times \Omega \times \mathbf{R}^{3}$ set to $\mathbf{R\times }%
\{\gamma _{0}\cup \gamma _{-}\}.$ Hence, by (\ref{composition}), $|\Phi _{%
\mathbf{b}}(S)\setminus \gamma _{0}|=|\Phi _{\mathbf{b}}(s)\circ \Psi
_{s}[\Phi _{\mathbf{b}}^{-1}\{\Phi _{\mathbf{b}}(S)\setminus \gamma
_{0}\}]|=0.$ We therefore deduce $|\Phi _{\mathbf{b}}(S)|=0,$ for $S$ inside
a neighborhood of $(t,x,v).$

On the other hand, we take $S\in \mathbf{R\times }\{\gamma _{0}\cup \gamma
_{-}\}$ with $|S|=0$ and we need to show $|\Phi _{\mathbf{b}}^{-1}(S)|=0$ $.$
For any point $(t,x,v)\in \Phi _{\mathbf{b}}^{-1}(S)\setminus \gamma _{0},$ $%
v\cdot n(x)\neq 0,$ so that if $(t^{\prime },x^{\prime },v)=\Phi _{\mathbf{b}%
}(t,x,v),$ then $t^{\prime }<t.$ Hence, we again can find $t^{\prime }<s<t$
such that, near $(t,x,v),$ $\Phi _{\mathbf{b}}^{-1}=\Psi _{s}^{-1}\circ \Phi
_{\mathbf{b}}^{-1}(s)$ from (\ref{composition})$.$ Since $|\Phi _{\mathbf{b}%
}\{\Phi _{\mathbf{b}}^{-1}(S)\setminus \gamma _{0}\}|\leq |S|=0,$ for $\Phi
_{\mathbf{b}}^{-1}(S)$ near $(t,x,v),$ $\{\Phi _{\mathbf{b}%
}^{-1}(S)\setminus \gamma _{0}\}=\Psi _{s}^{-1}\circ \Phi _{\mathbf{b}%
}^{-1}(s)[\Phi _{\mathbf{b}}\{\Phi _{\mathbf{b}}^{-1}(S)\setminus \gamma
_{0}\}].$ It follows (\ref{composition}) again that $|\Phi _{\mathbf{b}%
}^{-1}(S)\setminus \gamma _{0}|=0$, $|\Phi _{\mathbf{b}}^{-1}(S)|=0$ for $S$
inside a neighborhood of $(t-t_{\mathbf{b}},x_{\mathbf{b}},v).$ The general
case follows from a countable covering for $\mathbf{R\times }\{\gamma
_{0}\cup \gamma _{+}\}$ and $\{s\}\times \Omega \times \mathbf{R}^{3}$.
\end{proof}

\begin{lemma}
\label{kernel}Recall (\ref{mu-k}) and the Grad estimate [Gr2] for hard
potentials:%
\begin{equation}
|K(v,v^{\prime })|\leq C\{|v-v^{\prime }|+|v-v^{\prime }|^{-1}\}e^{-\frac{1}{%
8}|v-v^{\prime }|^{2}-\frac{1}{8}\frac{||v|^{2}-|v^{\prime }|^{2}|^{2}}{%
|v-v^{\prime }|^{2}}}.  \label{grad}
\end{equation}%
Recall $w$ in (\ref{weight})$.$ Then there exists $0\leq \varepsilon (\theta
)<1$ and $C_{\theta }>0$ such that for $0\leq \varepsilon <\varepsilon
(\theta ),$ 
\begin{equation}
\int \{|v-v^{\prime }|+|v-v^{\prime }|^{-1}\}e^{-\frac{1-\varepsilon }{8}%
|v-v^{\prime }|^{2}-\frac{1-\varepsilon }{8}\frac{||v|^{2}-|v^{\prime
}|^{2}|^{2}}{|v-v^{\prime }|^{2}}}\frac{w(v)}{w(v^{\prime })}dv^{\prime
}\leq \frac{C}{1+|v|}.  \label{wk}
\end{equation}
\end{lemma}

\begin{proof}
By (\ref{weight}), we first notice that for some $C_{\rho ,\beta }>0,$ 
\begin{equation*}
\left\vert \frac{w(v)}{w(v^{\prime })}\right\vert \leq C[1+|v-v^{\prime
}|^{2}]^{|\beta |}e^{-\theta \{|v^{\prime }|^{2}-|v|^{2}\}}.
\end{equation*}%
Let $v-v^{\prime }=\eta $ and $v^{\prime }=v-\eta $ in the integral of (\ref%
{wk})$.$ By (\ref{grad}), we now compute the total exponent in $%
K(v,v^{\prime })\frac{w(v)}{w(v^{\prime })}$ as: 
\begin{eqnarray*}
&&-\frac{1}{8}|\eta |^{2}-\frac{1}{8}\frac{||\eta |^{2}-2v\cdot \eta |^{2}}{%
|\eta |^{2}}-\theta \{|v-\eta |^{2}-|v|^{2}\} \\
&=&-\frac{1}{4}|\eta |^{2}+\frac{1}{2}v\cdot \eta -\frac{1}{2}\frac{|v\cdot
\eta |^{2}}{|\eta |^{2}}-\theta \{|\eta |^{2}-2v\cdot \eta \} \\
&=&(-\theta -\frac{1}{4})|\eta |^{2}+(\frac{1}{2}+2\theta )v\cdot \eta -%
\frac{1}{2}\frac{\{v\cdot \eta \}^{2}}{|\eta |^{2}}.
\end{eqnarray*}%
Since $\theta <\frac{1}{4},$ the discriminant of the above quadratic form of 
$|\eta |$ and $\frac{v\cdot \eta }{|\eta |}$ is 
\begin{equation*}
\Delta =(\frac{1}{2}+2\theta )^{2}+2(-\theta -\frac{1}{4})=4\theta ^{2}-%
\frac{1}{4}<0.
\end{equation*}%
Hence, the quadratic form is negative definite. We thus have, for $%
\varepsilon >0$ sufficiently small, there is $C_{\theta }>0$ such that the
following perturbed quadratic form is still negative definite 
\begin{eqnarray}
&&-\frac{1-\varepsilon }{8}|\eta |^{2}-\frac{1-\varepsilon }{8}\frac{||\eta
|^{2}-2v\cdot \eta |^{2}}{|\eta |^{2}}-\theta \{|\eta |^{2}-2v\cdot \eta \} 
\notag \\
&\leq &-C_{\theta }\{|\eta |^{2}+\frac{|v\cdot \eta |^{2}}{|\eta |^{2}}%
\}=-C_{\theta }\{\frac{|\eta |^{2}}{2}+(\frac{|\eta |^{2}}{2}+\frac{|v\cdot
\eta |^{2}}{|\eta |^{2}})\}  \notag \\
&\leq &-C_{\theta }\{\frac{|\eta |^{2}}{2}+|v\cdot \eta |\}.
\label{exponent}
\end{eqnarray}%
Therefore, for given $|v|\geq 1,$ we make another change of variable $\eta
_{\shortparallel }=\{\eta \cdot \frac{v}{|v|}\}\frac{v}{|v|},$ and $\eta
_{\perp }=\eta -\eta _{||}$ so that $|v\cdot \eta |=|v||\eta
_{\shortparallel }|$ and $|v-v^{\prime }|\geq |\eta _{\perp }|.$ We can
absorb $\{1+|\eta |^{2}\}^{|\beta |}$, $|\eta |\{1+|\eta |^{2}\}^{|\beta |}$
by $e^{\frac{C_{\theta }}{4}|\eta |^{2}}$, and bound the integral in (\ref%
{wk}) by (\ref{exponent}): 
\begin{eqnarray*}
&&C_{\beta }\int_{\mathbf{R}^{2}}(\frac{1}{|\eta _{_{\perp }}|}+1)e^{-\frac{%
C_{\theta }}{4}|\eta |^{2}}\left\{ \int_{-\infty }^{\infty }e^{-C_{\theta
}|v|\times |\eta _{||}|}d|\eta _{||}|\right\} d\eta _{\perp } \\
&\leq &\frac{C_{\beta }}{|v|}\int_{\mathbf{R}^{2}}(\frac{1}{|\eta _{_{\perp
}}|}+1)e^{-\frac{C_{\theta }}{4}|\eta _{\perp }|^{2}}\left\{ \int_{-\infty
}^{\infty }e^{-C_{\theta }|y|}dy\right\} d\eta _{\perp }\text{ \ \ }%
(y=|v|\times |\eta _{||}|).
\end{eqnarray*}%
We thus deduce our lemma since both integrals are finite.
\end{proof}

\begin{lemma}
\label{analytic}Let $\kappa (y)$ be a real analytic function of $y\in 
\mathbf{R}^{n}$ in a region such that $\kappa (y)$ is not identically zero.
Then the set $\{y:\kappa (y)=0\}$ has zero $n-$dimensional Lebesque measure.
\end{lemma}

\begin{proof}
We use an induction on the dimension $n.$ When $n=1,$ we assume $\kappa
(y^{0})=0.$ Since $\kappa $ is analytic, for $y$ near $y^{0},$ we have 
\begin{equation*}
\kappa (y)=\kappa (y^{0})+\sum_{k=1}^{\infty }\frac{\kappa ^{(k)}(y^{0})}{k!}%
(y-y^{0})^{k}.
\end{equation*}%
Since $\kappa (y)$ is not identically zero, we can always assume a smallest $%
k_{1}$ such that $\frac{\kappa ^{(k_{1})}(y^{0})}{k_{1}!}\neq 0.$ We
therefore can rewrite 
\begin{equation*}
\kappa (y)=(y-y^{0})^{k_{1}}\times \left\{ \sum_{k\geq k_{1}}^{\infty }\frac{%
\kappa ^{(k)}(y^{0})}{k!}(y-y^{0})^{k-k_{1}}\right\} .
\end{equation*}%
Hence $\kappa (y)=0$ for $y-y^{0}$ sufficiently small implies $y=y^{0}$ (an
isolated point), which has zero one dimensional measure.

Assume that the lemma is valid for $m.$ For $m+1$ dimensional case, we
assume $\kappa (y^{0})=0.$ We first notice that by finite open coverings for
any compact subset, it suffices to show that for any $y^{0}$ such that $%
\kappa (y^{0})=0,$ then there is a ball $\{y:|y-y^{0}|<\delta \}$ such that $%
|\{y:|y-y^{0}|<\delta $, $\kappa (y)=0\}|=0.$

Now for any $y\neq y^{0}$ and $|y-y^{0}|<\delta ,$ since $\kappa (y)$ is
real analytic, we have 
\begin{equation*}
\kappa (y)=\kappa (y_{0})+\sum_{|k|=1}^{\infty }\frac{\kappa ^{(k)}(y^{0})}{%
k!}(y-y^{0})^{k}
\end{equation*}%
where the multi-index $k=[k_{1},k_{2},...,k_{m}],$ $%
k!=k_{1}!k_{2}!...k_{m}!, $ while $(y-y^{0})^{k}=\Pi
_{l=1}^{m}(y_{l}-y_{l}^{0})^{k_{l}}.$ Since $\kappa (y)$ is not identically
zero, there exists $\bar{k}$ such that $\kappa ^{(\bar{k})}(y^{0})\neq 0.$
Without loss of generality, we can further assume that $\bar{k}_{1}\neq 0$,
so that $(y-y^{0})^{k_{1}}$ contains the factor $(y_{1}-y_{1}^{0})^{\bar{k}%
_{1}}.$ Furthermore, we can assume $\bar{k}_{1}\geq 0$ is the smallest among
those non-zero terms, so that every term $(y-y^{0})^{k}$ contains the common
factor $(y_{1}-y_{1}^{0})^{\bar{k}_{1}}.$ We therefore can rewrite$:$%
\begin{equation*}
\kappa (y)=(y_{1}-y_{1}^{0})^{\bar{k}_{1}}\left\{ \sum \frac{\kappa
^{(k)}(y^{0})}{k!}(y-y^{0})^{k-\bar{k}_{1}}\right\} .
\end{equation*}%
For $y_{1}\neq y_{1}^{0}$, $\kappa (y)=0$ implies 
\begin{equation}
\kappa _{1}(y)\equiv \sum \frac{\kappa ^{(k)}(y^{0})}{k!}(y-y^{0})^{k-\bar{k}%
_{1}}=0.  \label{q1}
\end{equation}%
Clearly, for any given $y_{1},$ $\kappa _{1}(y_{1},y_{2},...,y_{m+1})$ is an
analytical function for $m$ variables $\tilde{y}=[y_{2},...,y_{m+1}]$.
Therefore, for fixed $y_{1},$ we can expand $\kappa _{1}(y_{1},\tilde{y})$
around $[y_{2}^{0},...,y_{m+1}^{0}]$ to get 
\begin{equation*}
\kappa _{1}(y_{1},\tilde{y})\equiv \sum_{k=[0,k_{2},k_{3},...,k_{m+1}]}\frac{%
\kappa _{1}^{(k)}(y_{1},y_{2}^{0},...,y_{m+1}^{0})}{k!}(\tilde{y}-\tilde{y}%
^{0})^{k}.
\end{equation*}%
Since by our choices of $\bar{k}$ and $\bar{k}_{1}$, the multi-index $\bar{k}%
-\bar{k}_{1}=[0,\bar{k}_{2},...,\bar{k}_{m+1}]$, and we can consider the
term 
\begin{equation*}
\frac{\kappa _{1}^{(\bar{k}-\bar{k}_{1})}(y_{1},y_{2}^{0},...,y_{m+1}^{0})}{(%
\bar{k}-\bar{k}_{1})!}(\tilde{y}-\tilde{y}^{0})^{\bar{k}-\bar{k}_{1}}.
\end{equation*}%
We compute $\kappa _{1}^{(\bar{k}-\bar{k}_{1})}$ from (\ref{q1}) as 
\begin{equation*}
\kappa _{1}^{(\bar{k}-\bar{k}%
_{1})}(y_{1},y_{2}^{0},...,y_{m+1}^{0})|_{y_{1}=y_{1}^{0}}=\sum \frac{\kappa
^{(k)}(y^{0})}{k!}\partial ^{\bar{k}-\bar{k}_{1}}\{(y-y^{0})^{k-\bar{k}%
_{1}}\}=\frac{(\bar{k}-\bar{k}_{1})!\kappa ^{(\bar{k})}(y^{0})}{\bar{k}!}%
\neq 0,
\end{equation*}%
by our choices of $\bar{k}$ and $\bar{k}_{1}.$ From the continuity of $%
\kappa _{1}^{(\bar{k}-\bar{k}_{1})}(y_{1},y_{2}^{0},...,y_{m+1}^{0})$ with
respect to $y_{1},$ for $|y_{1}-y_{1}^{0}|<\delta $ small, we deduce that%
\begin{equation*}
\kappa _{1}^{(\bar{k}-\bar{k}_{1})}(y_{1},y_{2}^{0},...,y_{m+1}^{0})\neq 0.
\end{equation*}%
Therefore, $\kappa _{1}(y_{1},y_{2},...,y_{m+1})$ is an analytical function
which is not identically zero for all $|y-y^{0}|<\delta $ sufficiently small$%
.$ By the induction hypothesis, for any fixed $y_{1},$ $\kappa
_{1}(y_{1,}y_{2},...,y_{m+1})=0$ is a set of $m-$dimensional Lebesque zero
measure set for $y_{2},y_{3},...y_{m+1}.$ We now apply the Fubini theorem to
compute 
\begin{eqnarray*}
|\{\kappa _{1}(y) &=&0,|y-y^{0}|<\delta \}|=\int_{\mathbf{R}^{m+1}}\mathbf{1}%
_{\{\kappa _{1}(y)=0\}}(y_{1},y_{2},...,y_{m+1})\mathbf{1}_{|y-y^{0}|<\delta
} \\
&=&\int_{\mathbf{R}}\left( \int_{\mathbf{R}^{m}}\mathbf{1}_{\{\kappa
_{1}(y)=0\}}(y_{1},y_{2},...,y_{m+1})\mathbf{1}_{|y-y^{0}|<\delta
}dy_{2}dy_{3}...dy_{m+1}\right) dy_{1} \\
&\leq &\int_{\mathbf{R}}|\{(y_{2},y_{3},...,y_{m+1}):\kappa
_{1}(y_{1},y_{2},...,y_{m+1})=0,\text{ }|y-y^{0}|<\delta \}|dy_{1} \\
&=&0.
\end{eqnarray*}%
Therefore, inside $|y-y^{0}|<\delta $ for $\delta $ sufficiently small, we
have 
\begin{equation*}
\{y:\kappa (y)=0\}\subset \{y_{1}=y_{1}^{0}\}\cup \{\kappa _{1}(y)=0\},
\end{equation*}%
and both sets have zero $(m+1)-$ dimensional Lebesque measure.
\end{proof}

\begin{lemma}
\label{collision}Recall (\ref{weight}) and (\ref{gamma}). We have 
\begin{equation}
|w\Gamma \lbrack g_{1,}g_{2}](v)|\leq C\{w(v)(|v|+1)^{\gamma
}|g_{1}(v)|+||wg_{1}||_{\infty }\}||wg_{2}||_{\infty }.  \notag
\end{equation}
\end{lemma}

\begin{proof}
First consider the second term $\Gamma _{\mathrm{loss}}$ in (\ref{gamma}).
We have 
\begin{equation*}
\int_{\mathbf{R}^{3}}|u-v|^{\gamma }|\mu ^{1/2}(u)g_{2}(x,u)|du\leq
C\{|v|+1\}^{\gamma }||wg_{2}||_{\infty },
\end{equation*}%
Hence $w\Gamma _{\mathrm{loss}}[g_{1,}g_{2}]$ is bounded by 
\begin{equation*}
w|g_{1}|\int |u-v|^{\gamma }|\mu ^{1/2}(u)g_{2}(x,u)|du\leq
Cw(v)\{|v|+1\}^{\gamma }|g_{1}(v)|\times ||wg_{2}||_{\infty }.
\end{equation*}%
For $\Gamma _{\mathrm{gain}}$ in (\ref{gamma}), by $|u^{\prime
}|^{2}+|v^{\prime }|^{2}=|u|^{2}+|v|^{2},$ $w(v)\leq Cw(v^{\prime
})w(u^{\prime }),$ and 
\begin{gather*}
\int q_{0}(\theta )|u-v|^{\gamma }e^{-|u-v|^{2}/4}w(v)|g_{1}(u^{\prime
})g_{2}(v^{\prime })|d\omega du \\
\leq \int q_{0}(\theta )|u-v|^{\gamma }e^{-|u-v|^{2}/4}w(u^{\prime
})w(v^{\prime })|g_{1}(u^{\prime })g_{2}(v^{\prime })|d\omega du \\
\leq ||wg_{1}||_{\infty }\times ||w_{2}g_{2}||_{\infty }\int |u-v|^{\gamma
}e^{-|u-v|^{2}/4}du.
\end{gather*}%
Since $0\leq \gamma \leq 1,$ this completes the proof.
\end{proof}

\section{$L^{2}$ Decay Theory}

We define the boundary integration for $g(x,v),$ $x\in \partial \Omega ,$%
\begin{equation}
\int_{\gamma _{\pm }}gd\gamma =\int_{\gamma _{\pm }}g(x,v)|n_{x}\cdot
v|dS_{x}dv  \label{dgamma}
\end{equation}%
where $dS_{x}$ is the standard surface measure on $\partial \Omega .$ We
also define $||h||_{\gamma }=||h||_{\gamma _{+}}+||h||_{\gamma _{-}}$ to be
the $L^{2}(\gamma )$ with respect to the measure $|n_{x}\cdot v|dS_{x}dv.$
For fixed $x\in \partial \Omega ,$ denote the boundary inner product over $%
\gamma _{\pm }$ in $v$ as 
\begin{equation*}
\langle g_{1},g_{2}\rangle _{\gamma _{\pm }}(t,x)=\int_{\pm v\cdot
n(x)>0}g_{1}(t,x,v)g_{2}(t,x,v)|n(x)\cdot v|dv.
\end{equation*}%
By (\ref{diffusenormal}), we also define a different $L_{v}^{2}$-projection
for any boundary function $g(x,v)$ toward the unit vector$\sqrt{c_{\mu }\mu
(v)}$ with respect to $\langle \cdot ,\cdot \rangle $ as: 
\begin{equation}
P_{\gamma }g=\{\int_{n(x)\cdot v>0}g(t,x,v)\sqrt{\mu (v)}n(x)\cdot
vdv\}c_{\mu }\sqrt{\mu (v)}.  \label{pboundary}
\end{equation}

Our main theorem of this section is

\begin{theorem}
\label{L2decay} Let $f(t,x,v)\in L^{2}$ be the (unique) solution to the
linear Boltzmann equation (\ref{lboltzmann}) with trace $f_{\gamma }\in L_{%
\text{loc }}^{2}(\mathbf{R}_{+};L^{2}(\gamma )).$

(1) If $f$ satisfies the in-flow boundary condition (\ref{inflow}), then
there exists $\lambda >0$ and $C>0$ such that for all $0\leq t\leq \infty ,$ 
\begin{equation*}
e^{2\lambda t}||f(t)||^{2}\leq 2\{||f(0)||^{2}+\int_{0}^{t}e^{2\lambda
s}||g(s)||_{\gamma _{-}}^{2}ds\}.
\end{equation*}

(2) Let $f$ satisfy the bounce-back boundary condition (\ref{bounceback}),
then there exists $\lambda >0$ and $C>0$ such that $\sup_{0\leq t\leq \infty
}\{e^{2\lambda t}||f(t)||^{2}\}\leq 2||f(0)||^{2}.$

(3) Let $f$ satisfy the specular reflection condition (\ref{specular}), and
the mass and energy conservation laws (\ref{mass}) and (\ref{energy}). In
the case $\Omega $ has any axis of rotational symmetry (\ref{axis}), we
further require that the corresponding conservation of the angular momentum (%
\ref{axiscon}). Then there exists $\lambda >0$ and $C>0$ such that $%
\sup_{0\leq t\leq \infty }\{e^{2\lambda t}||f(t)||^{2}\}\leq 2||f(0)||^{2}.$

(4) If $f$ satisfies the diffusive boundary condition (\ref{diffuse}) \ and
the mass conservation (\ref{mass}), then there exists $\lambda >0$ and $C>0$
such that $\sup_{0\leq t\leq \infty }\{e^{2\lambda t}||f(t)||^{2}\}\leq
2||f(0)||^{2}.$
\end{theorem}

We remark that the existence of such a $L^{2}$ solution $f$ $\ $with its
trace $f_{\gamma }\in L_{\text{loc }}^{2}(\mathbf{R}_{+};L^{2}(\gamma ))$
(which guarantees the uniqueness) is in general not known for the
bounce-back and specular boundary conditions within $L^{2}$ framework. This
is due to the possible blow-up of $L_{\text{loc }}^{2}(\mathbf{R}%
_{+};L^{2}(\gamma ))$ at the grazing set $\gamma _{0}.$ See [BP], [CIP] and
[U1] for more details. On the other hand, $\int_{0}^{t}||f(s)||_{\gamma
}^{2}ds<\infty $ will be established by the study of (\ref{lboltzmannh}) in
Theorems \ref{bouncebackrate}, \ref{specularrate} and \ref{duffusiverate}
with property $h=wf\in L^{\infty }$ and $h=wf\in L^{\infty }(\gamma ).$ We
mainly will establish the following:

\begin{proposition}
\label{lower}

(1) There exists $M>0$ such that for any solution $f(t,x,v)$ to the
linearized Boltzmann equation (\ref{lboltzmann}), 
\begin{equation}
\int_{0}^{1}||\mathbf{P}f(s)||_{\nu }^{2}ds\leq M\{\int_{0}^{1}||(\mathbf{I}-%
\mathbf{P)}f(s)||_{\nu }^{2}ds+\int_{0}^{1}||f(s)||_{\gamma }^{2}ds\}.
\label{inflowpositive}
\end{equation}

(2) There exists $M>0$ such that for any solution $f(t,x,v)$ to the
linearized Boltzmann equation (\ref{lboltzmann}) satisfying the bounce-back
boundary condition (\ref{bounceback}) and the mass-energy conservation laws (%
\ref{mass}) and (\ref{energy})$,$ we have 
\begin{equation}
\int_{0}^{1}||\mathbf{P}f(s)||_{\nu }^{2}ds\leq M\int_{0}^{1}||(\mathbf{I}-%
\mathbf{P)}f(s)||_{\nu }^{2}ds.  \label{bbpositive}
\end{equation}

(3) There exists $M>0$ such that for any solution $f(t,x,v)$ to the
linearized Boltzmann equation (\ref{lboltzmann}) satisfying the specular
reflection condition (\ref{specular}) and the mass-energy conservation laws (%
\ref{mass}) and (\ref{energy}), (in the case $\Omega $ has any axis of
rotational symmetry (\ref{axis}), we further assume conservation of the
angular momentum (\ref{axiscon})), estimate (\ref{bbpositive}) is valid.

(4) There exists $M>0$ such that for any solution $f(t,x,v)$ solution to the
linearized Boltzmann equation (\ref{lboltzmann}) satisfying the diffusive
boundary condition (\ref{diffuse}) and the mass conservation (\ref{mass}), 
\begin{equation}
\int_{0}^{1}||\mathbf{P}f(s)||_{\nu }^{2}ds\leq M\{\int_{0}^{1}||(\mathbf{I}-%
\mathbf{P)}f(s)||_{\nu }^{2}ds+\int_{0}^{1}||\mathbf{\{}I-P_{\gamma
}\}f(s)||_{\gamma _{+}}^{2}ds\}.  \label{diffusepositive}
\end{equation}
\end{proposition}

We first show that Proposition \ref{lower} implies Theorem \ref{L2decay}.

\begin{proof}
\textbf{of Theorem }\ref{L2decay}: For any solution $f$ to the linear
Boltzmann equation (\ref{lboltzmann}), $e^{\lambda t}f(t)$ satisfies 
\begin{equation}
\{\partial _{t}+v\cdot \nabla _{x}+L\}\{e^{\lambda t}f\}-\lambda e^{\lambda
t}f=0.  \label{lambdaf}
\end{equation}

Let $0\leq $ $N\leq t\leq N+1,$ $N$ being an integer. We split $[0,t]=$ $%
[0,N]$ $\cup \lbrack N,t].$

For the time interval $[N,t],$ since $f_{\gamma }\in L_{\text{loc }}^{2}(%
\mathbf{R}_{+};L^{2}(\gamma )),$ $\int_{N}^{t}||f(s)||_{\gamma
}^{2}ds<\infty .$ We then establish the $L^{2}$ energy estimate for $[N,t]$
as 
\begin{equation}
||f(t)||^{2}+\int_{N}^{t}(Lf,f)ds+\int_{N}^{t}\int_{\gamma
_{+}}f^{2}(s)d\gamma ds=||f(N)||^{2}+\int_{N}^{t}\int_{\gamma
_{-}}f^{2}(s)d\gamma ds.  \label{ntot}
\end{equation}

For the time interval $[0,N],$ (we may assume $N\geq 1$), since $f_{\gamma
}\in L_{\text{loc }}^{2}(\mathbf{R}_{+};L^{2}(\gamma )),$ $%
\int_{0}^{N}||f(s)||_{\gamma }^{2}ds<\infty .$ We multiply $e^{\lambda t}f$
with (\ref{lambdaf}) and take $L^{2}$ energy estimate over $0\leq s\leq N:$%
\begin{eqnarray*}
&&e^{2\lambda N}||f(N)||^{2}+\int_{0}^{N}e^{2\lambda s}(Lf,f)ds-\lambda
\int_{0}^{N}e^{2\lambda s}||f(s)||^{2}ds \\
&=&||f(0)||^{2}+\int_{0}^{N}\int_{\gamma _{-}}e^{2\lambda s}f^{2}(s)d\gamma
ds-\int_{0}^{N}\int_{\gamma _{+}}e^{2\lambda s}f^{2}(s)d\gamma ds.
\end{eqnarray*}%
Dividing the time interval into $\cup _{k=0}^{N-1}[k,k+1)$ and letting $%
f_{k}(s,x,v)\equiv f(k+s,x,v)$ for $k=0,1,2...N-1,$ we deduce 
\begin{eqnarray}
&&e^{2\lambda N}||f(N)||^{2}+\sum_{k=0}^{N-1}\int_{0}^{1}\left\{ e^{2\lambda
\{k+s\}}(Lf_{k},f_{k})-\lambda e^{2\lambda \{k+s\}}||f_{k}(s)||^{2}\right\}
ds  \label{L2energy} \\
&=&||f(0)||^{2}+\sum_{k=0}^{N-1}\left\{ \int_{0}^{1}\int_{\gamma
_{-}}e^{2\lambda \{k+s\}}f_{k}^{2}(s)d\gamma ds-\int_{0}^{1}\int_{\gamma
_{+}}e^{2\lambda \{k+s\}}f_{k}^{2}(s)d\gamma ds\right\} .  \notag
\end{eqnarray}%
Notice that $f_{k}(k+s,x,v)\,$\ satisfies the same linearized Boltzmann
equation (\ref{lboltzmann}) for $0\leq s\leq 1$.

\textbf{In-flow boundary condition (\ref{inflow}):} Multiplying $\delta
_{0}e^{2\lambda k}$ with (\ref{inflowpositive}) to each $f_{k}(s,x,v)$ and
then summing up over $k$ yields 
\begin{equation}
\frac{\delta _{0}}{2}\sum_{k=0}^{N-1}\{e^{2\lambda k}\int_{0}^{1}||\{\mathbf{%
I}-\mathbf{P}\}f_{k}||_{\nu }^{2}ds+e^{2\lambda k}\int_{0}^{1}\int_{\gamma
}f_{k}^{2}(s)d\gamma ds\}\geq \frac{\delta _{0}}{2M}\sum_{k=0}^{N-1}e^{2%
\lambda k}\int_{0}^{1}||\mathbf{P}f_{k}||_{\nu }^{2}ds.  \label{deltalower}
\end{equation}%
Note that $\int_{\gamma }=\int_{\gamma _{+}}+\int_{\gamma _{-}}$ , is the
total boundary integration. Since 
\begin{equation*}
(Lf_{k},f_{k})\geq \delta _{0}||\{\mathbf{I}-\mathbf{P}\}f_{k}||_{\nu }^{2}=%
\frac{\delta _{0}}{2}||\{\mathbf{I}-\mathbf{P}\}f_{k}||_{\nu }^{2}+\frac{%
\delta _{0}}{2}||\{\mathbf{I}-\mathbf{P}\}f_{k}||_{\nu }^{2}
\end{equation*}%
and $e^{2\lambda ks}\geq 1,$ we apply (\ref{deltalower}) to the first copy
of $\frac{\delta _{0}}{2}||\{\mathbf{I}-\mathbf{P}\}f_{k}||_{\nu }^{2}$ in (%
\ref{L2energy}), and move the boundary integrals in (\ref{deltalower}) to
the right hand side of (\ref{L2energy}). Hence, 
\begin{eqnarray*}
&&e^{2\lambda N}||f(N)||^{2}+\frac{\delta _{0}}{2M}\sum_{k=0}^{N-1}%
\int_{0}^{1}e^{2\lambda k}||\mathbf{P}f_{k}||_{\nu }^{2}ds+\frac{\delta _{0}%
}{2}\sum_{k=0}^{N-1}\int_{0}^{1}e^{2\lambda k}||\{\mathbf{I}-\mathbf{P}%
\}f_{k}||_{\nu }^{2}ds \\
&&-C_{\nu }\lambda \sum_{k=0}^{N-1}\int_{0}^{1}e^{2\lambda
\{k+s\}}||f_{k}(s)||_{\nu }^{2}ds \\
&\leq &||f(0)||^{2}+(1+\frac{\delta _{0}}{2})\sum_{k=0}^{N-1}\int_{0}^{1}%
\int_{\gamma _{-}}e^{2\lambda \{k+s\}}g_{k}^{2}(s)d\gamma ds-(1-\frac{\delta
_{0}}{2})\sum_{k=0}^{N-1}\int_{0}^{1}\int_{\gamma _{+}}e^{2\lambda
\{k+s\}}f_{k}^{2}(s)d\gamma ds.
\end{eqnarray*}%
Here we have used the fact $||\cdot ||\leq C_{\nu }||\cdot ||_{\nu }$ for
hard potentials, and the in-flow boundary condition $f_{k}=g_{k}$ on $\gamma
_{-}.$ Combining $\mathbf{P}f_{k}$ with $\{\mathbf{I}-\mathbf{P}\}f_{k},$
and note $e^{2\lambda k}=e^{2\lambda \{k+s\}}e^{-2\lambda s}\geq e^{2\lambda
\{k+s\}}e^{-2\lambda },$ we obtain a positive lower bound in the left hand
side: 
\begin{equation*}
\left( \min \{\frac{\delta _{0}}{2},\frac{\delta _{0}}{2M}\}e^{-2\lambda
}-C_{\nu }\lambda \right) \sum_{k=0}^{N-1}\int_{0}^{1}e^{2\lambda
\{k+s\}}||f_{k}(s)||_{\nu }^{2}ds>0
\end{equation*}%
for $C_{\nu }\lambda <\min \{\frac{\delta _{0}}{4},\frac{\delta _{0}}{4M}%
\}e^{-2\lambda }.$ Changing back to $f_{k}(t)=f(t+k)$ and letting $1-\frac{%
\delta _{0}}{2}>0,$ we deduce 
\begin{equation}
e^{2\lambda N}||f(N)||^{2}\leq ||f(0)||^{2}+(1+\frac{\delta _{0}}{2}%
)\int_{0}^{N}\int_{\gamma _{-}}e^{2\lambda s}g^{2}(s)d\gamma ds.
\label{ndecay}
\end{equation}

Notice that $e^{2\lambda t}\leq e^{2\lambda \{t-N\}}e^{2\lambda s}$ for $%
s\geq N$, and since $t\leq N+1,$ \ we can choose for $\delta _{0}$ and $%
\lambda $ small such that $e^{2\lambda (t-N)}(1+\frac{\delta _{0}}{2})\leq
2. $ Hence, multiplying $e^{2\lambda t}$ with (\ref{ntot}) and combining
with (\ref{ndecay}) yields 
\begin{eqnarray*}
&&e^{2\lambda t}||f(t)||^{2}+e^{2\lambda t}\int_{N}^{t}\int_{\gamma
_{+}}f^{2}(s)d\gamma ds\leq e^{2\lambda t}||f(N)||^{2}+e^{2\lambda
t}\int_{N}^{t}\int_{\gamma _{-}}g^{2}(s)d\gamma ds \\
&\leq &e^{2\lambda \{t-N\}}\{||f(0)||^{2}+(1+\frac{\delta _{0}}{2}%
)\int_{0}^{N}\int_{\gamma _{-}}e^{2\lambda s}g^{2}(s)d\gamma ds\} \\
&&+e^{2\lambda \{t-N\}}\int_{N}^{t}\int_{\gamma _{-}}e^{2\lambda
s}g^{2}(s)d\gamma ds\text{\ \ } \\
&\leq &2\{||f(0)||^{2}+\int_{0}^{t}\int_{\gamma _{-}}e^{2\lambda
s}g^{2}(s)d\gamma ds\}.
\end{eqnarray*}

\textbf{Bounce-back and specular reflections (\ref{bounceback}) and (\ref%
{specular})}. In both cases, the total boundary contribution in (\ref%
{L2energy}) vanishes: 
\begin{equation}
\sum_{k=0}^{N-1}\int_{0}^{1}\int_{\gamma _{-}}e^{2\lambda
\{k+s\}}f_{k}^{2}(s)d\gamma ds-\sum_{k=0}^{N-1}\int_{0}^{1}\int_{\gamma
_{+}}e^{2\lambda \{k+s\}}f_{k}^{2}(s)d\gamma ds=0.  \label{bvanish}
\end{equation}%
For $N\leq t<N+1,$ we use the same procedure as in the in-flow case (\ref%
{L2energy}) and the positivity (\ref{bbpositive}) to get 
\begin{equation*}
e^{2\lambda N}||f(N)||^{2}+\left( \min \{\frac{\delta _{0}}{2},\frac{\delta
_{0}}{2M}\}e^{-2\lambda }-C_{\nu }\lambda \right)
\sum_{k=0}^{N-1}\int_{0}^{1}e^{2\lambda \{k+s\}}||f_{k}(s)||_{\nu
}^{2}ds\leq ||f(0)||^{2}.
\end{equation*}%
For $C_{\nu }\lambda =\min \{\frac{\delta _{0}}{4},\frac{\delta _{0}}{4M}%
\}e^{-2\lambda }>0\,,$ changing back to the original $f(k+s)$ leads to%
\begin{equation}
e^{2\lambda N}||f(N)||^{2}\leq ||f(0)||^{2}.  \label{bbn}
\end{equation}%
By (\ref{ntot}) and (\ref{bvanish}), $||f(t)||\leq ||f(N)||.$ We deduce if $%
e^{2\lambda (t-N)}\leq 2$ 
\begin{equation*}
e^{2\lambda t}||f(t)||^{2}\leq e^{2\lambda t}||f(N)||^{2}\leq e^{2\lambda
(t-N)}||f(0)||^{2}\leq 2||f(0)||^{2}.
\end{equation*}

\textbf{Diffuse boundary reflection (\ref{diffuse}).} We note from (\ref%
{diffuse}) and\textbf{\ }(\ref{pboundary})\textbf{\ }that $\int_{\gamma
_{-}}f^{2}(s)d\gamma =\int_{\gamma _{+}}[P_{\gamma }f(s)]^{2}d\gamma ,$ so
that the boundary contribution in (\ref{L2energy}) is 
\begin{equation*}
\int_{0}^{1}\int_{\gamma _{-}}e^{2\lambda \{k+s\}}f(s)^{2}d\gamma
ds-\int_{0}^{1}\int_{\gamma _{+}}e^{2\lambda \{k+s\}}f^{2}(s)d\gamma
ds=-\int_{0}^{1}\int_{\gamma _{+}}e^{2\lambda s}[\{I-P_{\gamma
}\}f(s)]^{2}d\gamma ds.
\end{equation*}%
By the same procedure, we obtain from (\ref{L2energy}) and the positivity (%
\ref{diffusepositive}):%
\begin{eqnarray*}
&&e^{2\lambda N}||f(N)||^{2}+\left( \min \{\frac{\delta _{0}}{2},\frac{%
\delta _{0}}{2M}\}e^{-2\lambda }-C_{\nu }\lambda \right)
\sum_{k=0}^{N-1}\int_{0}^{1}e^{2\lambda \{k+s\}}||f_{k}(s)||_{\nu }^{2}ds \\
&\leq &||f(0)||^{2}-(1-\frac{\delta _{0}}{2})\sum_{k=0}^{N-1}\int_{0}^{1}%
\int_{\gamma _{+}}e^{2\lambda \{k+s\}}[\{I-P_{\gamma }\}f(s)]^{2}d\gamma ds.
\end{eqnarray*}%
For $\frac{\delta _{0}}{2}<1$ and $C_{\nu }\lambda <\min \{\frac{\delta _{0}%
}{4},\frac{\delta _{0}}{4M}\}\,e^{-2\lambda }$, we have 
\begin{equation*}
e^{2\lambda N}||f(N)||^{2}\leq ||f(0)||^{2}.
\end{equation*}%
Since for $[N,t],$ we have $||f(t)||^{2}\leq ||f(N)||^{2},$ from (\ref{ntot}%
). We therefore conclude the proposition for $e^{2\lambda (t-N)}\leq 2.$
\end{proof}

\subsection{Strategy for the Proof of Prop. \protect\ref{lower}}

The rest of this section is devoted entirely to the proof of the crucial
Proposition \ref{lower}. The proof of Proposition \ref{lower} is based on a
contradiction argument. If Proposition \ref{lower} were false, then there
are no $M$ exists as in Proposition \ref{lower} for every linear Boltzmann
solution. Hence, for any $k\geq 1,$ there exists a sequence of non-zero
solutions $f_{k}(t,x,v)$ to the linearized Boltzmann equation (\ref%
{lboltzmann}) to satisfy one of the following:

(1) \textbf{In the in-flow case:} $f_{k}$ satisfies (\ref{inflow}) and

\begin{equation*}
\int_{0}^{1}||\mathbf{P}f_{k}(s)||_{\nu }^{2}ds\geq k\{\int_{0}^{1}||(%
\mathbf{I}-\mathbf{P)}f_{k}(s)||_{\nu
}^{2}ds+\int_{0}^{1}||f_{k}(s)||_{\gamma }^{2}ds\}.
\end{equation*}%
Equivalently, in terms of normalization $Z_{k}(t,x,v)\equiv \frac{%
f_{k}(t,x,v)}{\sqrt{\int_{0}^{1}||\mathbf{P}f_{k}(s)||_{\nu }^{2}ds}},$ we
have 
\begin{equation}
\int_{0}^{1}||\mathbf{P}Z_{k}(s)||_{\nu }^{2}ds\equiv 1,  \label{1}
\end{equation}%
and 
\begin{equation}
\int_{0}^{1}||(\mathbf{I}-\mathbf{P)}Z_{k}(s)||_{\nu
}^{2}ds+\int_{0}^{1}||Z_{k}(s)||_{\gamma }^{2}ds\leq \frac{1}{k}.
\label{1/n}
\end{equation}%
We also have from $[\partial _{t}+v\cdot \nabla _{x}+L]f_{k}=0,$%
\begin{equation}
\lbrack \partial _{t}+v\cdot \nabla _{x}+L]Z_{k}=0.  \label{zn}
\end{equation}

(2) \ \textbf{In the bounce-back case:} $f_{k}$ satisfies (\ref{bounceback}%
), the mass-energy conservation laws (\ref{mass}) and (\ref{energy}), and 
\begin{equation}
\int_{0}^{1}||\mathbf{P}f_{k}(s)||^{2}ds\geq k\int_{0}^{1}||(\mathbf{I}-%
\mathbf{P)}f_{k}(s)||_{\nu }^{2}ds.  \label{bbblowup}
\end{equation}%
Hence, the normalized $Z_{k}$ satisfies (\ref{1}), (\ref{zn}), and 
\begin{equation}
\int_{0}^{1}||(\mathbf{I}-\mathbf{P)}Z_{k}(s)||_{\nu }^{2}ds\leq \frac{1}{k}
\label{bb1/n}
\end{equation}

(3) \textbf{In the specular reflection case:} $f_{k}$ satisfies (\ref%
{specular}), and the mass and energy conservation laws (\ref{mass}), (\ref%
{energy}). (If the domain $\Omega $ has any axis of rotation symmetry (\ref%
{axis}) then $f_{k}$ also satisfies (\ref{axiscon})). \ We note that (\ref%
{bbblowup}), (\ref{1}), (\ref{zn}) and (\ref{bb1/n}) are all valid for the
normalized $Z_{k}.$

\QTP{Body Math}
(4) \textbf{In the diffusive reflection case:} \ $f_{k}$ satisfies (\ref%
{diffuse}), and the mass conservation (\ref{mass}), (\ref{zn}), and 
\begin{equation}
\int_{0}^{1}||\mathbf{P}f_{k}(s)||_{\nu }^{2}ds\geq k\{\int_{0}^{1}||(%
\mathbf{I}-\mathbf{P)}f_{k}(s)||_{\nu }^{2}ds+\int_{0}^{1}||\{I-P_{\gamma
}\}f_{k}(s)||_{\gamma _{+}}^{2}ds\}.  \label{diffuseblowup}
\end{equation}%
The normalized $Z_{k}$ satisfies (\ref{1}) and 
\begin{equation}
\int_{0}^{1}||(\mathbf{I}-\mathbf{P)}Z_{k}(s)||_{\nu
}^{2}ds+\int_{0}^{1}||\{I-P_{\gamma }\}Z_{k}(s)||_{\gamma _{+}}^{2}ds\leq 
\frac{1}{k}.  \label{diffuse1/n}
\end{equation}

\QTP{Body Math}
In all four cases, there exists $Z(t,x,v)$ such that 
\begin{equation*}
Z_{k}\rightarrow Z\text{ weakly in }\int_{0}^{1}||\cdot ||_{\nu }^{2}ds,
\end{equation*}%
since $\sup_{k}\int_{0}^{1}||Z_{k}(s)||_{\nu }^{2}ds$\TEXTsymbol{<}+$\infty
, $ and from (\ref{1/n}), (\ref{bb1/n}), (\ref{diffuse1/n}) that 
\begin{equation}
\int_{0}^{1}||(\mathbf{I}-\mathbf{P)}Z_{k}(s)||_{\nu }^{2}ds\rightarrow 0.
\label{i-pto0}
\end{equation}%
Notice that it is straightforward to verify 
\begin{equation*}
\mathbf{P}Z_{k}\rightarrow \mathbf{P}Z\text{ weakly in }\int_{0}^{1}||\cdot
||_{\nu }^{2}ds.
\end{equation*}%
Therefore $(\mathbf{I}-\mathbf{P)}Z_{k}\rightarrow (\mathbf{I}-\mathbf{P)}Z$
weakly, and $(\mathbf{I}-\mathbf{P)}Z=0$ from (\ref{i-pto0}) so that 
\begin{equation}
Z(t,x,v)=\{a(t,x)+v\cdot b(t,x)+|v|^{2}c(t,x)\}\sqrt{\mu }.  \label{zabc}
\end{equation}%
Note $LZ_{k}=L(\mathbf{I}-\mathbf{P)}Z_{k}$ and $\int_{0}^{1}||(\mathbf{I}-%
\mathbf{P)}Z_{k}(s)||_{\nu }^{2}ds\rightarrow 0$ in four cases. Letting $%
k\rightarrow \infty $ in (\ref{zn}), we have, in the sense of distributions, 
\begin{equation}
\partial _{t}Z+v\cdot \nabla _{x}Z=0.  \label{zvlasov}
\end{equation}%
The main strategy is to show, on the one hand, $Z$ has to be zero from (\ref%
{i-pto0}) and one of the inherited boundary conditions (\ref{inflow}), (\ref%
{bounceback}), (\ref{specular}) and (\ref{diffuse}). On the other hand, $%
Z_{k}$ will be shown to converge strongly to $Z$ in $\int_{0}^{1}||\cdot
||^{2}ds,$ and $\int_{0}^{1}||Z||_{{}}^{2}ds\neq 0.$ This leads to a
contradiction.

\subsection{The Limit Function $Z(t,x,v)$}

\begin{lemma}
\label{limit} There exists constants $a_{0},c_{1},c_{2},$ and constant
vectors $b_{0},b_{1}$ and $\varpi $ such that $Z(t,x,v)$ takes the form:%
\begin{equation}
\left( \{\frac{c_{0}}{2}|x|^{2}-b_{0}\cdot
x+a_{0}\}+\{-c_{0}tx-c_{1}x+\varpi \times x+b_{0}t+b_{1}\}\cdot v+\{\frac{%
c_{0}t^{2}}{2}+c_{1}t+c_{2}\}|v|^{2}\right) \sqrt{\mu }.  \label{zlimit}
\end{equation}%
Moreover, these constants are finite:%
\begin{equation}
|a_{0}|+|c_{0}|+|c_{1}|+|c_{2}|+|b_{0}|+|b_{1}|+|\varpi |<+\infty .
\label{constantbound}
\end{equation}
\end{lemma}

\begin{proof}
We first derive deduce (\ref{zlimit}). Notice that by plugging (\ref{zabc})
into (\ref{zvlasov}) and comparing coefficients in front of $\sqrt{\mu },v%
\sqrt{\mu },|v|^{2}\sqrt{\mu },$ we deduce the macroscopic equations with $%
b=(b^{1},b^{2},b^{3})$: 
\begin{eqnarray}
\partial _{x_{i}}c &=&0,\text{ for }i=1,2,3  \label{c} \\
\partial _{t}c+\partial _{x_{i}}b^{i} &=&0,\text{ for }i=1,2,3  \label{ct} \\
\partial _{x_{j}}b^{i}+\partial _{x_{i}}b^{j} &=&0,\text{ for }i\neq j,
\label{b} \\
\partial _{x_{i}}a+\partial _{t}b^{i} &=&0,\text{ for }i=1,2,3  \label{bt} \\
\partial _{t}a &=&0.  \label{at}
\end{eqnarray}%
Since $\Omega $ is simply connected, from (\ref{c}), $c(t,x)\equiv c(t).$
Similarly, from (\ref{ct}), 
\begin{eqnarray*}
b^{1}(t,x) &=&-\partial _{t}c(t)x_{1}+\tilde{b}^{1}(t,x_{2},x_{3}), \\
b^{2}(t,x) &=&-\partial _{t}c(t)x_{2}+\tilde{b}^{2}(t,x_{1},x_{3}), \\
b^{3}(t,x) &=&-\partial _{t}c(t)x_{3}+\tilde{b}^{3}(t,x_{1},x_{2}).
\end{eqnarray*}

To determine $\tilde{b}^{1},$ we first make use of (\ref{b}) to get 
\begin{equation*}
\partial _{x_{2}}\tilde{b}^{1}(t,x_{2},x_{3})+\partial _{x_{1}}\tilde{b}%
^{2}(t,x_{1},x_{3})=0
\end{equation*}%
so that $\partial _{x_{2}}^{2}\tilde{b}^{1}(t,x_{2},x_{3})=0$. Therefore $%
\tilde{b}^{1}$ is linear with respect to $x_{2},$ and 
\begin{equation}
\tilde{b}^{1}(t,x_{2},x_{3})=j^{1}(t,x_{3})x_{2}+g^{1}(t,x_{3}).  \label{b1}
\end{equation}%
Similarly, we also have $\partial _{x_{1}}^{2}\tilde{b}^{2}(t,x_{1},x_{3})=0$
so that 
\begin{equation}
\tilde{b}^{2}(t,x_{1},x_{3})=-j^{1}(t,x_{3})x_{1}+g^{2}(t,x_{3}).  \label{b2}
\end{equation}

Next, we make use of another equation of (\ref{b}) to get 
\begin{equation*}
\partial _{x_{3}}\tilde{b}^{2}(t,x_{1},x_{3})+\partial _{x_{2}}\tilde{b}%
^{3}(t,x_{1},x_{2})=0
\end{equation*}%
with $\partial _{x_{3}}\tilde{b}^{2}(t,x_{1},x_{3})=-\partial
_{x_{3}}j^{1}(t,x_{3})x_{1}+\partial _{x_{3}}g^{2}(t,x_{3}).$ From $\partial
_{x_{2}}^{2}\tilde{b}^{3}(t,x_{1},x_{2})=0,$ we have 
\begin{equation}
\tilde{b}^{3}=j^{3}(t,x_{1})x_{2}+g^{3}(t,x_{1})  \label{b3}
\end{equation}%
so that $-\partial _{x_{3}}j^{1}(t,x_{3})x_{1}+\partial
_{x_{3}}g^{2}(x_{3})+j^{3}(t,x_{1})x_{2}=0.$ Taking one more $x_{3}$
derivative, we get 
\begin{equation*}
-\partial _{x_{3}}^{2}j^{1}(t,x_{3})x_{1}+\partial
_{x_{3}}^{2}g^{2}(t,x_{3})=0
\end{equation*}%
so that $\partial _{x_{3}}^{2}j^{1}(t,x_{3})=\partial
_{x_{3}}^{2}g^{2}(t,x_{3})\equiv 0.$ Furthermore, taking two more $x_{1}$
derivatives, we have $\partial _{x_{1}}^{2}j^{3}(t,x_{1})=0.$ Hence $j^{1}$
and $g^{2}$ can be expressed as 
\begin{eqnarray*}
j^{1}(t,x_{3}) &=&l^{1}(t)x_{3}+h^{1}(t), \\
j^{3}(t,x_{1}) &=&l^{1}(t)x_{1}-h^{2}(t), \\
g^{2}(t,x_{3}) &=&h^{2}(t)x_{3}+m^{2}(t).
\end{eqnarray*}%
Plugging back into (\ref{b1}), (\ref{b2}) and (\ref{b3}), we deduce 
\begin{eqnarray}
\tilde{b}^{1} &=&(l^{1}(t)x_{3}+h^{1}(t))x_{2}+g^{1}(t,x_{3}),  \notag \\
\tilde{b}^{2} &=&-(l^{1}(t)x_{3}+h^{1}(t))x_{1}+h^{2}(t)x_{3}+m^{2}(t),
\label{btilde} \\
\tilde{b}^{3} &=&l^{1}(t)x_{1}x_{2}-h^{2}(t)x_{2}+g^{3}(t,x_{1}).  \notag
\end{eqnarray}

Finally, from the remaining equation in (\ref{b}), 
\begin{equation*}
\partial _{x_{3}}\tilde{b}^{1}(t,x_{1},x_{3})+\partial _{x_{1}}\tilde{b}%
^{3}(t,x_{1},x_{2})=0.
\end{equation*}%
By (\ref{btilde}), $l^{1}x_{2}+\partial
_{x_{3}}g^{1}(t,x_{3})+l^{1}x_{2}+\partial _{x_{1}}g^{3}(t,x_{1})=0.$ Hence $%
l^{1}\equiv 0$, $g^{1}$ is a linear function of $x_{3}$ and $g^{3}$ is a
linear function of $x_{1}$: 
\begin{eqnarray*}
g^{3}(t,x_{1}) &=&h^{3}(t)x_{1}+m^{3}(t) \\
g^{1}(t,x_{3}) &=&-h^{3}(t)x_{3}+m^{1}(t).
\end{eqnarray*}%
Therefore, letting $\varpi (t)=-[h^{2}(t),h^{3}(t),h^{1}(t)]$ and $%
m(t)=[m^{1}(t),m^{2}(t),m^{3}(t)],$ we deduce from a direct computation that 
\begin{equation*}
b=-c^{\prime }(t)x+\varpi (t)\times x+m(t).
\end{equation*}%
We also have $\partial _{t}^{2}b(t,x)\equiv 0$ from (\ref{bt}) and (\ref{at}%
). Hence $\partial _{t}^{3}c(t)=0$ from (\ref{ct})$,$ and $c^{\prime
}(t)=c_{0}t+c_{1}$ so that 
\begin{equation*}
c=\frac{c_{0}t^{2}}{2}+c_{1}t+c_{2}.
\end{equation*}%
Hence $\varpi ^{\prime \prime }(t)\times x+m^{\prime \prime }(t)\equiv 0$
and $\varpi ^{\prime \prime }(t)=m^{\prime \prime }(t)\equiv 0.$ We can
denote 
\begin{equation*}
b=-\{c_{0}t+c_{1}\}x+\{\varpi ^{\prime }(0)t+\varpi \}\times x+b_{0}t+b_{1}.
\end{equation*}%
where $\varpi $ is a constant vector. Moreover, from (\ref{bt}), $\nabla
\times \partial _{t}b\equiv 0$ so that 
\begin{equation*}
\nabla \times \{-c_{0}x+\varpi ^{\prime }(0)\times x\}\equiv 0.
\end{equation*}%
This implies that $\varpi ^{\prime }(0)=0$ and from (\ref{bt}) again, 
\begin{equation*}
a=\frac{c_{0}|x|^{2}}{2}-b_{0}\cdot x+a_{0}.
\end{equation*}

Lastly, to prove (\ref{constantbound}), we note that for $1\leq i,j\leq 3,$
functions 
\begin{equation*}
|x|^{2}\sqrt{\mu },x_{i}\sqrt{\mu },\sqrt{\mu },tx\cdot v\sqrt{\mu },x\cdot v%
\sqrt{\mu },x\times v\sqrt{\mu },tv\sqrt{\mu },v\sqrt{\mu },t^{2}|v|^{2}%
\sqrt{\mu },t|v|^{2}\sqrt{\mu },|v|^{2}\sqrt{\mu }
\end{equation*}%
are linearly independent. Therefore, their coefficients $%
c_{0},c_{1},c_{2},a_{0},b_{0},b_{1},\varpi $ are bounded by $C\left\{
\int_{0}^{1}||Z(s)||^{2}ds\right\} ^{1/2},$ which is finite.
\end{proof}

\subsection{Interior Compactness}

\begin{lemma}
\label{interior}For any smooth function $\chi (t,x)$ such that $\sup \chi
\subset \subset (0,1)\times \Omega ,$ then up to a subsequence, $%
\lim_{k\rightarrow \infty }\int_{0}^{1}||\chi \{Z_{k}-Z\}(s)||^{2}ds=0.$
\end{lemma}

\begin{proof}
We multiply the equation (\ref{zn}) by $\chi $ to get%
\begin{equation*}
\lbrack \partial _{t}+v\cdot \nabla _{x}]\{\chi Z_{k}\}=\{[\partial
_{t}+v\cdot \nabla _{x}]\chi \}Z_{k}-\chi LZ_{k}.
\end{equation*}%
Since $\ \int_{0}^{1}||Z_{k}(s)||_{{}}^{2}ds$ is uniformly bounded for the
hard potentials, by (\ref{1}) and (\ref{i-pto0}), we deduce from the
Averaging Lemma [DL], $\int \chi (t,x)Z_{k}(v)\chi _{v}(v)dv$ are compact in 
$L^{2}([0,1]\times \Omega )$ for any smooth cutoff function $\chi _{v}(v)$
(see [G1]). It then follows that 
\begin{equation*}
\int \chi Z_{k}(v)[1,v,|v|^{2}]\sqrt{\mu }dv
\end{equation*}%
are compact in $L^{2}([0,1]\times \Omega ).$ Therefore, up to a subsequence,
the macroscopic parts of $Z_{k}$ satisfy $\chi \mathbf{P}Z_{k}\rightarrow
\chi \mathbf{P}Z=\chi Z$ strongly in $L^{2}([0,1]\times \Omega \times 
\mathbf{R}^{3}).$ Therefore, in light of $\int_{0}^{1}||(\mathbf{I}-\mathbf{%
P)}Z_{k}(s)||_{\nu }^{2}ds\rightarrow 0$ in (\ref{i-pto0}) for all four
boundary conditions, the remaining microscopic parts $\chi Z_{k}$ satisfy $%
\lim_{k\rightarrow \infty }\int_{0}^{1}||\chi \{\mathbf{I}-\mathbf{P}%
\}Z_{k}(s)||^{2}ds=0,$ and our lemma follows.
\end{proof}

\subsection{No Time Concentration}

We first establish $L^{\infty }$ in time estimate for $Z_{k}$ to rule out
possible concentration in $\,$time, near either $t=0$ or $t=1$.

\begin{lemma}
\label{zinfty} $\sup_{0\leq t\leq 1,k\geq 1}||Z_{k}(t)||<\infty .$
\end{lemma}

\begin{proof}
Since $\int_{0}^{1}||f_{k}(s)||_{\gamma }^{2}<\infty ,$ $%
\int_{0}^{1}||Z_{k}(s)||_{\gamma }^{2}<\infty .$ Therefore, by the standard $%
L^{2}$ estimate for (\ref{zn}), we obtain for $0\leq t\leq 1$:%
\begin{eqnarray}
&&||Z_{k}(t)||^{2}+\int_{0}^{t}||Z_{k}(s)||_{\gamma
_{+}}^{2}ds+2\int_{0}^{t}(LZ_{k},Z_{k})(s)ds  \notag \\
&=&||Z_{k}(0)||^{2}+\int_{0}^{t}||Z_{k}(s)||_{\gamma _{-}}^{2}ds.
\label{znenergy}
\end{eqnarray}

We first derive an upper bound for $Z_{k}(t).$ In the case of in flow case (%
\ref{inflow}), because of (\ref{1/n}), (\ref{1}) and $L\geq 0$, we deduce 
\begin{equation}
||Z_{k}(t)||^{2}\leq ||Z_{k}(0)||^{2}+\frac{1}{k}.  \label{z0bound}
\end{equation}

Note $\int_{0}^{t}||Z_{k}(s)||_{\gamma
_{+}}^{2}ds=\int_{0}^{t}||Z_{k}(s)||_{\gamma _{-}}^{2}ds$ for either
bounce-back or specular reflection (\ref{bounceback}) and (\ref{specular}),
hence (\ref{z0bound}) is clearly valid. In the case of diffuse reflection (%
\ref{diffuse}), we deduce (\ref{z0bound}) because 
\begin{equation*}
\int_{0}^{t}||Z_{k}(s)||_{\gamma _{-}}^{2}ds=\int_{0}^{t}||P_{\gamma
}Z_{k}(s)||_{\gamma _{+}}^{2}ds\leq \int_{0}^{t}||Z_{k}(s)||_{\gamma
_{+}}^{2}ds.
\end{equation*}

Next, we derive an upper bound for $Z_{k}(0).$ We note that 
\begin{equation*}
\int_{0}^{1}(LZ_{k}(t),Z_{k}(t))dt\leq C\int_{0}^{1}||\{\mathbf{I}-\mathbf{P}%
\}Z_{k}||_{\nu }^{2}dt\leq \frac{C}{k}.
\end{equation*}

In the case of the in-flow case (\ref{inflow}), by (\ref{1/n}) and (\ref%
{znenergy}), 
\begin{eqnarray}
||Z_{k}(t)||^{2} &\geq &||Z_{k}(0)||^{2}-\int_{0}^{1}||Z_{k}(s)||_{\gamma
_{+}}^{2}ds-\int_{0}^{1}(LZ_{k},Z_{k})(s)ds  \notag \\
&\geq &||Z_{k}(0)||^{2}-\frac{C}{k}.  \label{z0low}
\end{eqnarray}

Note that $\int_{0}^{t}||Z_{k}(s)||_{\gamma
_{+}}^{2}ds=\int_{0}^{t}||Z_{k}(s)||_{\gamma _{-}}^{2}ds$ for either
bounce-back and specular reflection (\ref{bounceback}) or (\ref{specular}),
so that (\ref{z0low}) is clearly valid. In the case of diffuse reflection (%
\ref{diffuse}), (\ref{z0low}) is valid because of (\ref{diffuse1/n}): 
\begin{equation*}
\int_{0}^{t}||Z_{k}(s)||_{\gamma
_{-}}^{2}ds-\int_{0}^{t}||Z_{k}(s)||_{\gamma
_{+}}^{2}ds=-\int_{0}^{t}||\{I-P_{\gamma }\}Z_{k}(s)||_{\gamma
_{+}}^{2}ds\geq -\frac{1}{k}.
\end{equation*}

Since $\int_{0}^{1}||Z_{k}(t)||^{2}dt\leq C\int_{0}^{1}||Z_{k}(t)||_{\nu
}^{2}dt\leq C\{1+\frac{1}{k}\}$ for hard potentials, integrating (\ref{z0low}%
) over $0\leq t\leq 1$ yields 
\begin{eqnarray*}
||Z_{k}(0)||^{2} &\leq &\int_{0}^{1}||Z_{k}(t)||^{2}dt+\frac{C}{k} \\
&\leq &C\int_{0}^{1}||Z_{k}(t)||_{\nu }^{2}dt+\frac{C}{k} \\
&\leq &C\{1+\frac{1}{k}\}+\frac{C}{k},
\end{eqnarray*}%
by (\ref{1}) and (\ref{i-pto0}). Our lemma thus follows from (\ref{z0bound}).
\end{proof}

\subsection{No Boundary Concentration}

The most delicate step is to prove that there is no concentration at the
boundary $\partial \Omega $ so that $Z_{k}\rightarrow Z$ strongly in $%
[0,1]\times \bar{\Omega}\times \mathbf{R}^{3}.\,\ $Let 
\begin{equation*}
\Omega _{\varepsilon ^{4}}\equiv \{x\in \Omega :\xi (x)<-\varepsilon ^{4}\}.
\end{equation*}%
To this end, we will establish a careful energy estimate in the thin
shell-like region near the boundary $[0,1]\times \{\Omega \setminus \Omega
_{\varepsilon ^{4}}\}\times \mathbf{R}^{3}.$

Recall $n(x)=\frac{\nabla \xi (x)}{|\nabla \xi (x)|}\neq 0,$ well-defined
and smooth on $\Omega \setminus \Omega _{\varepsilon ^{4}}$ for $\varepsilon 
$ small. For $m>1/2,$ for any $(x,v),$ we define the outward moving (inward
moving) indicator function $\chi _{+}$ ($\chi _{-}$) as 
\begin{eqnarray*}
\chi _{+}(x,v) &=&\mathbf{1}_{\Omega \setminus \Omega _{\varepsilon ^{4}}}(x)%
\mathbf{1}_{\{|v|\leq \varepsilon ^{-m},n(x)\cdot v>\varepsilon \}}(v) \\
\chi _{-}(x,v) &=&\mathbf{1}_{\Omega \setminus \Omega _{\varepsilon ^{4}}}(x)%
\mathbf{1}_{\{|v|\leq \varepsilon ^{-m},n(x)\cdot v<-\varepsilon \}}(v).
\end{eqnarray*}

Our main strategy is to show that the moving (non-grazing) part $\chi _{\pm
}Z_{k}$ are controlled by the inner boundary values of $Z_{k}$ on $\partial
\Omega _{\varepsilon ^{4}}=\{\xi (x)=-\varepsilon ^{4}\},$ which are further
controlled by the (compact!) interior parts of $Z_{k}.$ Hence, no
concentration is possible. On the remaining almost grazing part $\{1-\chi
_{\pm }\}Z_{k}$, thanks to the fact $\int_{0}^{1}||\{\mathbf{I}-\mathbf{P}%
\}Z_{k}(s)||_{\nu }^{2}ds\rightarrow 0,$ no concentration can occur for the
small velocity set $\{|v|\geq \varepsilon ^{-m}\}\cup \{|n(x)\cdot v|\leq
\varepsilon \}$.

\begin{lemma}
\bigskip \label{nomove}%
\begin{equation}
\sup_{k\geq 1}\int_{0}^{1}\int_{\Omega \setminus \Omega _{\varepsilon
^{4}}}\int_{\substack{ |n(x)\cdot v|\leq \varepsilon  \\ \text{or }|v|\geq
\varepsilon ^{-m}}}|Z_{k}(s,x,v)|^{2}dxdvds\leq C\varepsilon .
\label{smallnomove}
\end{equation}
\end{lemma}

\begin{proof}
Let $\mathbf{P}Z_{k}=\{a_{k}(t,x)+v\cdot b_{k}(t,x)+|v|^{2}c_{k}(t,x)\}\sqrt{%
\mu }.$ Since $\sup_{k}\int_{0}^{1}||Z_{k}(s)||^{2}ds~$is finite and $%
[1,v,|v|^{2}]\sqrt{\mu }$ are linearly independent, there is $C>0$
(independent of $k$) such that 
\begin{equation}
\int_{0}^{1}||a_{k}(s)||^{2}ds+\int_{0}^{1}||b_{k}(s)||^{2}ds+%
\int_{0}^{1}||c_{k}(s)||^{2}ds\leq C\int_{0}^{1}||Z_{k}(s)||^{2}ds\leq C.
\label{anbound}
\end{equation}%
By (\ref{i-pto0}), we can split: 
\begin{eqnarray*}
&&\int_{0}^{1}\int_{\Omega \setminus \Omega _{\varepsilon ^{4}}}\int 
_{\substack{ |n(x)\cdot v|\leq \varepsilon  \\ \text{or }|v|\geq \varepsilon
^{-m}}}|Z_{k}(s,x,v)|^{2}dxdvds \\
&\leq &\int_{\substack{ |n(x)\cdot v|\leq \varepsilon  \\ \text{or }|v|\geq
\varepsilon ^{-m}}}|\mathbf{P}Z_{k}(s,x,v)|^{2}+\int_{\substack{ |n(x)\cdot
v|\leq \varepsilon  \\ \text{or }|v|\geq \varepsilon ^{-m}}}|\{\mathbf{I}-%
\mathbf{P\}}Z_{k}(s,x,v)|^{2} \\
&\leq &\int_{0}^{1}\int_{\Omega \setminus \Omega _{\varepsilon ^{4}}}\int 
_{\substack{ |n(x)\cdot v|\leq \varepsilon  \\ \text{or }|v|\geq \varepsilon
^{-m}}}|\mathbf{P}Z_{k}(s,x,v)|^{2}dxdvds+\frac{C}{k}
\end{eqnarray*}%
Even with an extra weight $\{1+|v|^{2}\}^{l}$ ($l\geq 0$)$,$ the first term
can be bounded by the Fubini Theorem as 
\begin{eqnarray}
&&\int_{\substack{ |n(x)\cdot v|\leq \varepsilon  \\ \text{or }|v|\geq
\varepsilon ^{-m}}}\{1+|v|^{2}\}^{l}|\mathbf{P}Z_{k}(s,x,v)|^{2}dxdvds 
\notag \\
&\leq &\int_{0}^{1}\int_{\Omega \setminus \Omega _{\varepsilon
^{4}}}\{|a_{k}^{2}(s,x)|+|b_{k}^{2}(s,x)|+|c_{k}^{2}(s,x)|\}\times  \notag \\
&&\times \{\int_{\substack{ |n(x)\cdot v|\leq \varepsilon  \\ \text{or }%
|v|\geq \varepsilon ^{-m}}}\{1+|v|^{2}\}^{l+2}\mu dv\}dxds.  \label{pzn}
\end{eqnarray}%
We note that the inner $v-$integral above is bounded, uniformly in $x.$ In
fact, by a change of variable $v_{||}=\{n(x)\cdot v\}n(x),$ and $v_{\perp
}=v-v_{||}$ for $|n(x)\cdot v|\leq \varepsilon ,$ the inner integral is
bounded by the sum of 
\begin{eqnarray}
\int_{|n(x)\cdot v|\leq \varepsilon }\{1+|v|^{2}\}^{l+2}\mu dv &\leq
&C\int_{-\varepsilon }^{\varepsilon }dv_{||}\int_{\mathbf{R}%
^{2}}e^{-|v_{\perp }|^{2}/8}dv_{\perp }\leq C\varepsilon ,  \label{nvsmall}
\\
\text{ and \ \ }\int_{\text{ }|v|\geq \varepsilon
^{-m}}\{1+|v|^{2}\}^{l+2}\mu dv &\leq &C\varepsilon .  \notag
\end{eqnarray}%
Our lemma thus follows from (\ref{anbound}).
\end{proof}

To study the non-grazing parts $\chi _{\pm }Z_{k},$ we fix $(x,v)\in
\{\Omega \setminus \Omega _{\varepsilon ^{4}}\}\times \mathbf{R}^{3}$ and
any moment $s$ such that $\varepsilon \leq s\leq 1-\varepsilon ,$ and $.$ We
define for backward in time $0\leq t\leq s:$ 
\begin{equation}
\tilde{\chi}_{+}(t,x,v)=\mathbf{1}_{\Omega \setminus \Omega _{\varepsilon
^{4}}}(x-v\{t-s\})\mathbf{1}_{\{|v|\leq \varepsilon ^{-m},n(x-v\{t-s\})\cdot
v>\varepsilon \}}(v);\text{ }  \label{chi+}
\end{equation}%
and for forward in time $0\leq s\leq t:$%
\begin{equation}
\tilde{\chi}_{-}(t,x,v)=\mathbf{1}_{\Omega \setminus \Omega _{\varepsilon
^{4}}}(x-v\{t-s\})\mathbf{1}_{\{|v|\leq \varepsilon ^{-m},n(x-v\{t-s\})\cdot
v<-\varepsilon \}}(v).  \label{chi-}
\end{equation}%
Both $\tilde{\chi}_{\pm }$ solve the transport equations: 
\begin{equation}
\partial _{t}\tilde{\chi}_{\pm }+v\cdot \nabla _{x}\tilde{\chi}_{\pm }=0,%
\text{ \ \ \ }\tilde{\chi}_{\pm }(s,x,v)=\chi _{\pm }(x,v).
\label{vlasovchi}
\end{equation}%
We first prove that

\begin{lemma}
\label{chi} (1) For $0\leq s-\varepsilon ^{2}\leq t\leq s,$ if $\tilde{\chi}%
_{+}(t,x,v)\neq 0$ then $n(x)\cdot v>\frac{\varepsilon }{2}>0.$ Moreover, $%
\tilde{\chi}_{+}(s-\varepsilon ^{2},x,v)\equiv 0,$ for $x\in \Omega
\setminus \Omega _{\varepsilon ^{4}}.$ \ 

(2) For $s\leq t\leq s+\varepsilon ^{2}\leq 1,$ if $\tilde{\chi}%
_{-}(t,x,v)\neq 0,$ then $n(x)\cdot v<-\frac{\varepsilon }{2}<0.$ Moreover, $%
\tilde{\chi}_{-}(s+\varepsilon ^{2},x,v)\equiv 0,$ for $x\in \Omega
\setminus \Omega _{\varepsilon ^{4}}.$
\end{lemma}

\begin{proof}
It suffices to prove (1), the proof for (2) being exactly the same. First of
all, by (\ref{chi+}), if $\tilde{\chi}_{+}(t,x,v)\neq 0,\,$\ then $%
x-v(t-s)\in \Omega \setminus \Omega _{\varepsilon ^{4}},$ $%
n(x-v\{t-s\})\cdot v>\varepsilon ,$ and $|v|\leq \varepsilon ^{-m}.$ Hence
for $|t-s|\leq \varepsilon ^{2},$ for any $0\leq \theta \leq 1,$%
\begin{equation}
|x-\theta v(t-s)-\{x-v(t-s)\}|\leq \varepsilon ^{-m}\varepsilon ^{2}\leq
\varepsilon ,  \label{x}
\end{equation}%
for $2m<1.$ Therefore $x-\theta v(t-s)$ is also near $\partial \Omega $ and $%
|\nabla n(x-\theta v(t-s))|$ is uniformly bounded. Now, 
\begin{eqnarray}
n(x)\cdot v &=&n(x-v\{t-s\})\cdot v-\{n(x-v\{t-s\})-n(x)\}\cdot v  \notag \\
&>&\varepsilon -\sup_{0\leq \theta \leq 1}|\nabla n(x-\theta
v\{t-s\})|\times |t-s||v|^{2}.  \notag \\
&\geq &\varepsilon -C\varepsilon ^{-2m+2}  \notag \\
&=&\varepsilon \lbrack 1-C\varepsilon ^{-2m+1}]\geq \frac{\varepsilon }{2},
\label{nv1}
\end{eqnarray}%
for $2m<1.$ We thus conclude the first assertion.

To prove the second assertion, let $x\in \Omega \setminus \Omega
_{\varepsilon ^{4}}$ so that $-\varepsilon ^{4}\leq \xi (x)<0.$ If $\tilde{%
\chi}_{+}(s-\varepsilon ^{2},x,v)>0,$ by (\ref{chi+}), $-\varepsilon
^{4}\leq \xi (x-\varepsilon ^{2}v)<0$ and $|v|\leq \varepsilon ^{-m}.$ But 
\begin{equation*}
\xi (x-\varepsilon ^{2}v)=\xi (x)-\varepsilon ^{2}\nabla \xi (x)\cdot
v+\varepsilon ^{2}v\cdot \nabla ^{2}\xi (\bar{x})\cdot \varepsilon ^{2}v,
\end{equation*}%
for some $\bar{x}$ is between $x$ and $x-\varepsilon ^{2}v.$ Since $n(x)=%
\frac{\nabla \xi (x)}{|\nabla \xi (x)|},$ there exists a constant $C_{\xi
}>0 $ such that 
\begin{equation*}
-\varepsilon ^{2}\nabla \xi (x)\cdot v=-\varepsilon ^{2}|\nabla \xi
(x)|n(x)\cdot v<-\frac{\varepsilon ^{3}}{2}|\nabla \xi (x)|<-\frac{C_{\xi
}\varepsilon ^{3}}{2}
\end{equation*}%
for $x\in \Omega \setminus \Omega _{\varepsilon ^{4}}.$ Here we have used
the first assertion that $n(x)\cdot v\geq \frac{\varepsilon }{2}$. Therefore%
\begin{eqnarray*}
\xi (x-\varepsilon ^{2}v) &<&0-\frac{C_{\xi }\varepsilon ^{3}}{2}%
+\varepsilon ^{2}v\cdot \nabla ^{2}\xi (\bar{x})\cdot \varepsilon ^{2}v \\
&<&-\frac{C_{\xi }\varepsilon ^{3}}{2}+C|\varepsilon ^{4}v^{2}|=-\frac{%
C_{\xi }\varepsilon ^{3}}{2}\{1-C\varepsilon ^{1-2m}\}<-\varepsilon ^{4}
\end{eqnarray*}%
for $2m<1,$ and small $\varepsilon $. This is a contradiction to $%
-\varepsilon ^{4}\leq \xi (x-\varepsilon ^{2}v)<0.$
\end{proof}

\begin{lemma}
\label{strong}We have the strong convergence%
\begin{equation*}
\lim_{k\rightarrow \infty }\int_{0}^{1}||Z_{k}(s)-Z(s)||^{2}ds=0
\end{equation*}%
and $\int_{0}^{1}||Z(s)||_{\nu }^{2}>0.$ Moreover, $Z$ defined in (\ref%
{zlimit}) satisfies the corresponding boundary conditions (\ref{inflow})
with $g\equiv 0$, (\ref{bounceback}), (\ref{specular}) and (\ref{diffuse})
\end{lemma}

\begin{proof}
By (\ref{vlasovchi}), we multiply $\tilde{\chi}_{\pm }$ with (\ref{zn}) to
get 
\begin{equation}
\lbrack \partial _{t}+v\cdot \nabla _{x}]\{\tilde{\chi}_{\pm }Z_{k}\}=-%
\tilde{\chi}_{\pm }LZ_{k}.  \label{chizn}
\end{equation}%
Since $\int_{0}^{1}||Z_{k}(t)||_{\gamma }^{2}dt<\infty ,$ applying the $%
L^{2} $ estimate backward in time over the shell-like region $[s-\varepsilon
^{2},s]\times \{\Omega \setminus \Omega _{\varepsilon ^{4}}\}\times \mathbf{R%
}^{3}$ for outgoing part $\chi _{+}$, we obtain:%
\begin{eqnarray}
|| &&\chi _{+}Z_{k}(s)||^{2}+\int_{s-\varepsilon ^{2}}^{s}||\tilde{\chi}%
_{+}Z_{k}(t)||_{\gamma _{+}}^{2}dt-\int_{s-\varepsilon ^{2}}^{s}||\tilde{\chi%
}_{+}Z_{k}(t)||_{\gamma _{+}^{\varepsilon }}^{2}dt=||Z_{k}\tilde{\chi}%
_{+}(s-\varepsilon ^{2})||^{2}  \notag \\
&&+\int_{s-\varepsilon ^{2}}^{s}||\tilde{\chi}_{+}Z_{k}(t)||_{\gamma
_{-}}^{2}dt-\int_{s-\varepsilon ^{2}}^{s}||\tilde{\chi}_{+}Z_{k}(t)||_{%
\gamma _{-}^{\varepsilon }}^{2}dt-\int_{s-\varepsilon ^{2}}^{s}(\tilde{\chi}%
_{+}LZ_{k},Z_{k})(t)dt,  \label{L2chi+}
\end{eqnarray}%
where at the inner boundary $\gamma ^{\varepsilon }\equiv \{x:\xi
(x)=-\varepsilon ^{4}\}\times \mathbf{R}^{3}$, its normal vector is $n(x)=%
\frac{\nabla \xi (x)}{|\nabla \xi (x)|}.\,$\ We notice that by Lemma \ref%
{chi}, $\tilde{\chi}_{+}\equiv 0$ at $s-\varepsilon ^{2},$ while $\tilde{\chi%
}_{+}\equiv 0$ on $\gamma _{-}$ and $\gamma _{-}^{\varepsilon }$, since $%
n(x)\cdot v>0$ for $\tilde{\chi}_{+}(s,x,v)\neq 0.$ From $\int_{0}^{1}||(%
\mathbf{I}-\mathbf{P)}Z_{k}(s)||_{\nu }^{2}ds\leq \frac{1}{k},$ we get for $%
k $ large, 
\begin{eqnarray}
\int_{s-\varepsilon ^{2}}^{s}(\tilde{\chi}_{+}LZ_{k},Z_{k})(t)dt &\leq
&\int_{0}^{1}\int_{\Omega \times \mathbf{R}^{3}}|L\{\mathbf{I}-\mathbf{P}%
\}Z_{k}||Z_{k}|(t)dxdvdt  \notag \\
&\leq &C\left\{ \int_{0}^{1}||\{\mathbf{I}-\mathbf{P}\}Z_{k}(t)||_{\nu
}^{2}\right\} ^{1/2}\left\{ \int_{0}^{1}||Z_{k}(t)||_{\nu }^{2}\right\}
^{1/2}  \notag \\
&\leq &\frac{C}{\sqrt{k}}.  \label{lzz}
\end{eqnarray}%
Therefore we can simplify (\ref{L2chi+}) as 
\begin{equation*}
||\chi _{+}Z_{k}(s)||^{2}+\int_{s-\varepsilon ^{2}}^{s}||\tilde{\chi}%
_{+}Z_{k}(t)||_{\gamma _{+}}^{2}dt\leq \int_{s-\varepsilon ^{2}}^{s}||\tilde{%
\chi}_{+}Z_{k}(t)||_{\gamma _{+}^{\varepsilon }}^{2}dt+\frac{C}{\sqrt{k}}.
\end{equation*}

Similarly, we use $L^{2}$ estimate forward in time $[s,s+\varepsilon
^{2}]\times \{\Omega \setminus \Omega _{\varepsilon ^{4}}\}\times \mathbf{R}%
^{3}$ for incoming part $\chi _{-}$ to get%
\begin{eqnarray*}
&||&\tilde{\chi}_{-}Z_{k}(s+\varepsilon ^{2})||^{2}+\int_{s}^{s+\varepsilon
^{2}}||\tilde{\chi}_{-}Z_{k}(t)||_{\gamma
_{+}}^{2}dt-\int_{s}^{s+\varepsilon ^{2}}||\tilde{\chi}_{-}Z_{k}(t)||_{%
\gamma _{+}^{\varepsilon }}^{2}dt=||\chi _{-}Z_{k}(s)||^{2} \\
&&+\int_{s}^{s+\varepsilon ^{2}}||\tilde{\chi}_{-}Z_{k}(t)||_{\gamma
_{-}}^{2}dt-\int_{s}^{s+\varepsilon ^{2}}||\tilde{\chi}_{-}Z_{k}(t)||_{%
\gamma _{-}^{\varepsilon }}^{2}dt-\int_{s}^{s+\varepsilon ^{2}}(\tilde{\chi}%
_{-}L\{\mathbf{I}-\mathbf{P}\}Z_{k},Z_{k})(t)dt.
\end{eqnarray*}%
\ We notice that $\tilde{\chi}_{-}\equiv 0$ at $t=s+\varepsilon ^{2},$ while 
$\tilde{\chi}_{-}\equiv 0$ on the incoming part $\gamma _{+}$ and $\gamma
_{+}^{\varepsilon }$ by part (2) of Lemma \ref{chi}$.$ Therefore we deduce
from (\ref{lzz}) 
\begin{equation*}
||\chi _{-}Z_{k}(s)||^{2}+\int_{s}^{s+\varepsilon ^{2}}||\tilde{\chi}%
_{-}Z_{k}(t)||_{\gamma _{-}}^{2}dt\leq \int_{s}^{s+\varepsilon ^{2}}||\tilde{%
\chi}_{-}Z_{k}(t)||_{\gamma _{-}^{\varepsilon }}^{2}dt+\frac{C}{\sqrt{k}}.
\end{equation*}%
By Lemma $\ref{chi},$ the supports of $\tilde{\chi}_{\pm }Z_{k}(t)$ on $%
\gamma _{\pm }^{\varepsilon }$ are contained in $|n(x)\cdot v|\geq \frac{%
\varepsilon }{2}.$ Combining the $\pm $ cases, we are able to estimate $%
\tilde{\chi}_{\pm }Z_{k}$ in terms of the inner boundary contributions as 
\begin{eqnarray}
&&\int_{\Omega \setminus \Omega _{\varepsilon ^{4}}}\int_{\substack{ %
|n(x)\cdot v|>\varepsilon  \\ |v|\leq \varepsilon ^{-m}}}%
|Z_{k}(s,x,v)|^{2}dxdv  \notag \\
&\leq &\int_{s}^{s+\varepsilon ^{2}}||\tilde{\chi}_{-}Z_{k}(t)||_{\gamma
_{-}^{\varepsilon }}^{2}dt+\int_{s-\varepsilon ^{2}}^{s}||\tilde{\chi}%
_{+}Z_{k}(t)||_{\gamma _{+}^{\varepsilon }}^{2}dt+\frac{2C}{\sqrt{k}}  \notag
\\
&\leq &\int_{s-\varepsilon ^{2}}^{s+\varepsilon ^{2}}||\mathbf{1}_{\{|v|\leq
\varepsilon ^{-m},\text{ and }|n(x)\cdot v|\geq \frac{\varepsilon }{2}%
\}}Z_{k}(t)||_{\gamma ^{\varepsilon }}^{2}dt+\frac{2C}{\sqrt{k}}.
\label{gammae}
\end{eqnarray}

Since the outward normal at $x\in \partial \Omega _{\varepsilon ^{4}}$ is $%
n(x)$, the set $\{(x,v):x\in \partial \Omega _{\varepsilon ^{4}},|v\cdot
n(x)|\geq \frac{\varepsilon }{2}\}$ is away from the singular set $\gamma
_{0}^{\varepsilon }=\{(x,v):x\in \partial \Omega _{\varepsilon ^{4}},|v\cdot
n(x)|=0\}.$ Hence, by (\ref{tlower}) in Lemma \ref{huang}, both the backward
or forward trajectories emanating from $\partial \Omega _{\varepsilon
^{4}}\times \mathbf{\{|}v\mathbf{|\leq \varepsilon }^{-m},|v\cdot n(x)|\geq 
\frac{\varepsilon }{2}\}$ spend a positive period of time inside $\bar{\Omega%
}_{\varepsilon ^{4}}.$ Since 
\begin{equation}
\{\partial _{t}+v\cdot \nabla _{x}\}\left\{ \mathbf{1}_{\{|v|\leq
\varepsilon ^{-m}\}}(Z_{k}-Z)\right\} =-\mathbf{1}_{\{|v|\leq \varepsilon
^{-m}\}}L\{\mathbf{I}-\mathbf{P}\}Z_{k},  \label{nuzn}
\end{equation}%
we can apply Ukai's trace theorem (Theorem 5.1.1, [U1]) to $\mathbf{1}%
_{\{|v|\leq \varepsilon ^{-m}\}}(Z_{k}-Z)$ over $\bar{\Omega}_{\varepsilon
^{4}}$ to get 
\begin{eqnarray*}
&&\int_{s-\varepsilon ^{2}}^{s+\varepsilon ^{2}}||\mathbf{1}_{\{|v|\leq
\varepsilon ^{-m},|v\cdot n(x)|\geq \frac{\varepsilon }{2}%
\}}\{Z_{k}(t)-Z(t)\}||_{\gamma ^{\varepsilon }}^{2}ds \\
&=&\int_{s-\varepsilon ^{2}}^{s+\varepsilon ^{2}}||\mathbf{1}_{\{|v\cdot
n(x)|\geq \frac{\varepsilon }{2}\}}\{\mathbf{1}_{\{|v|\leq \varepsilon
^{-m}\}}(Z_{k}(t)-Z(t))\}||_{\gamma ^{\varepsilon }}^{2}ds \\
&\leq &C_{\varepsilon }\int_{\varepsilon }^{1-\varepsilon }\left\{ ||\mathbf{%
1}_{\{|v|\leq \varepsilon ^{-m}\}}(Z_{k}(t)-Z(t))||_{\Omega _{\varepsilon
^{4}}\times \mathbf{R}^{3}}^{2}+||\mathbf{1}_{\{|v|\leq \varepsilon
^{-m}\}}\{L\{\mathbf{I}-\mathbf{P}\}Z_{k}(t)\}||_{\Omega _{\varepsilon
^{4}}\times \mathbf{R}^{3}}^{2}\right\} dt \\
&\leq &C_{\varepsilon }\int_{\varepsilon }^{1-\varepsilon
}||Z_{k}(t)-Z(t)||_{\Omega _{\varepsilon ^{4}}\times \mathbf{R}%
^{3}}^{2}dt+C_{\varepsilon }\int_{0}^{1}||\{\mathbf{I}-\mathbf{P}%
\}Z_{k}(t)||_{\nu }^{2}dt \\
&\leq &C_{\varepsilon }\int_{\varepsilon }^{1-\varepsilon
}||Z_{k}(t)-Z(t)||_{\Omega _{\varepsilon ^{4}}\times \mathbf{R}^{3}}^{2}dt+%
\frac{C_{\varepsilon }}{k}.
\end{eqnarray*}%
Therefore, for fixed $\varepsilon ,$ we have from the interior compactness
in Lemma \ref{interior} 
\begin{equation*}
\lim_{k\rightarrow \infty }\int_{s-\varepsilon ^{2}}^{s+\varepsilon ^{2}}||%
\mathbf{1}_{\{|v|\leq \varepsilon ^{-m},|v\cdot n(x)|\geq \varepsilon
\}}\{Z_{k}(t)-Z(t)\}||_{\gamma ^{\varepsilon }}^{2}dt=0.
\end{equation*}%
Hence, for $k$ large, and for any $\varepsilon \leq s\leq 1-\varepsilon ,$
by (\ref{gammae})%
\begin{eqnarray*}
&&\int_{\Omega \setminus \Omega _{\varepsilon ^{4}}}\int_{\substack{ |v|\leq
\varepsilon ^{-m},  \\ |v\cdot n(x)|\geq \varepsilon }}|Z_{k}(s,x,v)|^{2}dxdv
\\
&\leq &2\int_{s-\varepsilon ^{2}}^{s+\varepsilon ^{2}}||\mathbf{1} 
_{\substack{ |v|\leq \varepsilon ^{-m}  \\ |v\cdot n(x)|\geq \varepsilon }}%
\{Z_{k}(t)-Z(t)\}||_{\gamma ^{\varepsilon }}^{2}dt+2\int_{s-\varepsilon
^{2}}^{s+\varepsilon ^{2}}||\mathbf{1}_{\substack{ |v|\leq \varepsilon ^{-m} 
\\ |v\cdot n(x)|\geq \varepsilon }}Z(t)||_{\gamma ^{\varepsilon }}^{2}dt+%
\frac{2C}{\sqrt{k}} \\
&\leq &\varepsilon +\int_{s-\varepsilon ^{2}}^{s+\varepsilon
^{2}}||Z(t)||_{\gamma ^{\varepsilon }}^{2}ds.
\end{eqnarray*}%
But from Lemma \ref{limit}, $Z(s,x,v)$ is smooth so its trace is given by (%
\ref{zlimit}) as well. By (\ref{constantbound}), since the time interval is
small, 
\begin{equation*}
\int_{s-\varepsilon ^{2}}^{s+\varepsilon ^{2}}||Z(t)||_{\gamma ^{\varepsilon
}}^{2}dt\leq 2\varepsilon ^{2}\times \sup_{0\leq t\leq 1}||Z(t)||_{\gamma
^{\varepsilon }}^{2}\leq C\varepsilon ^{2},
\end{equation*}%
where $C$ depends on $a_{0},c_{0},c_{1},c_{2},b_{0},b_{1}$ and $\varpi .$ We
thus deduce that for $\varepsilon \leq s\leq 1-\varepsilon ,$ for $k$ large, 
\begin{equation}
\int_{\Omega \setminus \Omega _{\varepsilon ^{4}}}\int_{\substack{ |n\cdot
v|\geq \varepsilon  \\ |v|\leq \varepsilon ^{-m}}}|Z_{k}(s,x,v)|^{2}dxdv\leq
C\varepsilon .  \label{moving}
\end{equation}

We are now ready to prove compactness of $Z_{k}$. We split 
\begin{equation*}
\int_{0}^{1}\int \int_{\Omega
}|Z_{k}(s,x,v)-Z(s,x,v)|^{2}dsdxdv=\int_{0}^{\varepsilon }+\int_{\varepsilon
}^{1-\varepsilon }+\int_{1-\varepsilon }^{1}.
\end{equation*}%
By Lemma \ref{zinfty}, we conclude that the integrals $\int_{0}^{\varepsilon
}+\int_{1-\varepsilon }^{1}$ are bounded by $C\varepsilon .$ On the other
hand, we further split the main part $\int_{\varepsilon }^{1-\varepsilon }$
as 
\begin{eqnarray*}
&&2\int_{\varepsilon }^{1-\varepsilon }\int_{\Omega \setminus \Omega
_{\varepsilon ^{4}}}\int_{\substack{ |n\cdot v|\geq \varepsilon  \\ |v|\leq
\varepsilon ^{-m}}}|Z_{k}(s,x,v)|^{2}+2\int_{\varepsilon }^{1-\varepsilon
}\int_{\Omega \setminus \Omega _{\varepsilon ^{4}}}\int_{\substack{ |n\cdot
v|\leq \varepsilon  \\ \text{or }|v|\geq \varepsilon ^{-m}}}%
|Z_{k}(s,x,v)|^{2} \\
&&+2\int_{\varepsilon }^{1-\varepsilon }\int_{\Omega \setminus \Omega
_{\varepsilon ^{4}}}\int |Z(s,x,v)|^{2}+\int_{\varepsilon }^{1-\varepsilon
}\int \int_{\Omega _{\varepsilon ^{4}}}|Z_{k}(s,x,v)-Z(s,x,v)|^{2}.
\end{eqnarray*}%
Clearly, the first term is bounded by (\ref{moving}), the second term is
bounded by $C\varepsilon $ thanks to Lemma \ref{nomove}; by (\ref%
{constantbound}), the third term is bounded by 
\begin{equation*}
\int_{\Omega \setminus \Omega _{\varepsilon ^{4}}}\int
|Z(t,x,v)|^{2}dxdv\leq C|\Omega \setminus \Omega _{\varepsilon ^{4}}|\leq
C\varepsilon ,
\end{equation*}%
where $C$\ depends on $a_{0},c_{0},c_{1},c_{2},b_{0},b_{1}$ and $\varpi .$
The last term goes to zero as $k\rightarrow \infty $ by Lemma \ref{interior}%
. We hence deduce the strong convergence 
\begin{equation*}
\int_{0}^{1}\int \int_{\Omega }|Z_{k}(s,x,v)-Z(s,x,v)|^{2}dsdxdv\rightarrow 0
\end{equation*}%
by first letting $\varepsilon $ small, then letting $k\rightarrow \infty .$
From our normalization $\int_{0}^{1}||\mathbf{P}Z_{k}(s)||_{\nu
}^{2}ds\equiv 1$ with $\mathbf{P}Z_{k}=\{a_{k}+v\cdot b_{k}+c_{k}|v|^{2}\}%
\sqrt{\mu },$ there exists $C>0$ independent of $k$ such that 
\begin{equation*}
\int_{0}^{1}||\mathbf{P}Z_{k}(s)||^{2}ds\geq C\int_{0}^{1}||\mathbf{P}%
Z_{k}(s)||_{\nu }^{2}ds\geq C>0,
\end{equation*}%
because both norms are equivalent to 
\begin{equation*}
\int_{0}^{1}|a_{k}(s,x)|^{2}dxds+\int_{0}^{1}|b_{k}(s,x)|^{2}dxds+%
\int_{0}^{1}|c_{k}(s,x)|^{2}dxds.
\end{equation*}%
Hence $\int_{0}^{1}||Z(s)||^{2}ds=\lim_{k\rightarrow \infty
}\int_{0}^{1}||Z_{k}(s)||^{2}ds\geq C>0.$

Finally, we study the boundary conditions which $Z$ satisfies. In fact,
recalling (\ref{nuzn}) and $\int_{0}^{1}||Z_{k}(t)-Z(t)||^{2}dt\rightarrow
0, $ we use Ukai's trace theorem to conclude, for any fixed $\varepsilon >0,$
\begin{eqnarray}
&&\lim_{k\rightarrow \infty }\int_{0}^{1}||\mathbf{1}_{\{|v\cdot n(x)|\geq 
\frac{\varepsilon }{2},|v|\leq \frac{1}{\varepsilon ^{m}}%
\}}Z_{k}(s)-Z(s)||_{\gamma }^{2}ds  \notag \\
&\leq &C\lim_{k\rightarrow \infty }[\int_{0}^{1}||\mathbf{1}_{|v|\leq \frac{1%
}{\varepsilon ^{m}}}\{Z_{k}(s)-Z(s)\}||^{2}ds+\int_{0}^{1}||[\partial
_{t}+v\cdot \nabla _{x}]\mathbf{1}_{|v|\leq \frac{1}{\varepsilon ^{m}}%
}\{Z_{k}(s)-Z(s)\}||^{2}ds]  \notag \\
&\leq &C\lim_{k\rightarrow \infty }\left( \int_{0}^{1}||\mathbf{1}_{|v|\leq 
\frac{1}{\varepsilon ^{m}}}\{L\{\mathbf{I}-\mathbf{P}\}Z_{k}(t)\}||^{2}dt%
\right) =0.  \label{gammalimit}
\end{eqnarray}

For the in-flow boundary case, by (\ref{1/n}) and the continuity of $Z,$ 
\begin{equation*}
\int_{0}^{1}||\mathbf{1}_{\{|v\cdot n(x)|\geq \frac{\varepsilon }{2}\}}%
\mathbf{1}_{|v|\leq \varepsilon ^{-m}}Z(s)||_{\gamma
}^{2}ds=\lim_{k\rightarrow \infty }\int_{0}^{1}||\mathbf{1}_{\{|v\cdot
n(x)|\geq \frac{\varepsilon }{2}\}}\mathbf{1}_{|v|\leq \varepsilon
^{-m}}Z_{k}(s)||_{\gamma }^{2}ds=0,\text{\ }
\end{equation*}%
so that $Z\equiv 0$ on $\gamma .$

For the bounce-back and specular reflections, $Z_{k}(t,x,v)=Z_{k}(t,x,-v),$
or $Z_{k}(t,x,v)=Z_{k}(t,x,R(x)v).$ Letting $k\rightarrow \infty ,$ we
deduce that $Z$ satisfies the same relation respectively for $\{|v\cdot
n(x)|\geq \frac{\varepsilon }{2}\}.$ Therefore, $Z(t,x,v)=Z(t,x,-v)$ or $%
Z(t,x,v)=Z(t,x,R(x)v)$ respectively by the continuity of $Z.$

For the diffusive reflection, notice that on $\gamma _{-},$ 
\begin{equation}
Z_{k}(t,x,v)=c_{\mu }\{\int_{n\cdot v^{\prime }>0}Z_{k}(t,x,v^{\prime })%
\sqrt{\mu }n\cdot v^{\prime }dv^{\prime }\}\sqrt{\mu (v)}\equiv \bar{a}%
_{k}(t,x)\sqrt{c_{\mu }\mu (v)}  \label{kdiffuse}
\end{equation}%
Fix $\varepsilon >0$ small and for any $x\in \partial \Omega ,$ on the set $%
\{v\cdot n(x)>\varepsilon \},Z_{k}\rightarrow Z$ in $L^{2}([0,1]\times
\gamma ).$ This implies that from (\ref{gammalimit}) 
\begin{equation*}
\left\{ \int_{0}^{1}\int_{\partial \Omega }|\bar{a}_{k}(t,x)|^{2}\int 
_{\substack{ v\cdot n(x)>\varepsilon  \\ \text{ and }|v|\geq \varepsilon
^{-m}}}c_{\mu }\mu dv\right\} dxdt=\int_{0}^{1}||\mathbf{1}_{\substack{ %
|v\cdot n(x)|\geq \varepsilon  \\ |v|\geq \varepsilon ^{-m}}}%
\{Z_{k}(t)\}_{\gamma _{-}}||^{2}dt<C_{\varepsilon }<\infty .
\end{equation*}%
Notice that for $\varepsilon $ small, $\int_{\substack{ v\cdot
n(x)>\varepsilon  \\ \text{ and }|v|\geq \varepsilon ^{-m}}}\mu dv$ is a
finite non-zero constant, independent of $x$. It follows that 
\begin{equation*}
\left\{ \int_{0}^{1}\int_{\partial \Omega }|\bar{a}_{k}(t,x)|^{2}\right\}
dxdt\leq C_{\varepsilon }<\infty .
\end{equation*}%
This implies that $P_{\gamma }\{Z_{k}\}_{\gamma _{+}}\equiv $ $\bar{a}%
_{k}(t,x)\sqrt{c_{\mu }\mu (v)}$ are uniformly bounded in $L^{2}([0,1]\times
\gamma _{+}).$ But from (\ref{diffuse1/n}), $\{I-P_{\gamma
}\}\{Z_{k}\}_{\gamma _{+}}\rightarrow 0$ in $L^{2}([0,1]\times \gamma _{+}),$
we deduce that $\{Z_{k}\}_{\gamma _{+}}$ are uniformly bounded in $%
L^{2}([0,1]\times \gamma _{+})~$\ with a weak limit$.$ But $%
\{Z_{k}\}_{\gamma _{+}}\rightarrow Z$ strongly in $L^{2}([0,1]\times
\{\gamma _{+}\setminus \gamma _{0}\})$ by the trace theorem$,$ so that $%
\{Z_{k}\}_{\gamma _{+}}\rightarrow Z$ weakly in $L^{2}([0,1]\times \gamma
_{+})$ since $\gamma _{0}$ has zero measure. Hence 
\begin{equation*}
c_{\mu }\{\int_{n\cdot v^{\prime }>0}Z_{k}(t,x,v^{\prime })\sqrt{\mu }n\cdot
v^{\prime }dv^{\prime }\}\sqrt{\mu (v)}\rightarrow c_{\mu }\{\int_{n\cdot
v^{\prime }>0}Z(t,x,v^{\prime })\sqrt{\mu }n\cdot v^{\prime }dv^{\prime }\}%
\sqrt{\mu (v)}
\end{equation*}%
weakly $\ L^{2}([0,1]\times \gamma _{+})$. We then recover (\ref{diffuse})
by letting $k\rightarrow \infty $ in (\ref{kdiffuse}).
\end{proof}

\subsection{Boundary Condition Leads to $Z=0.$}

Since $Z$ now satisfies one of the boundary conditions $Z_{\gamma }=0$, (\ref%
{bounceback}), (\ref{specular}), and (\ref{diffuse}), we will show that $Z$
in (\ref{zlimit}) has to be zero and this leads to a contradiction.

In the case of in-flow boundary (\ref{inflow}), since $Z=0$ on $\gamma ,$
from (\ref{zlimit}), for any $t$ and $x\in \partial \Omega ,$ and $v\in 
\mathbf{R}^{3},$ by comparing the coefficients in front of the polynomials
of $v,~$\ we deduce that $\{\frac{c_{0}t^{2}}{2}+c_{1}t+c_{2}\}\equiv 0$ and%
\begin{equation*}
\{-c_{0}tx-c_{1}x+\varpi \times x+b_{0}t+b_{1}\}\equiv \{\frac{c_{0}}{2}%
|x|^{2}-b_{0}\cdot x+a_{0}\}\equiv 0.
\end{equation*}%
Therefore $c_{0}=c_{1}=c_{2}=0,$ $\ $and $b_{0}=0.$ Then $a_{0}=0$ and $%
\varpi \times x+b_{1}\equiv 0,$ or 
\begin{equation}
\varpi ^{2}x_{3}-\varpi ^{3}x_{2}+b_{1}^{1}=-\varpi ^{1}x_{3}+\varpi
^{3}x_{1}+b_{1}^{2}=\varpi ^{1}x_{2}-\varpi ^{2}x_{1}+b_{1}^{3}\equiv 0
\label{cross}
\end{equation}%
for all $x\in \partial \Omega .$ Notice that since $\xi (x)=0$ is two
dimensional, so we may assume that $(x_{1},x_{2})$ are (locally)
independent. Hence $\varpi ^{1}=\varpi ^{2}=b_{1}^{3}=0$ then $\varpi
^{3}=b_{1}^{2}=0,$ and finally $b_{1}^{1}=0.$ Therefore we deduce $Z\equiv
0. $

In the case of the bounce-back case (\ref{bounceback}), for any fixed $t,$
because of (\ref{nuzn}), we apply Ukai's trace theorem over $[0,t]\times
\Omega \times \mathbf{R}^{3}$ to get, for any $0\leq t\leq 1,$ 
\begin{equation}
\lim_{k\rightarrow \infty }||\mathbf{1}_{\{|v|\leq \varepsilon
^{-1}\}}\{Z_{k}(t)-Z(t)\}||=0.  \label{zt}
\end{equation}%
Therefore, by (\ref{mass}) and (\ref{energy}), 
\begin{eqnarray}
\int Z(t)\sqrt{\mu } &=&\lim_{k\rightarrow \infty }\int_{|v|\leq \frac{1}{%
\varepsilon }}Z_{k}(t)\sqrt{\mu }+\lim_{k\rightarrow \infty }\int_{|v|\geq 
\frac{1}{\varepsilon }}Z_{k}(t)\sqrt{\mu }\equiv 0,  \label{zmass} \\
\int Z(t)|v|^{2}\sqrt{\mu } &\equiv &\lim_{k\rightarrow \infty
}\int_{|v|\leq \frac{1}{\varepsilon }}|v|^{2}Z_{k}(t)\sqrt{\mu }%
+\lim_{k\rightarrow \infty }{}\int_{|v|\geq \frac{1}{\varepsilon }%
}|v|^{2}Z_{k}(t)\sqrt{\mu },  \label{zenergy}
\end{eqnarray}%
because the integrations over $|v|\geq \frac{1}{\varepsilon }$ are bounded
by $C||Z_{k}(t)||\int_{|v|\geq \frac{1}{\varepsilon }}\mu
^{1/4}=C\varepsilon $, from the $L^{\infty }$ estimates in Lemma \ref{zinfty}%
. We therefore obtain that for all $t,$ 
\begin{eqnarray}
\int \{\frac{c_{0}}{2}|x|^{2}-b_{0}\cdot x+a_{0}\}\sqrt{\mu }+\{\frac{%
c_{0}t^{2}}{2}+c_{1}t+c_{2}\}|v|^{2}\sqrt{\mu } &\equiv &0,  \label{abcmass}
\\
\int \{\frac{c_{0}}{2}|x|^{2}-b_{0}\cdot x+a_{0}\}|v|^{2}\sqrt{\mu }+\{\frac{%
c_{0}t^{2}}{2}+c_{1}t+c_{2}\}|v|^{4}\sqrt{\mu } &\equiv &0.
\label{abcenergy}
\end{eqnarray}%
This implies that $c_{0}=c_{1}=0.$ Moreover, since from the bounce-back
boundary condition $Z(t,x,v)=Z(t,x,-v),$ we must have $b(t,x)\equiv 0$ in (%
\ref{hydro}), or 
\begin{equation*}
b(t,x)\equiv \varpi _{2}\times x+b_{0}t+b_{1}\equiv 0
\end{equation*}%
for all $x\in \partial \Omega .$ Clearly $b_{0}=0$ as a function of $t.$
From the argument after (\ref{cross}), $\varpi _{2}=0=b_{1}$. We therefore
deduce that from (\ref{abcmass}) and (\ref{abcenergy}) that%
\begin{equation*}
\int a_{0}\sqrt{\mu }dv+c_{2}\int |v|^{2}\sqrt{\mu }dv\equiv \int
a_{0}|v|^{2}\sqrt{\mu }+c_{2}\int |v|^{4}\sqrt{\mu }dv\equiv 0.
\end{equation*}%
We thus have $a_{0}=c_{2}=0,$ then $Z\equiv 0$ for the bounce-back case.

The specular reflection is more delicate. Using the same mass and
conservation laws (\ref{abcmass}) and (\ref{abcenergy}), we again have $%
c_{1}=c_{0}=0$ and $b(t,x)=$ $\varpi _{2}\times x+b_{0}t+b_{1}.$ Now from
the specular reflection, we have for any $x\in \partial \Omega ,$ $%
b(t,x)\cdot n\equiv 0$ or 
\begin{equation*}
\{\varpi \times x+b_{0}t+b_{1}\}\cdot n(x)\equiv 0.
\end{equation*}%
Hence $b_{0}=0$ for all $x\in \partial \Omega $ and 
\begin{equation}
\{\varpi \times x\}\cdot n(x)+b_{1}\cdot n(x)=0.  \label{bn}
\end{equation}%
In the case $\varpi =0,$ we have $b_{1}\cdot n(x)\equiv 0$ on $\partial
\Omega .$ We can choose $x^{\prime }\in \partial \Omega $ such that $b_{1}||$
$n(x^{\prime })$ by taking the minimizer of $\min_{\xi (x)=0}b_{1}\cdot x.$
Hence, $b_{1}\cdot n(x^{\prime })=0$ and $b_{1}=0.$

For $\varpi \neq 0,$ let's decompose $b_{1}=\beta _{1}\frac{\varpi }{|\varpi
|}+\beta _{2}\eta ,$ where $|\eta |=1$ and $\eta \perp \varpi .$ Then $\eta
=\{\frac{\varpi }{|\varpi |}\times \eta \}\times \frac{\varpi }{|\varpi |}.$
Hence 
\begin{equation*}
b_{1}=\beta _{1}\frac{\varpi }{|\varpi |}+\beta _{2}\eta =\beta _{1}\frac{%
\varpi }{|\varpi |}+\beta _{2}\{\frac{\varpi }{|\varpi |^{2}}\times \eta
\}\times \varpi \equiv \beta _{1}\frac{\varpi }{|\varpi |}-x_{0}\times
\varpi .
\end{equation*}%
where $x_{0}=-\frac{\beta _{2}}{|\varpi |^{2}}\varpi \times \eta .$ By
plugging this back into (\ref{bn}), we get 
\begin{equation*}
\beta _{1}\frac{\varpi }{|\varpi |}\cdot n(x)+\varpi \times (x-x_{0})\cdot
n(x)=0.
\end{equation*}%
Once again, we can choose a point $x^{\prime }\in \partial \Omega $ such
that $\varpi \parallel n(x^{\prime })$ (e.g., look for minimizer of $%
\min_{\xi (x)=0}\varpi \cdot x$). We then deduce $\varpi \times (x^{\prime
}-x_{0})\cdot n(x^{\prime })=0$ and hence $\beta _{1}=0.$ So 
\begin{equation}
Z=\varpi \times (x-x_{0})\cdot v\sqrt{\mu }  \label{zomega}
\end{equation}%
and $\varpi \times (x-x_{0})\cdot n(x)\equiv 0$ for all $x\in \partial
\Omega .$ If $\Omega $ is not rotational symmetric, there is no non-zero $%
\varpi $ and $x_{0}$ exist, then we deduce that $Z\equiv 0$ for the specular
case$.$

On the other hand, if $\Omega $ is rotational symmetric, there are such
non-zero $\varpi $ and $x_{0}$ for $\Omega $ so that (\ref{zomega}) is valid$%
,$ we then have to use the additional \ conservation law for angular
momentum:%
\begin{equation*}
\int \varpi \times (x-x_{0})\cdot vZ(t)\sqrt{\mu }dv=0
\end{equation*}%
as $k\rightarrow \infty $ of the same expression for $Z_{k}$ (see the proof
for (\ref{zmass}))$.$ Therefore, we combine (\ref{zomega}) to get 
\begin{equation*}
\int \{\varpi \times (x-x_{0})\cdot v\}^{2}\mu dxdv\equiv 0
\end{equation*}%
Therefore $\varpi \times (x-x_{0})\cdot v\equiv 0$ and $Z\equiv 0$ from (\ref%
{zomega}).

In the case of diffuse boundary condition (\ref{diffuse}), because of (\ref%
{mass}), we have (\ref{abcmass}) and $c_{1}=c_{0}=0.$ Moreover, we have 
\begin{equation*}
Z(t,x,v)=c_{\mu }\{\int_{n\cdot v^{\prime }>0}Z(t,x,v^{\prime })\sqrt{\mu }%
n\cdot v^{\prime }dv^{\prime }\}\sqrt{\mu }.
\end{equation*}%
on $\gamma _{-}.$ Since $v\sqrt{\mu },\sqrt{\mu },|v|^{2}\sqrt{\mu }$ are
linearly independent, this implies for all $t$ and $x\in \partial \Omega ,$ 
\begin{equation*}
b(t,x)\equiv \varpi \times x+b_{0}t+b_{1}\equiv 0,\text{ \ \ \ \ and \ \ }%
c_{2}\equiv 0.
\end{equation*}%
Therefore, $b_{0}=0,$ and $\varpi =b_{1}=0$ as in (\ref{cross}). Therefore,
we have from (\ref{abcmass}) $a_{0}\int \sqrt{\mu }dv=0.$ Hence \thinspace $%
Z\equiv 0.$

\section{$L^{\infty }$ Decay Theory}

\subsection{\protect\bigskip $L^{\infty }$ Decay For In-flow Boundary
Condition}

\subsubsection{G(t,0) and Continuity}

As outlined in Section 1.6, we study the $L^{\infty }$ (pointwise) decay for
the weighted $h=wf$ of the linear Boltzmann equation (\ref{lboltzmannh})
with the in-flow boundary condition. We first derive explicit formula for
solution operator $G(t,0)$ for the homogenous transport equation (\ref%
{transport}) with in-flow boundary condition. Note that for non-zero in-flow
datum at the boundary, $G(t,0)$ in general is not a semigroup.

\begin{lemma}
\label{ginflowdecay}Let $h_{0}(x,v)\in L^{\infty }$ and $wg\in L^{\infty }.$
Let $\{G(t,0)h_{0}\}$ be the solution to the transport equation (\ref%
{transport}) 
\begin{equation*}
\{\partial _{t}+v\cdot \nabla _{x}+\nu \}G(t,0)h_{0}=0,\text{ \ \ \ \ }%
G(0,0)h_{0}=h_{0},\text{ \ \ }\{G(t,0)h_{0}\}_{\gamma _{-}}=wg.
\end{equation*}%
For any $(x,v),$ with $x\in \bar{\Omega},$ let $t_{\mathbf{b}}(x,v)$ be its
back-time exit time defined in Definition \ref{exit}. Then for a.e. $(x,v),$ 
\begin{eqnarray}
\{G(t,0)h_{0}\}(t,x,v) &=&\mathbf{1}_{t-t_{\mathbf{b}}\leq 0}e^{-\nu
(v)t}h_{0}(x-tv,v)  \notag \\
&&+\mathbf{1}_{t-t_{\mathbf{b}}>0}e^{-\nu (v)t_{\mathbf{b}}}\{wg\}(t-t_{%
\mathbf{b}},x-t_{\mathbf{b}}v,v).  \label{inflowformula}
\end{eqnarray}%
Moreover,%
\begin{equation}
\sup_{t\geq 0}e^{\nu _{0}t}||G(t,0)h_{0}||_{\infty }\leq ||h_{0}||_{\infty
}+\sup_{s\geq 0}e^{\nu _{0}s}||wg(s)||_{\infty }.  \label{inflowgdecay}
\end{equation}
\end{lemma}

\begin{proof}
For almost every $x,v,$ along the characteristic line $\frac{dx}{d\tau }=v,%
\frac{dv}{d\tau }=0,$ 
\begin{equation*}
\frac{d}{d\tau }\{e^{\nu (v)\tau }G(\tau ,0)h_{0}\}\equiv 0.
\end{equation*}%
Hence $e^{\nu (v)\tau }G(\tau ,0)h_{0}$ is constant along the
characteristic. Choose any point $(t,x,v)$ in $[0,\infty )\times \Omega
\times \mathbf{R}^{3}$ with its backward exit point $(t-t_{\mathbf{b}},x_{%
\mathbf{b}},v)$. If $t-t_{\mathbf{b}}(x,v)\leq 0,$ then the backward
trajectory first hits on the initial plane $\{t=0\}$. On the other hand, if $%
t-t_{\mathbf{b}}(x,v)>0,$ then the backward trajectory first hits the
boundary. Since $\{G(\tau ,0)h_{0}\}_{\gamma _{-}}=wg$ a.e., from part (4)
of Lemma \ref{huang}, (\ref{inflowformula}) is clearly valid for almost
every $x,v,$ $x\in \bar{\Omega},$ and estimate (\ref{inflowgdecay}) follows
immediately from (\ref{inflowformula}) with $t_{\mathbf{b}}=t-(t-t_{\mathbf{b%
}})$.
\end{proof}

\begin{lemma}
\label{inflowcon}Let $\Omega $ be convex as in (\ref{convexity}). Let $%
h_{0}(x,v)$ be continuous in $\bar{\Omega}\times \mathbf{R}^{3}\setminus
\gamma _{0}\,,$ $g$ be continuous in $[0,\infty )\times \{\partial \Omega
\times \mathbf{R}^{3}\setminus \gamma _{0}\},$ $q(t,x,v)$ be continuous in
the interior of $[0,\infty )\times \Omega \times \mathbf{R}^{3}$ and $%
\sup_{[0,\infty )\times \Omega \times \mathbf{R}^{3}}|\frac{q(t,x,v)}{\nu (v)%
}|<\infty .$ Let $h(t,x,v)$ be the solution of 
\begin{equation*}
\{\partial _{t}+v\cdot \nabla _{x}+\nu \}h=q,\text{ \ \ \ \ }h(0)=h_{0},%
\text{ \ \ \ }h_{\gamma _{-}}=wg.
\end{equation*}%
Assume the compatibility condition on $\gamma _{-},$ 
\begin{equation}
h_{0}(x,v)=\{wg\}(0,x,v)  \label{inflowcom1}
\end{equation}%
Then $h(t,x,v)$ is continuous on $[0,\infty )\times \{\bar{\Omega}\times 
\mathbf{R}^{3}\setminus \gamma _{0}\}.$
\end{lemma}

\begin{proof}
Let $(x,v)$ $\notin \gamma _{0}$ and denote its backward exit time $[t-t_{%
\mathbf{b}},x_{\mathbf{b}},v]$. Since $\frac{d}{d\tau }\{e^{\nu (v)\tau
}G(\tau ,s)h\}=q$ along the characteristic $\frac{dx}{dt}=v,\frac{dv}{dt}=0,$
for $t-t_{\mathbf{b}}\leq 0,$ 
\begin{equation}
h(t,x,v)=e^{-\nu (v)t}h_{0}(x-vt,v)+\int_{0}^{t}e^{-\nu
(v)(t-s)}q(s,x-v(t-s),v)ds.  \label{inflowh1}
\end{equation}%
If $t-t_{\mathbf{b}}>0,$ we then have 
\begin{equation}
h(t,x,v)=e^{-\nu (v)t_{\mathbf{b}}}\{wg\}(t-t_{\mathbf{b}},x_{\mathbf{b}%
},v)+\int_{t-t_{\mathbf{b}}}^{t}e^{-\nu (v)(t-s)}q(s,x-v(t-s),v)ds.
\label{inflowh2}
\end{equation}

Since $(x,v)$ $\notin \gamma _{0},$ if $x$ $\notin \partial \Omega ,$ then $%
\xi (x)<0;$ and if $x\in \partial \Omega ,$ then $v\cdot \nabla \xi (x)\neq
0.$ This implies in (\ref{alpha}), $\alpha (t)>0.$ Since $\xi $ is convex
and $\xi (x_{\mathbf{b}})=0,$ we now apply Velocity Lemma \ref{velocity} to
get 
\begin{equation}
\alpha (t-t_{\mathbf{b}})=\{v\cdot \nabla \xi (x_{\mathbf{b}})\}^{2}\geq
c\alpha (t)>0.  \label{t-tb}
\end{equation}%
We thus conclude $v\cdot n(x_{\mathbf{b}})\neq 0$ and also $t_{\mathbf{b}%
}(x,v)>0$ by (\ref{tlower})$.$ Therefore, by Lemma \ref{huang}, $t_{\mathbf{b%
}}(x,v),$ $x_{\mathbf{b}}(x,v)$ are both smooth functions of $(x,v).$

Now take any point $(\bar{t},\bar{x},\bar{v})$ close to $(t,x,v)$ and we
separate three cases. If $t-t_{\mathbf{b}}(t,x,v)>0,$ when $(\bar{t},\bar{x},%
\bar{v})$ is close to $(t,x,v),$ $\bar{t}-t_{\mathbf{b}}(\bar{x},\bar{v})>0$
by continuity. Therefore 
\begin{equation}
h(\bar{t},\bar{x},\bar{v})=e^{-\nu (\bar{v})\bar{t}_{\mathbf{b}}}\{wg\}(\bar{%
t}-\bar{t}_{\mathbf{b}},\bar{x}_{\mathbf{b}},v)+\int_{\bar{t}-\bar{t}_{%
\mathbf{b}}}^{\bar{t}}e^{-\nu (\bar{v})(\bar{t}-s)}q(s,\bar{x}-\bar{v}(\bar{t%
}-s),\bar{v})ds.  \label{t1i}
\end{equation}%
From the continuity of $g$ away from $\gamma _{0},$ the second term above
tends to the second term in (\ref{inflowh2}). We split the third term into 
\begin{equation*}
\int_{\bar{t}-\bar{t}_{\mathbf{b}}}^{\bar{t}}=\int_{\bar{t}-\varepsilon }^{%
\bar{t}}+\int_{\bar{t}-\bar{t}_{\mathbf{b}}+\varepsilon }^{\bar{t}%
-\varepsilon }+\int_{\bar{t}-\bar{t}_{\mathbf{b}}}^{\bar{t}-\bar{t}_{\mathbf{%
b}}+\varepsilon },
\end{equation*}%
where $\varepsilon >0$ is small. The first and the third parts above are
small since $\frac{q}{\nu }$ is bounded, from our assumption. Notice that $%
x-v(t-s)$ is inside the interior of $\Omega $ for $\bar{t}-\bar{t}_{\mathbf{b%
}}+\varepsilon \leq s\leq \bar{t}-\varepsilon ,$ \ the middle term above
tends to $\int_{t-t_{\mathbf{b}}+\varepsilon }^{t-\varepsilon }e^{-\nu
(v)(t-s)}q(s,x-v(t-s),v)ds$ in (\ref{inflowh2}), from the interior
continuity of $q.$ Clearly $|h(t,x,v)-h(\bar{t},\bar{x},\bar{v})|\rightarrow
0$ as $(\bar{t},\bar{x},\bar{v})\rightarrow (t,x,v)$ in this case.

In the case $t-t_{\mathbf{b}}(x,v)<0,$ $x-vt\notin \partial \Omega .$ Then
for $(\bar{t},\bar{x},\bar{v})$ close to $(t,x,v),$ we have $\bar{t}-t_{%
\mathbf{b}}(\bar{x},\bar{v})<0,$ $\bar{x}-\bar{v}\bar{t}\notin \partial
\Omega ,$ and 
\begin{equation}
h(\bar{t},\bar{x},\bar{v})=e^{-\nu (\bar{v})\bar{t}}h_{0}(\bar{x}-\bar{v}%
\bar{t},\bar{v})+\int_{0}^{\bar{t}}e^{-\nu (\bar{v})(\bar{t}-s)}q(s,\bar{x}-%
\bar{v}(t-s),\bar{v})ds.  \label{t1b}
\end{equation}%
Since $h_{0}$ is continuous away from $\gamma _{0},$ and $q$ is continuous
and $\frac{q}{\nu }$ is bounded in the interior, we again deduce that $h(%
\bar{t},\bar{x},\bar{v})\rightarrow h(t,x,v)$ by the same argument as in the
first case $t-t_{\mathbf{b}}>0.$

Lastly, if $t-t_{\mathbf{b}}(x,v)=0$ and (\ref{inflowh1}) is valid. By (\ref%
{t-tb}), $x_{\mathbf{b}}=x-t_{\mathbf{b}}v=x-tv,$ and $(x-tv,v)\notin \gamma
_{0}.$ Then for any $(\bar{t},\bar{x},\bar{v})$ near $(t,x,v),$ $\bar{t}-%
\bar{t}_{\mathbf{b}}$ could be either $>0$ or $\leq 0.$ If $\bar{t}-\bar{t}_{%
\mathbf{b}}\leq 0,$ then $h(\bar{t},\bar{x},\bar{v})$ still has the same
expression (\ref{t1b}) as $h(t,x,v)$ and $h(\bar{t},\bar{x},\bar{v}%
)\rightarrow h(t,x,v)$ as before. On the other hand, if $\bar{t}-\bar{t}_{%
\mathbf{b}}>0,$ $h(\bar{t},\bar{x},\bar{v})$ is given by (\ref{t1i}). By the
Velocity Lemma \ref{velocity} and Lemma \ref{huang}, we have that $|\bar{t}-%
\bar{t}_{\mathbf{b}}|+|\bar{x}_{\mathbf{b}}-x_{\mathbf{b}}|\rightarrow 0,$
so that by the previous argument, 
\begin{equation*}
\lim_{(\bar{t},\bar{x},\bar{v})\rightarrow (t,x,v)}h(\bar{t},\bar{x},\bar{v}%
)=e^{-\nu (v)t}\{wg\}(0,x_{\mathbf{b}},v)+\int_{0}^{t}e^{-\nu
(v)t}q(s,x-v(t-s),v)ds.
\end{equation*}%
But $\{wg\}(0,x_{\mathbf{b}},v)=h_{0}(x_{\mathbf{b}},v)$ by the
compatibility condition (\ref{inflowcom1})$,$ hence this limit equals to $%
h(t,x,v)$ given by (\ref{inflowh1}).
\end{proof}

\subsubsection{Decay of In-flow $U(t,0)$}

\begin{theorem}
\label{inflowrate}Let $\{U(t,0)h_{0}\}$ be the solution to the weighted
linear Boltzmann equation (\ref{lboltzmannh}) as 
\begin{equation*}
\{\partial _{t}+v\cdot \nabla _{x}+\nu -K_{w}\}U(t,0)h_{0}=0,\text{ \ }%
U(0,0)h_{0}=h_{0},\text{ \ \ \ }\{U(t,0)h_{0}\}_{\gamma _{-}}=wg.
\end{equation*}%
There exists $0<\lambda <\lambda _{0}$ such that%
\begin{equation*}
\sup_{t\geq 0}e^{\lambda t}||U(t,0)h_{0}||_{\infty }\leq
C\{||h_{0}||_{\infty }+\sup_{0\leq s\leq \infty }e^{\lambda
_{0}s}||wg(s)||_{\infty }\}.
\end{equation*}
\end{theorem}

\begin{proof}
By (\ref{inflowh1}) and (\ref{inflowh2}), we have $\{U(t,0)h_{0}\}(t,x,v)=$ 
\begin{eqnarray*}
&&\mathbf{1}_{t-t_{\mathbf{b}}\leq 0}e^{-\nu (v)t}h_{0}(x-vt,v)+\mathbf{1}%
_{t-t_{\mathbf{b}}>0}e^{-\nu (v)t_{\mathbf{b}}}\{wg\}(t-t_{\mathbf{b}},x_{%
\mathbf{b}},v) \\
&&+\int_{\max \{0,t-t_{\mathbf{b}}\}}^{t}e^{-\nu
(v)(t-s_{1})}\{K_{w}U(s_{1},0)h_{0}\}(s_{1},x-v(t-s_{1}),v)ds_{1}.
\end{eqnarray*}%
Let $x_{1}=x-v(t-s_{1})$, $t_{\mathbf{b}}^{\prime }$ be the exit time for $%
(x_{1},v^{\prime })$ and $x_{\mathbf{b}}^{\prime }=x_{1}-v^{\prime }t_{%
\mathbf{b}}^{\prime }$. We now further iterate this formula to evaluate $%
\{K_{w}U(s_{1},0)h_{0}\}$ as 
\begin{eqnarray}
&&\int_{\mathbf{R}^{3}}K_{w}(v,v^{\prime
})\{U(s_{1},0)h_{0}\}(s_{1},x_{1},v^{\prime })dv^{\prime }  \label{double} \\
&=&\int_{\mathbf{R}^{3}}K_{w}(v,v^{\prime })\mathbf{1}_{s_{1}-t_{\mathbf{b}%
}^{\prime }\leq 0}e^{-\nu (v^{\prime })s_{1}}h_{0}(x_{1}-v^{\prime
}s_{1},v^{\prime })dv^{\prime }  \notag \\
&&+\int_{\mathbf{R}^{3}}K_{w}(v,v^{\prime })\mathbf{1}_{0<s_{1}-t_{\mathbf{b}%
}^{\prime }}e^{-\nu (v^{\prime })t_{\mathbf{b}}^{\prime }}\{wg\}(s_{1}-t_{%
\mathbf{b}}^{\prime },x_{\mathbf{b}}^{\prime },v^{\prime })\}dv^{\prime } 
\notag \\
&&+\int_{\max \{0,s_{1}-t_{\mathbf{b}}^{\prime }\}}^{s_{1}}e^{-\nu
(v^{\prime })(s_{1}-s)}\int K_{w}(v,v^{\prime })K_{w}(v^{\prime },v^{\prime
\prime })\{U(s,0)h_{0}\}(s,x_{1}-v^{\prime }(s_{1}-s),v^{\prime \prime
})dv^{\prime }dv^{\prime \prime }ds  \notag
\end{eqnarray}%
We note that $||K_{w}h||_{\infty }\leq C||h||_{\infty }$ from (\ref{wk}) in
Lemma \ref{kernel}. Clearly, $\ $since $\nu (v),\nu (v^{\prime })\geq \nu
_{0}>0$ for hard potentials, 
\begin{equation*}
e^{-\nu (t-s_{1})}e^{-\nu (v^{\prime })(s_{1}-s)}\leq e^{-\nu _{0}(t-s)},%
\text{ \ }e^{-\nu (v)(t-s_{1})}e^{-\nu (v^{\prime })t_{\mathbf{b}}^{\prime
}}\leq e^{-\nu _{0}t}e^{\nu _{0}(s_{1}-t_{\mathbf{b}}^{\prime })}.
\end{equation*}%
Plugging (\ref{double}) back into $\{U(t,0)h_{0}\}(t,x,v)$ yields that all
the terms except for the last term in (\ref{double}) are bounded by ($%
0<\lambda <\nu _{0}$): 
\begin{eqnarray}
&&e^{-\nu _{0}t}||h_{0}||_{\infty }+e^{-\nu _{0}t}\sup_{0\leq s\leq \infty
}e^{\lambda s}||wg(s)||_{\infty }+  \notag \\
&&+C_{K}\int_{\min \{0,t-t_{\mathbf{b}}\}}^{t}\{e^{-\nu
_{0}t}||h_{0}||_{\infty }+e^{-\nu _{0}t}\sup_{0\leq s\leq \infty }e^{\nu
_{0}s}||wg(s)||_{\infty }\}ds_{1}  \notag \\
&\leq &C_{K}\{t+1\}e^{-\nu _{0}t}\{||h_{0}||_{\infty }+\sup_{0\leq s\leq
\infty }e^{\nu _{0}s}||wg(s)||_{\infty }\}.  \label{gbound}
\end{eqnarray}

We now concentrate on the last term in (\ref{double}) and split the
velocity-time integration into several regions. We first consider the case $%
|v|\geq N.$

\textbf{CASE 1:} For $|v|\geq N.$ Since from (\ref{wk}) with $\varepsilon =0$
in Lemma \ref{kernel}, 
\begin{equation*}
\int \int K_{w}(v,v^{\prime })K_{w}(v^{\prime },v^{\prime \prime
})dv^{\prime }dv^{\prime \prime }\leq \frac{C_{K}}{1+|v|}\leq \frac{C_{K}}{N}%
,
\end{equation*}%
By Lemma \ref{kernel} again, the double-time integration $\int_{\max
\{0,t-t_{\mathbf{b}}\}}^{t}\int_{\max \{0,s_{1}-t_{\mathbf{b}}^{\prime
}\}}^{s_{1}}$ for $|v|\geq N$ is controlled by 
\begin{eqnarray}
\frac{C_{K}}{N}\int_{0}^{t}\int_{0}^{s_{1}}e^{-\nu
_{0}(t-s)}||U(s,0)h_{0}||_{\infty }dsds_{1} &\leq &  \label{inflowstep2} \\
\frac{C_{K}e^{-\frac{\nu _{0}t}{2}}}{N}\sup_{s}\{e^{\frac{\nu _{0}s}{2}%
}||U(s,0)h_{0}||_{\infty }\}\int_{0}^{t}\int_{0}^{s_{1}}e^{-\frac{\nu
_{0}(t-s)}{2}}dsds_{1} &\leq &\frac{C_{K}e^{-\frac{\nu _{0}t}{2}}}{N}%
\sup_{s}\{e^{\frac{\nu _{0}s}{2}}||U(s,0)h_{0}||_{\infty }\},  \notag
\end{eqnarray}%
where we have split the exponent as 
\begin{equation}
e^{-\nu _{0}(t-s)}=e^{-\frac{\nu _{0}t}{2}}e^{-\frac{\nu _{0}(t-s)}{2}}e^{%
\frac{\nu _{0}s}{2}},  \label{tsplit}
\end{equation}%
and used the fact $\int_{0}^{t}\int_{0}^{s_{1}}e^{-\frac{\nu _{0}(t-s)}{2}%
}dsds_{1}<+\infty $ by a direct computation.

\textbf{CASE 2:}\textit{\ }For $|v|\leq N,$ $|v^{\prime }|\geq 2N,$ or $%
|v^{\prime }|\leq 2N$, $|v^{\prime \prime }|\geq 3N.$ Notice that we have
either $|v^{\prime }-v|\geq N$ or $|v^{\prime }-v^{\prime \prime }|\geq N,$
and either one of the following are valid correspondingly: 
\begin{equation}
|K_{w}(v,v^{\prime })|\leq e^{-\frac{\varepsilon }{8}N^{2}}|K_{w}(v,v^{%
\prime })e^{\frac{\varepsilon }{8}|v-v^{\prime }|^{2}}|,\text{ \ \ \ \ \ }%
|K_{w}(v^{\prime },v^{\prime \prime })|\leq e^{-\frac{\varepsilon }{8}%
N^{2}}|K_{w}(v^{\prime },v^{\prime \prime })e^{\frac{\varepsilon }{8}%
|v^{\prime }-v^{\prime \prime }|^{2}}|.  \label{kwe}
\end{equation}%
From (\ref{wk}) in Lemma \ref{kernel}, both $\int |K_{w}(v,v^{\prime })e^{%
\frac{\varepsilon }{8}|v-v^{\prime }|^{2}}|$ and $\int |K_{w}(v^{\prime
},v^{\prime \prime })e^{\frac{\varepsilon }{8}|v^{\prime }-v^{\prime \prime
}|^{2}}|$ are still finite. By (\ref{tsplit}), we use (\ref{kwe}) to combine
the cases of $|v^{\prime }-v|\geq N$ or $|v^{\prime }-v^{\prime \prime
}|\geq N$ as: 
\begin{eqnarray}
&&\int_{\max \{0,t-t_{\mathbf{b}}\}}^{t}\int_{\max \{0,s_{1}-t_{\mathbf{b}%
}^{\prime }\}}^{s_{1}}\left\{ \int_{|v|\leq N,|v^{\prime }|\geq 2N,\text{ \
\ }}+\int_{|v^{\prime }|\leq 2N,|v^{\prime \prime }|\geq 3N}\right\}   \notag
\\
&\leq &C_{K}\int_{0}^{t}\int_{0}^{s_{1}}||U(s,0)h_{0}||_{\infty }\left\{
\int_{|v|\leq N,|v^{\prime }|\geq 2N,\text{ \ \ }}|K_{w}(v,v^{\prime
})|dv^{\prime }+\sup_{v^{\prime }}\int_{|v^{\prime }|\leq 2N,|v^{\prime
\prime }|\geq 3N}|K_{w}(v^{\prime },v^{\prime \prime })|dv^{\prime \prime
}\right\}   \notag \\
&\leq &C_{\varepsilon ,K}e^{-\frac{\varepsilon }{8}N^{2}}\int_{0}^{t}%
\int_{0}^{s_{1}}e^{-\nu _{0}(t-s)}||U(s,0)h_{0}||_{\infty }dsds_{1}  \notag
\\
&\leq &C_{\varepsilon ,K}e^{-\frac{\varepsilon }{8}N^{2}}e^{-\frac{\nu _{0}t%
}{2}}\sup_{s\geq 0}\{e^{\frac{\nu _{0}}{2}s}||U(s,0)h_{0}||_{\infty }\}.
\label{inflowstep3}
\end{eqnarray}

\textbf{CASE 3:}\textit{\ \thinspace }$s_{1}-s\leq \varepsilon ,$ for $%
\varepsilon >0$ small. We bound the last term in (\ref{double}) by%
\begin{eqnarray}
&&\int_{\min \{0,t-t_{\mathbf{b}}\}}^{t}\int_{s_{1}-\varepsilon
}^{s_{1}}C_{K}e^{-\nu _{0}(t-s)}||U(s,0)h_{0}||_{\infty }dsds_{1}  \notag \\
&\leq &C_{K}e^{\frac{-\nu _{0}t}{2}}\int_{0}^{t}\int_{s_{1}-\varepsilon
}^{s_{1}}e^{\frac{-\nu _{0}(t-s)}{2}}\{e^{\frac{\nu _{0}s}{2}%
}||U(s,0)h_{0}||_{\infty }\}dsds_{1}  \notag \\
&\leq &C_{K}e^{\frac{-\nu _{0}t}{2}}\sup_{s\geq 0}\{e^{\frac{\nu _{0}s}{2}%
}||U(s,0)h_{0}||_{\infty }\}\times \int_{0}^{t}\int_{s_{1}-\varepsilon
}^{s_{1}}e^{\frac{-\nu _{0}(t-s_{1})}{2}}dsds_{1}  \notag \\
&\leq &C_{K}e^{\frac{-\nu _{0}t}{2}}\sup_{s\geq 0}\{e^{\frac{\nu _{0}s}{2}%
}||U(s,0)h_{0}||_{\infty }\}\times \varepsilon \int_{0}^{t}e^{\frac{-\nu
_{0}(t-s_{1})}{2}}ds_{1}  \notag \\
&\leq &C_{K}\varepsilon e^{\frac{-\nu _{0}t}{2}}\sup_{s\geq 0}\{e^{\frac{\nu
_{0}s}{2}}||U(s,0)h_{0}||_{\infty }\}.  \label{inflowstep1}
\end{eqnarray}

\textbf{CASE 4.} $s_{1}-s\geq \varepsilon ,$ and $|v|\leq N,$ $|v^{\prime
}|\leq 2N,|v^{\prime \prime }|\leq 3N.$ This is the last remaining case
because if $|v^{\prime }|>2N,$ it is included in Case 2; while if $%
|v^{\prime \prime }|>3N,$ either $|v^{\prime }|\leq 2N$ or $|v^{\prime
}|\geq 2N$ are also included in Case 2. We now can bound the integral of the
third term in (\ref{double}) by 
\begin{equation*}
C\int_{\max \{0,t-t_{\mathbf{b}}\}}^{t}\int_{B}\int_{\max \{0,s_{1}-t_{%
\mathbf{b}}^{\prime }\}}^{s_{1}-\varepsilon }e^{-\nu
_{0}(t-s)}|K_{w}(v,v^{\prime })K_{w}(v^{\prime },v^{\prime \prime
})\{U(s,0)h_{0}\}(s,x_{1}-(s_{1}-s)v^{\prime },v^{\prime \prime })|
\end{equation*}%
where $B=\{|v^{\prime }|\leq 2N,$ $|v^{\prime \prime }|\leq 3N\}.$ By (\ref%
{grad}), $K_{w}(v,v^{\prime })$ has possible integrable singularity of $%
\frac{1}{|v-v^{\prime }|},$ we can choose $K_{N}(v,v^{\prime },v^{\prime
\prime })$ smooth with compact support such that 
\begin{equation}
\sup_{|p|\leq 3N}\int_{|v^{\prime }|\leq 3N}|K_{N}(p,v^{\prime
})-K_{w}(p,v^{\prime })|dv^{\prime }\leq \frac{1}{N}.  \label{approximate}
\end{equation}%
Splitting 
\begin{eqnarray*}
K_{w}(v,v^{\prime })K_{w}(v^{\prime },v^{\prime \prime })
&=&\{K_{w}(v,v^{\prime })-K_{N}(v,v^{\prime })\}K_{w}(v^{\prime },v^{\prime
\prime }) \\
&&+\{K_{w}(v^{\prime },v^{\prime \prime })-K_{N}(v^{\prime },v^{\prime
\prime })\}K_{N}(v,v^{\prime })+K_{N}(v,v^{\prime })K_{N}(v^{\prime
},v^{\prime \prime }),
\end{eqnarray*}%
we can use such an approximation (\ref{approximate}) to bound the above $%
s_{1},s$ integration by 
\begin{eqnarray}
&&\frac{Ce^{-\frac{\nu _{0}t}{2}}}{N}\sup_{s}\{e^{\frac{\nu _{0}}{2}%
s}||U(s,0)h_{0}||_{\infty }\}\times \left\{ \sup_{|v^{\prime }|\leq 2N}\int
|K_{w}(v^{\prime },v^{\prime \prime })|dv^{\prime \prime }+\sup_{|v|\leq
2N}\int |K_{N}(v,v^{\prime })|dv^{\prime }\}\right\}   \label{inflowstep41}
\\
&&+C\int_{\max \{0,t-t_{\mathbf{b}}\}}^{t}\int_{B}\int_{\max \{0,s_{1}-t_{%
\mathbf{b}}^{\prime }\}}^{s_{1}-\varepsilon }e^{-\nu
_{0}(t-s)}|K_{N}(v,v^{\prime })K_{N}(v^{\prime },v^{\prime \prime
})|\{U(s,0)h_{0}\}(s,x_{1}-(s_{1}-s)v^{\prime },v^{\prime \prime })|.  \notag
\end{eqnarray}%
Note that $x_{1}-(s_{1}-s)v^{\prime }\in \Omega $ for either $s_{1}-t_{%
\mathbf{b}}^{\prime }<0,$ $s\geq 0,$ or $0\leq s_{1}-t_{\mathbf{b}}^{\prime
}\leq s$. Split 
\begin{equation*}
\int_{\max \{0,s_{1}-t_{\mathbf{b}}^{\prime }\}}^{s_{1}-\varepsilon
}=\int_{0}^{s_{1}-\varepsilon }\{\mathbf{1}_{s_{1}-t_{\mathbf{b}}^{\prime
}<0}+\mathbf{1}_{0\leq s_{1}-t_{\mathbf{b}}^{\prime }\leq s}\}.
\end{equation*}%
for the last main term in (\ref{inflowstep41}). Since $K_{N}(v,v^{\prime
})K_{N}(v^{\prime },v^{\prime \prime })$ is bounded, we first integrate over 
$v^{\prime }$ to get 
\begin{eqnarray*}
&&C_{N}\int_{|v^{\prime }|\leq 2N}\{\mathbf{1}_{s_{1}-t_{\mathbf{b}}^{\prime
}<\tau }(v^{\prime })+\mathbf{1}_{0\leq s_{1}-t_{\mathbf{b}}^{\prime }\leq
s}(v^{\prime })\}|\{U(s,0)h_{0}\}(s,x_{1}-(s_{1}-s)v^{\prime },v^{\prime
\prime })|dv^{\prime } \\
&\leq &C_{N}\left\{ \int_{|v^{\prime }|\leq 2N}\mathbf{1}_{\Omega
}(x_{1}-(s_{1}-s)v^{\prime })|\{U(s,0)h_{0}\}(s,x_{1}-(s_{1}-s)v^{\prime
},v^{\prime \prime })|^{2}dv^{\prime }\right\} ^{1/2} \\
&\leq &\frac{C_{N}}{\varepsilon ^{3}}\left\{ \int_{\Omega
}|\{U(s,0)h_{0}\}(y,v^{\prime \prime })|^{2}dy\right\} ^{1/2}.
\end{eqnarray*}%
Here we have made a change of variable $y=x_{1}-(s_{1}-s)v^{\prime }\in
\Omega ,$ and for $s_{1}-s\geq \varepsilon ,$ $\frac{dy}{dv^{\prime }}\geq 
\frac{1}{\varepsilon ^{3}}.$ Denote $U(s,0)h_{0}=wf(s)$ so that $f$ \ is a $%
L^{2}$ solution to the linear Boltzmann equation (\ref{lboltzmann}) with $%
f(0)=\frac{h_{0}}{w}$ and $f_{\gamma _{-}}=g$. We then further control the
last term in (\ref{inflowstep41}) by:%
\begin{eqnarray}
&&\frac{C_{N}}{\varepsilon ^{3}}\int_{\max \{0,t-t_{\mathbf{b}%
}\}}^{t}\int_{0}^{s_{1}-\varepsilon }e^{-\nu _{0}(t-s)}\int_{|v^{\prime
\prime }|\leq 3N}\left\{ \int_{\Omega }|\{U(s,0)h_{0}\}(y,v^{\prime \prime
})|^{2}dy\right\} ^{1/2}dv^{\prime \prime }dsds_{1}  \notag \\
&\leq &\frac{C_{N}}{\varepsilon ^{3}}\int_{0}^{t}\int_{0}^{s_{1}-\varepsilon
}e^{-\nu _{0}(t-s)}\left\{ \int_{|v^{\prime \prime }|\leq 3N}\int_{\Omega
}|\{U(s,0)h_{0}\}(y,v^{\prime \prime })|^{2}dydv^{\prime \prime }\right\}
^{1/2}dsds_{1}  \notag \\
&\leq &\frac{C_{N}}{\varepsilon ^{3}}\int_{0}^{t}\int_{0}^{s_{1}-\varepsilon
}e^{-\nu _{0}(t-s)}\left\{ \int_{|v^{\prime \prime }|\leq 3N}\int_{\Omega
}|f(s,y,v^{\prime \prime })|^{2}dydv^{\prime \prime }\right\} ^{1/2}dsds_{1}
\notag \\
&\leq &\frac{C_{N}}{\varepsilon ^{3}}e^{-\lambda t}\sup_{s\geq
0}\{e^{\lambda s}||f(s)||\}\int_{0}^{t}\int_{0}^{s_{1}}e^{-\frac{\nu _{0}}{2}%
(t-s)}dsds_{1}=\frac{C_{N}}{\varepsilon ^{3}}e^{-\lambda t}\sup_{s\geq
0}\{e^{\lambda s}||f(s)||\}  \label{inflowstep42} \\
&\leq &\frac{C_{N}}{\varepsilon ^{3}}e^{-\lambda t}\left[ ||f(0)||+\left\{
\int_{0}^{s}e^{2\lambda \theta }||g(\theta )||_{\gamma _{-}}^{2}d\theta
\right\} ^{1/2}\right] ,  \notag
\end{eqnarray}%
where we have used crucially part (1) of Theorem \ref{L2decay} with some $%
0<\lambda <\frac{\nu _{0.}}{2}$ in the last line$.$ Note that since $%
\{1+|v|\}w^{-2}\in L^{1}(\mathbf{R}^{3}),$ $||f(0)||=||w^{-1}h_{0}||\leq
C||h_{0}||_{\infty },$ and $||g||_{\gamma _{-}}=||\frac{w}{w}g||_{\gamma
_{-}}\leq C||wg||_{\infty },$ and $\int_{0}^{s}e^{2\{\lambda -\lambda
_{0}\}\theta }d\theta <\infty ,$ where $\lambda _{0}$ is in Theorem \ref%
{inflownl}.  We can then further bound (\ref{inflowstep42}) by 
\begin{equation*}
\frac{C_{N,\lambda }}{\varepsilon ^{3}}e^{-\lambda t}\left[
||h_{0}||_{\infty }+\sup_{0\leq s\leq \infty }e^{\lambda
_{0}s}||wg(s)||_{\infty }\right] .
\end{equation*}

In summary, replacing $\nu _{0},\frac{\nu _{0}}{2}$ by $\lambda $ and
combining (\ref{gbound}), (\ref{inflowstep1}), (\ref{inflowstep2}), (\ref%
{inflowstep3}), (\ref{inflowstep41}) and (\ref{inflowstep42}), we have
established, for any $\varepsilon >0$ and large $N>0,$%
\begin{equation*}
\sup_{t\geq 0}\{e^{\lambda t}||U(t,0)h_{0}||_{\infty }\}\leq \{\varepsilon +%
\frac{C_{\varepsilon }}{N}\}\sup_{s\geq 0}\{e^{\lambda
s}||U(s,0)h_{0}||_{\infty }\}+C_{K}\sup_{0\leq s}e^{2\lambda
_{0}s}||wg(s)||_{\infty }+C_{\varepsilon ,N}||h_{0}||_{\infty }.
\end{equation*}%
First choosing $\varepsilon $ small, then $N$ sufficiently large so that $%
\{\varepsilon +\frac{C_{\varepsilon }}{N}\}<\frac{1}{2}$, 
\begin{equation*}
\sup_{t\geq 0}\{e^{\lambda t}||U(t,\tau )h||_{\infty }\}\leq
2C_{K}\sup_{0\leq s\leq \infty }e^{\lambda _{0}s}||wg(s)||_{\infty
}+2C_{\varepsilon ,N}||h_{0}||_{\infty },
\end{equation*}%
and we conclude our proof.
\end{proof}

\subsection{$L^{\infty }$ Decay for the Bounce-Back Reflection}

\subsubsection{Bounce-Back Cycles and Continuity of $G(t)$}

\begin{definition}
\label{bouncebackcycles}(\textbf{Bounce-Back Cycles}) Let $(t,x,v)\notin
\gamma _{0}.$ Let $(t_{0},x_{0},v_{0})=(t,x,v)$ and inductively define for $%
k\geq 1:$%
\begin{equation*}
(t_{k+1},x_{k+1},v_{k+1})=(t_{k}-t_{\mathbf{b}}(x_{k},v_{k}),x_{\mathbf{b}%
}(x_{k},v_{k}),-v_{k}).
\end{equation*}%
We define the back-time cycles as: 
\begin{equation}
X_{\mathbf{cl}}(s;t,x,v)=\sum_{k}\mathbf{1}_{[t_{k+1},t_{k})}(s)%
\{x_{k}+(s-t_{k})v_{k}\},\text{ \ }V_{\mathbf{cl}}(s;t,x,v)=\sum_{k}\mathbf{1%
}_{[t_{k+1},t_{k})}(s)v_{k}.  \label{bouncebackcycle}
\end{equation}
\end{definition}

\begin{remark}
Clearly, we have $v_{k+1}\equiv (-1)^{k+1}v,$ for $k\geq 1,$ 
\begin{equation}
x_{k}=\frac{1-(-1)^{k}}{2}x_{1}+\frac{1+(-1)^{k}}{2}x_{2},
\label{bouncebackx}
\end{equation}%
and let $d=t_{1}-t_{2},$ then for $k\geq 1,$ 
\begin{equation}
t_{k}-t_{k+1}=d\geq t-t_{\mathbf{b}}>0.  \label{bouncebackt}
\end{equation}
\end{remark}

We follow the outline in Section 1.6 and first establish an abstract lemma.

\begin{lemma}
\label{abstract}Let $\mathcal{M}$ be an operator on $L^{\infty }(\gamma
_{+})\rightarrow L^{\infty }(\gamma _{-})$ such that $||\mathcal{M}||_{%
\mathcal{L}(L^{\infty },L^{\infty })}=1.$ Then for any $\varepsilon >0,$
there exists $h(t)\in L^{\infty }$ and $h_{\gamma }\in L^{\infty }$ solving 
\begin{equation*}
\{\partial _{t}+v\cdot \nabla _{x}+\nu \}h=0,\text{ \ \ }h_{\gamma
_{-}}=(1-\varepsilon )Mh_{\gamma _{+}},\text{ \ }h(0,x,v)=h_{0}\in L^{\infty
}.
\end{equation*}
\end{lemma}

\begin{proof}
Fix $\varepsilon >0,$ we construct a solution by the following iteration
(with $h_{\gamma _{+}}^{0}\equiv 0$) for $k=0,1,2....$ 
\begin{equation*}
\{\partial _{t}+v\cdot \nabla _{x}+\nu \}h^{k+1}=0,\text{ \ \ }h_{\gamma
_{-}}^{k+1}=(1-\varepsilon )Mh_{\gamma _{+}}^{k},\text{ \ }%
h^{k+1}(0,x,v)=h_{0}.
\end{equation*}%
We now show $h^{k}$ $\ $and $h_{\gamma }^{k}$ is a Cauchy sequence. Taking
differences, we get 
\begin{equation*}
\{\partial _{t}+v\cdot \nabla _{x}+\nu \}\{h^{k+1}-h^{k}\}=0,\text{ \ \ }%
h_{\gamma _{-}}^{k+1}-h_{\gamma _{-}}^{k}=(1-\varepsilon )M\{h_{\gamma
_{+}}^{k}-h_{\gamma _{+}}^{k-1}\},\text{ \ }
\end{equation*}%
with zero initial datum \ $\{h^{k+1}-h^{k}\}_{t=0}=0.$ Note that from Lemma %
\ref{ginflowdecay},%
\begin{equation*}
\sup_{s}||h_{\gamma _{+}}^{k+1}(s)-h_{\gamma _{+}}^{k}(s)||_{\infty }\leq
(1-\varepsilon )\sup_{s}||h_{\gamma _{+}}^{k}(s)-h_{\gamma
_{+}}^{k-1}(s)||_{\infty }.
\end{equation*}%
Repeatedly using such inequality for $k=1,2,...,$ we obtain%
\begin{equation*}
\sup_{s}||h_{\gamma _{+}}^{k+1}(s)-h_{\gamma _{+}}^{k}(s)||_{\infty }\leq
(1-\varepsilon )^{k}\sup_{s}||h_{\gamma _{+}}^{1}(s)-h_{\gamma
_{+}}^{0}(s)||_{\infty }.
\end{equation*}%
Hence $\{h_{\gamma _{+}}^{k}\}$ is Cauchy in $L^{\infty }(\mathbf{R}\times
\gamma _{-}),~$\ and then both $\{h_{\gamma _{-}}^{k}\}$ and $\{h^{k}\}$ are
Cauchy respectively by Lemma \ref{ginflowdecay}. We deduce our lemma by
letting $k\rightarrow \infty .$
\end{proof}

\begin{lemma}
\label{bouncebackformula}Let $h_{0}\in L^{\infty }(\Omega \times \mathbf{R}%
^{3})$. There exists a unique solution $G(t)h_{0}$ of 
\begin{equation}
\{\partial _{t}+v\cdot \nabla _{x}+\nu \}\{G(t)h_{0}\}=0,\text{ \ \ \ \ }%
\{G(0)h_{0}\}=h_{0},  \label{gh0}
\end{equation}%
with the bounce-back reflection $\{G(t)h_{0}\}(t,x,v)=\{G(t)h_{0}\}(t,x,-v)$
for $x\in \partial \Omega .$ For almost any $(x,v)\in \bar{\Omega}\times 
\mathbf{R}^{3}\setminus \gamma _{0},$ 
\begin{equation}
\{G(t)h_{0}\}(t,x,v)=\sum_{k}\mathbf{1}_{[t_{k+1},t_{k})}(0)e^{-\nu
(v)t}h_{0}\left( X_{\mathbf{cl}}(0),V_{\mathbf{cl}}(0)\right) .
\label{bouncebackformular}
\end{equation}%
Moreover, $e^{\nu _{0}t}||G(t)h_{0}||_{\infty }\leq ||h_{0}||_{\infty }.$
\end{lemma}

\begin{proof}
For any $\varepsilon >0,$ by Lemma \ref{abstract}, there exists a solution $%
h^{\varepsilon }$ of 
\begin{equation*}
\{\partial _{t}+v\cdot \nabla _{x}+\nu \}h^{\varepsilon }=0,\text{ \ \ }%
h^{\varepsilon }(t,x,v)=(1-\varepsilon )h^{\varepsilon }(t,x,-v),\text{ \ }%
h^{\varepsilon }(0,x,v)=h_{0}.
\end{equation*}%
with finite $||h^{\varepsilon }(t,\cdot )||_{\infty }$ and $%
\sup_{t}||h_{\gamma }^{\varepsilon }(t,\cdot )||_{\infty }$. Such a solution
is necessary unique. This is because we can choose $w^{-2}\{1+|v|\}\in
L^{1}\,\ $so that $f^{\varepsilon }=\frac{h^{\varepsilon }}{w}\in L^{2}$ is
a $L^{2}$ solution to the same equation in (\ref{gh0}) with the same
boundary condition, with an additional property $\int_{0}^{t}||f^{%
\varepsilon }(s)||_{\gamma }^{2}ds<\infty .$ Then uniqueness follows from
the energy identity for $f^{\varepsilon }$.

Given any point $(t,x,v)\notin \gamma ^{0}$ and its back-time cycle $[X_{%
\mathbf{cl}}(s),V_{\mathbf{cl}}(s)].$ We notice $|V_{\mathbf{cl}}(s)|=|v|,$
for all $s,$ and $\frac{d}{ds}G(s)h_{0}\equiv -\nu G(s)h_{0}$ along the
back-time cycle $[X_{\mathbf{cl}}(s)$,$V_{\mathbf{cl}}(s)]$ for $t_{k+1}\leq
s<t_{k}.$ Together with the boundary condition at $s=t_{k}$, and part (4) of
Lemma \ref{huang}, we deduce that for almost every $(x,v),$ $e^{\nu
(v)t}G(s)h_{0}$ is constant along its back-time cycle $[X_{\mathbf{cl}}(s)$,$%
V_{\mathbf{cl}}(s)]$ in (\ref{bouncebackcycle})$.$ If $(x,v)\in \bar{\Omega}%
\times \mathbf{R}^{3}\setminus \gamma _{0},$ then $t_{\mathbf{b}}(x,v)>0,~$\
and 
\begin{equation*}
h^{\varepsilon }(t,x,v)=\sum_{k}\mathbf{1}_{[t_{k+1},t_{k})}(0)[1-%
\varepsilon ]^{k}e^{-\nu (v)t}h_{0}\left( X_{\mathbf{cl}}(0),V_{\mathbf{cl}%
}(0)\right) ,
\end{equation*}%
where the summation over $k$ is finite for finite $t$ by (\ref{bouncebackt})$%
.$ For all $\varepsilon ,$ 
\begin{equation*}
e^{\nu _{0}t}||h^{\varepsilon }(t)||_{\infty }\leq ||h_{0}||_{\infty },\text{
\ \ \ }\sup_{t\geq s,\gamma _{-}}|h^{\varepsilon }(t,x,v)|\leq \sup_{t\geq
s,\gamma _{+}}|h^{\varepsilon }(t,x,v)|\leq ||h_{0}||_{\infty },
\end{equation*}%
uniformly bounded. We thus can construct the solution $h$ to (\ref{gh0})
with the original bounce-back boundary condition by taking $w-\ast $ limit: $%
h(t,x,v)=\lim_{\varepsilon \rightarrow 0}h^{\varepsilon }(t,x,v),$ and $%
h_{\gamma }(t,x,v)=\lim_{\varepsilon \rightarrow 0}h_{\gamma }^{\varepsilon
}(t,x,v).$ We thus deduce our lemma by letting $\varepsilon \rightarrow 0.$
Once again, such a solution $h(t,x,v)$ is necessarily unique in the $%
L^{\infty }$ class because $f_{\gamma }=\frac{h_{\gamma }}{w}\in
L_{loc}^{2}(L^{2}(\gamma )).$
\end{proof}

\begin{lemma}
\label{bouncebackcon}Let $\xi $ be convex as in (\ref{convexity}). Let $%
h_{0} $ be continuous in $\bar{\Omega}\times \mathbf{R}^{3}\setminus \gamma
_{0}\,\ $\ and $q(t,x,v)$ be continuous in the interior of $[0,\infty
)\times \Omega \times \mathbf{R}^{3}$ and $\sup_{[0,\infty )\times \Omega
\times \mathbf{R}^{3}}|\frac{q(t,x,v)}{\nu (v)}|<\infty .$ Assume the
compatibility condition on $\gamma _{-}:$ $h_{0}(x,v)=h_{0}(x,-v)$. Then the
solution $h(t,x,v)$ of 
\begin{equation}
\{\partial _{t}+v\cdot \nabla _{x}+\nu \}h=q,\text{ \ \ \ \ }h(0,x,v)=h_{0},
\label{gq}
\end{equation}%
with $h(t,x,v)=h(t,x,-v)$, $x\in \partial \Omega $ is continuous on $%
[0,\infty )\times \{\bar{\Omega}\times \mathbf{R}^{3}\setminus \gamma
_{0}\}. $
\end{lemma}

\begin{proof}
Take any point $(t,x,v)$ $\notin \lbrack 0,\infty )\times \gamma _{0}$ and
denote its backward exit point $[t_{\mathbf{b}},x_{\mathbf{b}},v]$ along the
trajectory. Recall its back-time cycle and (\ref{bouncebackformula}). Assume 
$t_{m+1}\leq 0<t_{m}.$ Since $\frac{d}{d\tau }\{e^{\nu (v)}h\}=q$ along the
characteristics, $h(t,x,v)$ takes the form 
\begin{eqnarray}
&&e^{-\nu (v)t}h_{0}\left( x_{m}-t_{m}v_{m},v_{m}\right) +  \label{hg} \\
&&\sum_{k=0}^{m-1}\int_{t_{k+1}}^{t_{k}}e^{-\nu (v)(t-s)}q\left(
s,x_{k}+(s-t_{k})v_{k},v_{k}\right) ds+\int_{0}^{t_{m}}e^{-\nu
(v)(t-s)}q\left( s,x_{m}+(s-t_{m})v_{m},v_{m}\right) ds.  \notag
\end{eqnarray}%
Since $\Omega $ is convex and $(x,v)\notin \gamma _{0},$ then from the
Velocity Lemma \ref{velocity}, $n(x_{1})\cdot v_{1}\neq 0.$ Notice that $%
x_{k}\in \partial \Omega $ and $\xi (x_{k})=0$ for $k\geq 1$ so that 
\begin{equation*}
\alpha (t_{k})=(v_{k}\cdot \nabla \xi (x_{k}))^{2}.
\end{equation*}%
We now apply the Velocity Lemma \ref{velocity} to conclude $\alpha
(t_{k})>C\alpha (t)>0$ and $v_{k}\cdot n(x_{k})\neq 0$ for all $k\geq 1.$ By
Lemma \ref{huang}, $t_{k},$ and $x_{k}$ for $1\leq k\leq n$ are smooth
functions of $(t,x,v).$ For any other point $(\bar{t},\bar{x},\bar{v})$
which is close to $(t,x,v),$ we deduce that $\bar{t}_{m}>0.$

In the case that $t_{m+1}<0,$ or equivalently, $x_{m}-t_{m}x_{m}\in \Omega ,$
from continuity, $\bar{t}_{m+1}(\bar{t},\bar{x},\bar{v})<0.$ Therefore, $h(%
\bar{t},\bar{x},\bar{v})$ has the same expression as $h(t,x,v)$ in (\ref{hg}%
). Therefore, $h(\bar{t},\bar{x},\bar{v})\rightarrow h(t,x,v)$ because $\bar{%
x}_{k}\rightarrow x_{k},$ and $\bar{t}_{k}\rightarrow t_{k},\bar{v}%
_{k}\rightarrow v_{k}$ for $1\leq k\leq m+1,$ as in the proof of Lemma \ref%
{inflowcon}.

On the other hand, if $t_{m+1}=0,$ or equivalently, $%
x_{m+1}=x_{m}-t_{m}x_{m}\in \partial \Omega ,$ and $(x_{m+1},v_{m})\notin
\gamma _{0}.$ Then by continuity, we know that $\bar{t}_{m+1}(\bar{t},\bar{x}%
,\bar{v})$ is close to zero. In the case that $\bar{t}_{m+1}(\bar{t},\bar{x},%
\bar{v})\leq 0,$ then (\ref{hg}) is again valid and the continuity follows
as before. However, if $\bar{t}_{m+1}(\bar{t},\bar{x},\bar{v})>0,$ then $%
\bar{t}_{m+2}(\bar{t},\bar{x},\bar{v})<0$ (because $t_{m+2}^{{}}<0$), and we
have a different expression for $h(\bar{t},\bar{x},\bar{v})$ as: 
\begin{eqnarray}
&&e^{-\nu (\bar{v})\bar{t}}h_{0}\left( \bar{x}_{m+1}-\bar{t}_{m+1}\bar{v}%
_{m+1},\bar{v}_{m+1}\right) +\sum_{k=0}^{m}\int_{\bar{t}_{k+1}}^{\bar{t}%
_{k}}e^{-\nu (\bar{v})(\bar{t}-s)}q\left( s,\bar{x}_{k}+(s-\bar{t}_{k})\bar{v%
}_{k},\bar{v}_{k}\right) ds  \notag \\
&&+\int_{0}^{\bar{t}_{m+1}}e^{-\nu (\bar{v})(\bar{t}-s)}q\left( s,\bar{x}%
_{m+1}+(s-\bar{t}_{m+1})\bar{v}_{m+1},\bar{v}_{m+1}\right) ds.  \label{hgbar}
\end{eqnarray}%
The last term in (\ref{hgbar}) goes to zero because $\bar{t}%
_{m+1}\rightarrow 0,$ and the second term in (\ref{hgbar}) tends to the $q$
integration in (\ref{hg}) since $\int_{0}^{t_{m}}=\int_{t_{m+1}}^{t_{m}}$.
We now show that the first term (\ref{hgbar}) tends to the first term in (%
\ref{hg}) as well. Since $\bar{x}_{m+1}-\bar{t}_{m+1}\bar{v}%
_{m+1}\rightarrow \bar{x}_{m+1}\rightarrow x_{m+1}=x_{m}-t_{m}v_{m}\in
\partial \Omega ,$ the first term in (\ref{hgbar}) tends to 
\begin{equation*}
h_{0}\left( x_{m}-t_{m}x_{m},v_{m}\right) =h_{0}(x_{m}-t_{m}x_{m},-v_{m}),
\end{equation*}%
which is exactly the first term in (\ref{hg}), from the compatibility
condition $h_{0}(x,v)=h_{0}(x,-v)$ on $\gamma .$ We therefore conclude the
continuity.
\end{proof}

\subsubsection{Non-Grazing Condition $|S_{x}|=0.$}

The following lemma is due to Hongjie Dong:

\begin{lemma}
\label{s=0}For any $x\in \bar{\Omega},$ define the set 
\begin{equation}
S_{x}(v)=\{v\in \mathbf{R}^{3}:v\cdot n(x_{\mathbf{b}}(x,v))=0\}=\{v\in 
\mathbf{R}^{3}:v\cdot \nabla \xi (x-t_{\mathbf{b}}(x,v)v)=0\}.
\label{0measure}
\end{equation}%
If $\partial \Omega $ is $C^{1},$ then $|S_{x}(v)|=0,$ where $|\cdot |$ is
the Lebesaue measure.
\end{lemma}

\begin{proof}
We first note that if $v\in S_{x}(v),$ then $kv\in S_{x}(v)$ for all $k>0.$
It therefore suffices to show that the surface measure $|S_{x}(v)\cap 
\mathbf{S}^{2}|=0.$

We fix $x\in \bar{\Omega}$ and recall $x_{\mathbf{b}}=x-t_{\mathbf{b}}v\in
\partial \Omega $. If $x\in \partial \Omega ,$ we require $v\cdot n(x)\neq 0$
(a zero measure set)$.$ Hence $\frac{v}{|v|}=-\frac{x_{\mathbf{b}}-x}{|x_{%
\mathbf{b}}-x|}$ and $x_{\mathbf{b}}\neq x.$ Our goal is to show that, if $%
v\in S_{x}(v),$ then $\frac{v}{|v|}$ is a critical value of the mapping from 
$\partial \Omega \rightarrow \mathbf{S}^{2}:$ 
\begin{equation}
\phi (y)=-\frac{y-x}{|y-x|}
\end{equation}%
at $y=x_{\mathbf{b}}.$ Since $\phi (y)$ is smooth for $y\neq x,$ by Sard's
theorem, $\frac{v}{|v|}$ has zero measure in $\mathbf{S}^{2}$ and our lemma
is valid.

Indeed, we assume locally around $x_{\mathbf{b}},$ $\partial \Omega
=(y_{1},y_{2},\eta (y_{1},y_{2}))$ and if $v\in S_{x}(v),$ then at $%
(y_{1},y_{2})=(x_{\mathbf{b}1},x_{\mathbf{b}2}):$ 
\begin{equation}
0=\frac{v}{|v|}\cdot n(x_{\mathbf{b}})=\{x_{\mathbf{b}1}-x_{1}\}\partial
_{1}\eta +\{x_{\mathbf{b}2}-x_{2}\}\partial _{1}\eta -\eta (x_{\mathbf{b}%
1},x_{\mathbf{b}2})+x_{3}=0.  \label{para}
\end{equation}%
Clearly, since $x_{\mathbf{b}}\neq x,$ from (\ref{para}), $[x_{\mathbf{b}%
1}-x_{1},x_{\mathbf{b}2}-x_{2}]\neq 0.$ But $[\partial _{y_{1}}\phi
,\partial _{y_{2}}\phi ]=$ 
\begin{equation}
\begin{bmatrix}
-\frac{1}{|y-x|}+\frac{(y_{1}-x_{1})^{2}+(y_{1}-x_{1})(\eta -x_{3})\partial
_{1}\eta }{|y-x|^{3}} & +\frac{(y_{1}-x_{1})\{y_{2}-x_{2}+(\eta
-x_{3})\partial _{2}\eta \}}{|y-x|^{3}} \\ 
\frac{(y_{2}-x_{2})\{y_{1}-x_{1}+(\eta -x_{3})\partial _{1}\eta \}}{|y-x|^{3}%
} & -\frac{1}{|y-x|}+\frac{(y_{2}-x_{2})^{2}+(y_{2}-x_{2})(\eta
-x_{3})\partial _{2}\eta }{|y-x|^{3}} \\ 
-\frac{\partial _{1}\eta }{|y-x|}+\frac{(\eta -x_{3})\{y_{1}-x_{1}+(\eta
-x_{3})\partial _{1}\eta \}}{|y-x|^{3}} & -\frac{\partial _{2}\eta }{|y-x|}+%
\frac{(\eta -x_{3})\{y_{2}-x_{2}+(\eta -x_{3})\partial _{2}\eta \}}{|y-x|^{3}%
}%
\end{bmatrix}%
,
\end{equation}%
by (\ref{para}), a direct comupation yields 
\begin{equation}
\{x_{\mathbf{b}1}-x_{1}\}\partial _{y_{1}}\phi +\{x_{\mathbf{b}%
2}-x_{2}\}\partial _{y_{2}}\phi =0
\end{equation}%
at $(y_{1},y_{2})=(x_{\mathbf{b}1},x_{\mathbf{b}2}).$ This implies that $%
\frac{v}{|v|}$ is a critical value of $\phi $.
\end{proof}

\begin{lemma}
\label{bouncebacklower}Assume $|v|\leq 2N.$ Then for any $\varepsilon >0,$
there exist $\delta _{\varepsilon ,N}>0,$ and $l_{\varepsilon ,N,\xi }$
balls $B(x_{1};r_{1}),B(x_{2},r_{2})...,B(x_{l};r_{l})\subset \bar{\Omega}$,
as well as open sets $O_{x_{1}},O_{x_{2},}...O_{x_{l}}$ of the velocity $v$
with $|O_{x_{i}}|<\varepsilon $ for $1\leq i\leq l,$ such that for any $x\in 
\bar{\Omega},$ there exists $x_{i}$ so that $x\in B(x_{i};r_{i})\,\ $and for 
$v\notin O_{x_{i}},$ 
\begin{equation}
|v\cdot n(x-t_{\mathbf{b}}(x,v)v)|>\delta _{\varepsilon ,N}>0,\text{ \ \ }%
|v\cdot n(x+t_{\mathbf{b}}(x,-v)v)|>\delta _{\varepsilon ,N}>0.
\end{equation}
\end{lemma}

\begin{proof}
Fix $\varepsilon >0.$ For any $x\in \bar{\Omega},$ since $|S_{x}|=0$ by
Lemma \ref{s=0}, there exists an open set $O_{x}^{+}$ such that $%
|O_{x}^{+}|<\varepsilon /2,$ and $|v\cdot n(x-t_{\mathbf{b}}(x,v)v)|\neq 0,$
for $v\notin O_{x}^{+}.$ But from part (2) of Lemma \ref{huang}, this
implies that $v\cdot n(x-t_{\mathbf{b}}(x,v)v)$ is a smooth function on the
compact set $\{|v|\leq 2N\}\cap \{O_{x}^{+}\}^{c}.$ Hence, there exists $%
\delta _{x,\varepsilon ,N}>0,$ such that on the set $\{|v|\leq 2N\}\cap
\{O_{x}^{+}\}^{c},$ 
\begin{equation}
|v\cdot n(x-t_{\mathbf{b}}(x,v)v)|\geq \delta _{x,\varepsilon ,N}>0.
\end{equation}%
In particular, for $-v\notin O_{x}^{+}$ and $|v|\leq 2N,$ or equivalently, $%
v\notin -O_{x}^{+}\equiv \{-v^{\prime }:v^{\prime }\in O_{x}^{+}\},$ $%
|v\cdot n(x+t_{\mathbf{b}}(x,-v)v)|>\delta _{x,\varepsilon ,N}>0.$ We define 
$O_{x}\equiv O_{x}^{+}\cup \{-O_{x}^{+}\},$ clearly $|O_{x}|<\varepsilon .$
But by part (2) of Lemma \ref{huang}, for such $v\notin O_{x},$ both $t_{%
\mathbf{b}}(x,v)$ and $t_{\mathbf{b}}(x,-v)$ are smooth functions of both
variables $x$ and $v.$ In other words, there exists $B(x;r_{x})$ such that
if $y\in B(x;r_{x})$ and $v\notin O_{x}$%
\begin{equation*}
|v\cdot n(y\mp t_{\mathbf{b}}(y,\pm v)v)|>\delta _{x,\varepsilon ,N}/2>0.
\end{equation*}%
Now for any $x\in \bar{\Omega},$ all $B(x,r_{x})$ form an open covering for
the compact set $\bar{\Omega},$ hence there is a finite $l-$ subcovering $%
B(x_{1};r_{1}),B(x_{2},r_{2})...,B(x_{l};r_{l})$. From our construction, for
any $x\in \bar{\Omega},$ there exists $i,$ so that $x\in B(x_{i},r_{i})$ and
moreover, $|v\cdot \nabla n(x\mp t_{\mathbf{b}}(\pm v)v)|>\delta
_{x_{i},\varepsilon ,N}/2>0.$ We conclude our lemma by choosing $\delta
_{\varepsilon ,N}=\min_{i}\frac{\delta _{x_{i},\varepsilon ,N}}{2}.$
\end{proof}

\subsubsection{$L^{\infty }$ Decay of Bounce-back $U(t)$}

\begin{theorem}
\label{bouncebackrate}Assume $w^{-2}\{1+|v|\}\in L^{1}.$ Let $%
h_{0}=wf_{0}\in L^{\infty }.$ There exits a unique solution $f(t,x,v)$ to
the linear Boltzmann equation (\ref{lboltzmann}) satisfying $f(0,x,v)=f_{0},$
and $h(t,x,v)=U(t)h_{0}$ to the weighted linear Boltzmann equation (\ref%
{lboltzmannh}) satisfying $h(0,x,v)=h_{0}$, both with the bounce-back
boundary condition. Then there exist $\lambda >0$ and $C>0$ such that 
\begin{equation}
e^{\lambda t}||U(t)h_{0}||_{\infty }\leq C||h_{0}||_{\infty }.
\label{bouncebackdecay}
\end{equation}
\end{theorem}

For the well-posedness for both problems, we know from the Duhamel principle
(\ref{duhamel}), there exists a $L^{\infty }$ solution $h(t)=U(t)h_{0}$ to
the weighted linear Boltzamnn equation (\ref{lboltzmannh})$.$ By Ukai's
trace theorem, it follows that $h_{\gamma }$ is also in $L^{\infty }.$
Therefore, since $w^{-2}\{1+|v|\}\in L^{1},$ $\ f=\frac{h}{w}\in L^{2}$ and $%
f_{\gamma }=\frac{h_{\gamma }}{w}\in L_{loc}^{2}(L^{2}(\gamma ))$ is a
solution to the original linear Boltzmann equation (\ref{lboltzmann}), which
is unique by the standard energy estimate. Based on the $L^{2}$ decay
estimate for $f,$ to prove the decay estimate, it suffices to establish a
finite-time estimate (\ref{finitetime}).

\begin{lemma}
\label{bootstrap} Assume that there exists $\lambda >0$ so that the solution 
$f(t,x,v)$ of (\ref{lboltzmann}) satisfies $e^{\lambda t}||f(t)||\leq
C||f_{0}||.$ Let $h_{0}=wf_{0}\in L^{\infty }$ and $h(t)=U(t)h_{0}=wf(t)$ is
the solution of (\ref{lboltzmannh}) where $w^{-2}\in L^{1}.$ Assume there
exist $T_{0}>0$ and $C_{T_{0}}>0$ such that the satisfies 
\begin{equation}
||U(T_{0})h_{0}||_{\infty }\leq e^{-\lambda T_{0}}||h_{0}||_{\infty
}+C_{T_{0}}\int_{0}^{T_{0}}||f(s)||ds.  \label{finitetime}
\end{equation}%
Then (\ref{bouncebackdecay}) is valid.
\end{lemma}

\begin{proof}
It suffices to only prove (\ref{bouncebackdecay}) for $t\geq 1.$ For any $%
m\geq 1,$ we apply the finite-time estimate (\ref{finitetime}) repeatedly to
functions $h(lT_{0}+s)$ for $l=m-1,m-2,...0$:%
\begin{eqnarray*}
||h(mT_{0})||_{\infty } &\leq &e^{-\lambda T_{0}}||h(\{m-1\}T_{0})||_{\infty
}+C_{T_{0}}\int_{0}^{T_{0}}||f(\{m-1\}T_{0}+s)||ds \\
&=&e^{-\lambda T_{0}}||h(\{m-1\}T_{0})||_{\infty
}+C_{T_{0}}\int_{\{m-1\}T_{0}}^{mT_{0}}||f(s)||ds \\
&\leq &e^{-2\lambda T_{0}}||h(\{m-2\}T_{0})||_{\infty }+e^{-\lambda
T_{0}}C_{T_{0}}\int_{\{m-2\}T_{0}}^{\{m-1\}T_{0}}||f(s)||ds \\
&&+C_{T_{0}}\int_{\{m-1\}T_{0}}^{mT_{0}}||f(s)||ds \\
&\leq &e^{-m\lambda T_{0}}||h(0)||_{\infty
}+C_{T_{0}}\sum_{k=0}^{m-1}e^{-k\lambda
T_{0}}\int_{\{m-k-1\}T_{0}}^{\{m-k\}T_{0}}||f(s)||ds,
\end{eqnarray*}%
where $h(t)=U(t)h_{0}.$ Now by the $L^{2}$ decay assumption, in the interval 
$\{m-k-1\}T_{0}\leq s\leq \{m-k\}T_{0},$ we have $||f(s)||\leq e^{-\lambda
s}||f_{0}||\leq e^{-\lambda \{m-k-1\}T_{0}}||f_{0}||.$ Hence, $%
||h(mT_{0})||_{\infty }$ is further bounded by%
\begin{eqnarray*}
&&e^{-m\lambda T_{0}}||h_{0}||_{\infty
}+C_{T_{0}}\sum_{k=0}^{m-1}e^{-k\lambda
T_{0}}\int_{\{m-k-1\}T_{0}}^{\{m-k\}T_{0}}e^{-\lambda
\{m-k-1\}T_{0}}||f_{0}||ds \\
&\leq &e^{-m\lambda T_{0}}||h_{0}||_{\infty }+C_{T_{0}}e^{\lambda
T_{0}}mT_{0}e^{-m\lambda T_{0}}||f_{0}|| \\
&\leq &C_{T_{0},\lambda }e^{-\frac{m\lambda T_{0}}{2}}||h_{0}||_{\infty },
\end{eqnarray*}%
where by $w^{-2}\in L^{1},$ $||f_{0}||=||w^{-1}h_{0}||\leq
C||h_{0}||_{\infty }$ and $mT_{0}e^{-m\lambda T_{0}}\leq $ $e^{-\frac{%
m\lambda T_{0}}{2}}.$ For any $t,$ we can find $m$ such that $mT_{0}\leq
t\leq \{m+1\}T_{0},$ and 
\begin{equation*}
||h(t)||_{\infty }\leq C||h(mT_{0})||_{\infty }\leq C_{T_{0},\lambda }e^{-%
\frac{m\lambda T_{0}}{2}}||h_{0}||_{\infty }\leq \{C_{T_{0},\lambda }e^{%
\frac{\lambda T_{0}}{2}}\}e^{-\frac{\lambda }{2}t}||h_{0}||_{\infty },
\end{equation*}%
since $e^{-\frac{m\lambda T_{0}}{2}}\leq e^{-\frac{\lambda }{2}t}e^{\frac{%
\lambda T_{0}}{2}}$.
\end{proof}

\begin{proof}
of Theorem \ref{bouncebackrate}: By Lemma \ref{bootstrap}, we only need to
prove the finite-time estimate (\ref{finitetime}). We use the double Duhamal
Principle (\ref{duhamel2}) for semigroup $U(t)$ and $G(t).$ We first
estimate the first term in (\ref{duhamel2}) by Lemma \ref{bouncebackformula}%
, 
\begin{equation}
e^{\nu _{0}t}||G(t)h_{0}||_{\infty }\leq ||h_{0}||_{\infty }.
\label{bouncebackfirst}
\end{equation}

For the second term in (\ref{duhamel2}), we note that $||K_{w}h||_{\infty
}\leq C||h||_{\infty },$ then by Lemma \ref{bouncebackformula}, 
\begin{subequations}
\begin{equation}
\left\Vert \int_{0}^{t}G(t-s_{1})K_{w}G(s_{1})h_{0}ds_{1}\right\Vert
_{\infty }\leq \int_{0}^{t}e^{-\nu
_{0}(t-s_{1})}||K_{w}G(s_{1})h_{0}||_{\infty }ds_{1}=Cte^{-\nu
_{0}t}||h_{0}||_{\infty }.  \label{bouncebacksecond}
\end{equation}

We now concentrate on the third term in (\ref{duhamel2}) with the double
time integral. We now fix any point $(t,x,v)$ so that $(x,v)\notin \gamma
_{0}$ with its bounce-back cycle. Using (\ref{bouncebackformular}) twice, we
obtain 
\end{subequations}
\begin{eqnarray*}
&&G(t-s_{1})K_{w}G(s_{1}-s)K_{w}h(s) \\
&=&e^{-\nu (v)(t-s_{1})}\{K_{w}G(s_{1}-s)K_{w}h(s)\}\left( s_{1},X_{\mathbf{%
cl}}(s_{1}),V_{\mathbf{cl}}(s_{1})\right) \\
&=&e^{-\nu (v)(t-s_{1})}\int K_{w}(V_{\mathbf{cl}}(s_{1}),v^{\prime
})\{G(s_{1}-s)K_{w}h(s)\}\left( s_{1},X_{\mathbf{cl}}(s_{1}),v^{\prime
}\right) dv^{\prime } \\
&=&e^{-\nu (v)(t-s_{1})}\int \int K_{w}(V_{\mathbf{cl}}(s_{1}),v^{\prime
})K_{w}(V_{\mathbf{cl}}^{\prime }(s),v^{\prime \prime })e^{-\nu (v^{\prime
})(s_{1}-s)}h\left( s,X_{\mathbf{cl}}^{\prime }(s),v^{\prime \prime }\right)
dv^{\prime }dv^{\prime \prime }
\end{eqnarray*}%
where $X_{\mathbf{cl}}^{\prime }(s)\equiv X_{\mathbf{cl}}(s;s_{1},X_{\mathbf{%
cl}}(s_{1}),v^{\prime }),$ and $V_{\mathbf{cl}}^{\prime }(s)\equiv V_{%
\mathbf{cl}}(s;s_{1},X_{\mathbf{cl}}(s_{1}),v^{\prime }).$ In the case that $%
|v|\geq N,$ we use the same argument in Case 1, (\ref{inflowstep2}) in the
proof of Theorem \ref{inflowrate} to conclude:%
\begin{equation}
\int_{0}^{t}\int_{0}^{s_{1}}||G(t-s_{1})K_{w}G(s_{1}-s)K_{w}h(s)||_{\infty
}dsds_{1}\leq \frac{C_{K}}{N}e^{-\frac{\nu _{0}t}{2}}\sup_{s}\{e^{\frac{\nu
_{0}s}{2}}||h(s)||_{\infty }\}.  \label{bouncebackn}
\end{equation}

Moreover, \ since $|V_{\mathbf{cl}}(s_{1})|=|v|,$ $|V_{\mathbf{cl}}^{\prime
}(s_{1})|=|v^{\prime }|,$ hence as in Case 2, (\ref{inflowstep3}) in the
proof of Theorem \ref{inflowrate}, for $|v|\leq N,|v^{\prime }|\geq 2N$ or $%
|v|\leq N,|v^{\prime }|\leq 2N,|v^{\prime \prime }|\geq 3N,$ we deduce for $%
\varepsilon $ small,%
\begin{eqnarray}
&&\int_{0}^{t}\int_{0}^{s_{1}}\int_{\substack{ |v|\leq N,|v^{\prime }|\geq
2N  \\ \text{or }|v|\leq N,|v^{\prime }|\leq 2N,|v^{\prime \prime }|\geq 3N, 
}}e^{-\nu _{0}(t-s)}K_{w}(V_{\mathbf{cl}}(s_{1}),v^{\prime })K_{w}(V_{%
\mathbf{cl}}^{\prime }(s),v^{\prime \prime })h\left( s,X_{\mathbf{cl}%
}^{\prime }(s_{1}),v^{\prime \prime }\right)  \notag \\
&\leq &C_{\varepsilon ,N}e^{-\frac{\varepsilon }{8}N^{2}}e^{-\frac{\nu _{0}t%
}{2}}\sup_{s}\{e^{\frac{\nu _{0}}{2}s}||h(s)||_{\infty }\}.
\label{bouncebacknn}
\end{eqnarray}

We need to only consider the case $|v|\leq N,|v^{\prime }|\leq 2N,|v^{\prime
\prime }|\leq 3N,$ for which we can use the same approximation in Case 4, (%
\ref{inflowstep41}) to obtain an upper bound%
\begin{eqnarray}
&&\frac{C}{N}e^{-\frac{\nu _{0}t}{2}}\sup_{s}\{e^{\frac{\nu _{0}}{2}%
s}||h(s)||_{\infty }\}  \label{bouncebackbd} \\
&&+C_{N}\int_{0}^{t}\int_{0}^{s_{1}}\int_{\substack{  \\ \text{ }|v|\leq
N,|v^{\prime }|\leq 2N,|v^{\prime \prime }|\leq 3N,}}e^{-\nu
_{0}(t-s)}|h\left( s,X_{\mathbf{cl}}^{\prime }(s),v^{\prime \prime }\right)
|.  \notag
\end{eqnarray}

Recall $\Omega _{\varepsilon ^{4}}=\{x:\xi (x)<-\varepsilon ^{4}\}.$ We
focus on the second main term in (\ref{bouncebackbd}) and separate two cases:

\textbf{CASE 1:} $X_{\mathbf{cl}}(s_{1})\in \bar{\Omega}\setminus \Omega
_{\varepsilon ^{4}}$ and $\{v^{\prime }:|v^{\prime }\cdot \frac{\nabla \xi
(X_{\mathbf{cl}}(s_{1}))}{|\nabla \xi (X_{\mathbf{cl}}(s_{1}))|}|\leq
\varepsilon \}.$ In this case, since $|\nabla \xi (X_{\mathbf{cl}%
}(s_{1}))|\neq 0$ for $\varepsilon $ small$,$ the second term in (\ref%
{bouncebackbd}) is bounded by 
\begin{eqnarray}
&&C_{N}\int_{0}^{t}\int_{0}^{s_{1}}e^{-\nu _{0}(t-s)}\int_{\{v^{\prime
}:|v^{\prime }\cdot \frac{\nabla \xi (X_{\mathbf{cl}}(s_{1}))}{|\nabla \xi
(X_{\mathbf{cl}}(s_{1}))|}|\leq \varepsilon ,|v^{\prime }|+|v^{\prime \prime
}|\leq 5N\}}dv^{\prime }||h(s)||_{\infty }  \notag \\
&\leq &C_{N}\varepsilon e^{-\frac{\nu _{0}t}{2}}\sup_{s}\{e^{\frac{\nu _{0}s%
}{2}}||h(s)||_{\infty }\}\int_{0}^{t}\int_{0}^{s_{1}}e^{-\frac{\nu _{0}(t-s)%
}{2}}dsds_{1}  \notag \\
&\leq &C_{N}\varepsilon e^{-\frac{\nu _{0}t}{2}}\sup_{s}\{e^{\frac{\nu _{0}s%
}{2}}||h(s)||_{\infty }\}.  \label{bouncebacksmallangle}
\end{eqnarray}

\textbf{CASE 2: }$|v|\leq N,|v^{\prime }|\leq 2N,|v^{\prime \prime }|\leq 3N$
and either $X_{\mathbf{cl}}(s_{1})\in \Omega _{\varepsilon ^{4}}$ or $X_{%
\mathbf{cl}}(s_{1})\in \bar{\Omega}\setminus \Omega _{\varepsilon ^{4}}$ but 
$\{v^{\prime }:|v^{\prime }\cdot \frac{\nabla \xi (X_{\mathbf{cl}}(s_{1}))}{%
|\nabla \xi (X_{\mathbf{cl}}(s_{1}))|}|>\varepsilon \}$. We denote such a
set of $v,v^{\prime }$ and $v^{\prime \prime }$ by $A.$ By using the formula
for cycles (\ref{bouncebackx}), we get 
\begin{equation}
\int_{0}^{t}e^{-\nu (v)(t-s)}\int_{A}\sum_{k}\int_{t_{k+1}^{\prime
}}^{t_{k}^{\prime }}\mathbf{1}_{[0,s_{1}]}(s)h\left( s,X_{\mathbf{cl}%
}(s_{1})+\sum_{l=0}^{k-1}(t_{l+1}^{\prime }-t_{l}^{\prime
})(-1)^{l}v^{\prime }+(s-t_{k}^{\prime })(-1)^{k}v^{\prime },v^{\prime
\prime }\right) .  \label{onseta}
\end{equation}

We first claim that the number of bounces are bounded on $A:$ 
\begin{equation}
k\leq C_{T_{0},N,\varepsilon }.  \label{bouncebackclaim}
\end{equation}

\textit{Proof of the claim (\ref{bouncebackclaim}): }In the first case $X_{%
\mathbf{cl}}(s_{1})\in \Omega _{\varepsilon ^{4}},$ we have from the
mean-value theorem, 
\begin{equation*}
0=\xi (X_{\mathbf{cl}}(s_{1})-t_{\mathbf{b}}^{\prime }(X_{\mathbf{cl}%
}(s_{1}),v^{\prime })v^{\prime })=\xi (X_{\mathbf{cl}}(s_{1}))-t_{\mathbf{b}%
}^{\prime }v^{\prime }\cdot \nabla \xi (\bar{x}).
\end{equation*}%
Since $|v^{\prime }|\leq 2N,$ and $|\nabla \xi (\bar{x})|\leq C,$ 
\begin{equation*}
t_{\mathbf{b}}^{\prime }\geq \frac{|\xi (X_{\mathbf{cl}}(s_{1}))|}{%
|v^{\prime }\cdot \nabla \xi (\bar{x})|}\geq \frac{\varepsilon ^{4}}{C_{N}}.
\end{equation*}%
Because $t_{k}^{\prime }-t_{k+1}^{\prime }\geq t_{\mathbf{b}}^{\prime }$, $%
k\leq \frac{C_{N}T_{0}}{\varepsilon ^{4}}$ and (\ref{bouncebackclaim}) is
valid.

In the case $X_{\mathbf{cl}}(s_{1})\in \bar{\Omega}\setminus \Omega
_{\varepsilon ^{4}},$ and $\{v^{\prime }:|v^{\prime }\cdot \frac{\nabla \xi
(X_{\mathbf{cl}}(s_{1}))}{|\nabla \xi (X_{\mathbf{cl}}(s_{1}))|}%
|>\varepsilon \},$ denote $t_{\mathbf{b}}^{\prime }(v^{\prime })=t_{\mathbf{b%
}}^{\prime }(X_{\mathbf{cl}}(s_{1}),v^{\prime })$ and $t_{\mathbf{b}%
}^{\prime }(-v^{\prime })=t_{\mathbf{b}}^{\prime }(X_{\mathbf{cl}%
}(s_{1}),-v^{\prime }).$ We expand 
\begin{eqnarray*}
0 &=&\xi (X_{\mathbf{cl}}(s_{1})-t_{\mathbf{b}}^{\prime }(v^{\prime
})v^{\prime })=\xi (X_{\mathbf{cl}}(s_{1}))-t_{\mathbf{b}}^{\prime
}(v^{\prime })\nabla \xi (X_{\mathbf{cl}}(s_{1}))\cdot v^{\prime }+\frac{%
\{t_{\mathbf{b}}^{\prime }(v^{\prime })\}^{2}}{2}v^{\prime }\nabla ^{2}\xi
(x_{+})v^{\prime }, \\
0 &=&\xi (X_{\mathbf{cl}}(s_{1})+t_{\mathbf{b}}^{\prime }(-v^{\prime
})v^{\prime })=\xi (X_{\mathbf{cl}}(s_{1}))+t_{\mathbf{b}}^{\prime
}(-v^{\prime })\nabla \xi (X_{\mathbf{cl}}(s_{1}))\cdot v^{\prime }+\frac{%
\{t_{\mathbf{b}}^{\prime }(-v^{\prime })\}^{2}}{2}v^{\prime }\nabla ^{2}\xi
(x_{-})v^{\prime }.
\end{eqnarray*}%
Since $\nabla ^{2}\xi (x_{\pm })$ are bounded, $|v^{\prime }|\leq 2N,$ and $%
-\varepsilon ^{4}<\xi (X_{\mathbf{cl}}(s_{1}))<0,$ for some constant $C_{N},$
we have 
\begin{eqnarray*}
-t_{\mathbf{b}}^{\prime }(v^{\prime })\nabla \xi (X_{\mathbf{cl}%
}(s_{1}))\cdot v^{\prime }+C_{N}\{t_{\mathbf{b}}^{\prime }(v^{\prime
})\}^{2} &>&0, \\
t_{\mathbf{b}}^{\prime }(-v^{\prime })\nabla \xi (X_{\mathbf{cl}%
}(s_{1}))\cdot v^{\prime }+C_{N}\{t_{\mathbf{b}}^{\prime }(-v^{\prime
})\}^{2} &>&0.
\end{eqnarray*}%
We thus have $t_{\mathbf{b}}^{\prime }(v^{\prime })\geq \frac{\nabla \xi (X_{%
\mathbf{cl}}(s_{1}))\cdot v^{\prime }}{C_{N}}$ and $t_{\mathbf{b}}^{\prime
}(-v^{\prime })\geq -\frac{\nabla \xi (X_{\mathbf{cl}}(s_{1}))\cdot
v^{\prime }}{C_{N}}.$ Since $|\nabla \xi (X_{\mathbf{cl}}(s_{1}))\cdot
v^{\prime }|\geq C_{\xi }\varepsilon $, either $t_{\mathbf{b}}^{\prime
}(v^{\prime })\geq C_{\xi }\varepsilon $ or $t_{\mathbf{b}}^{\prime
}(-v^{\prime })\geq C_{\xi }\varepsilon .$ But for bounce-back cycles, 
\begin{equation}
t_{k}^{\prime }-t_{k+1}^{\prime }=t_{1}-t_{2}=t_{\mathbf{b}}^{\prime
}(v^{\prime })+t_{\mathbf{b}}^{\prime }(-v^{\prime })\geq C_{\xi
}\varepsilon ,  \label{tk-tk}
\end{equation}%
for all $k\geq 1.$ We therefore have verified the claim (\ref%
{bouncebackclaim}).

We are now ready to estimate (\ref{onseta}). By Lemma \ref{bouncebacklower},
for the given $\varepsilon >0,$ there is $\delta _{\varepsilon _{,}N}>0,$
and $[B(x_{i},r_{i}),O_{x_{i}}]$ for $i=1,...,l,$ $|O_{x_{i}}|<\varepsilon .$
For $X_{\mathbf{cl}}(s_{1})\in \bar{\Omega},$ there exists $i$ such that $X_{%
\mathbf{cl}}(s_{1})\in B(x_{i},r_{i})\,\ $and for $v^{\prime }\notin
O_{x_{i}}$ 
\begin{equation*}
|v^{\prime }\cdot n(X_{\mathbf{cl}}(s_{1})\mp t_{\mathbf{b}}^{\prime }(\pm
v^{\prime })v^{\prime })|\geq \delta _{\varepsilon ,N}>0.
\end{equation*}%
Hence $t_{1}^{\prime }=s_{1}-t_{\mathbf{b}}(x,v^{\prime })$, $t_{2}^{\prime
}=t_{1}-t_{\mathbf{b}}(x,v^{\prime })-t_{\mathbf{b}}(x,-v^{\prime })$ are
both also smooth with bounded derivatives over $(s,v^{\prime })\in \lbrack
0,T_{0}]\times \{O_{x_{i}}^{c}\cap |v^{\prime }|\leq 2N\}$. It thus follows
from Lemma \ref{huang} that $t_{l}^{\prime }$ and $x_{l}^{\prime }$ are all
smooth functions of $v^{\prime }$ $\notin $ $O_{x_{i}}.$ We then split $%
\{v^{\prime }:|v^{\prime }|\leq 2N\}$ into 
\begin{equation}
\int_{A}=\int_{\substack{  \\ A\cap \{v^{\prime }\in O_{x_{i}}\}}}+\int 
_{\substack{  \\ A\cap \{v^{\prime }\in O_{x_{i}}^{c}\}}}  \label{splita}
\end{equation}%
Since $\sum_{k}\int_{t_{k+1}^{\prime }}^{t_{k}^{\prime }}=\int_{0}^{s_{1}},$
the first part is bounded by 
\begin{eqnarray}
&&\int_{0}^{t}e^{-\nu _{0}(t-s)}\int_{\substack{  \\ v^{\prime }\in
O_{x_{i}},|v^{\prime \prime }|\leq 3N}}\sum_{k}^{{}}\int_{t_{k+1}^{\prime
}}^{t_{k}^{\prime }}\mathbf{1}_{[0,s_{1}]}(s)|h\left( s,X_{\mathbf{cl}%
}^{\prime }(s),v^{\prime \prime }\right) |dv^{\prime }dv^{\prime \prime } 
\notag \\
&\leq &C_{N}\sup_{s}\{e^{\frac{\nu _{0}}{2}s}||h(s)||_{\infty
}\}\int_{0}^{t}e^{-\nu _{0}t}\int_{\substack{  \\ v^{\prime }\in
O_{x_{i}},|v^{\prime \prime }|\leq 3N}}\sum_{k}\int_{t_{k+1}^{\prime
}}^{t_{k}^{\prime }}e^{\frac{\nu _{0}}{2}s}ds  \notag \\
&\leq &C_{N}|O_{x_{i}}|\sup_{s}\{e^{\frac{\nu _{0}}{2}s}||h(s)||_{\infty
}\}\int_{0}^{t}\int_{0}^{s_{1}}e^{-\nu _{0}t}e^{\frac{\nu _{0}}{2}s}dsds_{1}
\notag \\
&=&C_{N}\varepsilon e^{-\frac{\nu _{0}}{2}t}\times \sup_{s}\{e^{\frac{\nu
_{0}}{2}s}||h(s)||_{\infty }\}.  \label{smallo}
\end{eqnarray}%
By (\ref{bouncebackclaim}), we therefore only need to consider the second
part in (\ref{splita}): 
\begin{equation*}
\int_{0}^{t}\int_{\substack{  \\ v^{\prime }\notin O_{x_{i}},|v^{\prime
}|\leq 2N,|v^{\prime \prime }|\leq 3N}}\sum_{k}^{C_{T_{0},N,\varepsilon
}}\int_{t_{k+1}^{\prime }}^{t_{k}^{\prime }}e^{-\nu _{0}(t-s)}\mathbf{1}%
_{[0,s_{1}]}(s)|h\left( s,X_{\mathbf{cl}}^{\prime }(s),v^{\prime \prime
}\right) |dsds_{1}dv^{\prime }dv^{\prime \prime }.
\end{equation*}

We now wish to change variables as 
\begin{equation*}
X_{\mathbf{cl}}^{\prime }(s)\equiv X_{\mathbf{cl}}(s_{1})+%
\sum_{l=0}^{k-1}(t_{l+1}^{\prime }-t_{l}^{\prime })(-1)^{l}v^{\prime
}+(s-t_{k}^{\prime })(-1)^{k}v^{\prime }\rightarrow y.
\end{equation*}%
Since for $1\leq l\leq k,$ $t_{l}^{\prime }$ is a smooth function of $%
v^{\prime }$ on the set of integration: $\{v^{\prime }\notin
O_{x_{i}},|v^{\prime }|\leq 2N\},$ we can expand the determinant as a cubic
function of $s:$ 
\begin{equation*}
\det \left\vert \frac{\partial y}{\partial v^{\prime }}\right\vert
=(-1)^{k}s^{3}+q_{1}(v^{\prime })s^{2}+q_{2}(v^{\prime })s+q_{3}(v^{\prime
}),
\end{equation*}%
where $q_{i}(v)$ are smooth functions of $v^{\prime }.$ Therefore, by the
analytical formula of the algebraic cubic equation, there exists up to three
(real) continuous functions $\eta _{j}(v^{\prime })$ for $1\leq j\leq 3$ so
that 
\begin{equation*}
\left\{ (s,v^{\prime }):\det \left\vert \frac{\partial y}{\partial v^{\prime
}}\right\vert =0\right\} =\cup _{j}\{s:s=\eta _{j}(v^{\prime })\}.
\end{equation*}%
For $\varepsilon >0,$ we then split 
\begin{equation*}
\int_{0}^{t}\int_{\substack{  \\ v^{\prime }\notin O_{x_{i}},|v^{\prime
}|\leq 2N,|v^{\prime \prime }|\leq 3N}}\sum_{k}\int_{t_{k+1}^{\prime
}}^{t_{k}^{\prime }}\mathbf{1}_{\cup _{j}\{|s-\eta _{j}|\leq \varepsilon \}}+%
\mathbf{1}_{\cap _{i}\{|s-\eta _{j}|\geq \varepsilon \}}.
\end{equation*}%
The first part with a small $s$ interval $\{|s-\eta _{j}|\leq \varepsilon \}$
is bounded by%
\begin{eqnarray}
&&C\sum_{k,i}\int_{0}^{t}\int_{\substack{  \\ v^{\prime }\notin
O_{x_{i}},|v^{\prime }|\leq 2N,|v^{\prime \prime }|\leq 3N}}%
\int_{t_{k+1}^{\prime }}^{t_{k}^{\prime }}e^{-\nu (v)(t-s)}||h(s)||_{\infty }%
\mathbf{1}_{\{|s-\eta _{j}|\leq \varepsilon _{1},0\leq s\leq s_{1}\}}(s) 
\notag \\
&\leq &C\sup_{s}\{e^{\frac{\nu _{0}}{2}s}||h(s)||_{\infty
}\}\int_{0}^{t}\int _{\substack{  \\ |v^{\prime }|\leq 2N,|v^{\prime \prime
}|\leq 3N}}e^{-\nu _{0}t}\left\{ \sum_{k,j}\int_{t_{k+1}^{\prime
}}^{t_{k}^{\prime }}\mathbf{1}_{\{|s-\eta _{j}|\leq \varepsilon ,0\leq s\leq
s_{1}\}}(s)e^{\frac{\nu _{0}}{2}s}ds\right\} ds_{1}  \notag \\
&\leq &C_{N}\varepsilon e^{-\frac{\nu _{0}}{2}t}\sup_{s}\{e^{\frac{\nu _{0}}{%
2}s}||h(s)||_{\infty }\}.  \label{bouncebacks}
\end{eqnarray}

On the other hand, for the second main term, we notice that on the compact
set in $0\leq s\leq T_{0}$ and $v^{\prime }\in \cap _{j}\{|s-\eta _{j}|\geq
\varepsilon \}\cap \{v^{\prime }\notin O_{x_{i}},|v^{\prime }|\leq 2N\}$,
the function $J\{\frac{\partial y}{\partial v^{\prime }}\}$ is uniformly
continuous$,$ with uniformly bounded derivatives for $s,v^{\prime }.$ There
exists a $\zeta _{\varepsilon ,N,T_{0}}>0$ such that 
\begin{equation}
J\{\frac{\partial y}{\partial v^{\prime }}\}\geq \zeta _{\varepsilon
,N,T_{0}}>0.  \label{juniform}
\end{equation}%
And for any point $(s,v^{\prime })\in \lbrack 0,T_{0}]\times \{\cap
_{j}\{|s-\eta _{j}|\geq \varepsilon \}\cap \{v^{\prime }\in
O_{x_{i}},|v^{\prime }|\leq 2N\},$ there exists open set $O_{s,v^{\prime }}$
such that \thinspace $v^{\prime }\rightarrow y$ is one-to-one and
invertible. We therefore have a finite covering (depending on $\varepsilon ,$
$T_{0}$,$N$) $O_{s_{m},v_{m}^{\prime }}$ such that 
\begin{equation*}
\cap _{j}\{|s-r_{j}|\geq \varepsilon \}\cap \{v^{\prime }\in
O^{c},|v^{\prime }|\leq 2N\}\subset \cup _{m}O_{s_{m},v_{m}^{\prime }}
\end{equation*}%
and $v^{\prime }\rightarrow y$ is invertible on each $O_{s_{m},v_{m}^{\prime
}}.$ We therefore can change variable $v^{\prime }\rightarrow y$ locally as 
\begin{eqnarray*}
&&\int_{0}^{t}e^{-\nu _{0}(t-s)}\sum_{k,m}\int_{t_{k+1}^{\prime
}}^{t_{k}^{\prime }}\int_{O_{s_{m},v_{m}^{\prime }}}\int_{|v^{\prime \prime
}|\leq 3N}\mathbf{1}_{[0,s_{1}]}(s)h\left( s,X_{\mathbf{cl}}^{\prime
}(s),v^{\prime \prime }\right) dv^{\prime }dv^{\prime \prime }dsds_{1} \\
&\leq &C_{N}\int_{0}^{t}e^{-\nu _{0}t}\sum_{k,m}\int_{t_{k+1}^{\prime
}(y)}^{t_{k}^{\prime }(y)}\int_{O_{s_{m},v_{m}^{\prime }}}\int_{|v^{\prime
\prime }|\leq 3N}\mathbf{1}_{[0,s_{1}]}(s)e^{\nu _{0}s}h\left( s,y,v^{\prime
\prime }\right) \frac{1}{J\{\frac{\partial y}{\partial v^{\prime }}\}}%
dydv^{\prime \prime }dsds_{1} \\
&\leq &C_{N},_{\varepsilon ,T_{0}}\int_{0}^{t}e^{-\nu
_{0}t}\int_{0}^{s_{1}}\int_{|v^{\prime \prime }|\leq 3N}e^{\nu _{0}s}\left\{
\int_{\Omega }h^{2}\left( s,y,v^{\prime \prime }\right) dy\right\}
^{1/2}dv^{\prime \prime }\text{ }dsds_{1} \\
&\leq &C_{N},_{\varepsilon ,T_{0}}\int_{0}^{t}e^{-\nu
_{0}t}\int_{0}^{s_{1}}e^{\nu _{0}s}\left\{ \int_{\Omega \times |v^{\prime
\prime }|\leq 3N}h^{2}\left( s,y,v^{\prime \prime }\right) dydv^{\prime
\prime }\right\} ^{1/2}dsds_{1} \\
&\leq &C_{N},_{\varepsilon ,T_{0}}\int_{0}^{t}e^{-\nu
_{0}t}\int_{0}^{s_{1}}e^{\nu _{0}s}||f(s)||dsds_{1}\text{ \ \ \ \ \ \ \ \ \
\ \ \ \ \ \ \ \ }(f=\frac{h}{w}) \\
&\leq &C_{N},_{\varepsilon ,T_{0}}\times \int_{0}^{T_{0}}||f(s)||ds,
\end{eqnarray*}%
where $k\leq C_{T_{0},N,\varepsilon }$ and $m\leq C_{T_{0},N,\varepsilon }.$
We thus conclude from (\ref{bouncebacks}), (\ref{smallo}), (\ref%
{bouncebacksmallangle}), (\ref{bouncebackbd}), (\ref{bouncebacknn}), (\ref%
{bouncebackn}), for $t\leq T_{0},$ $e^{\frac{\nu _{0}}{2}t}||h(t)||_{\infty
} $ is bounded by 
\begin{equation*}
C_{K}(1+t)e^{-\frac{\nu _{0}}{2}t}||h_{0}||_{\infty }+(\frac{C_{K}}{N}%
+C_{N}\varepsilon )\sup_{s\leq t}e^{\frac{\nu _{0}}{2}s}||h(s)||_{\infty
}+C_{N},_{\varepsilon ,T_{0}}\int_{0}^{T_{0}}||f(s)||ds.
\end{equation*}%
We first choose $T_{0}$ so that $2C_{K}(1+T_{0})e^{-\frac{\nu _{0}}{2}%
T_{0}}=e^{-\lambda T_{0}},$ next choose $N$ large, then $\varepsilon $
sufficiently small to get $(\frac{C_{K}}{N}+C_{N}\varepsilon )<\frac{1}{2}.$
We therefore conclude 
\begin{equation*}
\sup_{0\leq t\leq T_{0}}e^{\frac{\nu _{0}}{2}t}||h(t)||_{\infty }\leq
2C_{K}(1+T_{0})||h_{0}||_{\infty }+2C_{N},_{\varepsilon
,T_{0}}\int_{0}^{T_{0}}||f(s)||ds.
\end{equation*}%
We thus conclude our theorem by letting $t=T_{0}$ on the left hand side.
\end{proof}

\subsection{$L^{\infty }$ Decay for Specular Reflection}

\subsubsection{Specular Cycles and Continuity of $G(t)$}

\begin{definition}
\label{specularcycles}Fix any point $(t,x,v)\notin \gamma _{0},$ and define $%
(t_{0},x_{0},v_{0})=(t,x,v)$, and for $k\geq 1$ 
\begin{equation}
(t_{k+1},x_{k+1},v_{k+1})=(t_{k}-t_{\mathbf{b}}(t_{k},x_{k},v_{k}),x_{%
\mathbf{b}}(x_{k},v_{k}),R(x_{k+1})v_{k}),  \label{specularcycle}
\end{equation}%
where $R(x_{k+1})v_{k}=v_{k}-2(v_{k}\cdot n(x_{k+1}))n(x_{k+1}).$ And we
define the specular back-time cycle 
\begin{equation}
X_{\mathbf{cl}}^{{}}(s)\equiv \sum_{k=1}^{\infty }\mathbf{1}%
_{[t_{k},t_{k+1})}(s)\left\{ x_{k}+v_{k}(s-t_{k})\right\} ,\text{ \ \ }V_{%
\mathbf{cl}}^{{}}(s)\equiv \sum_{k=1}^{\infty }\mathbf{1}%
_{[t_{k},t_{k+1})}(s)v_{k}.  \label{backtimecycle}
\end{equation}
\end{definition}

\begin{lemma}
\label{speculargdecay}Let $\Omega $ be convex (\ref{convexity}). Let $%
h_{0}\in L^{\infty }$ and $G(t)h_{0}$ solves (\ref{gh0}) with specular
boundary condition $h(t,x,v)=h(t,x,R(x)v)$ for $x\in \partial \Omega .$ Then
for almost all $(x,v)\notin \gamma _{0},$%
\begin{eqnarray}
&&\{G(t)h_{0}\}(t,x,v)=e^{-\nu (v)t}h_{0}(X_{\mathbf{cl}}^{{}}(0),V_{\mathbf{%
cl}}^{{}}(0))  \notag \\
&=&\sum_{k}\mathbf{1}_{[t_{k+1},t_{k})}(0)e^{-\nu
(v)t}h_{0}(x_{k}-t_{k}v_{k},v_{k}).  \label{specularformula}
\end{eqnarray}%
And $e^{\nu _{0}t}||G(t)h_{0}||_{\infty }\leq ||h_{0}||_{\infty }.$
\end{lemma}

\begin{proof}
The existence and uniqueness of the solution follows exactly the argument in
the proof in Lemma \ref{bouncebackformula}, with the bounce-back condition
replaced by the specular reflection.

If $(x,v)\in \bar{\Omega}\times \mathbf{R}^{3}\setminus \gamma _{0},$ then $%
t_{\mathbf{b}}(x,v)>0.$ We consider the back-time specular cycles of $%
(t,x,v) $ as $[X_{\mathbf{cl}}^{{}}(s),V_{\mathbf{cl}}^{{}}(s)]$ as in (\ref%
{specularcycle}). Clearly, $|V_{\mathbf{cl}}(s)|\equiv v.$ Since $\frac{d}{ds%
}\{e^{-\nu (v)}G(s)h_{0}\}\equiv 0$ for $t_{k+1}<s<t_{k},$ any $k,$ by part
4 of Lemma \ref{huang} and the specular boundary condition at $t_{k+1}$ and $%
t_{k},$ it follows that $e^{-\nu (v)}G(s)h_{0}$ is a constant along the
cycle $[X_{\mathbf{cl}}^{{}}(s),V_{\mathbf{cl}}^{{}}(s)]$.

We now show for fixed $t,$ the number of bounces $k$ is finite. Since $%
(x,v)\in \bar{\Omega}\times \mathbf{R}^{3}\setminus \gamma _{0},$ by (\ref%
{alpha}), $\alpha (t)>0.$ By repeatedly applying Velocity Lemma \ref%
{velocity} along the back-time cycle $[X_{\mathbf{cl}}^{{}}(s),V_{\mathbf{cl}%
}^{{}}(s)],$ we have for all $k\geq 1:$ $e^{-Ct_{k}}\alpha (t_{k})\geq
e^{-Ct_{k-1}}\alpha (t_{k-1})\geq ...\geq e^{-Ct}\alpha (t)>0.$ But $\alpha
(t_{k})=\{v_{k}\cdot \nabla \xi (x_{k})\}^{2},$ we then have 
\begin{equation}
\{v_{k}\cdot n(x_{k})\}^{2}\geq C\alpha (t)>0,  \label{vknk}
\end{equation}%
for all $k\geq 1,$ where $C$ depends on $t$ and $v.$ Therefore by (\ref%
{tlower}) in Lemma \ref{huang}, that $t_{k}-t_{k+1}\geq \frac{\delta (t)}{%
C(t,v)|v|^{2}}>0.$ So that the summation over $k$ is finite$.$
\end{proof}

\begin{lemma}
\label{specularcon}Let $\xi $ be convex as in (\ref{convexity}). Let $h_{0}$
be continuous in $\bar{\Omega}\times \mathbf{R}^{3}\setminus \gamma _{0}\,\ $
and $q(t,x,v)$ be continuous in the interior of $[0,\infty )\times \Omega
\times \mathbf{R}^{3}$ and $\sup_{[0,\infty )\times \Omega \times \mathbf{R}%
^{3}}|\frac{q(t,x,v)}{\nu (v)}|<\infty .$ Assume that on $\gamma _{-},$ $%
h_{0}(x,v)=h_{0}(x,R(x)v)$. Then the specular solution $h(t,x,v)$ to (\ref%
{gq}) is continuous on $[0,\infty )\times \{\bar{\Omega}\times \mathbf{R}%
^{3}\setminus \gamma _{0}\}.$
\end{lemma}

\begin{proof}
We only sketch the proof, which is similar to that for Lemma \ref%
{bouncebackcon}. Take any point $(t,x,v)$ $\notin \lbrack 0,\infty )\times
\gamma _{0}$ and consider its specular back-time cycle $[X_{\mathbf{cl}%
}^{{}}(s),V_{\mathbf{cl}}^{{}}(s)]$ as in (\ref{specularformula}). By
repeatedly applying the Velocity Lemma \ref{velocity} and Lemma \ref{huang},
it follows that $t_{k}(t,x,v),x_{k}(t,x,v)$ and $v_{k}(t,x,v)$ are all
smooth functions of $(t,x,v).$ We assume that $t_{m+1}\leq 0<t_{m},$ then $%
h(t,x,v)$ is given by (\ref{hg}) with specular cycles $[t_{k},x_{k},v_{k}]%
\in \lbrack X_{\mathbf{cl}}^{{}}(s),V_{\mathbf{cl}}^{{}}(s)].$ For any other
point $(\bar{t},\bar{x},\bar{v})$ which is close to $(t,x,v).$ We now show
that $h(\bar{t},\bar{x},\bar{v})$ is close to $h(t,x,v)$ by separating two
different cases.

In the case that $t_{m+1}<0,$ or equivalently, $x_{m}-(t_{m}-s)v_{m}\in
\Omega ,$ away from the boundary. By continuity, $\bar{t}_{m+1}<0$.
Therefore, we have the same expression for $h(\bar{t},\bar{x},\bar{v})$ as $%
h(t,x,v)$ in (\ref{hg}) with $t_{k},x_{k},v_{k}$ replaced by $\bar{t}_{k},%
\bar{x}_{k},\bar{v}_{k}.$ Therefore, since $|v_{l}|\equiv |v|$, $h(\bar{t},%
\bar{x},\bar{v})\rightarrow h(t,x,v)$ following from the continuity of $\bar{%
t}_{l}\rightarrow t_{l}$, $\bar{x}_{l}\rightarrow x_{l},$ $\bar{v}%
_{l}\rightarrow v_{l}.$

On the other hand, in the case $t_{m+1}=0,$ $x_{m+1}=x_{m}-t_{m}v_{m}\in
\partial \Omega .$ From (\ref{vknk}), $(x_{k+1},v_{k})\notin \gamma _{0}.$
Then by continuity, we know that $\bar{t}_{m}>0,$ and $\bar{t}_{m+1}$ is
close to zero. In the case that $\bar{t}_{m+1}(\bar{t},\bar{x},\bar{v})<0,$
then (\ref{hg}) is still valid and the continuity follows. However, if $\bar{%
t}_{m+1}>0,$ then $\bar{t}_{m+2}<0,$ due to $t_{m+2}<t_{m+1}=0.$ Therefore $%
h(\bar{t},\bar{x},\bar{v})$ is given by a different expression (\ref{hgbar})
with specular cycles $[t_{k},x_{k},v_{k}]\in \lbrack X_{\mathbf{cl}%
}^{{}}(s),V_{\mathbf{cl}}^{{}}(s)]$. We notice that since $\bar{t}%
_{m+1}\rightarrow 0,$ the $q-$integrals in (\ref{hgbar}) tend to $q-$%
integrals in (\ref{hg}) because of $\int_{0}^{t_{m}}=\int_{t_{m+1}}^{t_{m}}.$
On the other hand, since $\bar{x}_{m+1}-\bar{t}_{m+1}\bar{v}%
_{m+1}\rightarrow x_{m+1},$ \ \ $\bar{v}_{m+1}\rightarrow
v_{m+1}=R(x_{m})v_{m},$ the first term in (\ref{hgbar}) tends to the first
term in (\ref{hg}) as 
\begin{equation*}
h_{0}\left( \bar{x}_{m+1}-\bar{t}_{m+1}\bar{v}_{m+1},\bar{v}_{m+1}\right)
\rightarrow h_{0}\left( x_{m+1},R(x_{m})v_{m}\right) =h_{0}\left(
x_{m},v_{m}\right) ,
\end{equation*}%
from $h_{0}(x,v)=h_{0}(x,R(x)v)$ on $\gamma .$ We thus complete the proof.
\end{proof}

\subsubsection{$\det \left( \frac{\partial v_{k}}{\partial v_{1}}\right) $
Near $\partial \Omega $}

Assume $\Omega $ is convex as in (\ref{convexity}). We now compute $\det (%
\frac{\partial v_{k}}{\partial v_{1}})$ for a carefully chosen specular
back-time cycle near the boundary $\partial \Omega $. We assume $x_{1}\in
\partial \Omega .$ Given $\varepsilon _{0}$ small, we choose $v_{1}$ such
that 
\begin{equation}
|v_{1}|=\varepsilon _{0},\text{ \ }v_{1}\cdot n(x_{1})=\frac{v\cdot \nabla
\xi (x_{1})}{|\nabla \xi (x_{1})|}=\varepsilon _{0}^{2}.  \label{v1}
\end{equation}

We shall analyze the specular back-time cycle of $%
(0,x_{1},v_{1}):(t_{k},x_{k},v_{k}).$ Letting $s_{k}=t_{\mathbf{b}%
}(x_{k},v_{k})$, we have $\xi (x_{1}-s_{1}v_{1})=0,$ $x_{2}=x_{1}-s_{1}v_{1}%
\in \partial \Omega $ and for $k\geq 2:$ 
\begin{equation*}
\xi (x_{1}-\sum_{j=1}^{k}s_{j}v_{j})=0,\text{ \ \ \ \ }v_{k}=R(x_{k})v_{k-1},%
\text{ \ }x_{k}=x_{k-1}-s_{k}v_{k}\in \partial \Omega .
\end{equation*}

\begin{proposition}
\label{jacobian} For any finite $k\geq 1,$%
\begin{equation}
\frac{\partial v_{k}^{i}}{\partial v_{1}^{l}}=\delta _{li}+\zeta
(k)n^{i}(x_{1})n^{l}(x_{1})+O(\varepsilon _{0}),  \label{vkv1}
\end{equation}%
where $O$ depends on $k,$ and $\zeta $ is defined as $\zeta (1)=0,$ 
\begin{equation}
\zeta (k)=4\sum_{p=1}^{k-2}(-1)^{k-p+1}+4\sum_{p=1}^{k-2}(-1)^{k-1-p}\zeta
(p)+2+3\zeta (k-1),\text{ for }k\geq 2.  \label{zetak}
\end{equation}%
In particular, $\zeta (k)$ is an even integer so that 
\begin{equation*}
\det \left( \frac{\partial v_{k}^{i}}{\partial v_{1}^{l}}\right) =\{\zeta
(k)+1\}+O(\varepsilon _{0})\neq 0.
\end{equation*}
\end{proposition}

\begin{proof}
From $n(x_{j})=\frac{\nabla \xi (x_{j})}{|\nabla \xi (x_{j})|},$ since $%
v_{j}=v_{j-1}-2\{n(x_{j})\cdot v_{j-1}\}n(x_{j}),$ we define 
\begin{equation}
d_{j}\equiv v_{j}\cdot \nabla \xi (x_{j})=-v_{j-1}\cdot \nabla \xi
(x_{j})\geq 0.  \label{dj}
\end{equation}%
By the Velocity Lemma \ref{velocity} and our choice of $v_{1}$ in (\ref{v1}%
), if $\sum_{j=1}^{k-1}s_{j}<C,$ we have $C_{1}\alpha (0)\leq \alpha
(\sum_{p=1}^{j}s_{p})\leq C_{2}\alpha (0),$ for all $j=1,2,...,k-1.$ But $%
\alpha (0)=\{v_{1}\cdot \nabla \xi (x_{1})\}^{2}\backsim \varepsilon
_{0}^{4} $ and $\alpha (\sum_{p=1}^{j}s_{p})=\{v_{j}\cdot \nabla \xi
(x_{j})\}^{2},~$\ we then have 
\begin{equation}
v_{j}\cdot n(x_{j})=-v_{j-1}\cdot n(x_{j})\backsim C\varepsilon _{0}^{2}.
\label{d}
\end{equation}%
We therefore deduce that, by denoting $n_{j}=n(x_{j}),$ 
\begin{eqnarray}
|v_{j}-v_{1}| &\leq &|v_{j}-v_{j-1}|+|v_{j-1}-v_{j-2}|+...+|v_{2}-v_{1}| 
\notag \\
&\leq &2|v_{j-1}\cdot n_{j}|+2|v_{j-2}\cdot n_{j-1}|+...+2|v_{1}\cdot n_{2}|
\notag \\
&\leq &2jC\varepsilon _{0}^{2}.  \label{j-1}
\end{eqnarray}%
With the assumption $\sum_{j}^{k-1}s_{j}<C,$ by $|v_{1}|=\varepsilon _{0},$
we deduce that 
\begin{equation}
|x_{k}-x_{1}|\leq C\sum_{j=1}^{k-1}|v_{j}|\leq C_{k}\varepsilon _{0}.
\label{xe}
\end{equation}

We first estimate the next $s_{k}.$ Note that for $k\geq 2,$ 
\begin{equation*}
\xi (x_{k}+s_{k-1}v_{k-1})=0,\text{ \ \ \ \ }\xi (x_{k}-s_{k}v_{k})=0.
\end{equation*}%
We then use Taylor expansion at $x_{k}$ to get 
\begin{eqnarray*}
\xi (x_{k}+s_{k-1}v_{k-1}) &=&\xi (x_{k})+s_{k-1}v_{k-1}\cdot \nabla \xi
(x_{k})+\frac{1}{2}s_{k-1}^{2}v_{k-1}\nabla ^{2}\xi
(x_{k})v_{k-1}+O(s_{k-1}^{3}v_{k-1}^{3}); \\
\xi (x_{k}-s_{k}v_{k}) &=&\xi (x_{k})-s_{k}v_{k}\cdot \nabla \xi (x_{k})+%
\frac{1}{2}s_{k}^{2}v_{k}\nabla ^{2}\xi (x_{k})v_{k}+O(s_{k}^{3}v_{k}^{3}).
\end{eqnarray*}%
But $\frac{\nabla \xi (x_{k})}{|\nabla \xi (x_{k})|}=n_{k},$ $\xi (x_{k})=0,$
$|v_{k}|=|v_{k+1}|=O(\varepsilon _{0}),$ by (\ref{dj}), (\ref{d}) and (\ref%
{j-1}) , we have 
\begin{eqnarray*}
1-\frac{1}{2}s_{k-1}^{{}}\frac{v_{k-1}\nabla ^{2}\xi (x_{k})v_{k-1}}{d_{k}}%
+O(\varepsilon _{0})s_{k-1}^{2} &=&0, \\
1-\frac{1}{2}s_{k}^{{}}\frac{v_{k}\nabla ^{2}\xi (x_{k})v_{k}}{d_{k}}%
+O(\varepsilon _{0})s_{k}^{2} &=&0,
\end{eqnarray*}%
where $d_{k}\backsim \varepsilon _{0}^{2}.$ Therefore, by (\ref{j-1}) and (%
\ref{xe}), 
\begin{eqnarray}
s_{k-1} &=&\frac{2d_{k}}{v_{k-1}\nabla ^{2}\xi (x_{k})v_{k-1}}+O(\varepsilon
_{0})=\frac{2d_{k}}{v_{1}\nabla ^{2}\xi (x_{1})v_{1}}+O(\varepsilon _{0}), 
\notag \\
s_{k} &=&\frac{2d_{k}}{v_{k}\nabla ^{2}\xi (x_{k})v_{k}}+O(\varepsilon _{0})=%
\frac{2d_{k}}{v_{1}\nabla ^{2}\xi (x_{1})v_{1}}+O(\varepsilon _{0}),
\label{tj}
\end{eqnarray}%
so that $s_{k+1}-s_{k}=O(\varepsilon _{0}),$ for finite $k.$

We now compute from $v_{k}=v_{k-1}-2\{n(x_{k})\cdot v_{k-1}\}n(x_{k}),$ 
\begin{equation}
\frac{\partial v_{k}^{i}}{\partial v_{1}^{l}}=\frac{\partial v_{k-1}^{i}}{%
\partial v_{1}^{l}}-2(v_{k-1}\cdot n_{k})\partial
_{v_{1}^{l}}n_{k}^{i}-2(v_{k-1}\cdot \partial _{v_{1}^{l}}n_{k})n_{k}^{i}-2(%
\frac{\partial v_{k-1}}{\partial v_{1}^{l}}\cdot n_{k})n_{k}^{i}.
\label{partialv}
\end{equation}

To compute $\partial _{v_{1}^{l}}n_{k}^{m}$ in (\ref{partialv})$,$ we note $%
n^{m}(y)=\frac{\partial _{m}\xi (y)}{|\nabla \xi (y)|}$ and 
\begin{eqnarray}
\partial _{v_{1}^{l}}n_{k}^{m} &=&\partial
_{v_{1}^{l}}\{n^{m}(x_{1}-\sum_{j=1}^{k-1}s_{j}v_{j})\}  \notag \\
&=&\partial _{q}n^{m}(x_{k})\times \{-\sum_{j=1}^{k-1}\partial
_{v_{1}^{l}}s_{j}v_{j}^{q}-\sum_{j=1}^{k-1}s_{j}\frac{\partial v_{j}^{q}}{%
\partial v_{1}^{l}}\}  \notag \\
&=&(\frac{\partial _{mq}\xi }{|\nabla \xi |}-\frac{n^{m}\partial _{oq}\xi
n^{o}}{|\nabla \xi |})|_{x_{k}}\times \{-\sum_{j=1}^{k-1}\partial
_{v_{1}^{l}}s_{j}v_{j}^{q}-\sum_{j=1}^{k-1}s_{j}\frac{\partial v_{j}^{q}}{%
\partial v_{1}^{l}}\}.  \label{partialn}
\end{eqnarray}

To compute $\partial _{v_{1}^{l}}s_{j},$ we recall $\xi (x_{j+1})=\xi
(x_{j})=0$ so that 
\begin{equation*}
\xi (x_{1}-\sum_{p=1}^{j}s_{p}v_{p})=0,\text{ \ \ }\xi
(x_{1}-\sum_{p=1}^{j-1}s_{p}v_{p})=0.
\end{equation*}%
Taking their $v_{1}^{l}$ derivatives, we split $\sum_{p=1}^{j}$ into $%
\sum_{p=1}^{j-1}+$ $\sum_{p=j}^{j}~$\ to get 
\begin{eqnarray*}
\sum_{o}\partial _{o}\xi (x_{j+1})\{-\sum_{p=1}^{j-1}\frac{\partial v_{p}^{o}%
}{\partial v_{1}^{l}}s_{p}-\sum_{p=1}^{j-1}v_{p}^{o}\partial
_{v_{1}^{l}}s_{p}\} &=&\sum_{o}\partial _{o}\xi (x_{j+1})\{\frac{\partial
v_{j}^{o}}{\partial v_{1}^{l}}s_{j}+v_{j}^{o}\partial _{v_{1}^{l}}s_{j}\}, \\
\sum_{o}\partial _{o}\xi (x_{j})\{-\sum_{p=1}^{j-1}\frac{\partial v_{p}^{o}}{%
\partial v_{1}^{l}}s_{p}-\sum_{p=1}^{j-1}v_{p}^{o}\partial
_{v_{1}^{l}}s_{p}\} &=&0.
\end{eqnarray*}%
Subtracting these two identities, we deduce%
\begin{eqnarray*}
\sum_{o}\partial _{o}\xi (x_{j+1})v_{j}^{o}\partial _{v_{1}^{l}}s_{j}
&=&-\sum_{o}\partial _{o}\xi (x_{j+1})\frac{\partial v_{j}^{o}}{\partial
v_{1}^{l}}s_{j} \\
&&+\sum_{o}\{\partial _{o}\xi (x_{j+1})-\partial _{o}\xi
(x_{j})\}\{-\sum_{p=1}^{j-1}\frac{\partial v_{p}^{o}}{\partial v_{1}^{l}}%
s_{p}-\sum_{p=1}^{j-1}v_{p}^{o}\partial _{v_{1}^{l}}s_{p}\}.
\end{eqnarray*}%
But by the Taylor expansion and (\ref{xe}),%
\begin{eqnarray*}
\partial _{o}\xi (x_{j+1})-\partial _{o}\xi (x_{j}) &=&\partial _{oe}\xi
(x_{j})(x_{j+1}^{e}-x_{j}^{e})+O(|x_{j+1}-x_{j}|^{2}) \\
&=&-\partial _{oe}\xi (x_{j})s_{j}v_{j}^{e}+O(\varepsilon _{0}^{2}).
\end{eqnarray*}%
Rewriting $\sum_{o}\partial _{o}\xi (x_{j+1})v_{j}^{o}=v_{j}\cdot \nabla \xi
(x_{j+1}),$ we therefore have 
\begin{eqnarray*}
v_{j}\cdot \nabla \xi (x_{j+1})\partial _{v_{1}^{l}}s_{j}
&=&-\sum_{o}\partial _{o}\xi (x_{j+1})\frac{\partial v_{j}^{o}}{\partial
v_{1}^{l}}s_{j} \\
&&+\sum_{o,e}\partial _{oe}\xi (x_{j})s_{j}v_{j}^{e}\{\sum_{p=1}^{j-1}\frac{%
\partial v_{p}^{o}}{\partial v_{1}^{l}}s_{p}+\sum_{p=1}^{j-1}v_{p}^{o}%
\partial _{v_{1}^{l}}s_{p}\} \\
&&+O(\varepsilon _{0}^{2})\{\sum_{p=1}^{j-1}\frac{\partial v_{p}^{0}}{%
\partial v_{1}^{l}}s_{p}+\sum_{p=1}^{j-1}v_{p}^{0}\partial
_{v_{1}^{l}}s_{p}\}.
\end{eqnarray*}%
Since $v_{j}\cdot \nabla \xi (x_{j+1})=-d_{j+1},$ from (\ref{tj}), $%
\sum_{o,e}\frac{\partial _{oe}\xi (x_{j})s_{j}v_{j}^{e}v_{p}^{o}}{-d_{j+1}}%
=-2+O(\varepsilon _{0})$ and 
\begin{eqnarray}
\partial _{v_{1}^{l}}s_{j} &=&\sum_{o}\frac{\partial _{o}\xi (x_{j+1})}{%
d_{j+1}}\frac{\partial v_{j}^{o}}{\partial v_{1}^{l}}s_{j}-\{2-O(\varepsilon
_{0})\}\sum_{p=1}^{j-1}\partial _{v_{1}^{l}}s_{p}-\sum_{o,e}\sum_{p=1}^{j-1}%
\frac{\partial _{oe}\xi (x_{j})s_{j}v_{j}^{e}}{d_{j+1}}\frac{\partial
v_{p}^{o}}{\partial v_{1}^{l}}s_{p}  \notag \\
&&+O(1)\{\sum_{p=1}^{j-1}\frac{\partial v_{p}^{o}}{\partial v_{1}^{l}}%
s_{p}+\sum_{p=1}^{j-1}v_{p}^{o}\partial _{v_{1}^{l}}s_{p}\}.
\label{partialt}
\end{eqnarray}%
We first claim that for $1\leq j\leq k$%
\begin{equation}
|\partial _{v_{1}^{l}}s_{j}|\leq \frac{C_{k}}{\varepsilon _{0}^{2}},\text{ \
\ }|\frac{\partial v_{j}^{0}}{\partial v_{1}^{l}}|\leq C_{k}.
\label{claimbound}
\end{equation}%
We shall prove this via an induction of $j.$ In fact, when $j=1,$ $\frac{%
\partial v_{1}^{0}}{\partial v_{1}^{l}}=\delta _{ol},$ and from $\xi
(x_{1}-s_{1}v_{1})=0,$ we deduce 
\begin{equation}
\partial _{v_{1}^{l}}s_{1}=\frac{\partial _{l}\xi (x_{2})s_{1}}{d_{2}}%
=O(\varepsilon _{0}^{-2}).  \label{s1v1}
\end{equation}%
And a simple induction leads to the desired result (\ref{claimbound}).

From (\ref{claimbound}) and (\ref{partialt}), we have 
\begin{equation*}
\partial _{v_{1}^{l}}s_{j}=\sum_{o}\frac{\partial _{o}\xi (x_{j+1})}{d_{j+1}}%
\frac{\partial v_{j}^{o}}{\partial v_{1}^{l}}s_{j}-2\sum_{p=1}^{j-1}\partial
_{v_{1}^{l}}s_{p}+O(\frac{1}{\varepsilon _{0}}).
\end{equation*}%
By letting $A_{j}\equiv \sum_{p=1}^{j}\partial _{v_{1}^{l}}s_{p}$ and moving
one copy of $\sum_{p=1}^{j-1}\partial _{v_{1}^{l}}s_{p}$ to the right hand
side, we deduce%
\begin{equation*}
A_{j}=\sum_{o}\frac{\partial _{o}\xi (x_{j+1})}{d_{j+1}}\frac{\partial
v_{j}^{o}}{\partial v_{1}^{l}}s_{j}-A_{j-1}+O(\frac{1}{\varepsilon _{0}})
\end{equation*}%
so that we can obtain explicit formula for $A_{j}$ as 
\begin{eqnarray}
A_{j} &=&\sum_{o}\frac{\partial _{o}\xi (x_{j+1})}{d_{j+1}}\frac{\partial
v_{j}^{o}}{\partial v_{1}^{l}}s_{j}-A_{j-1}+O(\frac{1}{\varepsilon _{0}}) 
\notag \\
&=&\sum_{o}\frac{\partial _{o}\xi (x_{j+1})}{d_{j+1}}\frac{\partial v_{j}^{o}%
}{\partial v_{1}^{l}}s_{j}-\sum_{o}\frac{\partial _{o}\xi (x_{j})}{d_{j}}%
\frac{\partial v_{j-1}^{o}}{\partial v_{1}^{l}}s_{j-1}+A_{j-2}+O(\frac{1}{%
\varepsilon _{0}})...  \notag \\
&=&\sum_{p=1}^{j}\sum_{o}(-1)^{j-p}\frac{\partial _{o}\xi (x_{p+1})}{d_{p+1}}%
\frac{\partial v_{p}^{o}}{\partial v_{1}^{l}}s_{p}+O(\frac{1}{\varepsilon
_{0}}),  \label{aj}
\end{eqnarray}%
we have used the fact by (\ref{s1v1}), $A_{1}=\partial
_{v_{1}^{l}}s_{1}=\sum_{o}\frac{\partial _{o}\xi (x_{2})}{d_{2}}\frac{%
\partial v_{1}^{o}}{\partial v_{1}^{l}}s_{1}.$ Now finally we recall (\ref%
{partialn}), $\partial _{v_{1}^{l}}n_{k}^{i}=O(\frac{1}{\varepsilon _{0}}),$
so that the second term on the right hand side in (\ref{partialv}) is of the
order $O(\varepsilon _{0})$. Hence 
\begin{eqnarray*}
\frac{\partial v_{k}^{i}}{\partial v_{1}^{l}} &=&\frac{\partial v_{k-1}^{i}}{%
\partial v_{1}^{l}}-2(v_{k-1}^{m}(\frac{\partial _{mq}\xi }{|\nabla \xi |}-%
\frac{n^{m}\partial _{oq}\xi n^{o}}{|\nabla \xi |})|_{x_{k}}\times
\{-\sum_{j=1}^{k-1}\partial _{v_{1}^{l}}s_{j}v_{j}^{q}-\sum_{j=1}^{k-1}s_{j}%
\frac{\partial v_{j}^{q}}{\partial v_{1}^{l}}\})n_{k}^{i} \\
&&-2(\frac{\partial v_{k-1}^{i}}{\partial v_{1}^{l}}\cdot
n_{k})n_{k}^{i}+O(\varepsilon _{0}).
\end{eqnarray*}%
Since $\sum_{m}v_{k-1}^{m}n^{m}(x_{k})=O(\varepsilon _{0}^{2}),$ the second
term on the right hand side is%
\begin{eqnarray*}
&&2\sum_{j=1}^{k-1}\frac{v_{k-1}^{m}\partial _{mq}\xi (x_{k})v_{j}^{q}}{%
|\nabla \xi (x_{k})|}\times \partial _{v_{1}^{l}}s_{j}n_{k}^{i}-2(\frac{%
\partial v_{k-1}^{{}}}{\partial v_{1}^{l}}\cdot
n_{k})n_{k}^{i}+O(\varepsilon _{0}) \\
&=&2\frac{v_{1}^{m}\partial _{mq}\xi (x_{1})v_{1}^{q}}{|\nabla \xi (x_{1})|}%
\times \{\sum_{j=1}^{k-1}\partial _{v_{1}^{l}}s_{j}\}n_{1}^{i}-2(\frac{%
\partial v_{k-1}^{{}}}{\partial v_{1}^{l}}\cdot
n_{1})n_{1}^{i}+O(\varepsilon _{0})\text{ by (\ref{xe})} \\
&=&2\frac{v_{1}^{m}\partial _{mq}\xi (x_{1})v_{1}^{q}}{|\nabla \xi (x_{1})|}%
\times \sum_{j=1}^{k-1}(-1)^{k-p-1}\frac{\partial _{o}\xi (x_{p+1})}{d_{p+1}}%
\frac{\partial v_{p}^{o}}{\partial v_{1}^{l}}s_{p}n_{1}^{i}+O(\varepsilon
_{0})\text{ by (\ref{aj}).}
\end{eqnarray*}%
Note $\frac{s_{p}}{d_{p+1}}v_{1}^{m}\partial _{mq}\xi
(x_{1})v_{1}^{q}=4+O(\varepsilon _{0})$ from (\ref{dj}) and (\ref{tj}), we
deduce 
\begin{equation*}
\frac{\partial v_{k}^{i}}{\partial v_{1}^{l}}=\frac{\partial v_{k-1}^{i}}{%
\partial v_{1}^{l}}+4\sum_{j=1}^{k-1}(-1)^{k-p-1}n_{1}^{o}\frac{\partial
v_{p}^{o}}{\partial v_{1}^{l}}n_{1}^{i}-2(\frac{\partial v_{k-1}^{i}}{%
\partial v_{1}^{l}}\cdot n_{1})n_{1}^{i}+O(\varepsilon _{0})\text{.}
\end{equation*}%
Clearly, $\xi (1)=0,$ and assume (\ref{vkv1}) is valid up to $k-1.$ Then 
\begin{eqnarray*}
\frac{\partial v_{k}^{i}}{\partial v_{1}^{l}} &=&\delta _{li}+\zeta
(k-1)n_{1}^{i}n_{1}^{l}+4\sum_{p}^{k-1}(-1)^{k-p-1}n_{1}^{o}(\delta
_{ol}+\zeta (p)n_{1}^{o}n_{1}^{l})n_{1}^{i} \\
&&-2(\delta _{lo}+\zeta (k-1)n_{1}^{o}n_{1}^{l})\cdot
n_{1}^{o}n_{1}^{i}+O(\varepsilon _{0}).
\end{eqnarray*}%
Notice that $\sum_{o}n_{1}^{o}n_{1}^{o}=1,$ and splitting $%
\sum_{j}^{k-1}=\sum_{j}^{k-2}+\sum_{j=k-1}^{k-1}$ and collecting terms, we
conclude our proposition.
\end{proof}

\subsubsection{$L^{\infty }$ Decay for $U(t)$}

We now fix any point $(t,x,v)$ so that $(x,v)\notin \gamma _{0}.$ Let the
back-time specular cycle of $(t,x,v)$ be $[X_{\mathbf{cl}}^{{}}(s_{1}),V_{%
\mathbf{cl}}^{{}}(s_{1})].$ By (\ref{duhamel2}), we use twice (\ref%
{specularformula}) to derive $h(t,x,v)=$ 
\begin{eqnarray}
&&e^{-\nu (v)t}h_{0}(X_{\mathbf{cl}}^{{}}(0),V_{\mathbf{cl}}^{{}}(0))  \notag
\\
&&+\int_{0}^{t}e^{-\nu (v)(t-s_{1})}\int K_{w}(V_{\mathbf{cl}%
}^{{}}(s_{1}),v^{\prime })e^{-\nu (v^{\prime })s_{1}}h_{0}\left( X_{\mathbf{%
cl}}^{\prime }(0),V_{\mathbf{cl}}^{\prime }(0)\right) dv^{\prime }
\label{hspecular} \\
&&+\int_{0}^{t}\int_{0}^{s_{1}}\int e^{-\nu (v)(t-s_{1})-\nu (v^{\prime
})(s_{1}-s)}K_{w}(V_{\mathbf{cl}}(s_{1}),v^{\prime })K_{w}(V_{\mathbf{cl}%
}^{\prime }(s),v^{\prime \prime })h\left( X_{\mathbf{cl}}^{\prime
}(s),v^{\prime \prime }\right) .  \notag
\end{eqnarray}%
where the back-time specular cycle from $(s_{1},X_{\mathbf{cl}%
}^{{}}(s_{1}),v^{\prime })$ is denoted by 
\begin{equation}
X_{\mathbf{cl}}^{\prime }(s)=X_{\mathbf{cl}}(s;s_{1},X_{\mathbf{cl}%
}(s_{1}),v^{\prime }),\text{ \ \ \ \ }V_{\mathbf{cl}}^{\prime }(s)=V_{%
\mathbf{cl}}^{{}}(s;s_{1},X_{\mathbf{cl}}^{{}}(s_{1}),v^{\prime }).
\label{x'}
\end{equation}%
More explicitly, let $t_{k}$ and $t_{k^{\prime }}^{\prime }$ be the
corresponding times for both specular cycles as in (\ref{specularcycle}).
For $t_{k+1}\leq s_{1}<t_{k},$ $t_{k^{\prime }+1}^{\prime }\leq
s<t_{k^{\prime }}^{\prime }$ 
\begin{equation}
X_{\mathbf{cl}}^{\prime }(s)=X_{\mathbf{cl}}(s;s_{1},X_{\mathbf{cl}%
}(s_{1}),v^{\prime })\equiv x_{k^{\prime }}^{\prime }+(s-t_{k^{\prime
}}^{\prime })v_{k^{\prime }}^{\prime }  \label{x'explicit}
\end{equation}%
where $x_{k^{\prime }}^{\prime }=X_{\mathbf{cl}}(t_{k^{\prime
}};s_{1},x_{k}+(s_{1}-t_{k})v_{k},v^{\prime }),v_{k^{\prime }}^{\prime }=V_{%
\mathbf{cl}}(t_{k^{\prime }};s_{1},x_{k}+(s_{1}-t_{k})v_{k},v^{\prime }).\,$%
\ Recall $\alpha $ in (\ref{alpha}) and define naturally 
\begin{equation}
\alpha (x,v)\equiv \alpha (t)=\xi ^{2}(x)+[v\cdot \nabla \xi
(x)]^{2}-[v\nabla ^{2}\xi (x)v]\xi (x).  \label{alphaxv}
\end{equation}%
We define the main set 
\begin{equation}
A_{\alpha }=\{(x,v):x\in \bar{\Omega},\text{ }\frac{1}{N}\leq |v|\leq N,%
\text{ and }\alpha (x,v)\geq \frac{1}{N}\}.  \label{aalpha}
\end{equation}

We remark that for $x$ is near $\partial \Omega ,$ $\det \{\frac{\partial
v_{2}}{\partial v_{1}}\}\backsim 3$ in Lemma \ref{jacobian} for $v_{1}$ is
almost tangential to $n(x).$ On the other hand, by (\ref{tbderivative}), it
is easy to compute for $v_{1}=n(x),$ $\det \{\frac{\partial v_{2}}{\partial
v_{1}}\}\backsim -1$ since $t_{\mathbf{b}}$ $\backsim 0.$ This implies from
continuity that there is $v_{1}$ such that $\det \{\frac{\partial v_{2}}{%
\partial v_{1}}\}=0$ even after one specular reflection. However, as shown
next, such a zero set is small if $\Omega $ is both analytic and convex.

\begin{lemma}
\label{specularlower} Fix $k$ and $k^{\prime }.$ Define for $t_{k+1}\leq
s_{1}\leq t_{k},s\in \mathbf{R}$%
\begin{equation*}
J\equiv J_{k,k^{\prime }}(t,x,v,s_{1},s,v^{\prime })\equiv \det \left( \frac{%
\partial \{x_{k^{\prime }}^{\prime }+(s-t_{k^{\prime }}^{\prime
})v_{k^{\prime }}^{\prime }\}}{\partial v^{\prime }}\right) .
\end{equation*}%
For any $\varepsilon >0$ sufficiently small, there is $\delta (N,\varepsilon
,T_{0},k,k^{\prime })>0$ and an open covering $\cup
_{i=1}^{m}B(t_{i},x_{i},v_{i};r_{i})$ of $[0,T_{0}]\times A_{\alpha }$ and
corresponding open sets $O_{t_{i},x_{i},v_{i}}$ for $[t_{k+1}+\varepsilon
,t_{k}-\varepsilon ]\times \mathbf{R}\times \mathbf{R}^{3}$ with $%
|O_{t_{i},x_{i},v_{i}}|<\varepsilon ,$ such that 
\begin{equation*}
|J_{k,k^{\prime }}(t,x,v,s_{1},s,v^{\prime })|\geq \delta >0,
\end{equation*}%
for $0\leq t\leq T_{0},$ $(x,v)\in A_{\alpha }$ and $(s_{1},s,v^{\prime })$
in 
\begin{equation*}
O_{t_{i,}x_{i},v_{i}}^{c}\cap \lbrack t_{k+1}+\varepsilon ,t_{k}-\varepsilon
]\times \lbrack 0,T_{0}]\times \{|v^{\prime }|\leq 2N\}.
\end{equation*}
\end{lemma}

\begin{proof}
Fix $(t,x,v)$ such that $x,v\in A_{\alpha }$ in (\ref{aalpha}), and fix $%
k,k^{\prime }.$ Since $x,v\notin \gamma _{0}$, by Velocity Lemma \ref%
{velocity}, we deduce that $\alpha (t_{k})=\{n(x_{k})\cdot v_{k}\}^{2}\neq 0$
so that $t_{k}-t_{k+1}>0$ from (\ref{tlower}). We note since $|v^{\prime
}|\geq \frac{1}{N}$ from (\ref{aalpha}), $t_{k}-t_{k+1}\leq N$diam$\Omega .$
Since for $t_{k+1}+\frac{\varepsilon }{2}\leq s_{1}\leq t_{k}-\frac{%
\varepsilon }{2},$ $x_{k}-(s_{1}-t_{k})v_{k}\in \Omega \mathbf{,}$ the
interior of the domain with $\varepsilon $ sufficiently small. From (\ref%
{alphaxv}) $\alpha (x_{k}+(s_{1}-t_{k})v_{k},v^{\prime })>0$ for all $%
v^{\prime }$. This implies that along its back-time specular cycle $[X_{%
\mathbf{cl}}^{\prime }(s),V_{\mathbf{cl}}^{\prime }(s)],$ $\alpha
(t_{l^{\prime }}^{\prime })>0$ and $v_{l^{\prime }}^{\prime }\cdot
n(x_{l^{\prime }}^{\prime })\neq 0$ from the Velocity Lemma \ref{velocity}$.$
Clearly, by the Velocity Lemma \ref{velocity} and part (2) of Lemma \ref%
{huang}, $t_{l}^{\prime },x_{l}^{\prime },v_{l}^{\prime }$ are analytical
functions of $s_{1},s,v^{\prime }.$ Therefore the function $J_{k,k^{\prime
}}(t,x,v,s_{1},s,v^{\prime })$ is well-defined, and analytic for all $%
v^{\prime }\in \mathbf{R}^{3},$ $s\in \mathbf{R}$, and $t_{k+1}+\frac{%
\varepsilon }{2}\leq s_{1}\leq t_{k}-\frac{\varepsilon }{2}$. Moreover,
expanding as a polynomial of $s,$ we obtain 
\begin{equation*}
J(t,x,v,s_{1},s,v^{\prime })=\det \left( \frac{\partial v_{k^{\prime
}}^{\prime }}{\partial v^{\prime }}\right) s^{3}+p_{1}s^{2}+p_{2}s+p_{3}
\end{equation*}%
where $p_{i}=p_{i}(t,x,v,s_{1},v^{\prime })$ is an analytical function of $%
(s_{1},v^{\prime })\in (t_{k+1}+\frac{\varepsilon }{2},t_{k}-\frac{%
\varepsilon }{2})\times \mathbf{R}^{3}.$ But at $s_{1}=t_{k+1}$, $X_{\mathbf{%
cl}}(s_{1})=x_{k+1}\in \partial \Omega .$ From Proposition \ref{jacobian},
there exists $v_{0}^{\prime }$ with $\alpha (t_{k+1})=\{v_{0}^{\prime }\cdot
n(x_{k+1})\}^{2}=\varepsilon _{0}^{4}>0$ such that $\det \left( \frac{%
\partial v_{k^{\prime }}^{\prime }}{\partial v^{\prime }}\right)
|_{v^{\prime }=v_{0}}\neq 0.$ Since $v_{0}\cdot n(x_{k+1})\neq 0,$ by the
Velocity lemma \ref{velocity} and Lemma \ref{huang} for such $v_{0}^{\prime
},$ $\det \left( \frac{\partial v_{k}^{\prime }}{\partial v^{\prime }}%
\right) $ is continuous with respect to $y$ near $x_{k+1}.$ In particular,
along $[X_{\mathbf{cl}}^{\prime }(s),V_{\mathbf{cl}}^{\prime }(s)]$ in (\ref%
{x'}), $\det \left( \frac{\partial v_{k^{\prime }}^{\prime }}{\partial
v^{\prime }}\right) |_{v^{\prime }=v_{0}}\neq 0$ for some $s_{1}$ at $%
t_{k+1}+\frac{3\varepsilon }{4},$ for $\varepsilon $ sufficiently small so
that $x_{k}+(s_{1}-t_{k})v_{k}\backsim x_{k}+(t_{k+1}-t_{k})v_{k}\backsim
x_{k+1}$. Therefore, $\det \left( \frac{\partial v_{k}^{\prime }}{\partial
v^{\prime }}\right) $ is an analytical function which is not identically
zero, so is $J_{k,k^{\prime }}(t,x,v,s_{1},s,v^{\prime })$ as an analytical
function of $(s_{1},s,v^{\prime })\in (t_{k+1}+\frac{\varepsilon }{2},t_{k}-%
\frac{\varepsilon }{2})\times \mathbf{R\times R}^{3}$. By Lemma \ref%
{analytic}, for each $(t,x,v),$ there exists an open set $O_{t,x,v}$ of $%
s_{1},s,v^{\prime }$ in $(t_{k+1}+\frac{\varepsilon }{2},t_{k}-\frac{%
\varepsilon }{2})\times \mathbf{R}\times \mathbf{R}^{3}$ such that $%
|O_{t,x,v}|<\varepsilon ,$ and for $(s_{1},s,v^{\prime })\notin O_{t,x,v},$ $%
J(t,x,v,s_{1},s,v^{\prime })\neq 0.$ Therefore, by continuity of $%
J(t,x,v,s_{1},s,v^{\prime })$ with respect to $s_{1},s,v^{\prime }$, there
exists $\delta _{t,x,v,N,T_{0},\varepsilon ,\varepsilon _{1},k,k^{\prime
}}>0,$ such that 
\begin{equation*}
|J(t,x,v,s_{1},s,v^{\prime })|>\delta _{t,x,v,N,T_{0},\varepsilon
,k,k^{\prime }}>0
\end{equation*}%
for the compact set: 
\begin{equation*}
(s_{1},s,v^{\prime })\in O_{t,x,v}^{c}\cap \lbrack t_{k+1}+\frac{%
3\varepsilon }{4},t_{k}-\frac{3\varepsilon }{4}]\times \lbrack
0,T_{0}]\times \{|v^{\prime }|\leq 2N\}.
\end{equation*}

Since $\alpha (x,v)\geq \frac{1}{N},$ by the Velocity Lemma \ref{velocity}
and par (2) of Lemma \ref{huang}, $t_{k},x_{k},$ and $v_{k}$ are analytic
functions respect to $(t,x,v),$ and $x_{k^{\prime }}^{\prime }$ and $%
t_{k^{\prime }}^{\prime }~$are analytic with respect to $(t,x,v)$ as well$.$
Therefore, there exists an open ball $B(t,x,v;r_{(t,x,v,\varepsilon )})$
such that if $(\tau ,y,w)\in B(t,x,v;r_{(t,x,v,\varepsilon )}),$ 
\begin{eqnarray}
t_{k+1}(\tau ,y,w) &<&t_{k+1}(t,x,v)+\frac{\varepsilon }{2}%
<s_{1}<t_{k}(t,x,v)-\frac{\varepsilon }{2}<t_{k}(\tau ,y,w),  \notag \\
t_{k+1}(t,x,v)+\frac{\varepsilon }{2} &<&t_{k+1}(\tau ,y,w)+\varepsilon ,%
\text{ \ }t_{k}(\tau ,y,w)-\varepsilon <t_{k}(t,x,v)-\frac{\varepsilon }{2}.
\label{taut}
\end{eqnarray}%
Hence by (\ref{taut}), $J_{k,k^{\prime }}(\tau ,y,z,s_{1},s,v^{\prime })$ is
well-defined, continuous, and we may then assume 
\begin{equation*}
|J_{k,k^{\prime }}(\tau ,y,z,s_{1},s,v^{\prime })|>\frac{\delta
_{t,x,v,T_{0},N,\varepsilon ,k,k^{\prime }}}{2}>0,
\end{equation*}%
in $B(t,x,v;r_{(t,x,v,\varepsilon )})\times O_{t,x,v}^{c}\cap \lbrack
t_{k+1}(t,x,v)+\frac{\varepsilon }{2},t_{k}(t,x,v)-\frac{\varepsilon }{2}%
]\times \lbrack 0,T_{0}]\times \{|v^{\prime }|\leq 2N\},$ and clearly also
on the smaller set (by (\ref{taut})): 
\begin{equation*}
B(t,x,v;r_{(t,x,v,\varepsilon )})\times O_{t,x,v}^{c}\cap \lbrack
t_{k+1}(\tau ,y,z)+\varepsilon ,t_{k}(\tau ,y,z)-\varepsilon ]\times \lbrack
0,T_{0}]\times \{|v^{\prime }|\leq 2N\}.
\end{equation*}%
Now by a finite covering for the compact set $[0,T_{0}]\times A_{\alpha }$
by such $B(t,x,v;r),$ there are $[t_{1},x%
\,_{1},v_{1}],...[t_{m},x_{m},v_{m}] $ such that $[0,T_{0}]\times A_{\alpha
}\subset \cup _{i=1}^{m}B(t_{i},x_{i},v_{i};r_{i}).$ For any point $(t,x,v),$
there is $i$ so $(t,x,v)\in B(t_{i},x_{i},v_{i};r_{i}(\varepsilon ))$ and%
\begin{equation*}
|J(t,,x,v,s_{1,}s,v^{\prime })|>\min_{1\leq i\leq m}\frac{\delta
_{i,T_{0},N,\varepsilon .k,k^{\prime }}}{2}>0
\end{equation*}%
for $(s_{1},s,v^{\prime })\in O_{t_{i},x_{i},v_{i}}^{c}\cap \lbrack
t_{k+1}+\varepsilon ,t_{k}-\varepsilon ]\times \lbrack 0,T_{0}]\times
\{|v^{\prime }|\leq 2N\}.$
\end{proof}

\begin{theorem}
\label{specularrate}Assume $w^{-2}\{1+|v|\}\in L^{1}.$ Assume that $\xi $ is
both strictly convex (\ref{convexity}) and analytic, and the mass (\ref{mass}%
) and energy (\ref{energy}) are conserved. In the case of $\Omega $ has
rotational symmetry (\ref{axis}), we also assume conservation of
corresponding angular momentum (\ref{axiscon}). Let $h_{0}\in L^{\infty }.$
There exits a unique solution to both the (\ref{lboltzmann}) and (\ref%
{lboltzmannh}) with boundary specular condition, and the exponential decay (%
\ref{bouncebackdecay}) is valid.
\end{theorem}

\begin{proof}
The well-posedness follows from the exact argument in the proof of Theorem %
\ref{bouncebackrate}. Thanks to Lemma \ref{bootstrap}, we only need to
establish the finite time estimate (\ref{finitetime}). Recall $A_{\alpha }$
in (\ref{aalpha}).

\textbf{STEP 1:}\textit{\ } Estimate of $h(t,x,v)\mathbf{1}_{A_{\alpha }}.$
We first express and estimate the main part $h(t,x,v)\mathbf{1}_{A}$ through
(\ref{hspecular}). As in the case of bounce-back reflection, the first and
the second terms in (\ref{hspecular}) are bounded by (\ref{bouncebackfirst})
and (\ref{bouncebacksecond}) respectively.

For the third main contribution in (\ref{hspecular}), notice that along the
back-time specular cycles $[X_{\mathbf{cl}}^{{}}(s),V_{\mathbf{cl}}^{{}}(s)]$
and $[X_{\mathbf{cl}}^{\prime }(s),V_{\mathbf{cl}}^{\prime }(s)]$ in (\ref%
{x'}), $|V_{\mathbf{cl}}^{{}}(s_{1})|\equiv |v|$ and $|V_{\mathbf{cl}%
}^{\prime }(s_{1})|\equiv |v^{\prime }|.$ Hence, the integration over $%
|v^{\prime }|\geq 2N$ or $|v^{\prime }|\leq 2N$ but $|v^{\prime \prime
}|\geq 3N$ are bounded by (\ref{bouncebacknn}). By using the same
approximation, we only need to concentrate on the bounded set $\{|v^{\prime
}|\leq 2N$ and $|v^{\prime \prime }|\leq 3N\}$ as in (\ref{bouncebackbd}) of%
\begin{eqnarray*}
&&\int_{0}^{t}\int_{0}^{s_{1}}\int_{|v^{\prime }|\leq 2N,|v^{\prime \prime
}|\leq 3N}e^{-v_{0}(t-s)}\mathbf{1}_{A_{\alpha }}|h\left( s,X_{\mathbf{cl}%
}^{\prime }(s),v^{\prime \prime }\right) |dv^{\prime }dv^{\prime \prime
}ds_{1}ds \\
&=&\int_{\substack{ \alpha (X_{\mathbf{cl}}(s_{1}),v^{\prime })<\varepsilon 
\\ |v^{\prime }|\leq 2N,|v^{\prime \prime }|\leq 3N,}}+\int_{\substack{ %
\alpha (X_{\mathbf{cl}}(s_{1}),v^{\prime })\geq \varepsilon  \\ |v^{\prime
}|\leq 2N,|v^{\prime \prime }|\leq 3N,}}
\end{eqnarray*}%
where we have further spit into $\alpha (X_{\mathbf{cl}}(s_{1}),v^{\prime
})\leq \varepsilon $ and $\alpha (X_{\mathbf{cl}}(s_{1}),v^{\prime
})>\varepsilon .$

In the case $\alpha (X_{\mathbf{cl}}(s_{1}),v^{\prime })\leq \varepsilon $, $%
\xi ^{2}(X_{\mathbf{cl}}(s_{1}))+[v^{\prime }\cdot $ $\nabla \xi (X_{\mathbf{%
cl}}(s_{1}))]^{2}\leq \varepsilon .$ Hence for $\varepsilon $ small, $X_{%
\mathbf{cl}}(s_{1})\backsim \partial \Omega ,$ and $|\nabla \xi (X_{\mathbf{%
cl}}(s_{1}))|\geq \frac{1}{2}$ . The first part integral is bounded by 
\begin{eqnarray*}
&&C_{N}\int_{0}^{t}\int_{0}^{s_{1}}e^{-v_{0}(t-s)}||h(s)||_{\infty
}dsds_{1}\int_{\substack{ \alpha (X_{\mathbf{cl}}(s_{1}),v^{\prime })\leq
\varepsilon  \\ |v^{\prime }|\leq 2N,|v^{\prime \prime }|\leq 3N,}} \\
&\leq &C_{N}\sup_{t\geq s}e^{-\frac{v_{0}}{2}(t-s)}||h(s)||_{\infty
}\int_{|v^{\prime }\cdot \frac{\nabla \xi (X_{\mathbf{cl}}(s_{1}))}{|\nabla
\xi (X_{\mathbf{cl}}(s_{1}))|}|\leq 2\varepsilon ,|v^{\prime }|\leq
2N,|v^{\prime \prime }|\leq 3N}\leq C_{N}\varepsilon \sup_{t\geq s}e^{-\frac{%
v_{0}}{2}(t-s)}||h(s)||_{\infty }.
\end{eqnarray*}

Finally, from (\ref{x'explicit}), we bound the first main term $\alpha (X_{%
\mathbf{cl}}(s_{1}),v^{\prime })\geq \varepsilon $ as 
\begin{eqnarray*}
&&C_{N}\int_{0}^{t}e^{-\nu _{0}(t-s)}\int_{0}^{s_{1}}\int_{\substack{ \alpha
(X_{\mathbf{cl}}(s_{1}),v^{\prime })\geq \varepsilon  \\ |v^{\prime }|\leq
2N,|v^{\prime \prime }|\leq 3N}}\mathbf{1}_{A_{\alpha }}|h\left( s,X_{%
\mathbf{cl}}^{\prime }(s),v^{\prime \prime }\right) |dv^{\prime }dv^{\prime
\prime } \\
&=&C_{N}\sum_{k,k^{\prime }}\int_{t_{k+1}}^{t_{k}}\int_{t_{k^{\prime
}+1}^{\prime }}^{t_{k^{\prime }}^{\prime }}\int_{\substack{ \alpha (X_{%
\mathbf{cl}}(s_{1}),v^{\prime })\geq \varepsilon  \\ |v^{\prime }|\leq
2N,|v^{\prime \prime }|\leq 3N}}\mathbf{1}_{A_{\alpha }}e^{-\nu
_{0}(t-s)}|h\left( s,x_{k^{\prime }}^{\prime }+(s-t_{k^{\prime }}^{\prime
})v_{k^{\prime }}^{\prime },v^{\prime \prime }\right) |,
\end{eqnarray*}%
where $[t_{k^{\prime }}^{\prime },x_{k^{\prime }}^{\prime },v_{k^{\prime
}}^{\prime }]$ is the back-time cycle of $%
(s_{1},x_{k}+(s_{1}-t_{k})v_{k},v_{k}),$ for $t_{k+1}\leq s_{1}\leq t_{k}.$

We now study $x_{k^{\prime }}^{\prime }+(s-t_{k^{\prime }}^{\prime
})v_{k^{\prime }}^{\prime }.$ By the repeatedly using Velocity Lemma \ref%
{velocity}, we deduce for $(t,x,v)\in A_{\alpha }$ and $\alpha
(X(s_{1}),v^{\prime })\geq \varepsilon ,|v^{\prime }|\leq 2N:$ 
\begin{eqnarray*}
\alpha (t_{l}) &=&\{v_{l}\cdot n_{x_{l}}\}^{2}\geq e^{-\{C_{\xi
}N-1\}T_{0}}\alpha (s_{1})\geq C_{T_{0},\xi ,N}>0; \\
\alpha (t_{l}^{\prime }) &=&\{v_{l}^{\prime }\cdot n_{x_{l}^{\prime
}}\}^{2}\geq e^{-\{C_{\xi }N-1\}T_{0}}\alpha (X_{\mathbf{cl}%
}(s_{1}),v^{\prime })\geq C_{T_{0},N,\xi }\varepsilon >0.
\end{eqnarray*}%
Therefore, applying (\ref{tlower}) in Lemma \ref{huang} yields $%
t_{l}-t_{l+1}\geq \frac{c_{T_{0},\xi ,N}}{N^{2}}$ and $t_{l}^{\prime
}-t_{l+1}^{\prime }\geq \frac{c_{T_{0},\xi ,N}\varepsilon }{4N^{2}}$ so that 
\begin{equation}
k\leq \frac{T_{0}N^{2}}{c_{T_{0},\xi ,N}}=C_{T_{0},\xi ,N},\text{ \ \ }%
k^{\prime }\leq \frac{T_{0}N^{2}}{c_{T_{0},\xi ,N}\varepsilon }=C_{T_{0},\xi
,N,\varepsilon }.  \label{kk'bound}
\end{equation}

We therefore further split the $s_{1}-$integral as 
\begin{eqnarray*}
&&C_{K,N}\int_{t_{k+1}}^{t_{k}}\int_{\substack{  \\ |v^{\prime }|\leq
2N,|v^{\prime \prime }|\leq 3N}}\sum_{k\leq C_{T_{0},N},k^{\prime }\leq
C_{T_{0},N,}\varepsilon }\int_{t_{k^{\prime }+1}^{\prime }}^{t_{k^{\prime
}}^{\prime }}\mathbf{1}_{A_{\alpha }}e^{-\nu _{0}(t-s)}|h\left(
s,x_{k^{\prime }}^{\prime }+(s-t_{k^{\prime }}^{\prime })v_{k^{\prime
}}^{\prime },v^{\prime \prime }\right) | \\
&=&\int_{t_{k+1}+\varepsilon }^{t_{k}-\varepsilon
}+\int_{t_{k}}^{t_{k}-\varepsilon }+\int_{t_{k+1}}^{t_{k+1}+\varepsilon }.
\end{eqnarray*}

Since $\sum_{k^{\prime }}\int_{t_{k^{\prime }+1}^{\prime }}^{t_{k^{\prime
}}^{\prime }}=\int_{0}^{s_{1}},$ the last two terms make small contribution
as 
\begin{equation*}
\varepsilon C_{K,N}\sup_{0\leq s\leq t}e^{-\nu _{0}(t-s)}||h(s)||_{\infty
}\int_{0}^{T_{0}}\int_{\substack{  \\ |v^{\prime }|\leq 2N,|v^{\prime \prime
}|\leq 3N}}=\varepsilon C_{K,N}\sup_{0\leq s\leq t}e^{-\nu
_{0}(t-s)}||h(s)||_{\infty }.
\end{equation*}%
For the main contribution $\int_{t_{k+1}+\varepsilon }^{t_{k}-\varepsilon },$
By Lemma \ref{specularlower}, $\ $\ on the set $O_{t_{i},x_{i},v_{i}}^{c}%
\cap \lbrack t_{k+1}+\varepsilon ,t_{k}-\varepsilon ]\times \lbrack
0,T_{0}]\times \{|v^{\prime }|\leq N\},$ we can define a change of variable 
\begin{equation*}
y\equiv x_{k^{\prime }}^{\prime }+(s-t_{k^{\prime }}^{\prime })v_{k^{\prime
}}^{\prime },
\end{equation*}%
so that $\det (\frac{\partial y}{\partial v^{\prime }})>\delta $ on the same
set. By the Implicit Function Theorem, there are an finite open covering $%
\cup _{j=1}^{m}V_{j}$ of $O_{t_{i},x_{i},v_{i}}^{c}\cap \lbrack
t_{k+1}+\varepsilon ,t_{k}-\varepsilon ]\times \lbrack 0,T_{0}]\times
\{|v^{\prime }|\leq N\}$, and smooth function $F_{j}$ such that $v^{\prime
}=F_{j}(t,x,v,y,s_{1},s)$ in $V_{j}$. We therefore have 
\begin{equation*}
\sum_{k,k^{\prime }}\int_{t_{k+1}+\varepsilon }^{t_{k}-\varepsilon }\int 
_{\substack{  \\ |v^{\prime }|\leq 2N,}}\int_{t_{k^{\prime }+1}^{\prime
}}^{t_{k^{\prime }}^{\prime }}\leq \sum_{k,k^{\prime
}}\int_{t_{k+1}+\varepsilon }^{t_{k}-\varepsilon }\int_{\substack{  \\ %
|v^{\prime }|\leq 2N,}}\int_{t_{k^{\prime }+1}^{\prime }}^{t_{k^{\prime
}}^{\prime }}\mathbf{1}_{O_{t_{i},x_{i},v_{i}}}+\sum_{j,k,k^{\prime
}}\int_{t_{k+1}+\varepsilon }^{t_{k}-\varepsilon }\int_{\substack{  \\ %
|v^{\prime }|\leq 2N,}}\int_{t_{k^{\prime }+1}^{\prime }}^{t_{k^{\prime
}}^{\prime }}\mathbf{1}_{V_{j}}.
\end{equation*}%
Since $\sum_{k^{\prime }}\int_{t_{k^{\prime }+1}^{\prime }}^{t_{k^{\prime
}}^{\prime }}=\int_{0}^{s_{1}}\leq \int_{0}^{T_{0}}$ and $%
|O_{t_{i},x_{i},v_{i}}|<\varepsilon ,$ the first part is bounded by $%
C_{N}\varepsilon e^{-\frac{\nu _{0}}{2}t}\sup_{s}\{e^{\frac{\nu _{0}}{2}%
s}||h(s)||_{\infty }\}$ from Lemma \ref{specularlower}.

For the second part, we can make a change of variable $v^{\prime
}\rightarrow y=x_{k^{\prime }}^{\prime }+(s-t_{k^{\prime }}^{\prime
})v_{k^{\prime }}^{\prime }$ on each $V_{j}$ to get 
\begin{eqnarray*}
&&C_{\varepsilon ,T_{0},N}\sum_{j,k,k^{\prime }}\int_{V_{j}}\int_{|v^{\prime
\prime }|\leq 3N}e^{-\nu (v)(t-s)}|h\left( s,x_{k^{\prime }}^{\prime
}+(s-t_{k^{\prime }}^{\prime })v_{k^{\prime }}^{\prime },v^{\prime \prime
}\right) | \\
&=&C_{\varepsilon ,T_{0},N}\sum_{j}\int_{V_{j}}\int_{|v^{\prime \prime
}|\leq 3N}e^{-\nu (v)(t-s)}|h\left( s,y,v^{\prime \prime }\right) |\frac{1}{%
|\det \{\frac{\partial y}{\partial v^{\prime }}\}|}dydv^{\prime \prime
}dsds_{1} \\
&\leq &\frac{C_{\varepsilon ,T_{0},N}}{\delta }\int_{0}^{t}%
\int_{0}^{s_{1}}e^{-\nu _{0}t}\int_{|v^{\prime \prime }|\leq 3N}e^{\nu
_{0}s}\left\{ \int_{\Omega }h^{2}\left( s,y,v^{\prime \prime }\right)
dy\right\} ^{1/2}dv^{\prime \prime }dsds_{1} \\
&\leq &C_{\varepsilon ,T_{0},N}\int_{0}^{t}||f(s)||ds,
\end{eqnarray*}%
where $f=\frac{h}{w}.$ We therefore conclude, summing over $k$ and $%
k^{\prime }$ and collecting terms 
\begin{eqnarray}
||h(t,x,v)\mathbf{1}_{A_{\alpha }}||_{\infty } &\leq &\{1+C_{K}t\}e^{-\nu
_{0}t}||h_{0}||_{\infty }+\{\frac{C}{N}+C_{N,T_{0}}\varepsilon \}\sup_{s}e^{-%
\frac{\nu _{0}}{2}\{t-s\}}||h(s)||_{\infty }  \notag \\
&&+C_{\varepsilon ,N,T_{0}}\int_{0}^{t}||f(s)||ds.  \label{hm}
\end{eqnarray}

\textbf{STEP 2:} Estimate of $h(t,x,v).$ We now further plug (\ref{hm}) back
in: $h(t,x,v)=G(t,s)h_{0}+\int_{0}^{t}G(t,s_{1})\{K_{w}h(s_{1})\}ds_{1}$ to
get 
\begin{equation}
||h(t)||_{\infty }\leq e^{-\nu _{0}t}||h_{0}||_{\infty }+\int_{0}^{t}e^{-\nu
_{0}\{t-s_{1}\}}||K_{w}h||_{\infty }(s_{1})ds.  \label{hduhamel}
\end{equation}%
But $\{K_{w}h\}(s_{1},x,v)=\int K_{w}(v,v^{\prime })h(s_{1},x,v^{\prime
})dv^{\prime }$ and we split it as 
\begin{equation*}
\int K_{w}(v,v^{\prime })h(s_{1},x,v^{\prime })\{1-\mathbf{1}_{A_{\alpha
}(x,v^{\prime })}\}dv^{\prime }+\int K_{w}(v,v^{\prime })h(s_{1},x,v^{\prime
})\mathbf{1}_{A_{\alpha }(x,v^{\prime })}dv^{\prime }.
\end{equation*}%
By the definition of $A_{\alpha }$ in (\ref{aalpha}), the first term is
bounded by 
\begin{equation*}
\left( \int_{|v^{\prime }|\geq N,\text{ or }|v^{\prime }|\leq \frac{1}{N}%
}|K_{w}(v,v^{\prime })|dv^{\prime }+\int_{\alpha (x,v^{\prime })\leq \frac{1%
}{N}}|K_{w}(v,v^{\prime })|\right) ||h(s_{1})||_{\infty }.
\end{equation*}%
By (\ref{wk}) and (\ref{approximate}) if necessary, $\int_{|v^{\prime }|\geq
N,\text{ or }|v^{\prime }|\leq \frac{1}{N}}|K_{w}(v,v^{\prime })|dv^{\prime
}=o(1)$ as $N\rightarrow \infty .$ From $\alpha (x,v^{\prime })\leq \frac{1}{%
N}$, $\xi ^{2}(x)+[v^{\prime }\cdot $ $\nabla \xi (x)]^{2}\leq \frac{1}{N}.$
For $N$ large, $x\backsim \partial \Omega $ and $|\nabla \xi (x)|\geq \frac{1%
}{2}$ so that 
\begin{equation*}
\int_{\alpha (x,v^{\prime })\leq \frac{1}{N}}|K_{w}(v,v^{\prime
})|dv^{\prime }\leq \int_{|v^{\prime }\cdot \frac{\nabla \xi (x)}{|\nabla
\xi (x)|}|\leq \frac{2}{\sqrt{N}}}|K_{w}(v,v^{\prime })|dv^{\prime }=o(1)
\end{equation*}%
as $N\rightarrow \infty .$ We apply (\ref{hm}) to the second term to bound $%
||\{K_{w}h\}(s_{1})||_{\infty }$ as%
\begin{equation*}
\{1+C_{K}s_{1}\}e^{-\nu _{0}s_{1}}||h_{0}||_{\infty }+\{o(1)+\frac{C}{N}%
+C_{N,T_{0}}\varepsilon \}\sup_{s}e^{-\frac{\nu _{0}}{2}\{s_{1}-s%
\}}||h(s)||_{\infty }+C_{\varepsilon ,N,T_{0}\,}\int_{0}^{s_{1}}||f(s)||ds.
\end{equation*}%
Hence, by (\ref{hduhamel}), $||h(t)||_{\infty }$ is bounded by 
\begin{eqnarray*}
&&e^{-\nu _{0}t}||h_{0}||_{\infty }+\int_{0}^{t}e^{-\nu
_{0}t}\{1+C_{K}s_{1}\}||h_{0}||_{\infty }+ \\
&&+\int_{0}^{t}e^{-\nu _{0}\{t-s_{1}\}}\{o(1)+\frac{C}{N}+C_{N,T_{0}}%
\varepsilon \}\sup_{s\leq t}e^{-\frac{\nu _{0}}{2}\{s_{1}-s\}}||h(s)||_{%
\infty }+C_{\varepsilon ,N,T_{0}}\int_{0}^{s_{1}}||f(s)||ds\}ds_{1} \\
&\leq &\{1+C_{K}t^{2}\}e^{-\nu _{0}t}||h_{0}||_{\infty }+C\{o(1)+\frac{1}{N}%
+C_{N,T_{0}}\varepsilon \}\sup_{s\leq t}\{e^{-\frac{\nu _{0}}{4}%
\{t-s\}}||h(s)||_{\infty }\}+C_{\varepsilon ,N,T_{0}}\int_{0}^{t}||f(s)||ds.
\end{eqnarray*}%
We choose $T_{0}$ large such that $2\{1+C_{K}T_{0}^{2}\}e^{-\frac{\nu _{0}}{4%
}T_{0}}=e^{-\lambda T_{0}},$ for some $\lambda >0.$ We then further choose $%
N $ large, and then $\varepsilon $ sufficiently small such that $C\{o(1)+%
\frac{1}{N}+C_{N,T_{0}}\varepsilon \}<\frac{1}{2}.$ We there have\ 
\begin{equation*}
\sup_{0\leq s\leq t}\{e^{\frac{\nu _{0}}{4}s}||h(s)||_{\infty }\}\leq
2\{1+C_{K}t^{2}\}||h_{0}||_{\infty }+C_{T_{0}}\int_{0}^{t}||f(s)||ds.
\end{equation*}%
Choosing $s=t=T_{0},$ we deduce the finite-time estimate (\ref{finitetime}),
and our theorem follows from Lemma \ref{bootstrap}.
\end{proof}

\subsection{$L^{\infty }\,$Decay of Diffuse Reflection}

\subsubsection{Infinite Cycles and $L^{\infty }$ Bound for Diffuse $G(t)$}

In this section, we study the $L^{\infty }$ decay of the diffuse reflection.
Define $h=fw$ to satisfy 
\begin{eqnarray}
\{\partial _{t}+v\cdot \nabla _{x}+\nu \}h &=&0,\text{ }h(t,x,v)|_{\gamma
_{-}}=\frac{1}{\tilde{w}(v)}\int_{\mathcal{V}(x)}h(t,x,v^{\prime })\tilde{w}%
(v^{\prime })d\sigma ,  \label{hdiffuse} \\
\text{where }\mathcal{V}(x) &\mathcal{=}&\mathcal{\{}v^{\prime }\in \mathbf{R%
}^{3}:v^{\prime }\cdot n(x)>0\},\text{ }\tilde{w}(v)\equiv \frac{1}{w(v)%
\sqrt{\mu (v)}},\text{ \ }  \label{wtilde}
\end{eqnarray}%
and by (\ref{diffusenormal}), the probability measure $d\sigma =d\sigma (x)$
is given by 
\begin{equation}
d\sigma (x)=c_{\mu }\mu (v^{\prime })\{n(x)\cdot v^{\prime }\}dv^{\prime }.
\label{prob}
\end{equation}%
For $\frac{1}{4}-\theta >0$ and $\theta >\theta _{0},$ and for $\rho $
sufficiently small, 
\begin{equation}
\tilde{w}=\frac{e^{\{\frac{1}{4}-\theta \}|v|^{2}}}{(1+\rho |v|^{2})^{\beta }%
}\geq 1,\text{ \ }\int_{\mathcal{V}}\tilde{w}^{2}d\sigma \backsim c_{\mu
}\int_{\mathcal{V}}e^{-2\theta |v|^{2}}\{n(x)\cdot v\}dv=\frac{1}{16\theta
^{2}}<\frac{1}{16\theta _{0}^{2}}.  \label{w>1}
\end{equation}%
We have used the normalization (\ref{diffusenormal}) and a change of
variable $v=\sqrt{\frac{1}{4\theta }}v^{\prime }$ to evaluate the integral.

\begin{definition}
\label{diffusecycles}Fix any point $(t,x,v)\notin \gamma _{0},$ and let $%
(t_{0},x_{0},v_{0})=(t,x,v)$. Define the back-time cycle as 
\begin{equation}
(t_{k+1},x_{k+1},v_{k+1})=(t_{k}-t_{\mathbf{b}}(x_{k},v_{k}),x_{\mathbf{b}%
}(x_{k},v_{k}),v_{k+1}),\text{ for }v_{k+1}\in \mathcal{V}%
_{k+1}=\{v_{k+1}\cdot n(x_{k+1})>0\}.  \label{diffusecycle}
\end{equation}%
And 
\begin{equation*}
X_{\mathbf{cl}}(s;t,x,v)=\sum_{k}\mathbf{1}_{[t_{k+1},t_{k})}(s)%
\{x_{k}+(s-t_{k})v_{k}\},\text{ \ }V_{\mathbf{cl}}(s;t,x,v)=\sum_{k}\mathbf{1%
}_{[t_{k+1},t_{k})}(s)v_{k}.
\end{equation*}%
We define the iterated integral for $k\geq 2$ 
\begin{equation}
\int_{\Pi _{l=1}^{k-1}\mathcal{V}_{l}}\Pi _{l=1}^{k-1}d\sigma _{l}\equiv
\int_{\mathcal{V}_{1}}...\left\{ \int_{\mathcal{V}_{k-1}}d\sigma
_{k-1}\right\} d\sigma _{1}  \label{sigma}
\end{equation}
\end{definition}

We note that each $v_{l}$ ($l=1,2,...$) are all independent variables,
however, the phase space $\mathcal{V}_{l}$ implicitly depends on $%
(t,x,v,v_{1},v_{2},...v_{l-1}).$ We first show that the set in the phase
space $\Pi _{l=1}^{k-1}\mathcal{V}_{l}$ not reaching $t=0$ after $k$ bounces
is small when $k$ is large.

\begin{lemma}
\label{small}For any $\varepsilon >0,$ there exists $k_{0}(\varepsilon
,T_{0})$ such that for $k\geq k_{0},$ for all $(t,x,v),0\leq t\leq T_{0},$ $%
x\in \bar{\Omega}$ and $v\in \mathbf{R}^{3},$%
\begin{equation*}
\int_{\Pi _{l=1}^{k-1}\mathcal{V}_{l}}\mathbf{1}_{\mathcal{\{}%
t_{k}(t,x,v,v_{1},v_{2}...,v_{k-1})>0\}}\Pi _{l=1}^{k-1}d\sigma _{l}\leq
\varepsilon .
\end{equation*}
\end{lemma}

\begin{proof}
Choosing $0<\delta <1$ sufficiently small$,$ we further define non-grazing
sets for $1\leq l\leq k-1$ as%
\begin{equation*}
\mathcal{V}_{l}^{\delta }=\{v_{l}\in \mathcal{V}_{l}:\text{ }v_{l}\cdot
n(x_{l})\geq \delta \}\cap \{v_{l}\in \mathcal{V}_{l}:\text{ }|v_{l}|\leq 
\frac{1}{\delta }\}.
\end{equation*}%
Clearly, by the same argument in (\ref{nvsmall}), 
\begin{equation}
\int_{\mathcal{V}_{l}\setminus \mathcal{V}_{l}^{\delta }}d\sigma _{l}\leq
\int_{v_{l}\cdot n(x_{l})\leq \delta }d\sigma _{l}+\int_{|v_{l}|\geq \frac{1%
}{\delta }}d\sigma _{l}\leq C\delta ,  \label{v-v}
\end{equation}%
where $C$ is independent of $l.$ On the other hand, if $v_{l}\in \mathcal{V}%
_{l}^{\delta },$ then from diffusive back-time cycle (\ref{diffusecycle}),
we have $x_{l}-x_{l+1}=(t_{l}-t_{l+1})v_{l}.$ By (\ref{tlower}) in Lemma \ref%
{huang}$,$ since $|v_{l}|\leq \frac{1}{\delta },$ and $v_{l}\cdot
n(x_{l})\geq \delta ,$ 
\begin{equation*}
(t_{l}-t_{l+1})\geq \frac{\delta ^{3}}{C_{\xi }}.
\end{equation*}%
Therefore, if $t_{k}(t,x,v,v_{1},v_{2}...,v_{k-1})>0,$ then there can be at
most $\left[ \frac{C_{\xi }T_{0}}{\delta ^{3}}\right] +1$ number of $v_{l}$ $%
\in \mathcal{V}_{l}^{\delta }$ for $1\leq l\leq k-1.$ We therefore have 
\begin{eqnarray*}
&&\int_{\mathcal{V}_{1}}...\left\{ \int_{\mathcal{V}_{k-1}}\mathbf{1}_{%
\mathcal{\{}t_{k}>0\}}d\sigma _{k-1}\right\} d\sigma _{k-2}...d\sigma _{1} \\
&\leq &\sum_{j=1}^{\left[ \frac{C_{\xi }T_{0}}{\delta ^{3}}\right]
+1}\int_{\{\text{There are exactly }j\text{ of }v_{l_{i}}\in \mathcal{V}%
_{l_{i}}^{\delta },\text{ and }k-1-j\text{ of }v_{l_{i}}\notin \mathcal{V}%
_{l_{i}}^{\delta }\}}\Pi _{l=1}^{k-1}d\sigma _{l} \\
&\leq &\sum_{j=1}^{\left[ \frac{C_{\xi }T_{0}}{\delta ^{3}}\right] +1}\binom{%
k-1}{j}|\sup_{l}\int_{\mathcal{V}_{l}^{\delta }}d\sigma _{l}|^{j}\left\{
\sup_{l}\int_{\mathcal{V}_{l}\setminus \mathcal{V}_{l}^{\delta }}d\sigma
_{l}\right\} ^{k-j-1}.
\end{eqnarray*}%
Since $d\sigma $ is a probability measure, $\int_{\mathcal{V}_{l}^{\delta
}}d\sigma _{l}\leq 1,$ and 
\begin{equation*}
\left\{ \int_{\mathcal{V}_{l}\setminus \mathcal{V}_{l}^{\delta }}d\sigma
_{l}\right\} ^{k-j-1}\leq \left\{ \int_{\mathcal{V}_{l}\setminus \mathcal{V}%
_{l}^{\delta }}d\sigma _{l}\right\} ^{k-2-\left[ \frac{C_{\xi }T_{0}}{\delta
^{3}}\right] }\leq \{C\delta \}^{k-2-\left[ \frac{C_{\xi }T_{0}}{\delta ^{3}}%
\right] }.
\end{equation*}%
But $\binom{k-1}{j}\leq \{k-1\}^{j}\leq \{k-1\}^{\left[ \frac{C_{\xi }T_{0}}{%
\delta ^{3}}\right] +1},$ we deduce that 
\begin{equation*}
\int \mathbf{1}_{\{t_{k}>0\}}\Pi _{l=1}^{k-1}d\sigma _{l}\leq \{k-1\}^{\left[
\frac{C_{\xi }T_{0}}{\delta ^{3}}\right] +1}\{C\delta \}^{k-2-\left[ \frac{%
C_{\xi }T_{0}}{\delta ^{3}}\right] }.
\end{equation*}%
For $\varepsilon >0,$ our lemma follows for $C\delta <1,$ and $k>>\left[ 
\frac{C_{\xi }T_{0}}{\delta ^{3}}\right] +1.$
\end{proof}

\begin{lemma}
Assume that $h,\frac{q}{\nu }\in L^{\infty }$ satisfy $\{\partial
_{t}+v\cdot \nabla _{x}+\nu \}h=q(t,x,v),$ with the diffuse boundary
condition (\ref{hdiffuse}). Recall the diffusive cycles in (\ref%
{diffusecycle}). Then for any $0\leq s\leq t,$ for almost every $x,v$, if $%
t_{1}(t,x,v)\leq s,$ 
\begin{equation}
h(t,x,v)=e^{\nu (s-t)}h(s,x-v(t-s),v)+\int_{s}^{t}e^{\nu (\tau -t)}q(\tau
,x-v(t-\tau ),v)d\tau ;  \label{t1le0}
\end{equation}%
If $t_{1}(t,x,v)>s,$ then for $k\geq 2,$ 
\begin{equation*}
h(t,x,v)=\int_{t_{1}}^{t}e^{\nu (\tau -t)}q(\tau ,x-v(t-\tau ),v)d\tau +%
\frac{e^{\nu (v)(t_{1}-t)}}{\tilde{w}(v)}\int_{\Pi _{j=1}^{k-1}\mathcal{V}%
_{j}}H
\end{equation*}%
where $H$ is given by 
\begin{eqnarray}
&&\sum_{l=1}^{k-1}\mathbf{1}_{\{t_{l}>s,t_{l+1}\leq
s\}}h(s,x_{l}+(s-t_{l})v_{l},v_{l})d\Sigma _{l}(s)  \notag \\
&&+\sum_{l=1}^{k-1}\int_{s}^{t_{l}}\mathbf{1}_{\{t_{l}>s,t_{l+1}\leq
s\}}q(\tau ,x_{l}+(\tau -t_{l})v_{l},v_{l})d\Sigma _{l}(\tau )d\tau  \notag
\\
&&+\sum_{l=1}^{k-1}\int_{t_{l+1}}^{t_{l}}\mathbf{1}_{\{t_{l+1}>s\}}q(\tau
,x_{l}+(\tau -t_{l})v_{l},v_{l})d\Sigma _{l}(\tau )d\tau  \notag \\
&&+\mathbf{1}_{\{t_{k}>s\}}h(t_{k},x_{k},v_{k-1})d\Sigma _{k-1}(t_{k}),
\label{diffuseformula}
\end{eqnarray}%
and $d\Sigma _{k-1}(t_{k})$ is evaluated at $s=t_{k}$ of 
\begin{equation}
d\Sigma _{l}(s)=\{\Pi _{j=l+1}^{k-1}d\sigma _{j}\}\{e^{\nu (v_{l})(s-t_{l})}%
\tilde{w}(v_{l})d\sigma _{l}\}\Pi _{j=1}^{l-1}\{e^{\nu
(v_{j})(t_{j+1}-t_{j})}d\sigma _{j}\}.  \label{dsigma}
\end{equation}
\end{lemma}

\begin{proof}
When $k=2,$ if $t_{1}(t,x,v)\leq s,$ then (\ref{t1le0}) is clearly valid. If 
$t_{1}(t,x,v)>s,$ 
\begin{equation}
h(t,x,v)\mathbf{1}_{\{t_{1}>s\}}=e^{\nu
(v)(t_{1}-t)}h(t_{1},x_{1},v)+\int_{t_{1}}^{t}e^{\nu (v)(\tau -t)}q(\tau
,x+(\tau -t)v,v)d\tau .
\end{equation}%
Since $\frac{d\{e^{\nu s}h\}}{ds}=e^{\nu s}q$ along a trajectory $\frac{dx}{%
dt}=v,\frac{dv}{dt}=0$ almost everywhere, the first term can be expressed
(almost everywhere) by the diffuse boundary condition (\ref{hdiffuse}) and
part 4 of Lemma \ref{huang} as 
\begin{eqnarray*}
&&\frac{e^{\nu (v)(t_{1}-t)}}{\tilde{w}(v)}\int_{\mathcal{V}%
_{1}}h(t_{1},x_{1},v_{1})\tilde{w}(v_{1})d\sigma _{1} \\
&=&\frac{e^{\nu (v)(t_{1}-t)}}{\tilde{w}(v)}\int_{\mathcal{V}_{1}}\mathbf{1}%
_{\{t_{1}>s,t_{2}\leq s\}}e^{\nu
(v_{1})(s-t_{1})}h(s,x_{1}+v_{1}(s-t_{1}),v_{1})\tilde{w}(v_{1})d\sigma _{1}
\\
&&+\frac{e^{\nu (v)(t_{1}-t)}}{\tilde{w}(v)}\int_{s}^{t_{1}}\int_{\mathcal{V}%
_{1}}\mathbf{1}_{\{t_{1}>s,t_{2}\leq s\}}e^{\nu (v_{1})(\tau -t_{1})}q(\tau
,x_{1}+v_{1}(\tau -t_{1}),v_{1})\tilde{w}(v_{1})d\sigma _{1}d\tau \\
&&+\frac{e^{\nu (v)(t_{1}-t)}}{\tilde{w}(v)}\int_{\mathcal{V}_{1}}\mathbf{1}%
_{\{t_{2}>s\}}e^{\nu
(v_{1})\{t_{2}-t_{1}\}}h(t_{2},x_{1}+v_{1}(t_{2}-t_{1}),v_{1})\tilde{w}%
(v_{1})d\sigma _{1} \\
&&+\frac{e^{\nu (v)(t_{1}-t)}}{\tilde{w}(v)}\int_{t_{2}}^{t_{1}}\int_{%
\mathcal{V}_{1}}\mathbf{1}_{\{t_{2}>s\}}e^{\nu (v_{1})(\tau -t_{1})}q(\tau
,x_{1}+v_{1}(\tau -t_{1}),v_{1})\tilde{w}(v_{1})d\sigma _{1}d\tau .
\end{eqnarray*}%
Therefore, the formula (\ref{diffuseformula}) is valid for $k=2.$ Assume
that (\ref{diffuseformula}) is valid for $k\geq 2$, then for $k+1,$ we
further split the last term in (\ref{diffuseformula}) with $t_{k}>0$ into 
\begin{equation*}
h(t_{k},x_{k},v_{k-1})\tilde{w}(v_{k-1})=\int_{\mathcal{V}%
_{k}}h(t_{k},x_{k},v_{k})\tilde{w}(v_{k})d\sigma _{k}=\int_{\mathcal{V}_{k}}%
\mathbf{1}_{\{t_{k}>s,t_{k+1}\leq s\}}+\mathbf{1}_{\{t_{k+1}>s\}}.
\end{equation*}%
For the first term, we further integrate along the characteristics $\frac{dx%
}{dt}=v,\frac{dv}{dt}=0$ to reach the plane $t=s$ as 
\begin{equation*}
\int_{\mathcal{V}_{k}}\mathbf{1}_{t_{k+1}\leq s<t_{k}}\{e^{\nu
(v_{k})(s-t_{k})}h(s,x_{k}+(s-t_{k})v_{k},v_{k})+\int_{s}^{t_{k}}e^{\nu
(v_{k})(\tau -t_{k})}q(\tau ,x_{k}+(\tau -t_{k})v_{k},v_{k})d\tau \}\tilde{w}%
(v_{k})d\sigma _{k};
\end{equation*}%
For the second term, we integrate along the characteristics $\frac{dx}{dt}=v,%
\frac{dv}{dt}=0$ to $t=t_{k+1}>s$ to get 
\begin{equation*}
\int_{\mathcal{V}_{k}}\mathbf{1}_{t_{k+1}>s}\{e^{\nu
(v_{k})(t_{k+1}-t_{k})}h(t_{k+1},x_{k+1},v_{k})+\int_{t_{k+1}}^{t_{k}}e^{\nu
(v_{k})(\tau -t_{k})}q(\tau ,x_{k}+(\tau -t_{k})v_{k},v_{k})d\tau \}\tilde{w}%
(v_{k})d\sigma _{k}.
\end{equation*}%
We then deduce our lemma by adding $\int_{\mathcal{V}_{k}}d\sigma _{k}=1$
inside the rest of the terms so that all the integrations are over $\Pi
_{l=1}^{k}\mathcal{V}_{l}$ instead of $\Pi _{l=1}^{k-1}\mathcal{V}_{l}.$
\end{proof}

\begin{lemma}
\label{diffuseg0}Let $h_{0}\in L^{\infty }$ and\ assume (\ref{wdiffuse})$.$%
There exits a unique solution $h(t)=G(t)h_{0}\in L^{\infty }$ to (\ref%
{transport}) with the diffuse boundary condition (\ref{hdiffuse}) and 
\begin{equation}
\sup_{0\leq t\leq 1}\{e^{\nu _{0}t}||\{G(t)h_{0}\}\mathbf{1}%
_{t_{1}>0}||_{\infty }\}\leq e^{\frac{\nu _{0}}{2}}||\frac{h_{0}}{\tilde{w}}%
||_{\infty },\text{ \ \ \ }||\{G(t)h_{0}\}\mathbf{1}_{t_{1}\leq 0}||_{\infty
}\leq ||e^{-t\nu (v)}h_{0}||_{\infty }  \label{diffuserate}
\end{equation}%
In particular, 
\begin{equation}
\sup_{t\geq 1}e^{\frac{\nu _{0}}{2}t}||G(t)h_{0}||_{\infty }\leq e^{\nu
_{0}}\max \{||\frac{h_{0}}{\tilde{w}}||_{\infty },||e^{-\nu (v)+\nu
_{0}}h_{0}||_{\infty }\}.  \label{t>1}
\end{equation}
\end{lemma}

\begin{proof}
Given any $m\geq 1,$ we first construct a solution to $\{\partial
_{t}+v\cdot \nabla _{x}+\nu \}h^{m}=0,$ with the following approximate
boundary and initial conditions as%
\begin{eqnarray}
h^{m}(t,x,v) &=&\left\{ 1-\frac{1}{m}\right\} \frac{1}{\tilde{w}(v)}\int_{%
\mathcal{V}}h^{m}(t,x,v^{\prime })\tilde{w}(v^{\prime })d\sigma (x),
\label{diffusecutoff} \\
h^{m}(0,x,v) &=&h_{0}\mathbf{1}_{\{|v|\leq m\}}.  \notag
\end{eqnarray}%
Then $\tilde{h}^{m}\equiv h^{m}\tilde{w}$ satisfies $\{\partial _{t}+v\cdot
\nabla _{x}+\nu \}\tilde{h}^{m}=0$ but with 
\begin{equation}
\tilde{h}^{m}(t,x,v)=\left\{ 1-\frac{1}{m}\right\} \int_{\mathcal{V}}\tilde{h%
}^{m}(t,x,v^{\prime })d\sigma (x).  \label{1-1/m}
\end{equation}%
Clearly, since $\int d\sigma =1,$ this boundary operator maps $L^{\infty }$
to $L^{\infty }$ with norm bounded by $1-\frac{1}{m},$ and initially 
\begin{equation*}
||\tilde{h}^{m}(0)||_{\infty }=\sup |h^{m}(0,x,v)\tilde{w}|=||h_{0}\mathbf{1}%
_{\{|v|\leq m\}}\tilde{w}||_{\infty }\leq C_{m,\theta }||h_{0}||_{\infty
}<\infty .
\end{equation*}%
Therefore, by Lemma \ref{abstract}, there exists a solution $\tilde{h}%
^{m}(t,x,v)\in L^{\infty }$ to (\ref{transport}) with (\ref{1-1/m}), so that
we have constructed $h^{m}=$ $\frac{\tilde{h}^{m}}{\tilde{w}}$ with (\ref%
{diffusecutoff})$,$ which obviously is bounded. Such a solution is unique by
the transformation $f^{m}=\frac{h^{m}}{w}\in L^{2}$ with $%
\int_{0}^{t}||f^{m}(s)||_{\gamma }^{2}ds<\infty .$

In order to take $m\rightarrow \infty $ in (\ref{diffusecutoff})$,$ we need
to obtain an uniform $L^{\infty }$ bound (\ref{diffuserate}) and (\ref{t>1})
for $h^{m},$ which is more delicate. We first claim that it suffices to show
(\ref{diffuserate}) to derive (\ref{t>1}) for $h^{m}$. Letting $t=1$ in two
parts of (\ref{diffuserate}), \ since $\tilde{w}\geq 1$ from (\ref{w>1}) and 
$e^{-\frac{\nu _{0}}{2}}e^{\nu _{0}}\geq 1,$ we have 
\begin{equation}
||h^{m}(1)||_{\infty }\leq e^{-\frac{\nu _{0}}{2}}\max \{||h(0)||_{\infty
},||e^{-\nu (v)+\nu _{0}}h(0)||_{\infty }\}\leq e^{-\frac{\nu _{0}}{2}%
}||h(0)||_{\infty }.  \label{01decay}
\end{equation}%
For any $l\leq t<l+1,$ we can repeatedly apply (\ref{01decay}) to get:%
\begin{equation*}
||h^{m}(l)||_{\infty }\leq e^{-\frac{\nu _{0}}{2}}||h^{m}(l-1)||_{\infty }.
\end{equation*}%
Therefore, by (\ref{diffuserate}), we deduce (\ref{t>1}) as%
\begin{eqnarray*}
||h^{m}(t)||_{\infty } &\leq &e^{\frac{\nu _{0}}{2}}||h^{m}(l)||_{\infty
}\leq ||h^{m}(l-1)||_{\infty } \\
&\leq &e^{-\frac{\nu _{0}(l-2)}{2}}||h^{m}(1)||_{\infty } \\
&\leq &e^{-\frac{\nu _{0}(l-1)}{2}}\max \{||\frac{h(0)}{\tilde{w}}||_{\infty
},||e^{-\nu (v)+\nu _{0}}h(0)||_{\infty }\} \\
&\leq &e^{\nu _{0}}e^{-\frac{\nu _{0}}{2}t}\max \{||\frac{h(0)}{\tilde{w}}%
||_{\infty },||e^{-\nu (v)+\nu _{0}}h(0)||_{\infty }\}.
\end{eqnarray*}

The rest of the proof is devoted to the validity of (\ref{diffuserate}). If $%
t_{1}(t,x,v)\leq 0,$ we have $\{G(t)h_{0}^{m}\}(t,x,v)=e^{-\nu
(v)t}h_{0}^{m}(x-tv,v)$ and (\ref{diffuserate}) is clearly valid.

We now consider $t_{1}(t,x,v)>0,$ then the back-time trajectory first hits
the boundary. As in the proof of (\ref{diffuseformula}) with $q\equiv 0$, we
can ignore powers of the factor $1-\frac{1}{m}$ to bound $|h^{m}(t,x,v)|$ by 
$\frac{e^{\nu (v)(t_{1}-t)}}{\tilde{w}}\times $%
\begin{eqnarray}
&&\sum_{l=1}^{k-1}\int_{\Pi _{l=1}^{k-1}\mathcal{V}_{l}}\mathbf{1}%
_{\{t_{l}>0,t_{l+1}\leq 0\}}|h^{m}(0,x_{l}-t_{l}v_{l},v_{l})|d\Sigma _{l}(0)
\notag \\
&&+\int_{\Pi _{l=1}^{k-1}\mathcal{V}_{l}}\mathbf{1}_{\{t_{k}>0%
\}}|h^{m}(t_{k},x_{k},v_{k-1})|d\Sigma _{k-1}(t_{k}).  \label{hmformula}
\end{eqnarray}

Over the second small set $\mathbf{1}_{\{t_{k}>0\}}$, choose any $%
\varepsilon (\nu _{0})>0$ such that 
\begin{equation}
(1-2\sqrt{\varepsilon })e^{\frac{\nu _{0}}{2}}>1\text{,\ \ \ \ }  \label{emu}
\end{equation}%
then choose $k_{0}(\varepsilon )$ by Lemma \ref{small} with $T_{0}=1$. By
Lemma \ref{small} with $T_{0}=1$, for $k=k_{0}(\varepsilon )$, $\int_{\Pi
_{l=1}^{k-1}\mathcal{V}_{l}}\mathbf{1}_{\{t_{k}>0\}}\Pi _{l=1}^{k-1}d\sigma
_{l}<\varepsilon .$ We further split the second integral in (\ref{hmformula}%
) into $\{t_{k}>0,t_{k+1}\leq 0\}$ and $\{t_{k+1}>0\}$ in $\mathcal{V}_{k}$
with $d\sigma _{k}.$ Integrating along the characteristic for the first part 
$\{t_{k}>0,t_{k+1}\leq 0\}$ yields: $\frac{e^{\nu (v)(t_{1}-t)}}{\tilde{w}}%
\times $ 
\begin{eqnarray}
&&\int_{\Pi _{l=1}^{k}\mathcal{V}_{l}}\mathbf{1}_{\{t_{k}>0,t_{k+1}\leq
0\}}|h^{m}(0,x_{k}-v_{k}t_{k},v_{k})|d\Sigma _{k}(0)  \notag \\
&&+\int_{\Pi _{l=1}^{k}\mathcal{V}_{l}}\mathbf{1}_{\{t_{k+1}>0%
\}}|h^{m}(t_{k},x_{k},v_{k-1})|d\Sigma _{k-1}(t_{k})d\sigma _{k}.  \label{tk}
\end{eqnarray}%
Since $t_{1}(t_{k},x_{k},v_{k})>0$ over the set $\{t_{k+1}>0\},$ we deduce
that 
\begin{equation*}
\mathbf{1}_{\{t_{k+1}>0\}}|h^{m}(t_{k},x_{k},v_{k-1})|\leq
\sup_{x,v}|h^{m}(t_{k},x,v)\mathbf{1}_{\{t_{1}>0\}}|.
\end{equation*}%
From (\ref{dsigma}), the exponential in $d\Sigma _{l}(s)$ is bounded by $%
e^{\nu _{0}(s-t_{1})}.$ Since $\int_{\mathcal{V}_{k}}d\sigma _{k}=1,$ by
Lemma \ref{small}, the last part in (\ref{tk}) is then bounded by 
\begin{eqnarray}
&&\frac{e^{\nu _{0}(t_{k}-t)}}{\tilde{w}}||h^{m}(t_{k})\mathbf{1}%
_{t_{1}>0}||_{\infty }\int_{\Pi _{l=1}^{k-1}\mathcal{V}_{l}}\mathbf{1}%
_{\{t_{k}>0\}}\tilde{w}(v_{k-1})\Pi _{l=1}^{k-1}d\sigma _{l}  \label{smallhm}
\\
&\leq &\sup_{0\leq s\leq t\leq 1}\{e^{\nu _{0}(s-t)}||h^{m}(s)\mathbf{1}%
_{t_{1}>0}||_{\infty }\}\left\{ \int_{\Pi _{l=1}^{k-1}\mathcal{V}_{l}}%
\mathbf{1}_{\{t_{k}>0\}}\Pi _{l=1}^{k-1}d\sigma _{l}\right\} ^{1/2}\left\{
\int \tilde{w}^{2}(v_{k-1})d\sigma _{k-1}\right\} ^{1/2}  \notag \\
&\leq &2\sqrt{\varepsilon }\sup_{0\leq s\leq t\leq 1}\{e^{\nu
_{0}(s-t)}||h^{m}(s)\mathbf{1}_{t_{1}>0}||_{\infty }\},  \notag
\end{eqnarray}%
We have used (\ref{w>1}) for $\theta <\theta _{0}\backsim \frac{1}{4}$.

On the other hand, inserting $\int_{\mathcal{V}_{k}}d\sigma _{k}=1$ into the
main contribution in (\ref{hmformula}), and combining with the first term in
(\ref{tk}) yields: 
\begin{eqnarray*}
&&\frac{e^{\nu (v)(t_{1}-t)}}{\tilde{w}}\int_{\Pi _{l=1}^{k}\mathcal{V}%
_{l}}\sum_{l=1}^{k}\mathbf{1}_{\{t_{l}>0,t_{l+1}\leq
0\}}|h^{m}(0,x_{l}-t_{l}v_{l},v_{l})| \\
&&\times \{\Pi _{j=l+1}^{k}d\sigma _{j}\}\{\tilde{w}(v_{l})e^{-\nu
(v_{l})t_{l}}d\sigma _{l}\}\Pi _{j=1}^{l-1}\{e^{\nu
(v_{j})(t_{j+1}-t_{j})}d\sigma _{j}\} \\
&\leq &\frac{e^{-\nu _{0}t}}{\tilde{w}}||\frac{h_{0}^{m}}{\tilde{w}}%
||_{\infty }\int \sum_{l=1}^{k}\mathbf{1}_{\{t_{l}>0,t_{l+1}\leq 0\}}\Pi
_{j=1}^{k}\tilde{w}^{2}(v_{j})d\sigma _{j},
\end{eqnarray*}%
since $\tilde{w}(v_{j})\geq 1$ by (\ref{wdiffuse}) and (\ref{w>1})$.$ Note $%
\sum_{l=1}^{k}\mathbf{1}_{\{t_{l}>0,t_{l+1}\leq 0\}}=\mathbf{1}%
_{\{t_{k+1}\leq 0\}}.$ By (\ref{emu}), we can further choose $\theta _{0}$
near $1/4~$such that $\left( \frac{1}{4\theta _{0}}\right)
^{2k_{0}(\varepsilon )}\leq (1-2\sqrt{\varepsilon })e^{\frac{\nu _{0}}{2}}$
in (\ref{w>1}). We therefore get 
\begin{eqnarray*}
&&\int_{\Pi _{l=1}^{k}\mathcal{V}_{l}}\sum_{l=1}^{k}\mathbf{1}%
_{\{t_{l}>0,t_{l+1}\leq 0\}}\Pi _{l=1}^{k}\tilde{w}^{2}(v_{l})d\sigma _{l} \\
&=&\int_{\Pi _{l=1}^{k}\mathcal{V}_{l}}\mathbf{1}_{\{t_{k+1}\leq 0\}}\Pi
_{l=1}^{k}\tilde{w}^{2}(v_{l})d\sigma _{l} \\
&\leq &\Pi _{l=1}^{k}\left\{ \int_{\mathcal{V}_{l}}\tilde{w}%
^{2}(v_{l})d\sigma _{l}\right\} \leq (1-2\sqrt{\varepsilon })e^{\frac{\nu
_{0}}{2}}.
\end{eqnarray*}%
for $k=k_{0}(\varepsilon )$. Hence, combining with (\ref{smallhm}), we have 
\begin{equation*}
\sup_{x,v}\{e^{\nu _{0}t}|h_{m}(t,x,v)\mathbf{1}_{\{t_{1}>0\}}|\}\leq 2\sqrt{%
\varepsilon }\sup_{0\leq s\leq 1}\{e^{\nu _{0}s}||h^{m}(s)\mathbf{1}%
_{t_{1}>0}||_{\infty }\}+(1-2\sqrt{\varepsilon })e^{\frac{\nu _{0}}{2}}||%
\frac{h_{0}^{m}}{\tilde{w}}||_{\infty }.
\end{equation*}%
Taking $\sup_{0\leq t\leq 1}$ and absorbing the first term on the right hand
side$,$ we obtain%
\begin{equation*}
\sup_{0\leq t\leq 1}e^{\nu _{0}t}||h_{m}(t)1_{\{t_{1}>0\}}||_{\infty }\leq
e^{\frac{\nu _{0}}{2}}||\frac{h_{0}^{m}}{\tilde{w}}||_{\infty }.
\end{equation*}%
Letting $t=1,$ and we deduce (\ref{diffuserate}) uniform in $m.$ We then
deduce our lemma by letting $m\rightarrow \infty .$
\end{proof}

\subsubsection{Continuity of Diffuse $G(t).$}

\begin{lemma}
\label{diffusegcon}Let $\Omega $ be strictly convex as in (\ref{convexity})
and (\ref{wdiffuse}) be valid. Let $h$ and $q$ satisfy $\{\partial
_{t}+v\cdot \nabla _{x}+\nu \}h=q(t,x,v),$ with the diffuse boundary
condition (\ref{hdiffuse}). Assume $h(0,x,v)=h_{0}(x,v),$ continuous for $%
(x,v)\notin \gamma _{0},$ and $q(t,x,v)$ is continuous in the interior of $%
[0,\infty ]\times \Omega \times \mathbf{R}^{3}$ with $\sup_{[0,\infty
]\times \Omega \times \mathbf{R}^{3}}|\frac{q(t,x,v)}{\nu (v)}|<\infty .$
Assume 
\begin{equation}
h_{0}(x,v)|_{\gamma _{-}}=\frac{1}{\tilde{w}(v)}\int_{\mathcal{V}%
}h_{0}(x,v^{\prime })\tilde{w}(v^{\prime })d\sigma .  \label{diffusecom}
\end{equation}%
Then for any $t,$ $(x,v)\notin \gamma _{0},$ $h(t,x,v)$ is continuous.
\end{lemma}

\begin{proof}
We fix $(t,x,v)$ such that $(x,v)\notin \gamma _{0},$ for any fixed $k,$ we
recall (\ref{diffuseformula}) with $s=0$ for the expression of $h(t,x,v).$
Now we take any point $(\bar{t},\bar{x},\bar{v})$ near $(t,x,v)$ and
evaluate $h(\bar{t},\bar{x},\bar{v})$ by (\ref{diffuseformula}) with the
same number of bounces ($k-$bounces), and with corresponding $\bar{t}_{l},%
\bar{x}_{l}$ and $\mathcal{\bar{V}}_{l}$ $\ $and $d\bar{\sigma}_{l}.$

\textbf{Step 1. Reduction to the approximate of phase spaces. }Since $\nu
(v)\backsim |v|$ for large $v$ and $\frac{1}{\tilde{w}}$ decays
exponentially, by Lemma \ref{diffuseg0} and the Duhamel principle, 
\begin{equation}
||h(t)||_{\infty }\leq ||G(t)h_{0}||_{\infty
}+\int_{0}^{t}||G(t-s)q(s)||_{\infty }ds\leq C(t,||h_{0}||_{\infty
},\sup_{[0,\infty ]\times \Omega \times \mathbf{R}^{3}}|\frac{q}{\nu }|),
\label{ht}
\end{equation}%
For any $\varepsilon >0,$ by Lemma \ref{small} and (\ref{ht}) , we can fix $%
k(\varepsilon ,t)$ sufficiently large, such that the last terms in (\ref%
{diffuseformula}) for both $h(t,x,v)$ and $h(\bar{t},\bar{x},\bar{v})$ are
bounded by 
\begin{equation*}
\{||h(t_{k})||_{\infty }+||h(\bar{t}_{k})||\}\int_{\Pi _{l=1}^{k-1}\mathcal{V%
}_{l}}\mathbf{1}_{\{t_{k}>0\}}\Pi _{l=1}^{k-1}d\sigma _{l}+\int_{\Pi
_{l=1}^{k-1}\mathcal{\bar{V}}_{l}}\mathbf{1}_{\{\bar{t}_{k}>0\}}\Pi
_{l=1}^{k-1}d\bar{\sigma}_{l}\leq \frac{\varepsilon }{2}.
\end{equation*}%
For the remaining sets $\mathbf{1}_{\{t_{k}\leq 0\}}$ and $\mathbf{1}_{\{%
\bar{t}_{k}\leq 0\}},$ for $\varepsilon _{1}<<\varepsilon ,$ define the
non-grazing sets as 
\begin{eqnarray*}
\mathcal{V}_{l}^{\varepsilon _{1}} &=&\{v_{l}:v_{l}\cdot n(x_{l})\geq
\varepsilon _{1}\text{ and }|v_{l}|\leq \frac{1}{\varepsilon _{1}}\}, \\
\mathcal{\bar{V}}_{l}^{\varepsilon _{1}} &=&\{v_{l}:v_{l}\cdot n(\bar{x}%
_{l})\geq \varepsilon _{1}\text{ and }|v_{l}|\leq \frac{1}{\varepsilon _{1}}%
\}.
\end{eqnarray*}%
We further split the integration region in (\ref{diffuseformula}) as 
\begin{eqnarray*}
\int \mathbf{1}_{\{t_{k}\leq 0\}} &=&\int_{\{\text{there exists a }v_{l}\in 
\mathcal{V}_{l}\setminus \mathcal{V}_{l}^{\varepsilon _{1}}\}}\mathbf{1}%
_{\{t_{k}\leq 0\}}+\int_{\mathcal{\{}\text{all }v_{l}\in \mathcal{V}%
_{l}^{\varepsilon _{1}}\}}\mathbf{1}_{\{t_{k}\leq 0\}}; \\
\int \mathbf{1}_{\{\bar{t}_{k}\leq 0\}} &=&\int_{\{\text{there exists a }%
v_{l}\in \mathcal{\bar{V}}_{l}\setminus \mathcal{\bar{V}}_{l}^{\varepsilon
_{1}}\}}\mathbf{1}_{\{\bar{t}_{k}\leq 0\}}+\int_{\mathcal{\{}\text{all }%
v_{l}\in \mathcal{\bar{V}}_{l}^{\varepsilon _{1}}\}}\mathbf{1}_{\{\bar{t}%
_{k}\leq 0\}}.
\end{eqnarray*}%
Clearly, by (\ref{v-v}), $\int_{\mathcal{V}_{l}\setminus \mathcal{V}%
_{l}^{\varepsilon _{1}}}d\sigma _{l}+\int_{\mathcal{\bar{V}}_{l}\setminus 
\mathcal{\bar{V}}_{l}^{\varepsilon _{1}}}d\bar{\sigma}_{l}\leq C\varepsilon
_{1},$ so that from the boundedness of $h_{0}$ and $\frac{q}{\nu },$ the
integrals in (\ref{diffuseformula}) over the almost grazing sets are small: 
\begin{equation*}
\left\vert \int_{\{\text{there exists }v_{l}\in \mathcal{V}_{l}\setminus 
\mathcal{V}_{l}^{\varepsilon _{1}}\}}\mathbf{1}_{\{t_{k}\leq
0\}}...\right\vert +\left\vert \int_{\{\text{there exists }v_{l}\in \mathcal{%
\bar{V}}_{l}\setminus \mathcal{\bar{V}}_{l}^{\varepsilon _{1}}\}}\mathbf{1}%
_{\{\bar{t}_{k}\leq 0\}}...\right\vert \leq C(h_{0},\frac{q}{\nu }%
,k)\varepsilon _{1}\leq \frac{\varepsilon }{4}.
\end{equation*}

Therefore, it suffices to compare only the integrations over the non-grazing
sets $\Pi _{l=1}^{k-1}\mathcal{V}_{l}^{\varepsilon _{1}}\cap \mathcal{\{}%
t_{k}\leq 0\}$ and $\Pi _{l=1}^{k-1}\mathcal{\bar{V}}_{l}^{\varepsilon
_{1}}\cap \mathcal{\{}\bar{t}_{k}\leq 0\}$. For any $1\leq l\leq k-1,$
recalling the back-time diffuse cycle (\ref{diffusecycle}), if $\alpha
(t_{l})=\{v_{l}\cdot n(x_{l})\}^{2}\geq \frac{\varepsilon _{1}}{2}>0,$ then
from convexity and the Velocity Lemma \ref{velocity}, we deduce 
\begin{equation*}
\alpha (t_{l+1})=\{v_{l}\cdot n(x_{l+1})\}^{2}\geq C\alpha (t_{l})\geq C%
\frac{\varepsilon _{1}}{2}>0.
\end{equation*}%
Hence $x_{l+1},t_{l+1}$ are smooth functions of $x_{l}$ and $v_{l}$ from
Lemma \ref{huang}. A simple induction for $l$ implies that $x_{l},t_{l}$ are
smooth functions of $(v_{1},...v_{l-1})\in \Pi _{j=1}^{l-1}\mathcal{V}_{j}^{%
\frac{\varepsilon _{1}}{2}}:$ 
\begin{equation}
|t_{l}-\bar{t}_{l}|+|x_{l}-\bar{x}_{l}|\rightarrow 0  \label{diffusecon}
\end{equation}%
as $(\bar{t},\bar{x},\bar{v})\rightarrow (t,x,v),$ uniformly in $\Pi
_{j=1}^{l-1}\mathcal{V}_{j}^{\frac{\varepsilon _{1}}{2}}.$ Clearly $\Pi
_{l=1}^{k-1}\mathcal{V}_{l}^{\varepsilon _{1}}\subset \Pi _{l=1}^{k-1}%
\mathcal{V}_{l}^{\frac{\varepsilon _{1}}{2}}.$ Since $\bar{x}_{1}\backsim
x_{1}$ by the Velocity Lemma \ref{velocity}, $\mathcal{\bar{V}}%
_{1}^{\varepsilon _{1}}\subset \mathcal{V}_{1}^{\frac{\varepsilon _{1}}{2}}.$
A simple induction leads to (\ref{diffusecon}) and $\Pi _{j=1}^{l-1}\mathcal{%
\bar{V}}_{j}^{\varepsilon _{1}}\subset \Pi _{j=1}^{l-1}\mathcal{V}_{j}^{%
\frac{\varepsilon _{1}}{2}}$ for \thinspace $1\leq l\leq k-1.$ Therefore, (%
\ref{diffusecon}) is valid on both $\Pi _{l=1}^{k-1}\mathcal{V}%
_{l}^{\varepsilon _{1}}$ and $\Pi _{l=1}^{k-1}\mathcal{\bar{V}}%
_{l}^{\varepsilon _{1}},$ subsets of $\Pi _{l=1}^{k-1}\mathcal{\bar{V}}_{l}^{%
\frac{\varepsilon _{1}}{2}}.$

Moreover, we have 
\begin{eqnarray*}
\mathcal{V}_{l}^{\varepsilon _{1}}\setminus \mathcal{\bar{V}}%
_{l}^{\varepsilon _{1}} &\equiv &\{v_{l}\cdot n(x_{l})\geq \varepsilon
_{1},v_{l}\cdot n(\bar{x}_{l})<\varepsilon _{1},\text{ and }|v_{l}|\leq 
\frac{1}{\varepsilon _{1}}\}. \\
\mathcal{\bar{V}}_{l}^{\varepsilon _{1}}\setminus \mathcal{V}%
_{l}^{\varepsilon _{1}} &\equiv &\{v_{l}\cdot n(x_{l})<\varepsilon
_{1},v_{l}\cdot n(\bar{x}_{l})\geq \varepsilon _{1},\text{ and }|v_{l}|\leq 
\frac{1}{\varepsilon _{1}}\}.
\end{eqnarray*}%
By continuity (\ref{diffusecon}), for $(\bar{t},\bar{x},\bar{v})\rightarrow
(t,x,v),$ $x_{l}\rightarrow \bar{x}_{l},$ and both sets are contained in 
\begin{equation*}
\{\varepsilon _{1}-C|x_{l}-\bar{x}_{l}|\leq v_{l}\cdot n(x_{l})\leq
\varepsilon _{1}+C|x_{l}-\bar{x}_{l}|,\text{ and }|v_{l}|\leq \frac{1}{%
\varepsilon _{1}}\}
\end{equation*}%
which have measure $\frac{C|x_{l}-\bar{x}_{l}|}{^{\varepsilon _{1}^{2}}}.$
We now define the approximate phase-spaces as: 
\begin{equation}
B=\Pi _{l=1}^{k-1}[\mathcal{V}_{l}^{\varepsilon _{1}}\cap \mathcal{\bar{V}}%
_{l}^{\varepsilon _{1}}]\cap \{t_{k}\leq 0,\bar{t}_{k}\leq 0\}.
\label{approximateb}
\end{equation}%
To estimate $\Pi _{l=1}^{k-1}\mathcal{V}_{l}^{\varepsilon _{1}}\setminus B$,
by an induction on $k,$ we get 
\begin{equation*}
|\Pi _{l=1}^{k-1}\mathcal{V}_{l}^{\varepsilon _{1}}\setminus \Pi
_{l=1}^{k-1}[\mathcal{V}_{l}^{\varepsilon _{1}}\cap \mathcal{\bar{V}}%
_{l}^{\varepsilon _{1}}]|+|\Pi _{l=1}^{k-1}\mathcal{V}_{l}^{\varepsilon
_{1}}\setminus \Pi _{l=1}^{k-1}\mathcal{\bar{V}}_{l}^{\varepsilon _{1}}|\leq
C(\varepsilon _{1},k)\sup_{1\leq l\leq k-1}|x_{l}-\bar{x}_{l}|.
\end{equation*}%
Notice that $\Pi _{l=1}^{k-1}\mathcal{V}_{l}^{\varepsilon _{1}}\subset
\lbrack \Pi _{l=1}^{k-1}\mathcal{V}_{l}^{\varepsilon _{1}}\setminus \Pi
_{l=1}^{k-1}\mathcal{\bar{V}}_{l}^{\varepsilon _{1}}]\cup \Pi _{l=1}^{k-1}%
\mathcal{\bar{V}}_{l}^{\varepsilon _{1}},$ $\Pi _{l=1}^{k-1}\mathcal{V}%
_{l}^{\varepsilon _{1}}\setminus B$ is contained in 
\begin{eqnarray*}
\Pi _{l=1}^{k-1}\mathcal{V}_{l}^{\varepsilon _{1}}\setminus \Pi _{l=1}^{k-1}[%
\mathcal{V}_{l}^{\varepsilon _{1}}\cap \mathcal{\bar{V}}_{l}^{\varepsilon
_{1}}]\cup \lbrack \Pi _{l=1}^{k-1}\mathcal{V}_{l}^{\varepsilon _{1}}\cap
\{t_{k} &>&0\}] \\
\cup \lbrack \Pi _{l=1}^{k-1}\mathcal{\bar{V}}_{l}^{\varepsilon _{1}}\cap \{%
\bar{t}_{k} &>&0\}]\cup \lbrack \Pi _{l=1}^{k-1}\mathcal{V}_{l}^{\varepsilon
_{1}}\setminus \Pi _{l=1}^{k-1}\mathcal{\bar{V}}_{l}^{\varepsilon _{1}}]
\end{eqnarray*}%
From Lemma \ref{small}, both $\int_{\Pi _{l=1}^{k-1}\mathcal{V}_{l}\cap
\{t_{k}>0\}}$ $\Pi _{l=1}^{k-1}d\sigma _{l}$ and $\int_{\Pi _{l=1}^{k-1}%
\mathcal{\bar{V}}_{l}\cap \{\bar{t}_{k}>0\}}$ $\Pi _{l=1}^{k-1}d\bar{\sigma}%
_{l}$ are bounded by $C\varepsilon .$ By similar splitting for the set $\Pi
_{l=1}^{k-1}\mathcal{\bar{V}}_{l}^{\varepsilon _{1}}\setminus B,$ as $|x-%
\bar{x}|+|t-\bar{t}|+|v-\bar{v}|\rightarrow 0,$ we deduce 
\begin{equation*}
\int_{\Pi _{l=1}^{k-1}\mathcal{V}_{l}^{\varepsilon _{1}}\setminus B}\Pi
_{l=1}^{k-1}d\sigma _{l}+\int_{\Pi _{l=1}^{k-1}\mathcal{\bar{V}}%
_{l}^{\varepsilon _{1}}\setminus B}\Pi _{l=1}^{k-1}d\bar{\sigma}%
_{l}<4C\varepsilon +C(\varepsilon _{1},k)\sup_{1\leq l\leq k-1}|x_{l}-\bar{x}%
_{l}|<5C\varepsilon .
\end{equation*}%
By $L^{2}$ bound for $\tilde{w}$ in (\ref{w>1}) and Cauchy-Schwarz's
inequality: 
\begin{equation*}
\int_{\Pi _{l=1}^{k-1}\mathcal{V}_{l}^{\varepsilon _{1}}\setminus B}\Pi
_{l=1}^{k-1}d\Sigma _{l}(s)+\int_{\Pi _{l=1}^{k-1}\mathcal{\bar{V}}%
_{l}^{\varepsilon _{1}}\setminus B}\Pi _{l=1}^{k-1}d\bar{\Sigma}_{l}(s)\leq
C(t)\sqrt{\varepsilon }.
\end{equation*}%
Thanks to the boundedness of $h_{0}$ and $\frac{q}{\nu },$ to prove the
continuity, it suffices to estimate the difference of $h(t,x,v)$ and $h(\bar{%
t},\bar{x},\bar{v})$ in (\ref{diffuseformula}), where the integrations are
over the same set $B.$

\textbf{Step 2. Continuity of }$h(t,x,v)$\textbf{\ over }$B.$

\textbf{Case 1: }\textit{\ }$t_{1}(t,x,v)\leq 0.$ In the case $t_{1}<0,$
then $\bar{t}_{1}<0$ by continuity over the set $B$ in (\ref{approximateb}).
Then both $h(t,x,v)$ and $h(\bar{t},\bar{x},\bar{v})$ are given by the same
formula (\ref{t1le0}). The continuity now follows from $(\bar{t},\bar{x},%
\bar{v})\rightarrow (t,x,v)$ and the continuity of $h_{0}$ and $q.$ Same
argument also applies to the situation $t_{1}=0$ and $t_{1}\leq 0.$

We only need to study the case $t_{1}=0$ but $\bar{t}_{1}>0$ in which $h(%
\bar{t},\bar{x},\bar{v})$ are given by the different expression (\ref%
{diffuseformula}). Over the set $B,$ since $|\bar{v}_{1}\cdot n(\bar{x}%
_{1})|\geq \varepsilon _{1}>0,$ from (\ref{tlower}) that $\bar{t}_{1}-\bar{t}%
_{2}\geq \frac{\varepsilon _{1}^{3}}{C_{\xi }}.$ But $\bar{t}_{1}\rightarrow
t_{1}=0,$ we therefore deduce that $\bar{t}_{2}<0.$ This implies for $k$
large 
\begin{equation*}
B=\{\bar{t}_{1}>0,\bar{t}_{2}\leq 0\}\cap B=\Pi _{l=1}^{k-1}\mathcal{V}%
_{l}^{\varepsilon _{1}}\cap \mathcal{\bar{V}}_{l}^{\varepsilon _{1}}.
\end{equation*}%
Applying (\ref{diffuseformula}) to $h(\bar{t},\bar{x},\bar{v})$ over the set 
$B$ with $\bar{t}_{2}<0,$ by Step 1, we obtain 
\begin{eqnarray}
h(\bar{t},\bar{x},\bar{v}) &\thicksim &\int_{\bar{t}_{1}}^{\bar{t}}e^{\nu
(\tau -\bar{t})}q(\tau ,\bar{x}-\bar{v}(\bar{t}-\tau ),\bar{v})d\tau + 
\notag \\
&&+\frac{e^{\nu (\bar{v})(\bar{t}_{1}-\bar{t})}}{\tilde{w}(\bar{v})}\int_{%
\mathcal{V}_{1}^{\varepsilon _{1}}\cap \mathcal{\bar{V}}_{1}^{\varepsilon
_{1}}}\mathbf{1}_{\{\bar{t}_{1}>0,\bar{t}_{2}\leq 0\}}h_{0}(\bar{x}_{1}-\bar{%
t}_{1}v_{1},v_{1})\tilde{w}(v_{1})e^{-\nu (v_{1})\bar{t}_{1}}d\bar{\sigma}%
_{1}  \label{barcase1} \\
&&+\frac{e^{\nu (\bar{v})(\bar{t}_{1}-\bar{t})}}{\tilde{w}(\bar{v})}\int_{%
\mathcal{V}_{1}^{\varepsilon _{1}}\cap \mathcal{\bar{V}}_{1}^{\varepsilon
_{1}}}\int_{0}^{\bar{t}_{1}}\mathbf{1}_{\{\bar{t}_{1}>0,\bar{t}_{2}\leq
0\}}q(\bar{x}_{1}+(\tau -\bar{t}_{1})v_{1},v_{1})\tilde{w}(v_{1})e^{\nu
(v_{1})(\tau -\bar{t}_{1})}d\bar{\sigma}_{1}d\tau .  \notag
\end{eqnarray}%
Since $\bar{t}_{1}\rightarrow t_{1}=0,$ it follows the last term above is
small from the boundedness of $\frac{q}{\nu }$. The first term on the right
hand side of (\ref{barcase1}) tends to 
\begin{equation*}
\int_{0}^{t}e^{\nu (\tau -t)}q(\tau ,x-v(t-\tau ),v)d\tau ,
\end{equation*}%
as second part of $h(t,x,v)$ in (\ref{t1le0}), from the continuity of $q.$
Since $\mathbf{1}_{\{\bar{t}_{1}>0,\bar{t}_{2}\leq 0\}}\equiv 1$ over $\Pi
_{l=1}^{k-1}\mathcal{V}_{l}^{\varepsilon _{1}}\cap \mathcal{\bar{V}}%
_{l}^{\varepsilon _{1}}$ in this case, and by $\bar{t}_{1}\backsim 0,$ $\bar{%
x}_{1}\backsim x_{1}\in \partial \Omega ,$ the second term on the right hand
side of (\ref{barcase1}) tends to 
\begin{equation*}
\frac{e^{-\nu (v)t}}{\tilde{w}(v)}\int_{\mathcal{V}_{1}^{\varepsilon
_{1}}}h_{0}(x_{1},v_{1})\tilde{w}(v_{1})d\sigma _{1}\backsim e^{-\nu
(v)t}h_{0}(x_{1},v)\thicksim e^{-\nu (v)t}h_{0}(x-tv,v),
\end{equation*}%
by the continuity of $h_{0}$ away from $\gamma _{0}$ and the compatibility
condition (\ref{diffusecom}). Therefore, we have shown $h(\bar{t},\bar{x},%
\bar{v})\rightarrow h(t,x,v)$ by (\ref{t1le0}).

\textbf{CASE 2:}\textit{\ }$t_{1}(t,x,v)>0.$ From continuity, $\bar{t}_{1}>0$
and 
\begin{equation*}
\int_{t_{1}}^{t}e^{\nu (\tau -t)}q(\tau ,x-v(t-\tau ),v)d\tau \backsim \int_{%
\bar{t}_{1}}^{\bar{t}}e^{\nu (\tau -\bar{t})}q(\tau ,\bar{x}-v(\bar{t}-\tau
),\bar{v})d\tau .
\end{equation*}%
It thus sufficient to only study integrals (\ref{diffuseformula}) over $B$
for both $h(t,x,v)$ and $h(\bar{t},\bar{x},\bar{v}).$ We further split 
\begin{equation*}
B=\sum_{i,m}B\cap \{t_{i+1}\leq 0,t_{i}>0;\bar{t}_{m+1}\leq 0,\bar{t}%
_{m}>0\}\equiv \sum_{i,m}B_{im},
\end{equation*}%
and $h(t,x,v)-h(\bar{t},\bar{x},\bar{v})\backsim \sum_{i,m}\int_{B_{im}}.$
It suffices to estimate the difference over each $B_{im},$ which can be
rewritten from (\ref{diffuseformula}) as: 
\begin{eqnarray}
&&\frac{e^{\nu (v)(t_{1}-t)}}{\tilde{w}(v)}\times
\{\int_{B_{im}}h_{0}(x_{i}-t_{i}v_{i},v_{i})d\Sigma _{i}(0)  \notag \\
&&+\int_{B_{im}}\sum_{j=1}^{i-1}\int_{t_{j+1}}^{t_{j}}q(\tau ,x_{j}+(\tau
-t_{j})v_{j},v_{j})d\Sigma _{j}(\tau )d\tau  \notag \\
&&+\int_{0}^{t_{i}}\int_{B_{im}}q(\tau ,x_{i}+(\tau
-t_{i})v_{i},v_{i})d\Sigma _{i}(\tau )d\tau \}  \label{h>0} \\
&&-\frac{e^{\nu (\bar{v})(\bar{t}_{1}-\bar{t})}}{\tilde{w}(\bar{v})}\times
\{\int_{B_{im}}h_{0}(\bar{x}_{m}-\bar{t}_{m}v_{m},v_{m})d\Sigma _{m}(s) 
\notag \\
&&+\int_{B_{im}}\sum_{j=1}^{m-1}\int_{\bar{t}_{j+1}}^{\bar{t}_{j}}q(\tau ,%
\bar{x}_{j}+(\tau -\bar{t}_{j})v_{j},v_{j})d\bar{\Sigma}_{j}(\tau )d\tau 
\notag \\
&&+\int_{0}^{\bar{t}_{m}}\int_{B_{im}}q(\tau ,\bar{x}_{m}+(\tau -\bar{t}%
_{m})v_{m},v_{m})d\bar{\Sigma}_{m}(\tau )d\tau \}.  \notag
\end{eqnarray}

By (\ref{tlower}) in Lemma \ref{huang}, $t_{i}-t_{i+1}\geq \frac{\varepsilon
_{1}^{3}}{C_{\xi }}>0.$ For $\varepsilon _{2}<<\varepsilon _{1},$ we further
split 
\begin{equation*}
\{t_{i}>0,t_{i+1}\leq 0\}=\{t_{i}>\varepsilon _{2},t_{i+1}\leq -\varepsilon
_{2}\}\cup \{0\leq t_{i}\leq \varepsilon _{2},t_{i+1}\leq 0\}\cup
\{t_{i}>\varepsilon _{2},-\varepsilon _{2}<t_{i+1}\leq 0\}.
\end{equation*}

\textbf{CASE 2a: }On the set $B_{im}\cap \{t_{i}>\varepsilon
_{2},t_{i+1}\leq -\varepsilon _{2}\}.$ By continuity (\ref{diffusecon}), for 
$(\bar{t},\bar{x},\bar{v})\rightarrow (t,x,v),$%
\begin{equation*}
\bar{t}_{i}>\frac{\varepsilon _{2}}{2},\bar{t}_{i+1}\leq -\frac{\varepsilon
_{2}}{2},
\end{equation*}%
then $h(\bar{t},\bar{x},\bar{v})$ has the same expression as $h(t,x,v)$ with 
$m=i$ in (\ref{h>0}), and the difference in (\ref{h>0}) is small over this
set, from the continuity of $h_{0}$ and $q.$

\textbf{CASE 2b: }On the set $B_{im}\cap \{0<t_{i}\leq \varepsilon
_{2},t_{i+1}\leq 0\}$. Now $\bar{t}_{i}\thicksim t_{i}\thicksim \varepsilon
_{2}$ and $\bar{t}_{i+1}<0$ for $\varepsilon _{2}<<\varepsilon _{1}.$ If $%
\bar{t}_{i}>0$ and $\bar{t}_{i+1}\leq 0,$ we the again have the same
expression so the difference in (\ref{h>0}) is small on this set again, from
the continuity of $h_{0}$ and $q$. Otherwise we have $\bar{t}_{i}\leq 0,$
and as $t\backsim \bar{t},$ $x\backsim \bar{x},$ $v\backsim \bar{v},$%
\begin{equation*}
\bar{t}_{i-1}\backsim t_{i-1}\geq \frac{\varepsilon _{1}^{3}}{2C_{\xi }}>0,
\end{equation*}%
from (\ref{tlower}). We define 
\begin{equation*}
B_{im}^{+}=B_{im}\cap \{0<t_{i}\leq \varepsilon _{2},t_{i+1}\leq 0,\text{
but }\bar{t}_{i}\leq 0,\bar{t}_{i-1}>0\}.
\end{equation*}%
By the continuity of $h_{0}$ and $q,$ the first term for $h(t,x,v)$ in (\ref%
{h>0}) is close to ($t_{i}\thicksim 0$): $\frac{e^{\nu (v)(t_{1}-t)}}{\tilde{%
w}(v)}\times $ 
\begin{eqnarray}
&&\int_{B_{im}^{+}}h_{0}(x_{i},v_{i})\{\Pi _{j=i+1}^{k-1}d\sigma _{j}\}\{%
\tilde{w}(v_{i})d\sigma _{i}\}\Pi _{j=1}^{i-1}\{e^{\nu
(v_{j})(t_{j+1}-t_{j})}d\sigma _{j}\}  \notag \\
&&+\int_{B_{im}^{+}}\sum_{j=1}^{i-1}\int_{t_{j+1}}^{t_{j}}q(\tau
,x_{j}+(\tau -t_{j})v_{j},v_{j})d\Sigma _{j}(\tau )d\tau .  \label{tlsmall}
\end{eqnarray}%
Since $\bar{t}_{i}\leq 0$, and $\bar{t}_{i-1}>0,$ hence $m=i-1$, $B_{im}^{+}$
is empty except for $m=i+1.$ The second term for $h(\bar{t},\bar{x},\bar{v})$
in (\ref{h>0}) is given by $\frac{e^{\nu (\bar{v})(\bar{t}_{1}-\bar{t})}}{%
\tilde{w}(\bar{v})}\times $ 
\begin{eqnarray*}
&&\int_{B_{im}^{+}}h_{0}(\bar{x}_{i-1}-\bar{t}_{i-1}v_{i-1},v_{i-1})\{\Pi
_{j=i}^{k-1}d\bar{\sigma}_{j}\}\{\tilde{w}(v_{i-1})e^{-\nu (v_{i-1})\bar{t}%
_{i-1}}d\bar{\sigma}_{i-1}\}\Pi _{j=1}^{i-2}\{e^{\nu (v_{j})(\bar{t}_{j+1}-%
\bar{t}_{j})}d\bar{\sigma}_{j}\} \\
&&+\int_{B_{im}^{+}}\sum_{j=1}^{i-2}\int_{t_{j+1}}^{t_{j}}q(\tau ,\bar{x}%
_{j}+(\tau -\bar{t}_{j})v_{j},v_{j})d\bar{\Sigma}_{j}(\tau )d\tau \\
&&+\int_{0}^{\bar{t}_{i-1}}\int_{B_{im}^{+}}q(\tau ,\bar{x}_{i-1}+(\tau -%
\bar{t}_{i-1})v_{i-1},v_{i-1})d\bar{\Sigma}_{i-1}(\tau )d\tau .
\end{eqnarray*}%
Since $\bar{x}_{i-1}-\bar{t}_{i-1}v_{i-1}\thicksim \bar{x}_{i-1}-(\bar{t}%
_{i-1}-\bar{t}_{i})v_{i-1}=\bar{x}_{i}\thicksim x_{i},$ $t_{i}\thicksim
0,\int_{0}^{\bar{t}_{i-1}}\thicksim \int_{t_{i}}^{t_{i-1}},$ the above is
simplified as%
\begin{eqnarray}
&\thicksim &\int_{B_{im}^{+}}h_{0}(x_{i},v_{i-1})\{\Pi _{j=i}^{k-1}d\sigma
_{j}\}\{\tilde{w}(v_{i-1})e^{\nu (v_{i-1})\{t_{i}-t_{i-1}\}}d\sigma
_{i-1}\}\Pi _{j=1}^{i-2}\{e^{\nu (v_{j})(t_{j+1}-t_{j})}d\sigma _{j}\} 
\notag \\
&&+\int_{B_{im}^{+}}\sum_{j=1}^{i-1}\int_{t_{j+1}}^{t_{j}}q(\tau
,x_{j}+(\tau -t_{j})v_{j},v_{j})d\Sigma _{j}(\tau )d\tau  \label{b+}
\end{eqnarray}%
By the boundary condition (\ref{hdiffuse}), 
\begin{equation*}
h_{0}(x_{i},v_{i-1})\tilde{w}(v_{i-1})=\int_{\mathcal{V}%
_{i}}h_{0}(x_{i},v_{i})\tilde{w}(v_{i})d\sigma _{i}\thicksim \int_{\mathcal{V%
}_{i}^{\varepsilon _{1}}\cap \mathcal{\bar{V}}_{i}^{\varepsilon
_{1}}}h_{0}(x_{i},v_{i})\tilde{w}(v_{i})d\sigma _{i}.
\end{equation*}%
For $v_{i}\in \mathcal{V}_{i}^{\varepsilon _{1}}\cap \mathcal{\bar{V}}%
_{i}^{\varepsilon _{1}},$ we have $t_{i}-t_{i+1}\geq \frac{\varepsilon
_{1}^{3}}{C_{\xi }}.$ But $t_{i}\leq \varepsilon _{2}<<\varepsilon _{1}$ in $%
B_{im}^{+},$ so that $t_{i+1}\leq 0.$ Hence 
\begin{equation*}
\mathcal{V}_{i}^{\varepsilon _{1}}\cap \mathcal{\bar{V}}_{i}^{\varepsilon
_{1}}=\{\mathcal{V}_{i}^{\varepsilon _{1}}\cap \mathcal{\bar{V}}%
_{i}^{\varepsilon _{1}}\}\cap \{0<t_{i}\leq \varepsilon _{2},t_{i+1}\leq 0,%
\text{ and }\bar{t}_{i}\leq 0,\bar{t}_{i-1}>0\},
\end{equation*}%
i.e., no restriction on $v_{i}\in \mathcal{V}_{i}^{\varepsilon _{1}}\cap 
\mathcal{\bar{V}}_{i}^{\varepsilon _{1}}$ in $B_{im}^{+},$ because $%
t_{i+1}\leq 0$ and $t_{i},$ $\bar{t}_{i}$ only depends on $v_{1},...v_{i-1},$
not on $v_{i}$ from (\ref{diffusecycle}). Moreover, since $\int_{\mathcal{V}%
_{i}^{\varepsilon _{1}}\cap \mathcal{\bar{V}}_{i}^{\varepsilon _{1}}}d\sigma
_{i}\thicksim 1,$ the first term in (\ref{b+}) 
\begin{equation*}
\thicksim \frac{e^{\nu (v)(t_{1}-t)}}{\tilde{w}(v)}%
\int_{B_{im}^{+}}h_{0}(x_{i},v_{i})\tilde{w}(v_{i})\{\Pi
_{j=i+1}^{k-1}d\sigma _{j}\}\Pi _{j=1}^{i-1}\{e^{\nu
(v_{j})(t_{j+1}-t_{j})}d\sigma _{j}\},
\end{equation*}%
which is exactly the first term in (\ref{tlsmall}).

\textbf{CASE 2c:} $B_{im}\cap \{t_{i}>\varepsilon _{2},-\varepsilon _{2}\leq
t_{i+1}\leq 0\}.$ Clearly $\bar{t}_{i}>0$ by continuity over $B_{im}$. If $%
\bar{t}_{i+1}\leq 0,$ then we again have $i=m$ in (\ref{h>0}) and the
difference is small on this set from the continuity of $h_{0}$ and $q.$ We
therefore only need to consider the set 
\begin{equation*}
B_{im}^{-}\equiv B_{im}\cap \{t_{i}>\varepsilon _{2},-\varepsilon _{2}\leq
t_{i+1}\leq 0,\text{ but }0<\bar{t}_{i+1},\bar{t}_{i+2}\leq 0\}.
\end{equation*}%
Since $B_{im}^{-}$ is empty except for $m=i+1,$ the contribution for $h(\bar{%
t},\bar{x},\bar{v})$ is given by $\frac{e^{\nu (\bar{v})(\bar{t}_{1}-\bar{t}%
)}}{\tilde{w}(\bar{v})}$ $\times $ 
\begin{eqnarray*}
&&\int_{B_{im}^{-}}h_{0}(\bar{x}_{i+1}-\bar{t}_{i+1}v_{i+1},v_{i+1})\{\Pi
_{j=i+2}^{k-1}d\sigma _{j}\}\{\tilde{w}(v_{i+1})e^{-\nu (v_{i+1})\bar{t}%
_{i+1}}d\sigma _{i+1}\}\Pi _{j=1}^{i}\{e^{\nu (v_{j})(\bar{t}_{j+1}-\bar{t}%
_{j})}d\sigma _{j}\} \\
&&+\int_{B_{im}^{-}}\sum_{j=1}^{i}\int_{\bar{t}_{j+1}}^{\bar{t}_{j}}q(\tau ,%
\bar{x}_{j}+(\tau -\bar{t}_{j})v_{j},v_{j})d\bar{\Sigma}_{j}(\tau )d\tau \\
&&+\int_{0}^{\bar{t}_{i+1}}\int_{B_{im}^{-}}q(\tau ,\bar{x}_{i+1}+(\tau -%
\bar{t}_{i+1})v_{i+1},v_{i+1})d\Sigma _{i+1}(\tau )d\tau .
\end{eqnarray*}%
Since on $B_{im}^{-},$ we have $\bar{t}_{i+1}\rightarrow t_{i+1}\thicksim
\varepsilon _{2},$ the last term is small from the continuity of $q$ and the
second term tends to the second term for $h(t,x,v)$ in (\ref{h>0}). Since $%
\bar{t}_{i+1}\thicksim \varepsilon _{2},$ $\bar{x}_{j+1}\thicksim x_{j+1%
\text{ }}$and $\bar{t}_{j+1}\thicksim \bar{t}_{j+1},$ the first term above
takes the form 
\begin{equation}
\thicksim \frac{e^{\nu (v)(t_{1}-t)}}{\tilde{w}(v)}%
\int_{B_{im}^{-}}h_{0}(x_{i+1},v_{i+1})\{\Pi _{j=i+2}^{k-1}d\sigma _{j}\}\{%
\tilde{w}(v_{i+1})d\sigma _{i}\}\Pi _{j=1}^{i}\{e^{\nu
(v_{j})(t_{j+1}-t_{j})}d\sigma _{j}\}.  \label{hbar}
\end{equation}

On the other hand, consider the first term\ for $h(t,x,v)$ in (\ref{h>0}).
Since $x_{i+1}-x_{i}=-v_{i}(t_{i}-t_{i+1}),$ and $t_{i+1}\thicksim
\varepsilon _{2},$ $x_{i}-t_{i}v_{i}-x_{i+1}\thicksim \varepsilon _{2},$ by
the continuity of $h_{0\text{ \ }}$(away from $\gamma _{0}$), it takes the
form 
\begin{equation}
\thicksim \frac{e^{\nu (v)(t_{1}-t)}}{\tilde{w}(v)}%
\int_{B_{im}}h_{0}(x_{i+1},v_{i})\{\Pi _{j=i+1}^{k-1}d\sigma _{j}\}\{\tilde{w%
}(v_{i})e^{-\nu (v_{i})t_{i})}d\sigma _{i}\}\Pi _{j=1}^{i-1}\{e^{\nu
(v_{j})(t_{j+1}-t_{j})}d\sigma _{j}\}  \label{hb-}
\end{equation}%
Since $\bar{t}_{i+1}\thicksim \varepsilon _{2}$ from continuity, for $%
\varepsilon _{2}<<\varepsilon _{1}$, $\bar{t}_{i+2}<0,$ and 
\begin{equation*}
\mathcal{V}_{i+1}^{\varepsilon _{1}}\cap \mathcal{\bar{V}}%
_{i+1}^{\varepsilon _{1}}\cap \{t_{i}>\varepsilon _{2},-\varepsilon _{2}\leq
t_{i+1}\leq 0,\text{ but }0<\bar{t}_{i+1},\bar{t}_{i+2}\leq 0\}=\mathcal{V}%
_{i+1}^{\varepsilon _{1}}\cap \mathcal{\bar{V}}_{i+1}^{\varepsilon _{1}},
\end{equation*}%
where $\bar{t}_{i+1}$ only depends on $v_{1},...v_{i},$ but not on $v_{i+1}$
by (\ref{diffusecycle})$.$ From the diffuse boundary condition (\ref%
{hdiffuse}) we have 
\begin{equation*}
h_{0}(x_{i+1},v_{i})\tilde{w}(v_{i})\thicksim \int_{\mathcal{V}%
_{i+1}^{\varepsilon _{1}}\cap \mathcal{\bar{V}}_{i+1}^{\varepsilon _{1}}\cap
B_{im}^{-}}h_{0}(x_{i+1},v_{i+1})\tilde{w}(v_{i+1})d\sigma _{i+1}.
\end{equation*}%
Since $\int_{\mathcal{V}_{i+1}^{\varepsilon _{1}}\cap \mathcal{\bar{V}}%
_{i+1}^{\varepsilon _{1}}\cap B_{im}^{-}}d\sigma _{i+1}\thicksim 1,$ (\ref%
{hb-}) reduces to (\ref{hbar}) and we conclude our proof.
\end{proof}

\subsubsection{Decay for Diffuse Reflection $U(t)$}

\begin{theorem}
\label{duffusiverate}Let $h_{0}\in L^{\infty }$ and assume (\ref{wdiffuse}).
There exits a unique solution to both the (\ref{lboltzmann}) and (\ref%
{lboltzmannh}) with the diffuse boundary condition (\ref{hdiffuse}). Let $%
U(t)h_{0}$ be the solution the the weighted linear Boltzmann equation (\ref%
{lboltzmannh}) with the diffuse boundary condition, then there exist $%
\lambda >0$ and $C>0$ such that the exponential decay (\ref{bouncebackdecay}%
) is valid$.$
\end{theorem}

\begin{proof}
Once again, with the $L^{\infty }$ solution $h(t)=U(t)h_{0}$ to the weighted
linear Boltzmann equation (\ref{lboltzmannh})$,$ from Lemma \ref{diffuseg0}
and Duhamel principle (\ref{duhamel}), we deduce from Ukai's trace theorem, $%
h_{\gamma }$ $\in L^{\infty }\,.$ So $f=\frac{h}{w}$ satisfies the original
linear Boltzmann equation (\ref{lboltzmann}) with $f\in L^{2}$ \ and $%
\int_{0}^{t}||f(s)||_{\gamma }^{2}ds<\infty .$ Hence uniqueness follows for $%
f$ by using the standard $L^{2}$ energy estimate. The well-posedness follows
from the exact argument in the proof of Theorem \ref{bouncebackrate}. By
Lemma \ref{bootstrap} and the $L^{2}$ decay for the diffuse boundary
problem, it suffices to establish the finite time estimate (\ref{finitetime}%
).

We apply the double-Duhamel's principle (\ref{duhamel2}). By Lemma \ref%
{diffuseg0}, the first two terms in (\ref{duhamel2}) are easily bounded by $%
C_{K}(t+1)e^{-\nu _{0}t}||h_{0}||_{\infty }.$

We concentrate on the third term 
\begin{equation}
\int_{0}^{t}\int_{0}^{s_{1}}G(t-s_{1})K_{w}G(s_{1}-s)K_{w}h(s)dsds_{1}.
\label{diffuse3}
\end{equation}%
We now fix any point $(t,x,v)$ so that $(x,v)\notin \gamma _{0}.$ From (\ref%
{diffuseformula}) with $q\equiv 0,$ the integrand above is given by 
\begin{eqnarray}
&&e^{\nu (v)(s_{1}-t)}\mathbf{1}_{\{t_{1}\leq
s_{1}\}}\{K_{w}G(s_{1}-s)K_{w}h(s)\}(s_{1},x-tv,v)  \notag \\
&&+\frac{e^{\nu (v)(t_{1}-t)}}{\tilde{w}(v)}\int \sum_{l=1}^{k-1}\mathbf{1}%
_{\{t_{l}>s_{1},t_{l+1}\leq
s_{1}\}}\{K_{w}G(s_{1}-s)K_{w}h(s)\}(s_{1},x_{l}+(s_{1}-t_{l})v_{l},v_{l})d%
\Sigma _{l}(s_{1})  \notag \\
&&+\frac{e^{\nu (v)(t_{1}-t)}}{\tilde{w}(v)}\int \mathbf{1}%
_{\{t_{k}>s_{1}\}}\{G(t-s_{1})K_{w}G(s_{1}-s)K_{w}h(s)%
\}(t_{k},x_{k},v_{k-1})d\Sigma _{k-1}(t_{k}).  \label{diffuseu}
\end{eqnarray}%
where $d\Sigma _{l}(s_{1})=\{\Pi _{j=l+1}^{k-1}d\sigma _{j}\}e^{\nu
(v_{l})(s_{1}-t_{l})}\tilde{w}(v_{l})d\sigma _{l}\Pi _{j=1}^{l-1}\{e^{\nu
(v_{j})(t_{j+1}-t_{j})}d\sigma _{j}\},$ and the exponential factor in $%
d\Sigma _{l}(s_{1})$ is bounded by $e^{\nu _{0}(s_{1}-t_{1})}.$ By Lemma \ref%
{diffuseg0}, 
\begin{equation*}
||\{G(t_{k}-s_{1})K_{w}G(s_{1}-s)K_{w}h(s)\}(t_{k})||_{\infty }\leq C_{K}e^{-%
\frac{\nu _{0}}{2}(t_{k}-s)}||h(s)||_{\infty }.
\end{equation*}%
Letting $k=k_{0}(\varepsilon ,T_{0})$ large in Lemma \ref{small}, as in (\ref%
{inflowstep1}), the integral of the last term is bounded by 
\begin{equation}
\varepsilon C_{K}e^{-\frac{\nu _{0}}{8}t}\sup_{0\leq t\leq T_{0}}\{e^{\frac{%
\nu _{0}}{8}s}||h(s)||_{\infty }\}.  \label{smalle}
\end{equation}

To estimate the first and second terms in (\ref{diffuseu}), we first derive
the formula for $K_{w}G(s_{1}-s)K_{w}h(s).$ Recall the back-time cycles of: $%
X_{\mathbf{cl}}(s)=x_{l}-(t_{l}-s)v_{l}.$ We denote $t_{l^{\prime }}^{\prime
}=t_{l^{\prime }}(s_{1},X_{\mathbf{cl}}(s_{1}),v^{\prime }),$ $\ $and $%
x_{l^{\prime }}^{\prime }=X_{\mathbf{cl}}(t_{l^{\prime }},(X_{\mathbf{cl}%
}(s_{1}),v^{\prime }),$ and $v_{l^{\prime }}$ $\in \mathcal{V}_{l^{\prime }},
$ for $1\leq l^{\prime }\leq k-1.$ By (\ref{diffuseformula}) again, 
\begin{eqnarray}
&&\{K_{w}G(s_{1}-s)K_{w}h(s)\}(s_{1},X_{\mathbf{cl}}(s_{1}),v_{l})  \notag \\
&=&\int K_{w}(v_{l},v^{\prime })\{G(s_{1}-s)K_{w}h(s)\}\left( s_{1},X_{%
\mathbf{cl}}(s_{1}),v^{\prime }\right) dv^{\prime }  \notag \\
&=&\int K_{w}(v_{l},v^{\prime })e^{\nu (v^{\prime })(s-s_{1})}\mathbf{1}%
_{\{t_{1}^{\prime }\leq s\}}\{K_{w}h(s)\}(s,X_{\mathbf{cl}%
}(s_{1})-(s_{1}-s)v^{\prime },v^{\prime })dv^{\prime }  \notag \\
&&+\int K_{w}(v_{l},v^{\prime })\frac{e^{\nu (v^{\prime })(t_{1}^{\prime
}-s_{1})}}{\tilde{w}(v^{\prime })}\sum_{l^{\prime }=1}^{k-1}\mathbf{1}%
_{t_{l^{\prime }+1}^{\prime }\leq s<t_{l^{\prime }}^{\prime }}\{\int
K_{w}(v_{l^{\prime }},v^{\prime \prime })h(s,x_{l^{\prime }}^{\prime
}-(t_{l^{\prime }}^{\prime }-s)v_{l^{\prime }},v^{\prime \prime })dv^{\prime
\prime }\}d\Sigma _{l^{\prime }}dv^{\prime }  \notag \\
&&+\int K_{w}(v_{l},v^{\prime })\frac{e^{\nu (v^{\prime })(t_{1}^{\prime
}-s_{1})}}{\tilde{w}(v^{\prime })}\int \mathbf{1}_{t_{k}^{\prime
}>s}\{G(s_{1}-s)K_{w}h(s)\}(t_{k}^{\prime },x_{k}^{\prime },v_{k})d\Sigma
_{k-1}(t_{k}^{\prime })dv^{\prime }.  \label{h'}
\end{eqnarray}%
Once again, since $||G(s_{1}-s)K_{w}h(s)(t_{k}^{\prime },x_{k}^{\prime
},v_{k})||_{\infty }\leq C_{K}e^{\frac{\nu _{0}}{2}\{t_{k}^{\prime
}-s\}}||h(s)||_{\infty },$ the last term is bounded by (\ref{smalle}) due to
Lemma \ref{small}, when $k\geq k_{0}(\varepsilon )$ large.

By inserting the main terms in (\ref{h'}) back to (\ref{diffuseu}), we
deduce that, up to $\varepsilon C_{K}e^{-\frac{\nu _{0}}{8}t}\sup_{0\leq
t\leq T_{0}}\{e^{\frac{\nu _{0}}{8}s}||h(s)||_{\infty }\},$ the third term (%
\ref{diffuse3}) in the double-duhamel representation (\ref{duhamel2}) is $%
\int_{0}^{t}\int_{0}^{s_{1}}\ast dsds_{1},$ where $\ast $ is 
\begin{eqnarray}
&&e^{\nu (v)(s_{1}-t)}\mathbf{1}_{t_{1}\leq s_{1}}\int K_{w}(v,v^{\prime
})e^{\nu (v^{\prime })(s-s_{1})}\mathbf{1}_{t_{1}^{\prime }\leq
s}\{K_{w}h(s)\}(s,X_{\mathbf{cl}}(s_{1})-(s_{1}-s)v,v^{\prime })dv^{\prime }
\notag \\
&&+e^{\nu (v)(s_{1}-t)}\mathbf{1}_{t_{1}\leq s_{1}}\int K_{w}(v,v^{\prime })%
\frac{e^{\nu (v^{\prime })(t-s_{1})}}{\tilde{w}(v^{\prime })}\sum_{l^{\prime
}=1}^{k-1}\mathbf{1}_{t_{l^{\prime }+1}^{\prime }\leq s<t_{l^{\prime
}}^{\prime }}\times   \notag \\
&&\times \{\int K_{w}(v_{l^{\prime }},v^{\prime \prime })h(s,x_{l^{\prime
}}^{\prime }-(t_{l^{\prime }}^{\prime }-s)v_{l^{\prime }},v^{\prime \prime
})dv^{\prime \prime }\}d\Sigma _{l^{\prime }}^{\prime }(s)dv^{\prime } 
\notag \\
&&+\frac{e^{\nu (v)(t_{1}-t)}}{\tilde{w}(v)}\int \sum_{l=1}^{k-1}\mathbf{1}%
_{t_{l+1}\leq s_{1}<t_{l}}K_{w}(v_{l},v^{\prime })\times   \label{diffusee}
\\
\times  &&e^{\nu (v^{\prime })(s-s_{1})}\mathbf{1}_{t_{1}^{\prime }\leq
s}K_{w}(v^{\prime },v^{\prime \prime })h(s,X_{\mathbf{cl}%
}(s_{1})-(s_{1}-s)v^{\prime },v^{\prime \prime })dv^{\prime \prime
}dv^{\prime }d\Sigma _{l}(s_{1})  \notag \\
&&+\frac{e^{\nu (v)(t_{1}-t)}}{\tilde{w}(v)}\int \sum_{l=1,l^{\prime
}=1}^{k-1}\mathbf{1}_{\{t_{l}>s_{1},t_{l+1}\leq
s_{1}\}}K_{w}(v_{l},v^{\prime })d\Sigma _{l}\times   \notag \\
&&\times \mathbf{1}_{\{t_{l^{\prime }}^{\prime }>s,t_{l^{\prime }+1}^{\prime
}\leq s\}}\frac{e^{\nu (v^{\prime })(t_{1}^{\prime }-s_{1})}}{\tilde{w}%
(v^{\prime })}K_{w}(v_{l^{\prime }},v^{\prime \prime })h(s,x_{l^{\prime
}}^{\prime }-(t_{l^{\prime }}^{\prime }-s)v_{l^{\prime }},v^{\prime \prime
})dv^{\prime \prime }d\Sigma _{l^{\prime }}^{\prime }ds_{1}ds.  \notag
\end{eqnarray}

We now estimate them term by term. For the first term in (\ref{diffusee}),
the back-time trajectories never touch the boundary. This term can be easily
estimated as the proof of Theorem \ref{inflowrate} for the in-flow case
(e.g. (\ref{inflowstep42})) by 
\begin{equation*}
\varepsilon C_{K}e^{-\frac{\nu _{0}}{8}t}\sup_{0\leq t\leq T_{0}}\{e^{\frac{%
\nu _{0}}{8}s}||h(s)||_{\infty }\}+C_{\varepsilon
,T_{0}}\int_{0}^{T_{0}}||f(s)||ds.
\end{equation*}

For the second term in (\ref{diffusee}), we fix $l^{\prime }$ and consider $%
v_{l^{\prime }},v^{\prime }$ and $v^{\prime \prime }$ integration. We recall 
$d\Sigma _{l^{\prime }}^{\prime }(s)=\{\Pi _{j=l^{\prime }+1}^{k-1}d\sigma
_{j}\}e^{\nu (v_{l^{\prime }})(s-t_{l^{\prime }}^{\prime })}\tilde{w}%
(v_{l^{\prime }})d\sigma _{l^{\prime }}^{\prime }\Pi _{j=1}^{l^{\prime
}-1}\{e^{\nu (v_{j})(t_{j+1}^{\prime }-t_{j}^{\prime })}d\sigma _{j}^{\prime
}\}$, and $\tilde{w}(v_{l^{\prime }})d\sigma _{l^{\prime }}=\tilde{w}%
(v_{l^{\prime }})\mu (v_{l^{\prime }})\{n(x_{l^{\prime }}^{\prime })\cdot
v_{l^{\prime }}\}dv_{l^{\prime }}$. The exponential factor in $d\Sigma
_{l^{\prime }}^{\prime }(s)$ is bounded by $e^{\nu _{0}(t_{1}^{\prime }-s)}.$
Notice that with $\tilde{w}(v_{l^{\prime }})\mu (v_{l^{\prime }})$ bounded
and integrable for $\theta <\frac{1}{4}$. Hence from (\ref{wk}), 
\begin{eqnarray}
&&\int_{|v_{l^{\prime }}|\geq N}\int |K_{w}(v_{l^{\prime }},v^{\prime \prime
})|dv^{\prime \prime }\tilde{w}(v_{l^{\prime }})\mu (v_{l^{\prime
}})\{n(x_{l^{\prime }}^{\prime })\cdot v_{l^{\prime }}\}dv_{l^{\prime }}
\label{vl'} \\
&\leq &\frac{C_{K}}{N}\int_{|v_{l^{\prime }}|\geq N}\tilde{w}(v_{l^{\prime
}})\mu (v_{l^{\prime }})\{n(x_{l^{\prime }}^{\prime })\cdot v_{l^{\prime
}}\}dv_{l^{\prime }}=\frac{C_{K}}{N}.  \notag
\end{eqnarray}%
By similar arguments as in Case 1 (\ref{inflowstep2}), Case 2 (\ref%
{inflowstep3}) in the proof of Theorem \ref{inflowrate}, we only need to
consider the case $|v|\leq N,$ $|v^{\prime }|\leq 2N,$ and $|v_{l^{\prime
}}|\leq N$ and $|v^{\prime \prime }|\leq 2N,$ for some large $N$. As in Case
4 in the proof of Theorem \ref{inflowrate}, we can also use the same
approximation (\ref{approximate}). Hence, for each fixed $l^{\prime },$ we
only need to control: 
\begin{eqnarray*}
&&\int_{t_{1}}^{t}\int_{\Pi _{j=1}^{l^{\prime }-1}\mathcal{V}_{j}^{\prime
}}\left\{ \int_{|v^{\prime \prime }|\leq 2N,|v_{l^{\prime }}|\leq
N}\int_{t_{l^{\prime }+1}^{\prime }}^{t_{l^{\prime }}^{\prime }}e^{\nu
_{0}(s-t)}|h(s,x_{l^{\prime }}^{\prime }-(t_{l^{\prime }}^{\prime
}-s)v_{l^{\prime }},v^{\prime \prime })|dv^{\prime \prime }dv^{\prime
}dv_{l^{\prime }}\right\} \Pi _{j=1}^{l^{\prime }-1}d\sigma _{j}^{\prime } \\
&\leq &\int_{t_{l^{\prime }+1}^{\prime }}^{t_{l^{\prime }}^{\prime }}{}%
\mathbf{1}_{\{t_{l^{\prime }}^{\prime }-s\geq \frac{1}{N}\}}+\mathbf{1}%
_{\{t_{l^{\prime }}^{\prime }-s\leq \frac{1}{N}\}} \\
&\leq &\int_{t_{l^{\prime }+1}^{\prime }}^{t_{l^{\prime }}^{\prime }}\mathbf{%
1}_{\{t_{l^{\prime }}^{\prime }-s\geq \frac{1}{N}\}}+\frac{C}{N}\sup_{0\leq
s\leq T_{0}}\{e^{-\nu _{0}(s-t_{l^{\prime }}^{\prime })}||h(s)||_{\infty }\}.
\end{eqnarray*}%
Here we have used the fact $\int \{\Pi _{j=l^{\prime }+1}^{k-1}d\sigma
_{j}\}=1$. Notice that $x_{l^{\prime }}^{\prime }$ and $t_{l^{\prime
}}^{\prime }$ depend only on $t,x,v,v_{1},...v_{l^{\prime }-1},$ not on $%
v_{l^{\prime }}.$ By making a change of variable $y=x_{l^{\prime }}^{\prime
}-(t_{l^{\prime }}^{\prime }-s)v_{l^{\prime }},$ for the first part, we use
Fubini's theorem and the fact $\int \{\Pi _{j=1}^{l^{\prime }-1}d\sigma
_{j}\}=1$ to majorize it as ($f=\frac{h}{w}):$ 
\begin{equation}
C_{N,T_{0}}\int_{0}^{T_{0}}\left\{ \int_{y\in \Omega ,|v^{\prime \prime
}|\leq 2N}|h(s,y,v^{\prime \prime })|^{2}dydv^{\prime \prime }\right\}
^{1/2}=C_{N,T_{0},\Omega }\int_{0}^{T_{0}}||f(s)||ds.  \label{diffusel2}
\end{equation}

For the third term in (\ref{diffusee}), for fixed $l,$ we consider the $%
v_{l},v^{\prime },v^{\prime \prime }$ integration. Recall $d\Sigma
_{l}(s_{1})=\{\Pi _{j=l+1}^{k-1}d\sigma _{j}\}e^{\nu (v_{l})(s_{1}-t_{l})}%
\tilde{w}(v_{l})d\sigma _{l}\Pi _{j=1}^{l-1}\{e^{\nu
(v_{j})(t_{j+1}-t_{j})}d\sigma _{j}\}$ and the exponential factor is bounded
by $e^{\nu _{0}(t_{1}-s_{1})}.$ Because $\tilde{w}(v_{l})\mu (v_{l})$ is
bounded and integrable, as in (\ref{vl'}), we can use the arguments in Case
1 (\ref{inflowstep2}), Case 2 (\ref{inflowstep3}) in the proof of Theorem %
\ref{inflowrate} to reduce to $|v_{l}|\leq N,$ $|v^{\prime }|\leq 2N$ and $%
|v^{\prime \prime }|\leq 3N.$ As in the second term, it then suffices to
estimate 
\begin{equation*}
e^{-\nu _{0}(t-s)}\int_{\Pi _{j=1}^{l-1}\mathcal{V}_{j}^{\prime
}}\int_{t_{l+1}}^{t_{l}}\int_{s}^{s_{1}}\int_{|v^{\prime \prime }|\leq
2N,|v^{\prime }|\leq N}|h(s,X_{\mathbf{cl}}(s_{1})-(s_{1}-s)v^{\prime
},v^{\prime \prime })|dsds_{1}dv^{\prime \prime }dv^{\prime }\Pi
_{j=1}^{l-1}d\sigma _{j}
\end{equation*}%
\newline
which is bounded by (\ref{diffusel2}) by the change of variable $y=X_{%
\mathbf{cl}}(s_{1})-(s_{1}-s)v^{\prime }$.

In the last integral in (\ref{diffusee}), both back-time diffuse cycles
first hit the boundary. For fixed $l$ and $l^{\prime },$ we consider the $%
v_{l},v_{l^{\prime }}$ and $v^{\prime }$,$v^{\prime \prime }$ integration.
We again recall $d\Sigma _{l}$ and $d\Sigma _{l^{\prime }}^{\prime \text{ }}$
and their exponential factors are bounded by $e^{\nu _{0}(t_{1}-s_{1})}$ \
and $e^{\nu _{0}(t_{1}^{\prime }-s)}$ respectively. Notice that both $\tilde{%
w}(v_{l})\mu (v_{l})$ and $\tilde{w}(v_{l^{\prime }})\mu (v_{l^{\prime }})$
are bounded and integrable. We use twice (\ref{vl'}), and as in the
reduction to (\ref{diffusel2}), we can reduce to the case of $|v_{l}|\leq
N,|v^{\prime }|\leq 2N,$ and $|v_{l^{\prime }}|\leq N$, $|v^{\prime \prime
}|\leq 2N.$ Therefore, by using the approximation (\ref{approximate}) and $%
\int \{\Pi _{j=l^{\prime }+1}^{k-1}d\sigma _{j}^{\prime }\}=1,$ $\int \{\Pi
_{j=l+1}^{k-1}d\sigma _{j}\}=1,$ we only need to control: 
\begin{equation*}
\int_{\Pi _{j=1}^{l-1}\mathcal{V}_{j}}\int_{\Pi _{j=1}^{l^{\prime }-1}%
\mathcal{V}_{j}^{\prime }}\int_{t_{l+1}}^{t_{l}}\int_{t_{l^{\prime
}+1}^{\prime }}^{t_{l^{\prime }}^{\prime }}\int_{|v_{l^{\prime
}}|,|v^{\prime \prime }|\leq 2N}e^{\nu _{0}(s-t)}|h(s,x_{l^{\prime
}}^{\prime }-(t_{l^{\prime }}^{\prime }-s)v_{l^{\prime }},v^{\prime \prime
})|dv^{\prime \prime }dv_{l^{\prime }}dsds_{1}|\Pi _{j=1}^{l^{\prime
}-1}d\sigma _{j}^{\prime }\Pi _{j=1}^{l-1}d\sigma _{j},
\end{equation*}%
which is again bounded by (\ref{diffusel2}) by the change of variable $%
y=x_{l^{\prime }}^{\prime }-(t_{l^{\prime }}^{\prime }-s)v_{l^{\prime }}$.

Summing over $l$, $l^{\prime }\leq k-1,$ and letting $N$ large, we summarize:%
\begin{equation*}
||h(t)||_{\infty }\leq C_{K}\{1+t\}e^{-\frac{\nu _{0}}{2}t}||h_{0}||_{\infty
}+\varepsilon C_{K}e^{-\frac{\nu _{0}}{8}t}\sup_{0\leq t\leq T_{0}}\{e^{%
\frac{\nu _{0}}{8}s}||h(s)||_{\infty }\}+C_{N,T_{0},\varepsilon
}\int_{0}^{T_{0}}||f(s)||ds.
\end{equation*}%
Choosing $\varepsilon $ small such that $\varepsilon C_{K}=\frac{1}{2},$ and 
$T_{0}$ large so that $2C_{K}\{1+T_{0}\}e^{-\frac{\nu _{0}}{2}%
T_{0}}=e^{-\lambda T_{0}},$ we deduce (\ref{finitetime}), and by Lemma \ref%
{bootstrap}, we deduce our theorem.
\end{proof}

\section{Nonlinear Exponential Decay}

\begin{proof}
(of Theorem \ref{inflownl}): \textbf{Step 1. Existence and Continuity}: Let $%
h^{0}\equiv 0,$ we use iteration 
\begin{equation}
\{\partial _{t}+v\cdot \nabla _{x}+\nu -K_{w}\}h^{m+1}=w\Gamma (\frac{h^{m}}{%
w},\frac{h^{m}}{w})  \label{interation1}
\end{equation}%
with $h^{m+1}|_{t=0}=h_{0},$ $h^{m+1}|_{\gamma ^{-}}=wg$. We further split $%
h^{m+1}=h_{g}^{m+1}+h_{\Gamma }^{m+1}$ for $m\geq 1,$ where $h_{g}^{m+1}$
satisfies the homogeneous linear weighted Boltzmann equation (\ref%
{lboltzmannh}) 
\begin{equation*}
\{\partial _{t}+v\cdot \nabla _{x}+\nu -K_{w}\}h_{g}^{m+1}=0,\text{ \ \ }%
h_{g}^{m+1}|_{\gamma ^{-}}=wg,\text{ \ }h_{g}^{m+1}|_{t=0}=h_{0};
\end{equation*}%
while $h_{\Gamma }^{m+1}$ satisfies the inhomogeneous linear weighted
Boltzmann equation (\ref{lboltzmannh}) with zero-boundary condition:%
\begin{equation*}
\{\partial _{t}+v\cdot \nabla _{x}+\nu -K_{w}\}h_{\Gamma }^{m+1}=w\Gamma (%
\frac{h^{m}}{w},\frac{h^{m}}{w}),\text{ }h_{\Gamma }^{m+1}|_{\gamma ^{-}}=%
\text{\ }h_{\Gamma }^{m+1}|_{t=0}=0.
\end{equation*}%
Clearly, from Theorem \ref{inflowrate}, for some $0<\lambda <\lambda _{0},$ 
\begin{equation*}
\sup_{0\leq t\leq \infty }e^{\lambda t}||h_{g}^{m+1}(t)||_{\infty }\leq
C\{||h_{0}||_{\infty }+\sup_{0\leq t\leq \infty }e^{\lambda
_{0}t}||wg(t)||_{\infty }\}.
\end{equation*}%
On the other hand, denote $U_{0}(t,s)$ and $G_{0}(t,s)$ to be solution
operators for linear weighted Boltzmann equation (\ref{lboltzmannh}) and (%
\ref{transport}) with \textit{zero} boundary condition respectively, then $%
U_{0}(t,s)=U_{0}(t-s,0)$ and $G_{0}(t,s)=G(t-s,0)$ are semigroups. Hence, by
the Duhamel Principle, 
\begin{equation}
h_{\Gamma }^{m+1}=\int_{0}^{t}U_{0}(t-s)w\Gamma (\frac{h^{m}}{w},\frac{h^{m}%
}{w})(s)ds.  \label{iteration2}
\end{equation}%
To avoid the extra weight function $\nu (v)\thicksim \{1+|v|\}^{\gamma }$ in
Lemma \ref{collision}, we further use the Duhamal Principle $%
U_{0}(t-s)=G_{0}(t-s)+\int_{s}^{t}G_{0}(t-s_{1})K_{w}U_{0}(s_{1}-s)ds_{1}$
to bound (\ref{iteration2}) by%
\begin{equation}
||\int_{0}^{t}G_{0}(t-s)w\Gamma (\frac{h^{m}}{w},\frac{h^{m}}{w}%
)(s)ds||_{\infty
}+||\int_{0}^{t}\int_{s}^{t}G_{0}(t-s_{1})K_{w}U_{0}(s_{1}-s)w\Gamma (\frac{%
h^{m}}{w},\frac{h^{m}}{w})(s)ds_{1}ds||_{\infty }.  \label{usplit}
\end{equation}%
For any $(t,x,v),$ by Lemma \ref{collision}:%
\begin{equation*}
|\{w\Gamma (\frac{h^{m}}{w},\frac{h^{m}}{w})\}(s,x-(t-s)v,v)|\leq C\nu
(v)||h^{m}||_{\infty }^{2},
\end{equation*}%
By the explicit formula (\ref{inflowformula}) with $g=0$, we therefore can
bound the first term in (\ref{usplit}) as 
\begin{eqnarray}
&&|\int_{0}^{t}e^{-\nu (v)(t-s)}\mathbf{1}_{t-t_{\mathbf{b}}(x,v)\leq
s}\{w\Gamma (\frac{h^{m}}{w},\frac{h^{m}}{w})\}(s,x-(t-s)v,v)ds|  \notag \\
&\leq &C\int_{0}^{t}e^{-\nu (v)(t-s)}\nu (v)||h^{m}(s)||_{\infty }^{2}ds 
\notag \\
&\leq &C\int e^{-\frac{\nu (v)}{2}(t-s)}\nu (v)ds\times \sup_{0\leq s\leq
t}\{e^{-\frac{\nu _{0}}{2}(t-s)}||h^{m}(s)||_{\infty }^{2}\}  \notag \\
&\leq &Ce^{-\frac{\nu _{0}}{2}t}\times \sup_{s}\{e^{\frac{\nu _{0}}{2}%
s}||h^{m}(s)||_{\infty }\}^{2}.  \label{g0}
\end{eqnarray}%
Here we have used the fact $\nu (v)\geq \nu _{0}\,,$ and $\int e^{-\frac{\nu
(v)}{2}(t-s)}\nu (v)ds<\infty .$

On the other hand, for the second term in (\ref{usplit}), let the semigroup $%
\tilde{U}(t)h_{0}$ solve 
\begin{equation}
\{\partial _{t}+v\cdot \nabla _{x}+\nu -K_{w/\sqrt{1+\rho |v|^{2}}}\}\{%
\tilde{U}(t)h_{0}\}=0,\text{ \ \ \ }  \label{utilde}
\end{equation}%
with $\{\tilde{U}(t)\tilde{h}_{0}\}|_{\gamma -}=0$ and $\tilde{U}(0)\tilde{h}%
_{0}=\tilde{h}_{0}.$ By a direct computation, $\sqrt{1+\rho |v|^{2}}\tilde{U}%
(t)$ solves the original linear Boltzmann equation (\ref{lboltzmannh}). By
uniqueness in the $L^{\infty }$ class, we deduce 
\begin{equation}
U(t)h_{0}\equiv \sqrt{1+\rho |v|^{2}}\tilde{U}(t)\{\frac{h_{0}}{\sqrt{1+\rho
|v|^{2}}}\}.  \label{equalutilde}
\end{equation}%
Therefore, we can rewrite $K_{w}U(s_{1},s)w\Gamma (\frac{h^{m}}{w},\frac{%
h^{m}}{w})(s)$ to get 
\begin{equation*}
\int_{0}^{t}\int_{s}^{t}e^{-\nu _{0}(t-s_{1})}||\left\{ \int
K_{w}(v,v^{\prime })\{\sqrt{1+\rho |v^{\prime }|^{2}}\}\right\} \tilde{U}%
(s_{1}-s)\{\frac{w}{\sqrt{1+\rho |v^{\prime }|^{2}}}\Gamma (\frac{h^{m}}{w},%
\frac{h^{m}}{w})(s)\}||_{\infty }ds_{1}dsdv^{\prime }.
\end{equation*}%
Since $w^{-2}(1+|v|)^{3}\in L^{1},$ $\ $\ $\{\frac{w}{\sqrt{1+\rho
|v^{\prime }|^{2}}}\}^{-2}(1+|v|)\in L^{1}$ so that Theorem \ref{inflowrate}
is valid for $\tilde{U}.$ Since $\frac{\nu (v^{\prime })}{\sqrt{1+\rho
|v^{\prime }|^{2}}}\leq C_{\rho }$, from the proof of Lemma \ref{collision}, 
\begin{equation}
\int K_{w}(v,v^{\prime })\sqrt{1+\rho |v^{\prime }|^{2}}dv^{\prime }\leq
C\int K_{w}(v,v^{\prime })\{|v-v^{\prime }|+|v|\}dv^{\prime }<+\infty .
\label{k}
\end{equation}%
Hence the second term in\ (\ref{usplit}) is bounded 
\begin{eqnarray}
&&C\int_{0}^{t}\int_{s}^{t}e^{-\nu _{0}(t-s_{1})}\tilde{U}(s_{1}-s)||\{\frac{%
w}{\sqrt{1+\rho |v^{\prime }|^{2}}}\Gamma (\frac{h^{m}}{w},\frac{h^{m}}{w}%
)(s)\}||_{\infty }ds_{1}ds  \notag \\
&\leq &C\int_{0}^{t}\int_{s}^{t}e^{-\lambda (t-s_{1})}||h^{m}(s)||_{\infty
}^{2}ds_{1}ds  \notag \\
&\leq &Ce^{-\frac{\lambda }{2}t}\times \{\sup_{0\leq s\leq \infty }e^{\frac{%
\lambda }{2}s}||h^{m}(s)||_{\infty }\}^{2}.  \label{u0tilde}
\end{eqnarray}%
Collecting terms for both $h_{g}^{m+1}$ and $h_{\Gamma }^{m+1},$ we obtain
for $0<\lambda <\lambda _{0},$ 
\begin{equation*}
\sup_{m}\sup_{0\leq t\leq \infty }\{e^{\lambda t}||h^{m+1}(t)||_{\infty
}\}\leq C\{||h_{0}||_{\infty }+\sup_{0\leq s\leq \infty }e^{\lambda
_{0}s}||g(s)||_{\infty }\}.
\end{equation*}%
for $||h_{0}||$ sufficiently small. Moreover, subtracting $h^{m+1}-h^{m}$
yields:%
\begin{equation}
\{\partial _{t}+v\cdot \nabla _{x}+\nu -K_{w}\}\{h^{m+1}-h^{m}\}=w\{\Gamma (%
\frac{h^{m}}{w},\frac{h^{m}}{w})-\Gamma (\frac{h^{m-1}}{w},\frac{h^{m-1}}{w}%
)\}  \label{difference}
\end{equation}%
with zero initial and boundary value. Therefore, splitting 
\begin{equation*}
\Gamma (\frac{h^{m}}{w},\frac{h^{m}}{w})-\Gamma (\frac{h^{m-1}}{w},\frac{%
h^{m-1}}{w})=\Gamma (\frac{h^{m}-h^{m-1}}{w},\frac{h^{m}}{w})-\Gamma (\frac{%
h^{m-1}}{w},\frac{h^{m-1}-h^{m}}{w})
\end{equation*}%
we can bound $||h^{m+1}(t)-h^{m}(t)||_{\infty }$ as in (\ref{g0}) and (\ref%
{u0tilde}): 
\begin{eqnarray}
&&C||\int_{0}^{t}U_{0}(t-s)w\Gamma (\frac{h^{m}-h^{m-1}}{w},\frac{h^{m}}{w}%
)(s)ds||_{\infty }  \label{hmdifference} \\
&&+C||\int_{0}^{t}U_{0}(t-s)w\Gamma (\frac{h^{m-1}}{w},\frac{h^{m-1}-h^{m}}{w%
})(s)ds||_{\infty }  \notag \\
&\leq &C\sup_{s}\{e^{\lambda s}\{||h^{m}(s)||_{\infty
}+||h^{m-1}(s)||\}\}\times e^{-\lambda t}\times \sup_{s}\{e^{\lambda
s}\{||h^{m}(s)-h^{m-1}(s)||_{\infty }\}.  \notag
\end{eqnarray}%
Hence $h^{m}$ is a Cauchy sequence and the limit $h$ is a desired unique
solution.

Moreover, if $\Omega $ is strictly convex, by Lemma \ref{inflowcon},
inductively, $h^{m}$ is continuous in $[0,\infty )\times \{\bar{\Omega}%
\times \mathbf{R}^{3}\setminus \gamma _{0}\}.$ It is straightforward and
routine to verify that $w\Gamma (\frac{h^{m}}{w},\frac{h^{m}}{w})$ is
continuous in the interior of $[0,\infty )\times \Omega \times \mathbf{R}%
^{3}.$ Moreover, from Lemma \ref{collision}, $\sup |\frac{w\Gamma (\frac{%
h^{m}}{w},\frac{h^{m}}{w})}{\nu }|$ is also finite. We therefore deduce that 
$h^{m+1}$ and hence $h$ is also continuous in $[0,\infty )\times \{\bar{%
\Omega}\times \mathbf{R}^{3}\setminus \gamma _{0}\}$ from the uniform
convergence.

\textbf{Step 2. Uniqueness and Positivity.} \ Assume that there is another
solution $\tilde{h}$ to the full Boltzmann equation with the same initial
and boundary condition as $h.$ Assume that $\sup_{0\leq t\leq T_{0}}||\tilde{%
h}(t)||_{\infty }$ is sufficiently small. Then 
\begin{equation*}
\{\partial _{t}+v\cdot \nabla _{x}+\nu -K_{w}\}\{h-\tilde{h}\}=w\{\Gamma (%
\frac{h}{w},\frac{h}{w})-\Gamma (\frac{\tilde{h}}{w},\frac{\tilde{h}}{w})\},
\end{equation*}%
so that $||h(t)-\tilde{h}(t)||_{\infty }$ is bounded by as in step 1: 
\begin{eqnarray*}
&&C||\int_{0}^{t}U_{0}(t-s)w\Gamma (\frac{h-\tilde{h}}{w},\frac{h}{w}%
)(s)ds||_{\infty }+C||\int_{0}^{t}U_{0}(t-s)w\Gamma (\frac{\tilde{h}}{w},%
\frac{h-\tilde{h}}{w})(s)ds||_{\infty }. \\
&\leq &C_{T_{0}}\sup_{0\leq s\leq T_{0}}\{\{||h(s)||_{\infty }+||\tilde{h}%
(s)||\}\}\sup_{0\leq s\leq T_{0}}\{||h(s)-\tilde{h}(s)||_{\infty }.
\end{eqnarray*}%
This implies that $\sup_{0\leq t\leq T_{0}}||h(t)-\tilde{h}(t)||_{\infty
}\equiv 0$ if $\sup_{0\leq t\leq T_{0}}||\tilde{h}(t)||_{\infty }$ and $%
\sup_{0\leq t\leq T_{0}}||h(t)||_{\infty }$ sufficiently small.

Finally, we address the positivity of the $F=\mu +\sqrt{\mu }f.$ We use a
different approximation for the original Boltzmann equation (\ref{boltzmann}%
) 
\begin{equation}
\{\partial _{t}+v\cdot \nabla _{x}\}F^{m+1}+\nu (F^{m})F^{m+1}=Q_{\text{gain}%
}(F^{m},F^{m}),\text{ \ }  \label{positive}
\end{equation}%
with the inflow boundary condition $\ F^{m+1}|_{\gamma ^{-}}=\mu +\sqrt{\mu }%
g\geq 0$ and $F^{m+1}|_{t=0}=\mu +\sqrt{\mu }f_{0},$ starting with $%
F^{0}\equiv \mu .$ Here 
\begin{equation*}
\nu (F^{m})=\int F^{m}(v)|v-u|^{\gamma }dud\omega .
\end{equation*}%
In terms of $f^{m}=\frac{F^{m}-\mu }{\sqrt{\mu }},$ by (\ref{mu-k}) and (\ref%
{gamma}), we have 
\begin{equation}
\{\partial _{t}+v\cdot \nabla _{x}+\nu \}f^{m+1}=Kf^{m}+\Gamma _{\text{gain}%
}(f^{m},f^{m})-\Gamma _{\text{loss}}(f^{m},f^{m+1}).  \label{fm}
\end{equation}%
It is routine and standard to show that $h^{m}=wf^{m}$ is convergent in $%
L^{\infty },$ locally in time $[0,T_{0}],$ where $T_{0}$ depends on $%
||h_{0}||_{\infty }=||\frac{F_{0}-\mu }{\sqrt{\mu }}||_{\infty }$ and $%
\sup_{0\leq t\leq T_{0}}||g(t)||_{\infty }.$ By induction on $m$, we can
show that if $F^{m}\geq 0,$ the $Q_{\text{gain }}(F^{m},F^{m})\geq 0.$
Denote the integration factor as 
\begin{equation}
I(t,s)=e^{-\int_{s}^{t}\nu (F^{m})(\tau ,X(\tau ),V(\tau ))d\tau },
\label{ifactor}
\end{equation}%
so that $\frac{d}{ds}\{I(t,s)F^{m+1}\}=I(t,s)Q_{\text{gain }}(F^{m},F^{m})$
almost everywhere along the back-time trajectory $[X(s),V(s)]$ of $t,x,v,$
inside $\bar{\Omega}$. As in (\ref{inflowformula}), we express $%
F^{m+1}(t,x,v)=$%
\begin{eqnarray*}
&&\mathbf{1}_{t-t_{\mathbf{b}}\leq
0}\{I(t,0)F_{0}(x-vt,v)+\int_{0}^{t}I(t,s)Q_{\text{gain }%
}(F^{m},F^{m})(X(s),V(s))ds\}+\mathbf{1}_{t-t_{\mathbf{b}}>0}\times \\
&&\{I(t,t-t_{\mathbf{b}})\{\mu +\sqrt{\mu }g\}(t-t_{\mathbf{b}},x-vt_{%
\mathbf{b}},v)+\int_{t-t_{\mathbf{b}}}^{t}I(t,s)Q_{\text{gain }%
}(F^{m},F^{m})(X(s),V(s)))ds\}
\end{eqnarray*}%
so that $F^{m+1}\geq 0$ on $[0,T_{0}].$ This implies that $F\geq 0$ in the
limit with $h=\frac{F-\mu }{\sqrt{\mu }}\in L^{\infty }.$ By the uniqueness
of our solutions with this class, $F$ is the same solution we constructed
earlier. We obtain $F\geq 0$ for all time by repeating $%
[0,T_{0}],[T_{0},2T_{0}]...[kT_{0},(k+1)T_{0}]$, from the uniform bound of $%
\sup_{t}||h(t)||_{\infty }.$
\end{proof}

\begin{proof}
(of Theorem \ref{bouncebacknl} and Theorem \ref{specularnl}): We use the
same iteration (\ref{interation1}) with either bounce-back or specular
reflection for $h^{m+1}.$ By the Duhamel Principle, 
\begin{equation}
h^{m+1}=U(t)h_{0}+\int_{0}^{t}U(t-s)w\Gamma (\frac{h^{m}}{w},\frac{h^{m}}{w}%
)(s)ds.  \label{hmduhamel}
\end{equation}%
Applying either Theorems \ref{bouncebackrate} or \ref{specularrate}
respectively, by Lemma \ref{collision}, we have 
\begin{equation*}
||h^{m+1}(t)||_{\infty }\leq Ce^{-\lambda t}||h_{0}||_{\infty
}+||\int_{0}^{t}U(t-s)w\Gamma (\frac{h^{m}}{w},\frac{h^{m}}{w}%
)(s)ds||_{\infty }.
\end{equation*}%
To avoid the extra weight function $\nu (v)\thicksim \{1+|v|\}^{\gamma }$ in
Lemma \ref{collision}, we use $U(t-s)=G(t-s)+%
\int_{s}^{t}G(t-s_{1})K_{w}U(s_{1}-s)ds_{1}$to estimate the second term
above. For any $(t,x,v)\notin \gamma _{0},$ and its back-time cycle $(X_{%
\mathbf{cl}}(s),V_{\mathbf{cl}}(s)),$ we use (\ref{bouncebackformular}) and (%
\ref{specularformula}) respectively to get 
\begin{eqnarray*}
&&|\int_{0}^{t}G(t-s)\{w\Gamma (\frac{h^{m}}{w},\frac{h^{m}}{w}%
)(s)\}ds|=|\int_{0}^{t}e^{-\nu (v)(t-s)}\{w\Gamma (\frac{h^{m}}{w},\frac{%
h^{m}}{w})(s,X(s),V(s))\}ds| \\
&\leq &C|\int_{0}^{t}e^{-\nu (v)(t-s)}\nu (v)||h^{m}(s)||_{\infty
}^{2}ds\leq Ce^{-\frac{\nu _{0}}{2}t}\times \{\sup_{0\leq s\leq \infty }e^{%
\frac{\nu _{0}}{2}s}||h^{m}(s)||_{\infty }\}^{2}
\end{eqnarray*}%
where we have used Lemma \ref{collision} and (\ref{g0}).

On the other hand, for the second term, we use $\tilde{U}$ as in (\ref%
{utilde}) with either the bounce-back or specular reflection. Hence (\ref%
{equalutilde}) still holds. Since $w^{-2}(1+|v|)^{3}\in L^{1},$ Theorems \ref%
{bouncebackrate} or \ref{specularrate} are valid for $\tilde{U},$ and we get
from (\ref{u0tilde}) 
\begin{eqnarray}
&&\int_{0}^{t}\int_{s}^{t}e^{-\nu _{0}(t-s_{1})}||K_{w}U(s_{1}-s)w\Gamma (%
\frac{h^{m}}{w},\frac{h^{m}}{w})(s)||_{\infty }dsds_{1}  \notag \\
&\leq &C\int_{0}^{t}\int_{s}^{t}e^{-\nu _{0}(t-s_{1})}\left\{ \int
K_{w}(v,v^{\prime })\sqrt{1+\rho |v^{\prime }|^{2}}dv^{\prime }\right\} 
\tilde{U}(s_{1}-s)||\{\frac{w\Gamma (\frac{h^{m}}{w},\frac{h^{m}}{w})(s)\}}{%
\sqrt{1+\rho |v^{\prime }|^{2}}}||_{\infty }dsds_{1}  \notag \\
&\leq &C\int_{0}^{t}\int_{s}^{t}e^{-\nu _{0}(t-s_{1})}\tilde{U}(s_{1}-s)||%
\frac{w\Gamma (\frac{h^{m}}{w},\frac{h^{m}}{w})(s)}{\sqrt{1+\rho |v^{\prime
}|^{2}}}||_{\infty }dsds_{1}  \notag \\
&\leq &Ce^{-\frac{\lambda }{2}t}\times \{\sup_{0\leq s\leq \infty }e^{\frac{%
\lambda }{2}s}||h^{m}(s)||_{\infty }\}^{2},  \label{2star}
\end{eqnarray}%
by (\ref{k}). This implies that $\sup_{m}\sup_{0\leq t\leq \infty }\{e^{%
\frac{\lambda }{2}t}||h^{m+1}(t)||_{\infty }\}\leq C||h_{0}||_{\infty }$ for 
$||h_{0}||$ sufficiently small. Moreover, by subtracting $h^{m+1}-h^{m},$ by
(\ref{difference}) and (\ref{hmdifference}), we deduce that $h^{m}$ is a
Cauchy sequence and the limit $h$ is the desired unique solution. If $\Omega 
$ is strictly convex, inductively, by Lemma \ref{bouncebackcon} and \ref%
{specularcon} respectively, $h^{m}$ is continuous in $[0,\infty )\times \{%
\bar{\Omega}\times \mathbf{R}^{3}\setminus \gamma _{0}\},$ and $w\Gamma (%
\frac{h^{m}}{w},\frac{h^{m}}{w})$ is continuous in $[0,\infty )\times \Omega
\times \mathbf{R}^{3}.$ Moreover, from Lemma \ref{collision}, $\sup |\frac{%
w\Gamma (\frac{h^{m}}{w},\frac{h^{m}}{w})}{\nu }|$ is also finite. We
therefore deduce that $h$ is also continuous in $[0,\infty )\times \{\bar{%
\Omega}\times \mathbf{R}^{3}\setminus \gamma _{0}\}$ from the uniform
convergence.

As for the positivity of $F,$ we follow the argument in the proof of the
inflow case by using a different iterative scheme (\ref{positive}) with
either bounce-back or specular reflection boundary conditions, so that $%
f^{m}=\frac{F^{m}-\mu }{\sqrt{\mu }}$ satisfies (\ref{fm}). Again, it
follows from a routine procedure to show that $h^{m}=wf^{m}$ is a Cauchy
sequence, local in time $[0,T_{0}]$. Assume $F^{m}\geq 0.$ For any $%
(t,x,v)\notin \gamma _{0},$ we denote its back-time cycle (either bounceback
or specular) as $[X_{\mathbf{cl}}(s),V_{\mathbf{cl}}(s)],$ and $%
t_{1},t_{2},...t_{l}>0,$ so that $t_{l+1}\leq 0.$ Recall (\ref{ifactor}), so
that $\frac{d}{ds}\{I(t,s)F^{m+1}\}=I(t,s)Q_{\text{gain }}(F^{m},F^{m})$
almost everywhere along the back-time cycles $[X(s),V(s)]$ of $t,x,v,$
inside $\bar{\Omega}$. As in Lemmas \ref{bouncebackformular} and \ref%
{specularformula}, we can express $F^{m+1}(t,x,v)$ as 
\begin{equation*}
I(t,t_{1})F^{m+1}(t_{1},x_{1},v_{1})+\int_{t_{1}}^{t}I(s,t)Q_{\text{gain }%
}(F^{m},F^{m})(X_{\mathbf{cl}}(s),V_{\mathbf{cl}}(s)ds\geq
I(t,t_{1})F^{m+1}(t_{1},x_{1},v_{1}).
\end{equation*}%
For $1\leq i\leq l,$ similarly, 
\begin{eqnarray*}
F^{m+1}(t_{i},x_{i},v_{i})
&=&I(t_{i},t_{i+1})F^{m+1}(t_{i+1},x_{i+1},v_{i+1})+%
\int_{t_{1}}^{t}I(s,t_{i})Q_{\text{gain }}(F^{m},F^{m})(X_{\mathbf{cl}%
}(s),V_{\mathbf{cl}}(s)ds, \\
F^{m+1}(t_{l},x_{l},v_{l}) &=&I(t_{l},0)F_{0}(X_{\mathbf{cl}}(0),V_{\mathbf{%
cl}}(0))+\int_{0}^{t_{l}}I(s,0)Q_{\text{gain }}(F^{m},F^{m})(X_{\mathbf{cl}%
}(s),V_{\mathbf{cl}}(s)ds\geq 0.
\end{eqnarray*}%
Hence by an induction over $i,$ $F^{m+1}\geq 0$ and we deduce $F\geq 0$ by
the uniqueness as in the proof of the in-flow case.
\end{proof}

\begin{proof}
of Theorem \ref{diffusenl}: We use the same iteration (\ref{interation1})
together with the diffusive boundary condition (\ref{hdiffuse}). By (\ref%
{hmduhamel}), we further bound its last term by the Duhamal principle as
before:%
\begin{equation}
||\int_{0}^{t}G(t-s)w\Gamma (\frac{h^{m}}{w},\frac{h^{m}}{w})(s)ds||_{\infty
}+||\int_{0}^{t}\int_{s}^{t}G(t-s_{1})K_{w}U(s_{1}-s)w\Gamma (\frac{h^{m}}{w}%
,\frac{h^{m}}{w})(s)dsds_{1}||_{\infty }.  \notag
\end{equation}%
We estimate the first term above in two parts: 
\begin{equation*}
|\int_{0}^{t}G(t-s)\{w\Gamma (\frac{h^{m}}{w},\frac{h^{m}}{w})(s)\}ds|\leq
|\int_{t-1}^{t}|+|\int_{0}^{t-1}|
\end{equation*}%
For $\int_{t-1}^{t},$ we use estimate (\ref{diffuserate}) and Lemma \ref%
{collision} to get 
\begin{eqnarray*}
&&|\int_{t-1}^{t}e^{-\nu _{0}(t-s)}\mathbf{1}_{\{t_{1}\leq s)}\{\frac{%
w\Gamma (\frac{h^{m}}{w},\frac{h^{m}}{w})(s)}{\tilde{w}}\}ds|+|%
\int_{t-1}^{t}e^{-\nu (v)(t-s)}\mathbf{1}_{\{t_{1}\leq s)}\{w\Gamma (\frac{%
h^{m}}{w},\frac{h^{m}}{w})(s)\}ds| \\
&\leq &C\sup |\int_{t-1}^{t}e^{-\nu (v)(t-s)}\nu (v)ds|\times
||h^{m}||_{\infty }^{2}+Ce^{-\frac{\nu _{0}}{2}t}\sup \{e^{\frac{\nu _{0}}{2}%
s}||h(s)||_{\infty }^{{}}\}^{2} \\
&\leq &Ce^{-\frac{\nu _{0}}{2}t}\sup \{e^{\frac{\nu _{0}}{2}%
s}||h(s)||_{\infty }^{{}}\}^{2}.
\end{eqnarray*}%
For $\int_{0}^{t-1},$ we use (\ref{t>1}) to get extra decay in $v$ as: 
\begin{equation*}
\int_{0}^{t-1}e^{-\frac{\nu _{0}}{2}(t-s)}\{||\frac{w\Gamma (\frac{h^{m}}{w},%
\frac{h^{m}}{w})(s)}{\tilde{w}}||_{\infty }+||e^{-\nu (v)}w\Gamma (\frac{%
h^{m}}{w},\frac{h^{m}}{w})(s)||_{\infty }\}ds\leq Ce^{-\frac{\nu _{0}}{2}%
t}\sup \{e^{\frac{\nu _{0}}{2}s}||h(s)||_{\infty }^{{}}\}^{2}.
\end{equation*}%
We have used the fact $\frac{1}{\tilde{w}}$ decays exponentially by (\ref%
{w>1}).

On the other hand, for the second term, we define $\tilde{U}$ ~again in (\ref%
{utilde}) with new weight $w_{1}=\frac{w}{\{1+\rho |v|^{2}\}^{\frac{1}{2}}}$
and the new diffuse boundary condition as 
\begin{equation*}
\{\tilde{U}\tilde{h}_{0}\}(t,x,v)|_{\gamma _{-}}=\frac{1}{\tilde{w}_{1}}%
\int_{v^{\prime }\cdot n(x)>0}\{\tilde{U}\tilde{h}_{0}\}(t,x,v^{\prime })%
\tilde{w}_{1}(v^{\prime })d\sigma .
\end{equation*}%
Hence (\ref{equalutilde}) still holds. By letting $\rho $ further
sufficienlty small in (\ref{wdiffuse}), we can apply Theorem \ref%
{duffusiverate} for the diffusive $\tilde{U}$ exactly as in (\ref{2star}).
Hence, $\sup_{m,0\leq t\leq \infty }\{e^{\lambda t}||h^{m+1}(t)||_{\infty
}\}\leq C||h_{0}||_{\infty }$ for $||h_{0}||$ sufficiently small. As before,
we can construct the desired unique solution and establish continuity when $%
\Omega $ is convex.

As for the positivity of $F,$ we follow use the different iterative scheme (%
\ref{positive}) with diffuse boundary condition%
\begin{equation*}
F^{m+1}(t,x,v)|_{\gamma _{-}}=c_{\mu }\mu (v)\int F^{m+1}(t,x,v^{\prime
})\{n\cdot v^{\prime }\}dv^{\prime },
\end{equation*}%
and $f^{m}=\frac{F^{m}-\mu }{\sqrt{\mu }}$ satisfies (\ref{fm}). Again, it
follows from a routine procedure that $h^{m}=wf^{m}$ is a Cauchy sequence,
local in time $[0,T_{0}]$. Assume $F^{m}\geq 0$ and recall (\ref{ifactor})
so that $\frac{d}{ds}\{I(t,s)F^{m+1}\}=I(t,s)Q_{\text{gain }}(F^{m},F^{m})$
almost everywhere along the back-time trajectory $[X(s),V(s)]$ of $t,x,v,$
inside $\bar{\Omega}$. Recall (\ref{prob}). For any $(t,x,v),$ consider its
back-time diffusive cycle $(X_{\mathbf{cl}}(s),V_{\mathbf{cl}}(s)).$ By a
similar derivation as in (\ref{diffuseformula}), if $t_{1}(t,x,v)\leq 0,$ 
\begin{equation*}
F^{m+1}(t,x,v)=I(t,0)F_{0}(x-vt,v)+\int_{0}^{t}I(s,0)Q_{\text{gain }%
}(F^{m},F^{m})(X_{\mathbf{cl}}(s),V_{\mathbf{cl}}(s))ds\geq 0.
\end{equation*}%
If $t_{1}(t,x,v)>0,$ then for $k\geq 2,$ 
\begin{equation*}
F^{m+1}(t,x,v)=\int_{t_{1}}^{t}I(s,t_{1})Q_{\text{gain }}(F^{m},F^{m})(X_{%
\mathbf{cl}}(s),V_{\mathbf{cl}}(s))ds+I(t,t_{1})c_{\mu }\mu (v)\int_{\Pi
_{j=1}^{k-1}\mathcal{V}_{j}}H^{m}
\end{equation*}%
where $H^{m}$ is given by 
\begin{eqnarray*}
&&\sum_{l=1}^{k-1}\mathbf{1}_{\{t_{l}>0,t_{l+1}\leq 0\}}F_{0}(X_{\mathbf{cl}%
}(0),V_{\mathbf{cl}}(0))d\Sigma _{l}^{m}(0) \\
&&+\sum_{l=1}^{k-1}\int_{t_{l+1}}^{t_{l}}\mathbf{1}_{\{t_{l+1}>0\}}Q_{\text{%
gain }}(F^{m},F^{m})(X_{\mathbf{cl}}(s),V_{\mathbf{cl}}(s))d\Sigma
_{l}^{m}(s)ds \\
&&+\sum_{l=1}^{k-1}\int_{s}^{t_{l}}\mathbf{1}_{\{t_{l}>0,t_{l+1}\leq 0\}}Q_{%
\text{gain }}(F^{m},F^{m})(X_{\mathbf{cl}}(s),V_{\mathbf{cl}}(s))d\Sigma
_{l}^{m}(s)ds \\
&&+\mathbf{1}_{\{t_{k}>0\}}F^{m+1}(t_{k},x_{k},v_{k-1})d\Sigma
_{k-1}^{m}(t_{k}),
\end{eqnarray*}%
where $d\Sigma _{l}(s)=\{\Pi _{j=l+1}^{k-1}d\sigma
_{j}\}\{I(s,t_{l})[n(x_{l})\cdot v_{l}]dv_{l}\}\Pi
_{j=1}^{l}\{I(t_{j},t_{j+1})d\sigma _{j}\}$. For any $\varepsilon >0,$ by
Lemma \ref{small}, $\int_{\Pi _{j=1}^{k-2}\mathcal{V}_{j}}\mathbf{1}%
_{\{t_{k-1}>0\}}\Pi _{j=1}^{k-2}d\sigma _{j}<\varepsilon $ for $k$ large.
Notice that $\{t_{k}>0\}\subset \{t_{k-1}>0\},$ and by (\ref{ifactor}), $%
I(s,t_{l})\leq C(\sup_{m,0\leq s\leq T_{0}}||F^{m}(s)||_{\infty },T_{0}).$
Since $F_{0}\geq 0$ and $Q_{\text{gain }}(F^{m},F^{m})\geq 0,$ we conclude $%
F^{m+1}(t,x,v)\geq $ 
\begin{eqnarray*}
&&C(\sup_{m,0\leq s\leq T_{0}}||F^{m}(s)||_{\infty },T_{0})\int_{\Pi
_{j=1}^{k-1}\mathcal{V}_{j}}\mathbf{1}_{\{t_{k}>0%
\}}F^{m+1}(t_{k},x_{k},v_{k})dv_{k-1}\Pi _{j=1}^{k-2}d\sigma _{j} \\
&\geq &-C(\sup_{m,0\leq s\leq T_{0}}||h^{m}(s)||_{\infty },T_{0})\int_{\Pi
_{j=1}^{k-2}\mathcal{V}_{j}}\mathbf{1}_{\{t_{k-1}>0\}}\left\{ \int_{\mathcal{%
V}_{k-1}}\{\mu +\sqrt{\mu }f^{m+1}\}(t_{k},x_{k},v_{k})dv_{k-1}\right\} \Pi
_{j=1}^{k-2}d\sigma _{j} \\
&\geq &-C(\sup_{m,0\leq s\leq T_{0}}||h^{m}(s)||_{\infty },T_{0})\varepsilon
.
\end{eqnarray*}%
Since $\varepsilon $ is arbitrary, we deduce that $F^{m+1}\geq 0$ and this
conclude the positivity of $F$ over $[0,T_{0}]$. We then conclude $F\geq 0$
by the uniqueness.
\end{proof}

\textbf{Acknowledgements. }We thank H-J. Hwang, C. Kim, N. Masmoudi and R.
Marra for stimulating discussions. We especially thank R. Esposito for
pointing out a simplication of the contradiction argument in $L^{2}$ decay
theory, the role of rotational symmetry, and the recent interesting
development for diffusive reflection in [AEMN].

\section{References}

[A] Arkeryd, L, On the strong $L^{1}$ trend to equilibrium for the Boltzmann
equation'. Studies in Appl. Math. (1992) 87, 283-288.

[AC] Arkeryd, L; Cercignani, C, A global existence theorem for the initial
boundary value problem for the Boltzmann equation when the boundaries are
not isothermal. Arch. Rational Mech. Anal. (1993) 125, 271-288.

[AEMN] Arkeryd, L., Esposito, R., Marra R. and Nouri, A., Stability of the
laminar solution of the Boltzmann equation for the Benard problem. Preprint
2007.

[AEP] Arkeryd, L.; Esposito, R.; Pulvirenti, M.: The Boltzmann equation for
weakly inhomogeneous data. Comm. Math. Phys. 111 (1987), no. 3, 393--407.

[ADVW] Alexandre, R., Desvillettes, L., Villani, C. and Wennberg, B.,
Entropy dissipation and long-range interactions. Arch. Ration. Mech. Anal.
152 (2000), 4, 327-355.

[AH] Arkeryd, L.; Heintz, A.: On the solvability and asymptotics of the
Boltzmann equation in irregular domains. Comm. Partial Differential
Equations 22 (1997), no. 11-12, 2129--2152.

[BP] Beals, R.; Protopopescu, V: Abstract time-depedendent transport
equations. J. Math. Anal. Appl. \ (1987) 212, 370-405.

[C1] Cercignani, C.: The Boltzmann Equation and Its Application,
Springer-Verlag, 1988.

[C2] Cercignani, C., Equilibrium States and the trend to equlibrium in a gas
according to the Boltzmann equation. Rend. Mat. Appl. (1990) 10, 77-95.

[C3] Cercignani, C., On the initial-boundary value problem for the Boltzmann
equation. Arch. Rational Mech. Anal. (1992) 116, 307-315.

[CC] Cannoe, R.; Cercignani, C, A trace theorem in kinetic theory. Appl.
Math. Letters. (1991) 4, 63-67.

[CIP] Cercignani, C., Illner, R. and Pulvirenti, M., The Mathematical Theory
of Dilute Gases, Springer-Verlag, 1994.

[D] Deimling, K., Nonlinear Functional Analysis. Springer-Verlag, 1988.

[De] Desvillettes, L: Convergence to equilibrium in large time for Boltzmann
and BGK equations. Arch. Rational Mech. Anal. (1990) 110, 73-91.

[DeV] Desvillettes, L.; Villani, C., On the trend to global equilibrium for
spatially inhomogeneous kinetic systems: the Boltzmann equation. Invent.
Math. 159 (2005), no. 2, 245--316.

[DL2] Diperna, R., Lions, P-L., On the Cauchy problem for the Boltzmann
equation. Ann. Math., (1989) 130, 321-366.

[DL2] Diperna, R., Lions, P-L., Global weak solution of Vlasov-Maxwell
systems, Comm. Pure Appl. Math., 42, 729-757 (1989).

[EGM] Esposito, L., Guo, Y., Marra, R., The Vlasov-Boltzamann system for
phase transition. In preparation.

[G1] Guo, Y., The Vlasov-Poisson-Boltzmann system near Maxwellians. Comm.
Pure Appl. Math. 55 (2002) no. 9, 1104-1135.

[G2] Guo, Y., The Vlasov-Maxwell-Boltzmann system near Maxwellians. Invent.
Math. 153 (2003), no. 3, 593--630.

[G3] Guo, Y., Singular solutions of the Vlasov-Maxwell system on a half
line. Arch. Rational Mech. Anal. 131 (1995), no. 3, 241--304.

[G4] Guo, Y., Regularity for the Vlasov equations in a half-space. Indiana
Univ. Math. J. 43 (1994), no. 1, 255--320.

[GL] Glassey, R., The Cauchy Problems in Kinetic Theory, SIAM, 1996.

[GS] Guo, Y., Strain, R., The relativistic Maxwell-Boltzmann system. In
preparation.

[Gr1] Grad, H.: Principles of the kinetic theory of gases. Handbuch der
Physik, XII, 205-294 (1958).

[Gr2] Grad, H.: Asymptotic theory of the Boltzmann equation. II. Rarefied
gas dynamics, 3rd Symposium, 26-59, Paris, 1962.

[Gui] Guiraud, J. P. : An $H-$theorem for a gas of rigid spheres in a
bounded domain. Theories cinetique classique et relativistes, (1975) G.
Pichon, ed., 29-58, CNRS, Paris.

[H] Hamdache, K: Initial boundary value problems for Boltzmann equation.
Global existence of week solutions. Arch. Rational Mech. Anal. (1992) 119,
309-353.

[HH] Hwang, H-J., Regularity for the Vlasov-Poisson system in a convex
domain. SIAM J. Math. Anal. 36 (2004), no. 1, 121--171.

[HV] Hwang, H-J., Velazquez, J., Global existence for the Vlasov-Poisson
system in bounded domain. Preprint 2007.

[LY] Liu, T-P.; Yu, S-H., Initial-boundary value problem for one-dimensional
wave solutions of the Boltzmann equation. Comm. Pure Appl. Math. 60 (2007),
no. 3, 295--356.

[M] Maslova, N. B. Nonlinear Evolution Equations, Kinetic Approach.
Singapore, World Scientific Publishing, 1993.

[Mi] Mischler, S., On the initial boundary value problem for the
Vlasov-Poisson-Boltzmann system. Commun. Math. Phys. 210 (2000) 447-466.

[MS] Masmoudi, N.; Saint-Raymond, L., From the Boltzmann equation to the
Stokes-Fourier system in a bounded domain. Comm. Pure Appl. Math. 56 (2003),
no. 9, 1263--1293.

[S] Shizuta, Y., On the classical solutions of the Boltzmann equation, Comm.
Pure Appl. Math., 36 (1983) 705-754.

[SA] Shizuta, Y., Asano, K., Global solutions of the Boltzmann equation in a
bounded convex domain. Proc. Japan Acad. (1977) 53A, 3-5.

[SG] Strain, R. Guo, Y., Exponential decay for soft potentials near \
Maxwellians. Arch. Rational Mech. Anal., in press, 2007.

[U1] Ukai, S., Solutions of the Boltzmann equation. Pattern and
waves-Qualitative Analysis of Nonlinear Differential Equations, (1986) 37-96.

[U2] Ukai, S., private communications.

[US] Ukai, S., Asano, K., On the initial boundary value problem of the
linearized Boltzmann equation in an exterior domain. Proc. Japan Acad.
(1980) 56, 12-17.

[V] Villani, C., Hypocoercivity. Memoir of AMS, to appear.

[Vi] Vidav, I., Spectra of perturbed semigroups with applications to
transport theory. J. Math. Anal. Appl., 30 (1970) 264-279.

[YZ] Yang, T.; Zhao, H-J., A half-space problem for the Boltzmann equation
with specular reflection boundary condition. Comm. Math. Phys. 255 (2005),
no. 3, 683--726.

\end{document}